\documentclass[11pt, reqno,a4paper]{report}

\usepackage[utf8]{inputenc}

\usepackage{bbm}

\usepackage[shortlabels]{enumitem}

\usepackage{hyperref}

\usepackage{graphicx}

\usepackage{epigraph}

\usepackage{eurosym}

\newenvironment{dedication}
  {\clearpage           
   \thispagestyle{empty}
   \vspace*{\stretch{1}}
   \itshape             
   \raggedleft          
  }
  {\par 
   \vspace{\stretch{3}} 
   \clearpage           
  }

\usepackage{amsmath,amsfonts,amssymb,amsthm, tikz}

\usepackage{mathabx}





\def\R{\mathbb{R}}

\def\N{\mathbb{N}}
\def\Z{\mathbb{Z}}

\def\Q{\mathbb{Q}}
\def\T{\mathbb{T}}

\newcommand{\G}{\mathcal{G}}
\newcommand{\PG}{\mathcal{PG}}

\def \ml {\begin{pmatrix}} 
\def \mr {\end{pmatrix}}

\def \L {{\mathcal L}}

\def \F {{\mathcal F}}
\def \r {{\mathbb R}}

\def \a {{\mathbb A}}
\def \s {{\widetilde S}}
\def \f {{\tilde f}}
\def\ww{\widetilde }
\def\wh{\widehat}

\def \feq#1 {\begin{center}\framebox{$\displaystyle #1$}\end{center}}

\def \dd {{\rm d}}

\def \T{\mathbb{T}}

\def \e{\varepsilon}

\def \D{{\mathcal D}}
\def \S{{\mathcal S}}
\def \BB{{\mathcal B}}
\def \FF{{\mathcal F}}

\def \LL{{\mathcal L}}
\def \PP{{\mathcal P}}
\def \AA{{\mathcal A}}
\def \MM{{\mathcal M}}
\def \PP{{\mathcal P}}

\def \SS{{\mathcal S}}
\def \PP{{\mathcal P}}

\newcommand{\AAA}{\mathfrak A}

\newcommand{\SSS}{\mathcal{S}' }
\newcommand{\A}{\mathbb A}
\newcommand{\DD}{\mathcal{D}}
\def \BB{{\mathcal B}}

\def \CC{{\mathcal C}}

\def \TT{{\mathcal T}}

\newcommand{\om}{\omega}

\newcommand{\cc}{{c[0]}}
\newcommand{\1}{\mathbbm 1}
\newcommand{\Tn}{T^{-n}}
\newcommand{\TTn}{T^{+n}}
\newcommand{\TTT}{\mathfrak T}

\newcommand*\tcircle[1]{%
  \raisebox{-0.5pt}{%
    \textcircled{\fontsize{7pt}{0}\fontfamily{phv}\selectfont #1}%
  }%
}

\reversemarginpar

\author{M. ZAVIDOVIQUE }

\newtheorem{df}{Definition}[section]

\newtheorem{lm}[df]{Lemma}
\newtheorem{pr}[df]{Proposition}
\newtheorem{Th}[df]{Theorem}
\newtheorem{co}[df]{Corollary}
\newtheorem{rem}[df]{Remark}
\newtheorem{ex}[df]{Example}

\title{Discrete Weak KAM Theory: \\an introduction through examples and its applications to twist maps}

\begin{document}
\begin{titlepage}

\begin{center}
%
%
%
%
%
%
\vspace{1cm}
{\Large \textbf{Maxime Zavidovique}}

%
%
\vspace{1cm}
{\Large \textbf{Discrete and Continuous Weak KAM Theory: \\an introduction through examples \\ and its  applications to twist maps.}}

%
%
%
%
%
%
%
%
%

\end{center}
\end{titlepage}
\thispagestyle{empty}
\newpage
\thispagestyle{empty}
\mbox{}
\newpage


\begin{dedication}
\`A mes trois rayons de soleil, \\
du matin, du midi et du soir.
\end{dedication}

\newpage
\thispagestyle{empty}
\mbox{}
\newpage

\begin{abstract}
The aim of these notes is to present a self contained account of discrete weak KAM theory.  Put aside the intrinsic elegance of this theory, it is also a toy model for classical weak KAM theory, where many technical difficulties disappear, but where central ideas and results persist. It can therefore serve as a good introduction to (continuous) weak KAM theory. After a general exposition of the general abstract theory, several examples are studied. The last section is devoted to the historical problem of conservative twist maps of the annulus. At the end of the  first three Chapters, the relations between the results proved in the discrete setting and the analogous theorems of classical weak KAM theory are discussed. Some key differences are also highlighted between the discrete and classical theory.
\end{abstract}

\newpage
\thispagestyle{empty}
\mbox{}
\newpage

\setcounter{page}{1}
\chapter*{Remerciements}
\vspace{-1cm}
Je voudrais avant tout remercier Thétis. C'est probablement grâce à sa négligence que j'ai eu le temps de me remettre à la rédaction de ce mémoire, confortablement immobilisé dans mon canapé.

Un grand merci également à Symplectix et à ses organisateurs de m'accueillir subrepticement le temps d'une soutenance.

As far as real living persons are concerned, my gratitude goes to Alessio Figalli, Hitoshi Ishii and Claude Viterbo for sacrificing part of their summer time to the task of reading my prose, thank you.

Merci à Alessio, Claude, Marie-Claude Arnaud, Pierre Cardaliaguet, Albert Fathi, Hélène Frankowska, Vadim Kaloshin et   Ludovic Rifford d'avoir accepté d'être dans le jury de cette HDR. Votre présence m'honore.

Je suis à jamais redevable à Albert de m'avoir introduit dans le monde KAM faible, il y a maintenant quinze ans. C'est un monde d'idées simples et profondes, un monde parfois technique mais rarement trop. L'explorer est une jubilation.

Questo testo è nato dopo un mini corso dato a Cortona en 2015. Inoltre di mangiare tartufo con Montelpuciano, mi ha fatto pensare a come organizzare la teoria di maniera pedagogica. Grazie agli organizzatori Andrea, Antonio e Alfonso per questa stupende conferenza.

Ce texte, et une partie de son contenu, n'aurait pas été écrit sans mes coauteurs, en particulier, Andrea, Albert, Antonio et bien sûr Marie-Claude. Apprendre grâce à vous et découvrir ensemble de nouveaux théorèmes suscite chez moi émerveillement et plaisir. Les mathématiques sont un sport d'équipe !

Parlant d'équipe, merci à Nicolas Bergeron, Alexandru Oancea et Patrice Le Calvez pour votre accompagnement et soutien à Jussieu.  Alex, merci aussi et surtout pour tes précieux conseils et la relecture de mon texte. \' Evoluer dans un environment professionnel chaleureux et apaisé est un luxe\footnote{que SU facture 380\euro} dont j'ai la chance de bénéficier. C'est grâce à tous mes collègues de AA. Une spéciale dédicace à Antonin et Pierre, merci pour votre aide précieuse avec les figures ! Je n'oublie pas les anciens collègues, Hélène, Erwan  et surtout Grégory pour la formation continue en chocolat et pâtisserie, et pour tout le reste. 

Cette soutenance n'aurait pas existé (et moi non plus) sans mes parents. Merci pour vos encouragements pendant la longue phase de rédaction. Ils furent réguliers, soutenus mais toujours avec une pudique retenue. Jamais plus de trois fois par jour.

Mes frères et sœurs m'accompagnent tout le temps, je pense à eux.  Tout particulièrement à Laurent, merci pour le grand flacon.

Enfin, et surtout, Anne et les enfants illuminent mon quotidien. Si les enfants sont mon soleil, Anne tu es mon \' Etoile polaire.

\newpage

\mbox{}
\newpage
\tableofcontents

\newpage

\mbox{}
\newpage

\chapter*{Introduction}
\addcontentsline{toc}{chapter}{\protect\numberline{}Introduction}

The present text initially emerged from lecture notes of a course given in Cortona in 2015 at the INDAM meeting entitled The Hamilton--Jacobi Equation:
At the crossroads of PDE, dynamical systems \& geometry. The goal of the lectures was to give a complete and thorough introduction to weak KAM theory through the prism of its discrete pendant. It culminated with the proof of convergence of the solutions of the discounted equation, which was new at the time. 

The pedagogical motivation is that discrete weak KAM theory is peculiarly elementary. Basic topology is the only prerequisite and the most advanced tools are the Arzel\`a--Ascoli Theorem and weak compactness of probability measures on a  compact metric space. However, all important results of weak KAM theory find their analogue in discrete weak KAM theory and the proofs being stripped of technicalities reveal the key ideas more clearly. After studying this toy model, the interested reader can then go on to learn more on  the major theories, as Calculus of Variations, Viscosity solutions of Hamilton--Jacobi equations, and Tonelli Hamiltonian Dynamical Systems. The latter are at the core of classical weak KAM theory.

We wish to start by explaining in which context weak KAM theory emerged in the 90's. Most notions described will be rigorously defined later in the body of the text. 

\section*{Weak KAM, a bridge between Aubry--Mather and Viscosity solutions}
\addcontentsline{toc}{section}{\protect\numberline{}Weak KAM, a bridge between Aubry--Mather and viscosity solutions}

Originally discovered by Albert Fathi in the 90's \cite{Fa4,Fa3,Fa5,Fa1}, weak KAM theory was designed to understand the dynamical objects of Aubry--Mather Theory for Tonelli Hamiltonian/Lagrangian systems through particular functions called {\it critical subsolutions} and {\it weak KAM solutions}. As explained by Fathi, the term weak KAM was chosen as KAM tori give rise to strong solutions of the Hamilton--Jacobi equation while weak KAM solutions are weak solutions of the  same equation. It turned out that weak KAM solutions and subsolutions fall in the realm of Viscosity Solutions and subsolutions of Hamilton--Jacobi equations, a theory founded by Crandall and Lions \cite{CrLi}.

\subsection*{Aubry--Mather theory}
\addcontentsline{toc}{subsection}{\protect\numberline{}Aubry--Mather theory}

This theory originated in the study of conservative twist maps and Frenkel--Kontorova models.  The objective was to understand invariant sets and the structure of minimizing orbits of such systems. Another type of problem was to construct  diffusion orbits connecting invariant sets, or to understand obstructions for such orbits to exist. Aubry \& Le Daeron \cite{Aubry} and Mather \cite{Mather} understood that orbits and invariant sets minimizing a certain energy verify similar properties as orbits of homeomorphisms of the circle. Hence they could apply Poincaré-Denjoy theory to classify and understand such invariant sets according to their rotation number. They went on to study the minimal average action of such minimizing orbits (a function only depending on the rotation number), invariant minimal measures (that would become Mather measures) and Mather gave several definitive answers to questions on the existence of connecting orbits \cite{MatherP, MatherMeasure,Mather43,MaDiff,MJAMS}. Amongst other important contributions let us mention Bangert \cite{Bang,BangR}.

Motivated by Moser, Mather developed a generalization of his theory to higher dimensional settings by introducing Minimizing Measures for a Tonelli Lagrangian defined on the tangent bundle of a compact manifold \cite{Mather1}. His next goal would then be to use this tool to tackle Arnol'd diffusion of such systems, a phenomenon highlighted by Arnol'd in his famous examples \cite{Arnold}. Let us present the philosophy. In an integrable system, all orbits are bounded and periodic or quasi--periodic. If one perturbs such a system, KAM theory implies that many quasi--periodic orbits persist. However, Arnol'd constructed examples where for small perturbations, some orbits have a huge drift in energy. He then conjectured that such a phenomenon should be typical. Mather led the way proposing groundbreaking variational mechanisms to construct generic diffusion (\cite{MatherF,  MaDif1,MaDif2}). Since then, there has been a huge literature trying to carry on Mather's program for diffusion.

\subsection*{Viscosity Solutions}
\addcontentsline{toc}{subsection}{\protect\numberline{}Viscosity Solutions}

As already mentioned, viscosity solutions were introduced by Crandall and Lions. They provide a simple, robust definition for weak solutions of first and second order PDE's. Strong existence and uniqueness results are obtained for wide classes of equations, including Hamilton--Jacobi equations (stationary and   evolutionary) making the solutions worthwhile studying. Let us mention amongst many others the founding works of Ishii, Crandall, Lions \cite{IsDUKE,CrIsLi,IsLi,IsCPAM}. The definition, that makes use of test functions that are super--tangent or sub--tangent to the solution  is very geometric, and allows easily to obtain stability results for viscosity solutions. References to learn more about basic (and more advanced) properties are \cite{Li,CrIsLiUG,barles}. For example, viscosity solution theory is so flexible as to apply to non--continuous functions. Stability allows to obtain very general convergence theorems of approximation schemes \cite{sou, basou}. Other references making use of this idea, in more weak KAM or variational contexts, are \cite{BFEZ,Ztwisted,Roos,Wei}.

Very soon in the development of the theory of Viscosity Solutions, Lions and collaborators realized the important role of dynamical programming properties for evolutionary equations and the links with Optimal Control theory that naturally emerge. Indeed, the Value function in optimal control is almost systematically a viscosity solution to some Hamilton--Jacobi equation \cite{BaCa}. This is fundamental in weak KAM theory as the Lax--Oleinik semigroup rediscovered by Fathi turns out to be the value function of an optimal control problem.

Another important problem that was solved early on by viscosity solution methods is that of Homogenization of  the Hamilton--Jacobi equation. Consider a continuous,  $\mathbb Z^N$--periodic in the first variable Hamiltonian $H : \R^N \times \R^N $ that is uniformly coercive in the second variable. Fix a bounded and uniformly continuous initial data $u_0 : \R^N \to \R$. Given $\varepsilon >0$, Lions, Papanicolaou and Varadhan consider the evolutionary Hamilton--Jacobi equation
\begin{equation}\label{evohom}
\begin{cases}
\partial_t U + H(\frac{x}{\varepsilon},\partial_x U) = 0 ,\\
U(0,\cdot) = u_0,\tag{EHJ$\varepsilon$}
\end{cases}
\end{equation}
that admits a unique viscosity solution $U_\varepsilon : [0,+\infty) \times \R^N\to \R$. They prove that as $\varepsilon \to 0$, the functions $U_\varepsilon$ converge locally uniformly to a function $U_0$ that is characterized as the unique solution to 

\begin{equation}\label{evohomo}
\begin{cases}
\partial_t U + \overline H(\partial_x U) = 0 ,\\
U(0,\cdot) = u_0,\tag{E$\overline{H}$J}
\end{cases}
\end{equation}
where $\overline H$ is called the {\it  effective Hamiltonian}. It is defined as follows: for $P\in \R^N$, $\overline H(P)$ is the only constant such that the cell problem
$H(x,P+\partial_x u) =  \overline H(P)$ admits a $\mathbb Z^N$--periodic viscosity solution $u : \R^N \to \R$. In Homogenization theory, the state variable $x\in \R^N$ takes values in the universal cover of the flat torus $\T^N$. A fundamental domain of the covering map is then the cell $[0,1)^N$ in the sense that the knowledge of a $\Z^N$--periodic function on $\R^N$ is equivalent to its restriction to the cell  $[0,1)^N$. The name cell problem comes from the fact that the unknown function is defined on a cell.

This effective Hamiltonian (in the case of convex Hamiltonians) coincides with the minimal average action, Mather's $\alpha$ function. The solutions to the cell problem in our terminology will be weak KAM solutions. This problem was revisited by Evans in \cite{Ev} where he introduced the perturbed test function method. Much later, it was studied under the light of symplectic topology by Viterbo and Montzner, Vichery, Zapolsky \cite{VitHom,MVZ}. Amongst many other follow ups in the spirit of Lions Papanicolaou and Varadhan, let us mention recent generalizations to other manifolds \cite{CIS,SHom}. 

\subsection*{The bridge}
\addcontentsline{toc}{subsection}{\protect\numberline{}The bridge}

As all good bridges, weak KAM theory quickly helped the development of both banks it joins. From a dynamical point of view, Fathi's first achievement was to construct connecting orbits through conjugate pairs of weak KAM solutions \cite{Fa3,Fa5}.

 From the PDE side, he proved long time convergence of solutions to the evolutionary Hamilon--Jacobi equation on compact manifolds \cite{Fa1}, for autonomous Hamiltonians. Though partial results  had been obtained by PDE means \cite{NR2}, the new idea he imported from the dynamical world is that long minimizing trajectories tend to accumulate on the support of minimizing Mather measures. This was followed by many generalizations, for instance \cite{DS,BS,Ishii08} where variational and PDE methods allow to weaken regularity hypotheses that Dynamical methods require. In a sense, this culminates in works (the first one being \cite{CGMT}) which use an idea of Evans \cite{evadj} where Mather measures are given a PDE definition and henceforth adapted to more general settings. Other related results are :
 \begin{itemize}
 \item for counterexamples in the  non--autonomous setting  \cite{FaMat,BS2};
 \item  for positive results in dimension $1$ \cite{BerAIF,BR}.
 \end{itemize}
 
 Another question raised in the theory of Hamilton--Jacobi equations was that of regularity of critical subsolutions. For Tonelli Hamiltonians, this was settled by Fathi and Siconolfi in \cite{FSC1} where existence of $C^1$ subsolutions is proved. Along the way, the authors also gave a simple proof of Ma\~n\' e's characterizations of Mather measures, as closed minimizing measures (\cite{MClosed,Mane}). This is the point of view used in this work. In \cite{FS05} the same authors extended their results to Lipschitz Hamiltonians. Finally, in \cite{BernardC11}, existence of $C^{1,1}$ subsolutions is obtained by some Lasry--Lions type approximation (\cite{LL,BeIl}). This could be expected as Fathi had also proved that $C^1$ solutions are automatically $C^{1,1}$ (\cite{Fa2}).
 
 The latest breakthrough of weak KAM theory in the PDE theory of Hamilton--Jacobi equations is probably the proof of convergence of the solutions of the discounted equations \cite{DFIZ2}, following some special cases in \cite{IS}. The discounting method was an approximation procedure used already in \cite{LPV} to prove existence of solutions to the cell problem. It has the advantage to approximate it by equations verifying a strong comparison principle and having exactly one solution. The convergence result of Davini, Fathi, Iturriaga and the author is that, as the perturbation goes to $0$, a particular weak KAM solution is selected when the Hamiltonian is convex in the second variable. Prior conditions on selected limiting solutions had been found by Gomes in \cite{Gom}. The result strikes by its generality (little regularity is assumed on the Hamiltonian, no strict convexity) and by the flexibility of its proof. It was naturally  followed by numerous generalizations and adaptations. Among them let us mention:
 \begin{itemize}
  \item the discrete setting that is presented later in this text \cite{DFIZ2}, 
  \item adaptations to Neumann problems \cite{IsNe},
  \item  non--compact settings \cite{IsSi}, 
  \item more abstract duality methods \cite{IMT1,IMT2}, 
  \item second order Hamilton--Jacobi equations \cite{MTdisc}, 
  \item weakly coupled systems of Hamilton--Jacobi equations \cite{DZdiscsys} using weak KAM tools from \cite{DSZ} and \cite{ IsSys,IsSys2} for more general results, 
  \item convergence from the negative direction \cite{DW},
  \item for more general nonlinear discount approximations \cite{CCIZ,WYZ, Chen, ZQJ},
  \item for discounted approximations on networks instead of manifolds \cite{PS}, 
  \item for mean field games \cite{CP} 
  \item  more recently, with degenerate discounting approximations \cite{Zdisc}.
    \end{itemize}
     Some limitations and counterexamples  exist nevertheless when convexity is dropped \cite{Ziliotto, IsCE} and last but not least, let us quote \cite{ArSu} for a more geometric perspective on the convergence result.

Crossing back to the dynamical systems world, Patrick Bernard used weak KAM solutions to push Mather's ideas in \cite{BerAIF,BerPseudo}. Probably the most definitive works on Arnol'd diffusion to this day are \cite{ Ber,BKZ}. Those works mix on the one hand dynamical strategies dating back to Arnol'd following chains of normally hyperbolic objects having transverse intersections of stable and unstable manifolds, and on the other hand analytical tools about regularity of weak KAM solutions and subsolutions.

In his founding works \cite{MClosed,Mane}, Ma\~n\'e proved that a generic Hamiltonian (in a sense he defined), has a unique Mather measure that is hence ergodic for the Lagrangian flow. He then asked if this measure is generically concentrated on a hyperbolic periodic orbit. This question is known as Ma\~n\' e's conjecture. It is still open but important steps were made by Figalli and Rifford \cite{FR1,FR2} who managed to bring together methods of optimal control and PDE with  the dynamical system theory of orbit closing lemma. They prove the conjecture in low regularity assuming the existence of a smooth critical subsolution. Then with Contreras \cite{CFR} they brought in Arnaud's theory of Green bundles \cite{Argreen} to prove that generically, the Aubry set is hyperbolic if the underlying manifold has dimension 2. The idea is that in this setting, if the Aubry set is not hyperbolic, then positive and negative Green bundles coincide and weak KAM solutions gain some extra regularity allowing to use ideas from the previous works.

 Mather had also raised a similar question about obtaining a uniform bound on the number of ergodic Mather measures when the cohomology varies, for generic Hamiltonians. This was solved by Bernard and Contreras \cite{BeCo,BerARMA}. More generally, understanding the shape of the Aubry set is a challenging question that is still to be understood. Progress on the quotiented Aubry set was obtained by Sorrentino and Fathi, Figalli, Rifford in \cite{So,FFR} and for the actual Aubry set, by Arnaud \cite{ArLy}.

\section*{Weak KAM, beyond Hamilton--Jacobi equations}
\addcontentsline{toc}{section}{\protect\numberline{}Weak KAM, beyond Hamilton--Jacobi equations}

It turns out that the philosophy of weak KAM theory and of Aubry--Mather theory applies to a variety of other areas.

\subsection*{Optimal control theory}
\addcontentsline{toc}{subsection}{\protect\numberline{}Optimal control theory}

Weak KAM theory's starting point is the fact that solutions to Hamilton--Jacobi equations of the form \eqref{evohom} (seen in the context of homogenization), on a manifold $M$, when the Hamiltonian has some convexity properties, are given by an explicit formula called {\it Lax--Oleinik semigroup} (for $\varepsilon = 1$):
\begin{equation}\label{LOintro}U(t ,x) = \inf_{\substack{\gamma:[-t,0]\to M \\ \gamma(0)=x}} u_0\big(\gamma(-t)\big) +\int_{-t}^0 L\big(\gamma(s),\dot\gamma(s)\big)\dd s,
\end{equation}
where the infimum is taken amongst all absolutely continuous curves. This formula (or its variations) is a convolution in the (min,+) semiring and is widely studied in Optimal Control theory. In this setting, $U$ is called the {\it value function}. Excellent introductory references on the subject are \cite{CaSi00,BaCa,clarke} and we already mentioned how Fathi and others made groundbreaking progress by importing ideas from Optimal Control (such as semiconcavity, regularity properties of minimizers...). Contributions in the other directions are also worth highlighting. Recently for instance, conjectures about propagation of singularities of solutions of Hamilton--Jacobi equations were solved by the use of the positive Lax--Oleinik semigroup \cite{CanCh,CCF}. Such results find amazing consequences in Riemannian geometry when applied to singularities of distance functions \cite{CCF2}.

\subsection*{Contact type and systems of Hamilton--Jacobi equations}
\addcontentsline{toc}{subsection}{\protect\numberline{}Contact type and systems of Hamilton--Jacobi equations}
Generalizations of weak KAM theory concern more general types of equations. The first family of generalizations we have in mind is that of contact type equations. This means that the Hamiltonian function $H : T^*M \times \R \to \R$ also depends on the value $u(t,x)$ of the unknown function. The terminology comes from the equations of characteristics that preserve a contact form (instead of the symplectic form for classical Hamiltonian equations). In this context, solutions are given by an implicit Lax--Oleinik semigroup and weak KAM arguments can therefore be used. This was exploited in many recent works including \cite{WWY,WWY2,CCLW,ZC,CCIZ,WWY3,ZQJ}. 

In a maybe more surprising way, weak KAM ideas also apply to some weakly coupled systems of Hamilton--Jacobi equations. This is more unexpected as weak KAM makes strong use of the order structure of $\R$ that is less clearly adaptable for systems where the values taken are in $\R^d$ for some $d>1$. The first evidence  of such a link was present in \cite{CamLey, Ley} and further developed in \cite{DZSys,MSTY,DSZ,SiZa,ISZ}.

\subsection*{Lorentzian geometry and Lyapunov functions}
\addcontentsline{toc}{subsection}{\protect\numberline{}Lorentzian geometry and Lyapunov functions}

As can be observed already from the definition of the Lax--Oleinik semigroup \eqref{LOintro} weak KAM theory studies minimization problems, objects verifying a family of inequalities and minimizers of those inequalities. As a matter of fact,  central objects in weak KAM theory are critical subsolutions. They are functions $u_0 : M\to \R$ for which the associated function $U(t,x)-\alpha(0)t$ (where $U$ is given by the Lax--Oleinik semigroup \eqref{LOintro} and $\alpha(0)$ is called the {\it critical constant}) is non--decreasing in $t$. Such functions are also characterized by the following two properties:
\begin{itemize}
\item $u_0$ is Lipschitz continuous,
\item $(x,D_x u_0)\in H^{-1}(-\infty,\alpha(0)] \cap\{x\} \times T_x^*M$ for almost every $x\in M$.
\end{itemize}
When $H$ is convex in the second variable, the sets $ H^{-1}(-\infty,\alpha(0)] \cap \{x\}\times T_x^*M$ are convex for all $x$. Therefore a natural generalization is to replace the Hamiltonian by a family of convex sets $C_x\in T^*_xM$ that verify suitable regularity properties. The question being to know if solutions exist, one studies differentiable inclusions of the form 
\begin{itemize}
\item $u$ is Lipschitz continuous,
\item $D_x u\in C_x$ for almost every $x\in M$.
\end{itemize}
In such settings, objects from Aubry--Mather theory such as the Aubry set appear as obstructions to finding such functions that verify extra conditions (as smoothness or replacing $C_x$ by its interior). Moreover, when possible, weak KAM methods make it possible to construct $C^1$ solutions to such differentiable inclusions. This was noticed by Fathi and Siconolfi \cite{FScone} who applied this philosophy to Lorentzian geometry, where the sets $C_x$ are cones provided by a section of non positive $2$--forms. This line of research was since developed in \cite{BS3,BS4,Suhr}.

Another fruitful extension of weak KAM ideas concerns Lyapunov functions and is closer to our subject of discrete weak KAM theory. Indeed, given a continuous transformation $ F$  of a metric space $X$, a Lyapunov function is $f : X\to \R$ that is non--increasing (or when possible decreasing) on the orbits of $F$. That is $f$ verifies the family of inequalities 
$$\forall x\in X , \quad f\circ F(x) \leqslant f(x).$$

Clearly, no function can be decreasing on a periodic orbit. More generally any reasonable notion of recurrence will provide an obstruction to the existence of strict Lyapunov functions. Hence being able to construct optimal Lyapunov functions is an important challenge related to fine dynamical properties. The study of this problem, importing weak KAM ideas, was done in Pageault's PhD that led to recovering earlier results of Conley, Akin, Auslander...  in simpler and more precise form. Results are to be found in \cite{ Pa,FaPa,FaPa2} and were followed by further studies \cite{BF,BF2,BFW}.

Let us also mention links with ergodic optimization and the analogue of Ma\~n\' e's conjecture that was recently solved by Contreras \cite{Contreras}. Other works related to weak KAM theory are \cite{GLT,GT}. 

\subsection*{Optimal Transportation}
\addcontentsline{toc}{subsection}{\protect\numberline{}Optimal Transportation}

Let us describe the original Monge problem in optimal transportation \cite{Monge}. The goal is to move some material described by a probability measure $\mu$ on a space $X$ into a certain configuration described by another probability measure $\nu$ on a space $Y$, by  means of a transport map $T : X\to Y$, knowing that the cost to move material from $x\in X$ to $y\in Y$ is given by $c(x,y)$. The problem is therefore to minimize $\int_{ X} c\big(x,T(x)\big) \ \dd \mu(x)$ on maps $T:X\to Y$ such that $T_* \mu = \nu$. A difficulty is that this problem might be ill posed. For example, if $\mu$ is a Dirac mass and $\nu$ is not, there is no map $T$ such that  $T_* \mu = \nu$. Moreover, even when such maps $T$ exist, the set of such maps does not verify good properties that allow to apply classical variational methods.

Kantorovitch's tour de force \cite{Kan} is twofold. He starts by relaxing the Monge problem in looking for transport plans. Those are probability measures $\gamma$ on $X\times Y$ whose marginals are given by $\pi_{1*} \gamma = \mu$ and $\pi_{2*}\gamma = \nu$. He then wants to minimize $\int_{ X\times Y} c(x,y) \ \dd \gamma(x,y)$ amongst such plans. The set of plans is always nonempty (as $\mu \otimes \nu$ is one such) and it is closed and convex when $X$ and $Y$ are compact for instance. Hence existence of an optimal (minimizing) plan can easily be proved under mild regularity hypotheses on $c$.
The second aspect of his contribution is to provide a dual equivalent problem. The minimal cost of a transport plan is given by 
$$\sup \left( \int_Y \varphi(y) \ \dd \nu(y) - \int_X \psi(x) \dd \mu(x) \right),$$
where the supremum is taken amongst pairs of continuous functions $\varphi : Y\to \R$ and $\psi : X\to \R$ such that 
$$\forall (x,y)\in X\times Y, \quad \varphi(y)-\psi(x) \leqslant c(x,y).$$

It will become clear in the next section that the minimizing problem of transport plans resembles that of minimizing Mather measures in Aubry--Mather theory. The dual problem of finding optimal Monge--Kantorovitch pairs of functions is transparently similar to the notion of subsolutions in weak KAM theory, especially when both spaces $X$ and $Y$ coincide. Finding a transport map $T$ amounts to proving that an optimal plan is concentrated on the graph of a function from $X$ to $Y$. This is analogous to Mather's Graph Theorem stating that Mather measures are concentrated on a graph. This deep parallel was drawn and studied by Bernard and Buffoni \cite{BeBu,BeBu1,BeBu2}.

\subsection*{Discrete weak KAM theory, with an economical twist}\label{economicaltwist}
\addcontentsline{toc}{subsection}{\protect\numberline{}Discrete weak KAM theory}

The main idea of  discrete weak KAM Theory is to directly discretize the Lax--Oleinik semigroup,  allowing the time to take  integer values only (or integer multiples of a given fixed value). This simple idea then allows to weaken hypotheses on the phase space. More precisely, coming back to the definition of the Lax--Oleinik semigroup \eqref{LOintro} for $t=1$, setting 
$$h_1 (x,y) = \inf_{\substack{\gamma:[-t,0]\to M \\ \gamma(0)=x\\ \gamma(1)=y}} \int_{0}^1 L\big(\gamma(s),\dot\gamma(s)\big)\dd s,$$
the  operator can be rewritten $ U(1,x) = \inf\limits_{y\in X} u_0(y) + h_1(y,x)$. To  write such a formula, very little structure is needed. Only an underlying space $X$ that we will assume to be metric  and a function $c : X\times X \to \R$ that will play the role of $h_1$. A first theoretical study of such an operator was done in \cite{BeBu}. Further analogies and results coming from classical weak KAM theory were the subject of the author's PhD thesis in which he also highlighted some fundamental differences.  

Here is an economical interpretation of discrete weak KAM theory. Assume that $X$ is the (uncountable) metric space whose points are wine stores in France. Let $DRC : X\to \R$ be the function that gives the price $DRC(x)$ of a bottle of Domaine de la Romanée Conti in the store $x$\footnote{Domaine de la Romanée Conti is maybe the best (as far as comparisons can be made between \emph{ œuvres d'art}) and certainly most prestigious wine in the world and it is a dream of the author to taste wines from this estate in his life. Unfortunately, in this randomly chosen example, the function $DRC$ is close to being  $+\infty$ everywhere but at Vosne Romanée and accessible at $x=$ Vosne Romanée (but with no available stock).}. Let now $c : X\times X \to \R$ be the function where $c(x,y)$ is the price for a 24 hour delivery of a wine bottle from $x$ to $y$. Then if Maxime lives at $x$, the  least price he will have to pay to obtain a bottle of Romanée Conti tomorrow is 
$$T^- DRC (x) = \inf_{y\in X} DRC(y) +c(y,x).$$
In this simple and simplistic  model, by iterating the previous operator $T^-$, one obtains the best price to have a bottle if one is willing to wait a long time. Finally, studying long optimal trajectories that a bottle will follow before reaching the patient Maxime will provide important objects of weak KAM theory and Aubry--Mather theory.

\section*{Organisation of the text}
\addcontentsline{toc}{section}{\protect\numberline{}Organisation of the text}

The first 3 Chapters are dedicated to presenting discrete weak KAM theory in a general setting. Each of them ends with a section where related classical weak KAM results are stated to give the reader an overview of the classical theory without proofs. We  believe that this back and forth between the discrete and the classical weak KAM theories is original. For well chosen costs, it highlights in a new way strong similarities and also key differences between the two versions of weak KAM theory. Discrete weak KAM theory as presented in these Chapters was developed by the author in \cite{Z,ZJMD} (following earlier works as \cite{BeBu}) in a non compact setting. Here we present the compact setting that is less technical, therefore easier for a first encounter with weak KAM theory. Yet all  key features and ideas of the classical theory persist and are better highlighted. Further and more precise results written hereafter were  obtained with coauthors, references being provided in the text. Those first Chapters should be accessible without any specialized background.

The First Chapter introduces the Lax--Oleinik semigroups (negative and positive). The discrete weak KAM Theorem is proved and the last part of the Chapter is dedicated to constructing continuous strict subsolutions, which are a fundamental tool in the theory. This allows to define the Aubry set.
In the last section, a proof of the weak KAM Theorem using the discrete one is given. Also, it is shown that for a natural cost function, weak KAM solutions and discrete weak KAM solutions coincide. The result is new to our knowledge. We also establish that for this cost, the projected Aubry set is equal to the classical one.

The Second Chapter aims at showing results of a more dynamical nature about the Aubry set. We introduce Peierls' barrier and characterize points of the projected Aubry set with it. Examples of points and chains of the Aubry set are given. Finally the problem of regularity of subsolutions (or lack thereof) is addressed. In a first part, in the general setting, we characterize the  projected Aubry set as the set where all subsolutions are continuous. Then by adding structure, we show existence of $C^{1,1}$ subsolutions, thus obtaining results similar to  Bernard's classical ones. Finally we provide Graph Theorems and by adding a twist condition (which replaces convexity) we show how to define a partial dynamics on the Aubry set.   

The Third Chapter is dedicated to Mather measures and to the crucial role they play in proving convergence of the discounted solutions. We start by giving two definitions of Mather measures and showing that they are equivalent. Then we prove convergence of solutions to the discounted equations. It is pointed out that the limit weak KAM solution for the positive Lax--Oleinik semigroup is not necessarily  the conjugate of the limit weak KAM solution for the negative Lax--Oleinik semigroup. We then study a degenerate discounted problem that is new in this setting and prove convergence of the solutions.

The Fourth Chapter provides examples in dimension 1. Those examples come from autonomous Hamiltonians and have the great merit that explicit computations can be made. We also show that the weak KAM solutions selected by the discounted approximation procedure may differ in the discrete setting and in the classical one. Such examples are folklore to experts. However we do not know of any reference where a detailed analysis is made under the scope of weak KAM theory. We believe that they are useful to have in mind in order to develop an intuition and test conejctures.

The Fifth and last Chapter puts back  in the context of discrete weak KAM theory the foundational problem of conservative twist maps of the annulus. We revisit results of Mather, Aubry, Bangert... from the perspective  of weak KAM solutions. In this unified  setting we gather proofs of well known results that are spread in various references. We hope this will make them more accessible. We also give a precise description of what those weak KAM solutions look like in this setting (results of Arnaud and the author). Finally we conclude with statements of results or Arnaud and the author opening to the world of weakly integrable twist maps.

\newpage 

\chapter{The discrete setting, weak KAM solutions and subsolutions}
The main idea of the discrete setting we  focus on is to directly discretize the Lax--Oleinik semigroup, allowing the time to take integer values only (or integer multiples of a given fixed value). This simple idea then allows to weaken the assumptions on the phase space. Most results and proofs of this Chapter are extracted as a particular compact case of \cite{Z}.
\section{Discrete setting and the Lax--Oleinik semigroup} 
We  focus our attention on the case where $(X,d )$ is a compact  metric space.
The analogue of the Lagrangian function is a cost function $c : X\times X \to \R$ which is assumed to be continuous. One can then define the Lax--Oleinik semigroup acting on the set of bounded functions $\BB(X,\R)$:
\begin{df}\rm
The Lax--Oleinik semigroup $T^- : \BB(X,\R) \to \BB(X,\R)$ associates to $f : X\to \R$ the function
$$T^- f : x\in X \mapsto T^-f(x) = \inf_{y\in X} f(y) + c(y,x).$$
\end{df}
\begin{rem}\rm
The Lax--Oleinik semigroup is a convolution with kernel $c$ in the $(\min,+)$ semiring. The $\inf$ plays the role of integration and the $+$ plays the role of multiplication in a classical convolution.

In particular, if the set $X$ is finite, the Lax--Oleinik semigroup reduces to a product (in the $(\min,+)$ semiring) with the matrix whose entries are given by $c$. 

If $f$ is a continuous function, then the infimum in the definition of $T^- f(x)$ is a minimum by compactness of $X$.
\end{rem}

We define the sup--norm $\|\cdot \|_\infty$ on the space $ \BB(X,\R)$ by setting $\|f\|_\infty = \sup\limits_{x\in X} |f(x)|$ for $f\in \BB(X,\R)$. The normed vector space $\big( \BB(X,\R),\|\cdot \|_\infty\big)$ is a Banach space.

We start by listing first basic properties of $T^-$:

\begin{pr}\label{propT}
\hspace{2em}
\begin{enumerate}[label=(\roman*)]
\item The image of $T^-$ consists of equicontinuous functions with uniformly bounded amplitude. \footnote{By amplitude of a function $f : X\to \R$ we mean $\sup f - \inf f$.}
\item The Lax--Oleinik semigroup commutes with addition of constant functions: if  $k\in \R$ and $f$ is a function, $T^- ( f+k) = (T^- f) +k$.
\item The Lax--Oleinik semigroup is order preserving: if $f\leqslant g$ then $T^-f \leqslant T^- g$.
\item The Lax--Oleinik semigroup is $1$--Lipschitz for the sup--norm $\|\cdot \|_\infty$.
\end{enumerate}

\end{pr}
\begin{proof}
Let us consider a modulus of uniform continuity $\om$ for $c$ ($X$ being compact). This is a non--decreasing function $\om : [0,+\infty) \to [0,+\infty)$ that is continuous at $0$, with $\om(0)=0$, such that 
$$\forall (x,y,x',y')\in X^4,\quad |c(x,y) - c(x',y')| \leqslant \om\big(d(x,x')+d(y,y')\big).$$
Without loss of generality, by triangular inequality we may assume that $\om$ is bounded and that $\|\om\|_\infty \leqslant 2\|c\|_\infty$.
Let $f : X\to \R$ be a bounded function and $\e>0$. Let $(x,x')\in X^2$. By definition of the Lax--Oleinik semigroup,  there exists a $y_\e$ such that $T^-f(x')\geqslant f(y_\e) + c(y_\e,x') -\e$. It follows that 
$$T^-f(x) -T^-f(x') \leqslant f(y_\e)+c(y_\e,x)  -f(y_\e)-c(y_\e,x')+\e \leqslant \om\big(d(x,x')\big) +\e.$$
Letting $\e\to 0$ yields that $  T^-f(x) -T^-f(x') \leqslant \om\big(d(x,x')\big)$. As $x$ and $x'$ play symmetrical roles we find that $|  T^-f(x) -T^-f(x')| \leqslant \om\big(d(x,x')\big)$. This is  $(i)$ as $x, x'$ are arbitrary.

Points $(ii)$ and $(iii)$ are obvious from the definition of $T^-$ and automatically imply $(iv)$. Indeed, if $f$ and $g$ are bounded functions, as $f-\|f-g\|_\infty \leqslant g\leqslant f+\|f-g\|_\infty$ we obtain
$$T^- f -\|f-g\|_\infty \leqslant T^- g\leqslant T^-f+\|f-g\|_\infty$$
which means that $\|T^-f - T^-g\|_\infty \leqslant \|f-g \|_\infty$. 
\end{proof}
\begin{rem}\label{remT}\rm 
The Lax--Oleinik semigroup can actually be defined on arbitrary functions $f : X\to \R$ (not necessarily bounded) with the only modification that $T^- f$ can take the value $-\infty$. However, it can be easily checked that, as $c$ is bounded, this may only happen if $f$ is unbounded from below. In this case, $T^-f$ is identically $-\infty$. Otherwise, the conclusions of the previous proposition (\ref{propT}) still hold, with the same proofs.
\end{rem}

Now that  those properties have been established, let us move on to the weak KAM theorem.

\section{The weak KAM Theorem and critical subsolutions}
In this section, we will introduce and  construct some of the most important objects of weak KAM theory. The first ones are  of course weak KAM solutions and are given by the following theorem:
\begin{Th}[weak KAM]\label{weak KAM}
There exists a unique constant $\cc\in \R$ for which the equation $u = T^- u + \cc$ admits solutions $u:X\to \R$.
\end{Th}
\begin{rem}\rm
Such functions are then called {\it weak KAM solutions}. The constant $\cc$ is called the {\it critical value}. 

It is immediate from Proposition \ref{propT}
 and Remark \ref{remT} that weak KAM solutions are automatically continuous.
\end{rem}
We will give two proofs of the existence part of the weak KAM Theorem. The first one is similar to the original proof of Fathi (\cite{Fa4,Fa}). The second one  is reminiscent of the work of Lions, Papanicolaou and Varadhan (\cite{LPV}) on homogenization, that actually appeared prior to the work of Fathi.

The uniqueness of the constant $\cc$ will be established in a second step.
\begin{proof}[First Proof]
Let us introduce $\mathcal E = \BB(X,\R) \slash \R\1$ the quotient of bounded functions by constant functions. The set $\mathcal E$ is clearly a vector space. If $f$ is a bounded function, we will denote by $\bar f$ its projection in $\mathcal E$. There is an induced norm on $\mathcal E$: if $f : X\to \R$ is a continuous function, denoting $\bar f \in \mathcal E$ its class in the quotient, we set $\|\bar f \|_0 = \min\{ \|f+k \|_\infty,\ \ \ k\in \R\}$.
As $T^-$ commutes with addition of constant functions, it induces an operator $\TT$ on $\mathcal E$ defined by $\TT \bar f = \overline{T^- f}$ which is independent of the the representative $f$ in the equivalence class $\bar f$. This new operator is also continuous. Indeed, if $f,g$ are two bounded functions, then for some suitably chosen constant $k\in \R$,
$$\|\TT \bar f - \TT \bar g\|_0 \leqslant  \|T^-f- T^- g + k\|_\infty \leqslant \|f-g+k\|_\infty = \|\bar f - \bar g \|_0.$$
Moreover, it follows from the fact that $T^-$ has values in equicontinuous functions with  uniformly bounded amplitude (see proposition \ref{propT}) and from the Arzel\` a--Ascoli theorem (\cite[Theorem 6.4 page 267]{dug}) that $\TT(\mathcal E)$ is relatively compact. We can therefore apply the Schauder--Tychonoff Theorem (\cite[Theorems 2.2 and 3.2 pages 414-415]{dug}) which asserts that $\TT$ has a fixed point. This exactly means that there exists a bounded function $u : X\to \R$ and a constant $C$ such that $u = T^-u+C$.
\end{proof}
The use of the Schauder--Tychonoff Theorem at the end of this first proof, though natural, is not really necessary. Indeed, fixed points of $1$--Lipschitz maps can be obtained by much simpler arguments, usually by perturbing the map, making it contracting and then passing to the limit. This is the spirit of the second proof in which we use an approximation called {\it discounted procedure}:

\begin{proof}[Second Proof]
Let $\lambda \in (0,1)$ and let us introduce the discounted operator $T^-_\lambda$ which acts on bounded functions as follows:
$$\forall f\in \BB(X,\R),\ \ \forall x\in X,\quad T^-_\lambda f (x) = \inf_{y\in X} \lambda f(y) + c(y,x) = T^- (\lambda f) (x).$$ 
Of course, the last formulation, together  with the $1$--Lipschitz nature of $T^-$ imply that $T^-_\lambda$ is now $\lambda$--Lipschitz. Hence, by the Banach fixed point theorem (\cite[Theorem 7.2 page 305]{dug}), as $\BB(X,\R)$ is a Banach space,  $T^-_\lambda$ admits a unique fixed point $u_\lambda$ which then verifies $T^-_\lambda u_\lambda = u_\lambda$. By Proposition \ref{propT}, the $(u_\lambda)_{\lambda \in (0,1)}$ are equicontinuous with uniformly bounded amplitude, as they all belong to the image of $T^-$. Moreover
$$\forall x\in X, \quad \min_{(y,z)\in X\times X} c(y,z) \leqslant (1-\lambda) u_\lambda (x) \leqslant \max_{z\in X} c(z,z).$$
To prove the left inequality, fix $\lambda$ and take $x_1$ such that $u_\lambda(x_1)$ is minimal. One then has for some $y\in X$,
$$ u_\lambda (x_1)  = \lambda u_\lambda (y) +c(y,x_1) \geqslant \lambda u_\lambda (x_1) + c(y,x_1).$$
The right inequality follows from the fact that by definition of $T^-_\lambda$ we obviously have $u_\lambda (x) = T^-_\lambda u_\lambda (x) \leqslant \lambda u_\lambda (x) + c(x,x)$.

Let us fix a point $x_0\in X$ and define $\hat u_\lambda = u_\lambda - u_\lambda(x_0)$ for all $\lambda \in (0,1)$. The previous remarks show that we can find a sequence $\lambda_n \to 1$ such that $(1-\lambda_n)u_{\lambda_n}(x_0)$ converges to a constant $C$ and $(\hat u_{\lambda_n})_{n\in \N}$ uniformly converges to a function $u$, as $n\to +\infty$. Note that in fact, as the functions $u_\lambda$ have bounded amplitude, one has $(1-\lambda_n)u_{\lambda_n}\to C$ uniformly. It now follows that
\begin{multline*}
\forall x \in X, \quad \hat u_{\lambda_n}(x) =   u_{\lambda_n}(x) - u_{\lambda_n}(x_0) \\
= T^-(\lambda_n u_{\lambda_n})(x)- u_{\lambda_n}(x_0) \\
= T^- (\lambda_n \hat u_{\lambda_n})(x)+(\lambda_n-1) u_{\lambda_n}(x_0) .
\end{multline*}
Letting $n\to +\infty$, by continuity of $T^-$, we conclude that $u = T^-u-C$.
\end{proof}
It remains to prove the uniqueness of the critical constant $\cc$. This is a direct consequence of the following proposition:
\begin{pr}\label{pr asympt}
Assume that $u = T^-u + \cc$ for some function $u$ and constant $\cc$, then
$$\forall v\in \BB(X,\R),\quad \frac{T^{-n} v }{n} \underset{n\to +\infty}{\longrightarrow}  -\cc,$$
where $\Tn=T^- \circ \cdots \circ T^-$ denotes the $n$-th iterate of the Lax--Oleinik semigroup, and the convergence is uniform.
\end{pr}
\begin{proof}
This result is a direct consequence of the non expansive character of the Lax--Oleinik semigroup. Indeed, by induction, one obtains that for all integers $n>0$, $u=\Tn u +n\cc$. Hence, if $v\in \BB(X,\R)$, then 
$$\|\Tn u - \Tn v \|_\infty =\| u-n\cc - \Tn v \|_\infty \leqslant \| u -  v \|_\infty.$$
 The result follows immediately as $\| n\cc + \Tn v \|_\infty \leqslant \| u -  v \|_\infty+ \| u\|_\infty$.   
\end{proof}

\begin{rem}\label{bounded}\rm
\hspace{2em}
\begin{enumerate}[label=(\roman*)]
\item In the previous Proposition, one derives that for all $v\in \BB(X,\R)$ the sequence  $(\Tn v +n\cc)_{n \in \N}$ is uniformly bounded.
\item It follows from the uniqueness of $\cc$ that in the second proof of the weak KAM Theorem (\ref{weak KAM}), the whole family $(1-\lambda)u_{\lambda}$ uniformly converges to $\cc$ as $\lambda \to 1$.
\item If $u : X\to \R$ is a weak KAM solution, it satisfies the following fundamental  inequalities:
$$\forall (x,y)\in X\times X,\quad u(y)-u(x) \leqslant c(x,y)+\cc.$$

\end{enumerate}

\end{rem}
This last Remark motivates the following definition:
\begin{df}\label{dfss}\rm
Given $C\in \R$, a function $v : X\to \R$ will be termed {\it $C$--subsolution} if
$$\forall (x,y)\in X\times X,\quad v(y)-v(x) \leqslant c(x,y)+C,$$
or equivalently, if
 $v\leqslant T^-v+C$.
 
 We will call $\cc$--subsolutions critical subsolutions, or just subsolutions when no confusion is possible. We will denote by $\S_C$ the set of $C$--subsolutions, and by $\S=\S_\cc$ the set of critical subsolutions.
\end{df}
Here are some first properties of subsolutions:
\begin{pr}\label{propS}
Given $C\in\R$, the following hold:
\begin{enumerate}[label=(\roman*)]
\item Any $C$--subsolution is bounded.
\item The set $\S_C$ of $C$--subsolutions is closed (with respect to pointwise convergence), convex and stable by the Lax--Oleinik semigroup: $T^-(\S_C)\subset \S_C$.
\end{enumerate}
\end{pr} 
\begin{proof}
The first point comes from the fact that if $u\in \S_C$ and $x_0\in X$, for all $y\in X$,
$$u(x_0)-c(y,x_0)-C\leqslant u(y) \leqslant u(x_0)+c(x_0,y)+C,$$
hence $\|u\|_\infty \leqslant |u(x_0)|+\|c+C\|_\infty $.

The fact that $\S_C$ is closed and convex is immediate from the definition. Only stability deserves an explanation.
It follows from the fact that $u\in\S_C$ if and only if $u\leqslant T^- u +C$ as can be checked by applying the definitions. Consequently, by the properties of the Lax--Oleinik semigroup (\ref{propT}), if $u\leqslant T^-u +C$ then 
$$T^-u \leqslant T^-(T^-u+C) = T^-(T^-u) +C,$$
 which means $T^-u\in \S_C$.
\end{proof}
We conclude this section by one last characterization of the critical constant $\cc$:
\begin{pr}\label{criticalSS}
The following holds: $\cc = \min\{ C \in \R, \ \ \S_C\neq \varnothing\}$.
\end{pr}
\begin{proof}
As follows from the weak KAM Theorem, $\S_\cc\neq \varnothing$, so we just have to prove that if for some constant $C$, $\S_C\neq \varnothing$, then $C\geqslant \cc$.
Let then $u\in \S_C$ for some $C\in \R$. As in the last proof, we get by induction that for all positive integer $n$, $u\leqslant T^{-n}u + nC$. Therefore, dividing by $n$, we infer that $u/n - C \leqslant T^{-n} u /n$. But the right hand side converges to $-\cc$ by Proposition \ref{pr asympt}. Hence passing to the limit we conclude that $-C\leqslant -\cc$.
\end{proof}

We illustrate once more our model with a random liquid example as in the introductory section page \pageref{economicaltwist}. Here again, $X$ is the space of wine stores in France and $c:X\times X \to \R$ provides the cost $c(x,y)$ of bringing a bottle of wine from $x$ to $y$ in a day. Let $P : X\to \R$ denote the price $P(x)$ of a bottle of Petrus\footnote{Petrus is a renowned and quite inaccessible wine from the Bordeaux region in France. More precisely it is the leading estate of the appellation Pomerol. The royalties earned by a lifetimes'work of the author would probably allow him to buy a quarter of a bottle of Petrus.}  at a location $x$. Let $n\in \N$, then if the author is at $x\in X$, the best price he will pay to have a bottle of Petrus in $n$ days is $T^{-n} P(x)$. Proposition \ref{pr asympt} then states that whatever the initial price $P$, for $n\to +\infty$ the amount $T^{-n} P(x)$ grows like $-nc[0]$. In this particular example, given the order of magnitude of $P$, the time $n$ would have to be very very large to compensate the initial price.

\section{The positive Lax--Oleinik semigroup}
In this section, we introduce the positive Lax--Oleinik semigroup and state its main properties. As they are analogous to the properties of the (negative) Lax--Oleinik semigroup,  the proofs are omitted and left to the reader.

\begin{df}\rm\label{dfT+}
The positive Lax--Oleinik semigroup $T^+ : \BB(X,\R) \to \BB(X,\R)$ maps to $f : X\to \R$ the function
$$T^+ f : x\in X \mapsto T^+f(x) = \sup_{y\in X} f(y) - c(x,y).$$
\end{df}
\begin{rem}\rm
From the definition, one checks that $T^+ f = -\big(T^-_{\hat c} (-f) \big)$ where $T^-_{\hat c}$ is the cost defined by $\hat c (x,y) =c(y,x)$. Hence the fact that $T^-$ and $T^+$ share very similar properties is not surprising.
\end{rem}

\begin{pr}\label{propT+}
\hspace{2em}
\begin{enumerate}[label=(\roman*)]
\item The image of $T^+$ consists of equicontinuous functions with uniformly bounded amplitude.
\item The positive Lax--Oleinik semigroup, $T^+$, commutes with addition of constant functions.
\item The positive Lax--Oleinik semigroup is order preserving: 

if $f\leqslant g$ then $T^+f \leqslant T^+ g$.
\item The positive Lax--Oleinik semigroup is $1$--Lipschitz for the sup--norm $\|\cdot \|_\infty$.
\end{enumerate}

\end{pr}
The positive semigroup also fulfills a weak KAM Theorem:
\begin{Th}[positive weak KAM]\label{weakKAM+}
The critical constant $\cc$ is the unique real value $c\in \R$ for which the equation $u = T^+ u - c$ admits solutions $u:X\to \R$. Moreover the constant $c[0]$ has the following caracterization: 
$$\forall v\in \BB(X,\R),\quad \frac{T^{+n} v }{n} \underset{n\to +\infty}{\longrightarrow} \cc,$$
where $T^{+n}=T^+ \circ \cdots \circ T^+$ denotes the $n$-th iterate of the positive Lax--Oleinik semigroup, and the convergence is uniform.

\end{Th}
\begin{rem}\rm
Such functions are then called {\it positive weak KAM solutions}. 

It is immediate from Proposition \ref{propT+}
  that positive  weak KAM solutions are always continuous.
\end{rem}
In the positive weak KAM theorem, the fact that the critical constant for $T^+$ is the same as that of $T^-$, that is $\cc$, deserves some explanation. It is checked that $u : X\to \R$ is a $C$--subsolution, for some $C\in \R$ if and only if $u\geqslant T^+u - C$. Hence, as for the Lax--Oleinik semigroup, one has that $T^+(\S_C) \subset \S_C$. 

Now, the fact that both critical constants coincide follows from Proposition \ref{criticalSS} which characterizes the critical constant only making use of the notion of subsolution, and not of either the positive, nor negative Lax--Oleinik semigroup. However, its proof can be done equivalently using the negative or the positive Lax--Oleinik semigroup.

 We end this section with a curiosity on the composition of positive and negative Lax--Oleinik semigroups:
 
 \begin{pr}\label{comp}
 Let $u\in \BB(X,\R)$, then $T^+\circ T^- u \leqslant u$ and  $T^-\circ T^+ u \geqslant u$.
 \end{pr}
 
 \begin{proof}
 Let us establish the first inequality. If $x\in X$ then
 $$T^+\circ T^- u(x) =  \sup_{y\in X} T^-u(y) - c(x,y) =  \sup_{y\in X} \inf_{z\in X} u(z)+c(z,y)- c(x,y) \leqslant u(x),$$
 where the last inequality is obtained by taking $z=x$.
 \end{proof}

\section{Strict subsolutions, Aubry sets}

In this section, we will focus our study on critical subsolutions, hence the adjective may be omitted from time to time, but is always implicit. Recall that the set $\S$ is the set of critical subsolutions. The goal will be to construct a special kind of subsolutions which are in some sense better than the others:
\begin{Th}\label{th-strict}
There exists a subsolution $u_0\in \S\cap C^0(X,\R)$ such that, if the equality $u_0(y) - u_0(x)  = c(x,y)+\cc$ holds  for some $(x,y)\in X\times X$, then
\begin{equation}\label{strict}
\forall u\in \S,\quad u(y) - u(x)  = c(x,y)+\cc.
\end{equation}
\end{Th}
A subsolution verifying this last property \eqref{strict} will be termed {\it strict}. The proof of this Theorem will occupy the rest of this section. It will be divided in two main parts. In a first one, we will construct the function $u_0$ and prove that it verifies \eqref{strict} for all other continuous subsolutions.

Then a digression is devoted to studying the structure of the set where the equality $u_0(y) - u_0(x)  = c(x,y)+\cc$ takes place. This is the Aubry set, a central object in weak KAM theory, and unsurprisingly in Aubry--Mather theory.

This being achieved, the proof of the Theorem ends rather easily.

\begin{proof}[Beginning of the Proof]
The set $(X,d)$ being metric and compact, the Banach space $\big(C^0(X,\R), \| \cdot \|_\infty\big)$ is itself separable. Therefore, $\S\cap C^0(X,\R) $ being a subset of a separable space is also separable. Let $(v_n)_{n\in \N}$ be a dense (with respect to the topology of uniform convergence) sequence in  $\S\cap C^0(X,\R) $. Now, let us define 
$u_0 = \sum\limits_{n\geqslant 0} a_n v_n$ where for $n>0$, we set $a_n = \min(2^{-n}, 2^{-n}/\|v_n\|_\infty)$ and $a_0 = 1-\sum\limits_{n>0}a_n$.

The function $u_0$ is defined by a  series of continuous functions that converges for the $\| \cdot\|_\infty$--norm, hence it is continuous. It is an infinite convex combination of critical subsolutions, therefore, by Proposition \ref{propT} it is a subsolution. We will prove it verifies the property we seek for.

Let us consider now $(x,y)\in X\times X$ such that $u_0(y) - u_0(x) = c(x,y)+\cc$. By definition of subsolutions, we know that $v_n(y) - v_n(x) \leqslant c(x,y)+\cc $, for  all $n\geqslant 0$. Multiplying each of these inequalities by $a_n$ and summing, we get
$$u_0(y) - u_0(x) = \sum_{n=0}^{+\infty} a_n\big(v_n(y)-v_n(x) \big) \leqslant \sum_{n=0}^{+\infty} a_n (c(x,y)+\cc ) = c(x,y)+\cc,$$
which is in fact an equality. Hence all the middle inequalities must be equalities, and the $a_n$ being all positive, we conclude that
$ v_n(y) - v_n(x) = c(x,y)+\cc$ for all integer $n$. The sequence $(v_n)_{n\in \N}$ being dense in $\S\cap C^0(X,\R)$ we conclude eventually that
$$\forall u\in \S\cap C^0(X,\R), \quad \Big[ u_0(y) - u_0(x) = c(x,y)+\cc \Big] \Longrightarrow \Big[  u(y) - u(x) = c(x,y)+\cc \Big].$$
\end{proof}
The rest of the proof will consist in extending this property to non continuous subsolutions.

 Let us start by defining some useful sets:
 \begin{df}[Aubry sets] \rm
  Let $u \in \S$ be a critical subsolution.
 \begin{itemize}
 \item The non--strict set of $u$ is 
 $$\mathcal{NS}_u = \{(x,y)\in X\times X,\quad u(y) - u(x) = c(x,y)+\cc\}.$$
 \item The Aubry set of $u$ is 
 $$\widetilde\AA_u = \Big\{(x_n)_{n\in \Z} \in X^\Z, \ \ \forall n<p,\quad u(x_p) - u(x_n) = \sum_{k=n}^{p-1} c(x_k,x_{k+1}) + (n-p)\cc \Big\}.$$
 \item The Aubry set is $\widetilde\AA = \widetilde\AA_{u_0}\subset X^\Z$ where $u_0$ is the peculiar subsolution previously constructed.
 \item The $2$--Aubry is $\widehat \AA = \mathcal{NS}_{u_0}\subset X\times X$.
 \item Finally, the projected Aubry set is $\AA = \pi_1(\widehat \AA )\subset X$ where $\pi_1 : X\times X \to X$ is the projection on the first factor: $(x,y) \mapsto x$.
 
 \end{itemize}

 \end{df}
 
\begin{rem}\label{rem-eg}\rm
\hspace{2em}
\begin{enumerate}[label=(\roman*)]
\item The beginning of the proof of Theorem \ref{th-strict} may be summed up in the equality $\widehat\AA = \bigcap\limits_{u\in \S\cap C^0(X,\R)} \mathcal{NS}_u $.
\item If $s : X^\Z \to X^Z$ is the shift operator: $(x_n)_{n\in \Z}\mapsto (x_{n+1})_{n\in \Z}$ then the sets introduced above are invariant by this shift: $\widetilde\AA_u=s(\widetilde\AA_u)$ and $\widetilde\AA = s(\widetilde\AA)$.
\item If $u\in \S$ and $(x_n)_{n\in \Z}\in X^\Z$ then one has for all $k\in \Z$, $u(x_{k+1})-u(x_k)\leqslant c(x_k,x_{k+1})+\cc$. Summing these inequalities, one gets 
$$\forall n<p,\quad u(x_p) - u(x_n) \leqslant \sum_{k=n}^{p-1} c(x_k,x_{k+1}) + (n-p)\cc.$$
Therefore if the previous inequality turns out to be an equality, it implies that $u(x_{k+1})-u(x_k)= c(x_k,x_{k+1})+\cc$ for $n\leqslant k<p$. 

\item Note that by taking $n=0$ and $p=1$ in the definition of $\widetilde\AA_u$ we find  that $\pi_{0,1}(\widetilde\AA_u ) \subset \mathcal{NS}_u$ where $\pi_{0,1} : X^\Z\to X\times X$ is the projection: $(x_n)_{n\in \Z}\mapsto (x_0,x_1)$.

\end{enumerate}
\end{rem}

Now let us study these Aubry sets. We start by basic topological properties:
\begin{pr}
The $2$--Aubry set $\widehat \AA$ is closed and non--empty.
\end{pr}
\begin{proof}
Being closed comes from the identity $\widehat \AA = F^{-1}\{0\}$ where $F(x,y) = u_0(y)-u_0(x) - c(x,y)-\cc$ is continuous.

Being non--empty is a consequence of the minimality of $\cc$ (Proposition \ref{criticalSS}). Indeed, as $u_0$ is a subsolution, the function $F$ is non--positive. By compactness and continuity, if $\widehat \AA = F^{-1}\{0\}$ were empty, there would be a small $\e>0$ such that $u_0(y)-u_0(x) \leqslant c(x,y)+\cc-\e$, for all $(x,y)\in X\times X$. But this means that $u_0\in \S_{\cc-\e}\neq \varnothing$ which contradicts Proposition \ref{criticalSS}.
\end{proof}
The next proposition states that elements of the $2$--Aubry set come in families, meaning that the Aubry set is not empty:
\begin{pr}\label{pr-sequence}
Let $(x,y)\in \widehat \AA$, then there exists a sequence $(x_n)_{n\in \Z} \in \widetilde\AA$ such that $(x,y) = \pi_{0,1}\big((x_n)_{n\in \Z}\big) $. In particular, the Aubry set is itself closed (for the product topology), not empty, and $\widehat \AA = \pi_{0,1}(\widetilde \AA)$. 
\end{pr}

The proof will make use of the following two lemmas. The first one's proof is a direct application of the definitions and is omitted:
\begin{lm}
Let $n$ be a positive integer and $f\in \BB(X,\R)$, then
$$\forall x\in X,\quad \Tn f(x) = \inf_{x_{-n},\cdots,x_0=x} f(x_{-n}) + \sum_{k=-n}^{-1} c(x_k,x_{k+1}),$$
$$\forall x\in X,\quad T^{+n} f(x) = \sup_{x=x_0, \cdots , x_n} f(x_{n}) - \sum_{k=0}^{n-1} c(x_k,x_{k+1}).$$

\end{lm}
\begin{lm}\label{A-cst1}
Let $u\in \S$ be a continuous subsolution and $(x,y)\in \widehat \AA$. We have $\Tn u(x) = u(x) - n\cc$ and $T^{+n}u (y) = u(y) + n\cc$ for all positive integer $n$. 
\end{lm}
\begin{proof}
As $\Tn u$ and $T^{+n} u$ are continuous subsolutions and $(x,y) \in \widehat \AA$, we now know that $\Tn u(y) - \Tn u(x) = \TTn u(y) - \TTn u(x) = c(x,y)+\cc$. It implies readily that
$$T^{-(n+1)} u (y)+\cc \geqslant \Tn u(y) = \Tn u(x)+c(x,y)+\cc \geqslant T^{-(n+1)} u (y)+\cc,$$
the first inequality coming from the fact that $\Tn u\in \S$ the second from the definition of $T^-$. Hence all inequalities turn out to equalities and the sequence $(\Tn u(y) + n\cc)_{n\in \N}$ is constant. 

The same holds for $\Tn u(x) + n\cc = \Tn u(y)+n\cc-c(x,y) -\cc$.

The proof of the rest of the lemma is established similarly.

\end{proof}

\begin{proof}[Proof of Proposition \ref{pr-sequence}]
The sequences $(x_n)$ for $n\leqslant 0$ and $n \geqslant 0$ are constructed separately. The first half will come from the negative Lax--Oleinik semigroup, the second one, from the positive Lax--Oleinik semigroup.

As $u_0 $ and $c$ are continuous and $X$ is compact, any supremum (resp. infimum) involving them is actually a maximum (resp. minimum). Therefore, for each positive $n$, there exist chains $x_{-n}^n,\cdots ,x_0^n = x,x_1^n = y , \cdots x_{n+1}^n$ such that
$$\Tn u_0(x) =u_0(x_{-n}^n)+\sum_{k=-n}^{-1} c(x_{k}^n,x_{k+1}^n)\ \ ;\quad T^{+n}u_0(y) = u_0(x^n_{n+1}) - \sum_{k=1}^n c(x_k^n,x_{k+1}^n).$$

By a diagonal argument, let $(N_n)_{n\geqslant 0}$ be an extracted sequence such that for all  $k\geqslant 0$ the sequences $(x_k^{N_n})_{n\geqslant k}$ and $(x_{-k}^{N_n})_{n\geqslant k}$ converge. We will denote by $x_k$ and $x_{-k}$ the respective limits. Obviously, $x_0=x$ and $x_1=y$. By definition of the Lax--Oleinik semigroup and Lemma \ref{A-cst1}, for all $k\leqslant n$, 
$$
u_0(x) = u_0(x_k^{N_n}) + \sum_{i=-k}^{-1} c(x_{i}^{N_n},x_{i+1}^{N_n})+k\cc\ ;$$
$$
 u_0(y) = u_0(x^{N_n}_{k+1}) - \sum_{i=1}^k c(x_i^{N_n},x_{i+1}^{N_n})-k\cc.
 $$
Letting $n\to +\infty$, we obtain that 
$$u_0(x) = u_0(x_k) + \sum_{i=-k}^{-1} c(x_{i},x_{i+1})+k\cc\ \ ;\quad u_0(y) = u_0(x_{k+1}) - \sum_{i=1}^k c(x_i,x_{i+1})-k\cc.$$
As in Remark \ref{rem-eg}, this implies that for all $n\in \Z$, $u_0(x_{n+1})-u_0(x_n) = c(x_n,x_{n+1})+\cc$ (the case $n<0$ is given by the left equalities and $n>0$ by the right equalities above, $n=0$ is because $(x,y)\in \widehat\AA$). Now, again as in Remark \ref{rem-eg}, by summing those equalities, one obtains that 
$$\forall n<p,\quad u_0(x_p) - u_0(x_n) = \sum_{k=n}^{p-1} c(x_k,x_{k+1}) + (n-p)\cc.$$
This exactly means that $(x_n)_{n\in \Z} \in \widetilde\AA$.
\end{proof}

\begin{rem}\label{remoublieee}\rm
Let $(x_n)_{n\in \Z} \in \widetilde\AA$ and $u$ be a subsolution. As for all $n\in \Z$, $(x_n,x_{n+1})\in \widehat \AA$, summing equalities $u(x_{n+1}) - u(x_n) = c(x_n,x_{n+1})+\cc$, one finds that
$$\forall n<p,\quad u(x_p) - u(x_n) = \sum_{k=n}^{p-1} c(x_k,x_{k+1}) + (n-p)\cc.$$
In other terms, $\widetilde \AA \subset \widetilde \AA_u$.

\end{rem}

As the Aubry set is invariant by the shift, we get the immediate:
\begin{co}
The projected Aubry set can be obtained by either projection, $\AA = \pi_1 (\widehat \AA) = \pi_2 (\widehat \AA)$, where $\pi_2$ is the projection on the second factor.
\end{co}

Let us  now complete the proof of the Theorem:
\begin{proof}[End of the Proof of Theorem \ref{th-strict}]
Let $u\in \S$ be any subsolution. Let $(x,y)\in \widehat \AA$ and let $(x_n)_{n\in \Z}\in \widetilde \AA$ such that $\pi_{0,1}\big((x_n)_{n\in \Z}\big)=(x,y)$. We have the following chain of inequalities, for all integer $n\in \Z$:
$$T^- u(x_{n+1}) \leqslant u(x_n)+c(x_n,x_{n+1})\leqslant T^-u (x_n)+c(x_n,x_{n+1})+\cc = T^- u(x_{n+1}),$$
$$T^+ u(x_{n-1}) \geqslant u(x_n)-c(x_{n-1},x_{n})\geqslant T^+u (x_n)-c(x_{n-1},x_{n})-\cc = T^+ u(x_{n-1}).$$
In each line, the first inequality is by definition of the Lax--Oleinik semigroups, the second holds because $u\in \S$ and the last equality comes from the fact that, both $T^-u$ and $T^+u$ being continuous subsolutions, the first part of the proof of Theorem \ref{th-strict} applies.

Hence all inequalities are equalities, and taking $n=0,1$ it follows both equalities
$u(x) = T^-u(x) + \cc$ and $u(y) = T^-u(y) +\cc$ and eventually that 
$$u(y) - u(x) = T^-u(y)-T^-u(x) = c(x,y)+\cc.$$
This ends the proof.
\end{proof}
Let us conclude by a corollary of this proof:
\begin{co}\label{cst}
Let $x\in X$. The following assertions are equivalent:
\begin{enumerate}[label=(\roman*)]
\item $x\in \AA$,
\item $\forall u \in \S, \ \ u(x) = T^-u(x)+\cc$, 
\item $\forall u \in \S, \ \ u(x) = T^+u(x)-\cc$, 
\item for any $u\in \S$ and $n>0$, $u(x) = \Tn u(x)+n\cc = \TTn u(x) -n\cc$.
\end{enumerate}

\end{co}

\section{Relations to the classical theory}
\subsection{Classical setting and Lax-Oleinik semigroup}

Classical weak KAM theory takes place originally in a smooth, connected and compact manifold $M$.  We will denote by $TM$ the tangent bundle of $M$  and denote points in this set by $(x,v)\in TM$, where $x\in M$ and $v\in T_x  M$ is a vector tangent to $M$ at $x$. Similarly, $T^*M$ is the cotangent bundle of $M$, and a point of this cotangent bundle will be written  $(x,p)\in T^*M$, where $x\in M$ and $p\in T^*_x  M$ is a linear form on $T_x M$. For convenience, we will equip $TM$ with a Riemannian metric and denote by $(x,v)\mapsto \|v\|_x$ the associated norm. As $M$ is compact, all Riemannian metrics are equivalent and all results are independent of this choice. The induced distance on $M$ will be denoted by $\dd(\cdot,\cdot)$.\footnote{A simple example to keep in mind is that of the flat torus $M=\T^n = \R^n / \Z^n$. In this case, both $TM$ and $T^*M$ are isomorphic to $\T^n\times \R^n$. As a Riemannian metric, one may use the canonical Euclidian scalar product  on $\R^n$ to define a metric both on $T\T^n$ and on $T^*\T^n$.  }  One considers a Tonelli Hamiltonian, that is a function $H : T^*M \to \R$ defined on the cotangent bundle of $M$ verifying the following set of conditions:
\begin{itemize}
\item $H$ is $C^2$,
\item $H$ is strictly convex in the momentum variable, meaning that for any $(x,p)\in T^*M$ the Hessian $\partial_{pp}H(x,p)$ is positive definite.
\item $H$ is superlinear, meaning that 
$$\forall x\in M, \quad \lim_{\|p\|_x\to +\infty }\frac{H(x,p)}{\|p\|_x} = +\infty.$$
\end{itemize}
Note that in the superlinearity condition the limit is automatically uniform in $x$, thanks to the convexity of $H$ and to the compactness of $M$. Moreover, this condition depends at first sight on the choice of the Riemannian metric on $TM$, the norm of $p\in T^*_xM$ being the operator norm of $p$ and  denoted again  $\|p\|_x$ to simplify notations. However, $M$ being compact, any two Riemannian metrics are equivalent, hence the notion of superlinearity becomes independent of the initial choice.

Given this Hamiltonian, one studies two equations called {\it Hamilton--Jacobi equations}.  The evolutionary Hamilton--Jacobi equation is:
\begin{equation}\label{evo}
\begin{cases}
\partial_t u + H(x,\partial_x u) = 0 ,\\
u(0,\cdot) = u_0. \tag{EHJ}
\end{cases}
\end{equation}
Above, the unknown function is $u(t,x) : [0,+\infty ) \times M \to \R$ and $u_0 : M\to \R$ is a given continuous function called initial condition. 

The stationary Hamilton--Jacobi equation is 
\begin{equation}\label{sta}
H(x,D_x u )=\alpha, \tag{SHJ}
\end{equation}
where the unknown is $u : M\to \R$ and $\alpha \in \R$ is a given constant.

Strong ($C^1$) solutions to those equations rarely exist. For instance, if $u_0$ is smooth, one can solve the evolutionary equation by using the method of characteristics. However, shocks appear almost systematically in finite time and the solution ceases to be smooth (results about this can be found in \cite{Fa, DZJDE}). Therefore a notion of weak solutions is required and the one we retain is that of viscosity solution. We provide it for the evolutionary equation even though it is not explicitly needed. It is  left  to the reader to infer the analogous definition for the stationary equation. A good introduction to the subject is \cite{barles}:

\begin{df}\rm
\hspace{2em}
\begin{itemize}
\item
A continuous function $u : [0,+\infty) \times M \to \R$ is a viscosity subsolution to \eqref{evo} if it verifies the initial condition and if for any $C^1$ function $\phi : (0,+\infty) \times M \to \R$, if $u-\phi$ has a local maximum at $(t_0,x_0)$ then
$$\partial_t \phi(t_0,x_0) + H\big(x_0,\partial_x \phi(t_0,x_0)\big) \leqslant 0.$$
\item 
A continuous function $u : [0,+\infty) \times M \to \R$ is a viscosity supersolution to \eqref{evo} if it verifies the initial condition and if for any $C^1$ function $\phi : (0,+\infty) \times M \to \R$, if $u-\phi$ has a local minimum at $(t_0,x_0)$ then
$$\partial_t \phi(t_0,x_0) + H\big(x_0,\partial_x\phi (t_0,x_0)\big) \geqslant 0.$$
\item A continuous function $u : [0,+\infty) \times M \to \R$ is a viscosity  solution to \eqref{evo} if it is both a subsolution and a supersolution.
\end{itemize}
\end{df}
In the rest of the exposition, unless otherwise  specified, any solution, subsolution or supersolution will be implicitly understood in the viscosity sense and the adjective will be omitted.

This notion is particularly adapted to our problem for several reasons. The first one is that it is reasonable in the sense that if a solution $u$ is differentiable at $(t_0,x_0) \in (0,+\infty)\times M$ then it solves the Hamilton--Jacobi equation at that point:
$$\partial_t u(t_0,x_0) + H\big(x_0,\partial_xu (t_0,x_0)\big) = 0.$$
As it can be proved that in our setting, solutions are locally Lipschitz, viscosity solutions turn out to be almost everywhere solutions. However, beware that the converse is not true. The following theorem makes viscosity solutions particularly handy (see \cite{barles, DZJDE}):

\begin{Th}
Given a continuous function $u_0 : M\to \R$, there exists a unique solution to \eqref{evo}. This solution will be denoted $(x,t)\mapsto S^-(t) (u _0)(x)$.
\end{Th}

For any fixed $t>0$, the operator $S^-(t)$ is acting on $C^0(M,\R)$. Due to the uniqueness of solutions and  to the fact that $H$ is autonomous, $S^-$ is a semigroup, meaning that $S^-(t+s) = S^-(t)\circ S^-(s)$. It turns out it enjoys properties very similar to the discrete Lax--Oleinik semigroup $T^-$:
\begin{pr}\label{prop S}
\hspace{2em}
\begin{enumerate}
\item For any $t>0$, there exists $K>0$ such that the set $S^-(t)\big(C^0(M,\R)\big)$ contains only $K$--Lipschitz functions.
\item For any $t>0$, $S^-(t)$ commutes with addition of constants: $S^-(t)(f+k)=S^-(t)(f)+k$, for all $f\in  C^0(M,\R)$ and $k\in \R$.
\item For any $t>0$, $S^-(t)$ is order preserving: if $f\leqslant g$ then $S^-(t)f\leqslant S^-(t)g$.
\item For any $t>0$, $S^-(t)$ is $1$--Lipschitz for the sup--norm.
\end{enumerate}

\end{pr}
At this stage, the similarities between discrete weak KAM theory and the Hamilton--Jacobi equations are not clear. It comes from an explicit control--theoretic representation formula of the operators $S^-(t)$. Let us define the Lagrangian as the convex dual of the Hamiltonian:

\begin{df}\rm
The Lagrangian $L:TM\to \R$ is defined by
$$\forall (x,v)\in TM,\quad L(x,v) = \sup_{p\in T^*_x M} p(v) - H(x,p).$$
\end{df}
In this definition, the supremum is actually a maximum. The Lagrangian $L$ is termed {\it Tonelli Lagrangian} as it enjoys the same properties as $H$:
\begin{itemize}
\item $L$ is $C^2$,
\item $L$ is strictly convex in the speed variable, meaning that for any $(x,v)\in TM$ the Hessian $\partial_{vv}L(x,v)$ is positive definite.
\item $L$ is superlinear, meaning that 
$$\forall x\in M, \quad \lim_{\|v\|_x\to +\infty }\frac{L(x,v)}{\|v\|_x} = +\infty.$$
\end{itemize}

Through the Lagrangian and Hamiltonian functions we can go from $TM $ to $T^*M$ thanks to the Fenchel transform  $\LL$  defined by
\begin{equation}\label{Fenchel}
\forall (x,v)\in TM,\quad \LL (x,v) = \big(x, \partial_vL(x,v)\big) \in T^*M.
\end{equation}
This transformation is a $C^1$ diffeomorphism under the Tonelli assumptions and its inverse is given by
$$\forall (x,p)\in T^*M,\quad \LL^{-1} (x,p) = \big(x, \partial_pH(x,p)\big) \in TM.$$
 Moreover, $H$ and $L$ are also related by the formulas
 $$H\big(x,\partial_vL(x,v)\big) = \partial_vL(x,v)(v) - L(x,v) \ \ ; \ \  L\big(x,\partial_pH(x,p)\big) = p\big(\partial_pH(x,p)\big) - H(x,p).$$

\begin{Th}
Let $u : M\to \R$ be any continuous function. For any $t>0$ and $x\in M$ the following holds:
\begin{equation}\label{LOSM}
S^-(t)u (x) = \inf_{\substack{\gamma:[-t,0]\to M \\ \gamma(0)=x}} u\big(\gamma(-t)\big) +\int_{-t}^0 L\big(\gamma(s),\dot\gamma(s)\big)\dd s.
\end{equation}
\end{Th}

In this formula, called the {\it Lax--Oleinik formula}, the infimum is taken amongst absolutely continuous curves. Tonelli theory asserts that the infimum is a minimum and any such minimum turns out to be $C^2$ and to verify the Euler--Lagrange equation: $\dfrac{\dd}{\dd t} \partial_v L(\gamma, \dot \gamma) = \partial_x L(\gamma,\dot \gamma)$ (see \cite{clarke}). This equation defines a complete flow on $TM$  denoted by $\varphi_L$. It is called the {\it Euler--Lagrange flow}. Its conjugate by the Fenchel transform $\varphi_H = \LL\circ \varphi_L\circ \LL^{-1}$ is a flow on $T^*M$ called {\it Hamiltonian flow}. Its trajectories solve Hamilton's equations:
\begin{equation}\label{Hamilton}\begin{cases}
\dot{x} = \partial_p H(x,p), \\
\dot{p} = -\partial_x H(x,p).
\end{cases}
\end{equation}

The infimum in \eqref{LOSM} can be split in two by first choosing a starting point $y$ for the curves and then minimizing between $y$ and $x$. More precisely, if we define the action functional
\begin{equation}\label{action}
\forall (t,y,x)\in [0,+\infty) \times M \times M,\quad h_t(y,x) = \inf_{\substack{\gamma:[-t,0]\to M \\ \gamma(0)=x \\ \gamma(-t)=y}}\int_{-t}^0 L\big(\gamma(s),\dot\gamma(s)\big)\dd s  ,
\end{equation}
then the formula for solutions of \eqref{evo} becomes 
$$S^-(t)u(x) = \inf_{y\in M} u(y) + h_t(y,x),$$ 
which is exactly the discrete Lax--Oleinik semigroup with cost funtion $c=h_t$.

\subsection{The weak KAM Theorem and critical subsolutions}

Another important fact follows from the simple remark that a function $u : M\to \R$ is solution to the stationary Hamilton--Jacobi equation \eqref{sta} with constant $\alpha$ if and only if the function $U(t,x) = u(x)-\alpha t$ is solution to the evolutionary equation \eqref{evo} with initial condition $u_0 = u$. Hence any such solution is characterized by the property
\begin{align*}
\forall (t,x)\in (0,+\infty)\times M,\quad u(x) & = \inf_{\substack{\gamma:[-t,0]\to M \\ \gamma(0)=x}} u\big(\gamma(-t)\big) +\int_{-t}^0 \Big[L\big(\gamma(s),\dot\gamma(s)\big)+\alpha\Big]\dd s \\
&=\inf_{y\in M}u(y) +h_t(y,x) +t\alpha.
\end{align*}
Note also that such solutions verify in particular that 
\begin{equation}\label{subsol}\forall t>0,\ \forall \gamma : [-t,0]\to M, \quad u\big(\gamma(0)\big) - u\big(\gamma(-t)\big)\leqslant  \int_{-t}^0 \Big[L\big(\gamma(s),\dot\gamma(s)\big)+\alpha\Big]\dd s .
\end{equation}
In fact, verifying the above family of inequalities characterizes $u$ to be a subsolution of \eqref{sta}.

With these facts in mind, it should not come as a surprise that the original weak KAM Theorem of Fathi  (\cite{Fa4}) is similar to the discrete weak KAM Theorem we stated:
\begin{Th}\label{cweakKAM}
There exists a unique constant $\alpha(0)\in \R$ for which the stationary equation \eqref{evo} admits a solution with right hand side equal to $\alpha=\alpha(0)$.\footnote{The notation $\alpha(0)$ is borrowed from Mather's $\alpha$ function. It is a function acting on the first cohomology group of $M$. Given a closed  $1$--form $c$, one can perturb the stationary equation \eqref{sta} by $H(x,c(x)+ D_xu)=\alpha$ and prove a weak KAM theorem for this equation. The critical constant found  depends only on the cohomology class $[c]$ and is $\alpha([c])$. Discrete analogues of this are discussed in \cite{ZJMD}.}
\end{Th}
\begin{proof}
For fun's sake, we provide yet another proof of this Theorem. We deduce it from the discrete weak KAM theorem, although it is highly unnatural.

Uniqueness of $\alpha(0)$ follows from the uniqueness of the critical constant in the discrete weak KAM theorem \ref{weak KAM} as a classical weak KAM solution is also a discrete weak KAM solution for the cost $h_1$ and same critical constant. 

The first (and central) part of the proof is to establish that for $t>0$ fixed, $h_t$ is Lipschitz continuous (hence continuous). Note that the second point of Proposition \ref{prop S} also follows from that assertion. We omit this technical (and central) aspect and refer to \cite{Fa}. 

We now apply the discrete weak KAM theorem \ref{weak KAM} which states that for all $n\in \N$, there exists a unique  constant $c_n$ and a function $u_n : M\to \R$ such that  $u_n = S^-(2^{-n}) u_n + c_n$. By using the semigroup property. One obtains that 
$$u_n = \big(S^-(2^{-n})\big)^{2^n} u_n + 2^nc_n = S^-(1)u_n +2^nc_n.$$
It follows by the uniqueness of $c_0$ that for all $n\geqslant 0$, $c_n = 2^{-n}c_0$. The same argument yields that 
\begin{equation}\label{diadic}
\forall n\geqslant 0, \ \forall t \in2^{-n} \N,\quad u_n = S^-(t)u_n + tc_0.
\end{equation}
Up to adding constants to the functions $u_n$ we may assume they all vanish at some point of $M$ which does not change their property of being weak KAM solutions for $S^-(2^{-n})$.
Moreover, as all the $u_n,n\in \N$ are in the image of $S^-(1)$ they form an equi--Lipschitz family of functions. Hence by the Arzel\` a--Ascoli theorem, we may find an extracted sequence $k_n$ such that $(u_{k_n})_{n\in \N}$ converges uniformly to a function $v$. By continuity of $u\mapsto S^-(t)u$ we may pass to the limit (as $m\to +\infty$) in the equalities
$$\forall m\geqslant n, \ \forall t \in 2^{-k_n}\N, \quad u_{k_{m}} = S^-(t)u_{k_{m}}+tc_0,$$ to obtain that $v=S^-(t)v + tc_0$ for any diadic number $t$. As diadic numbers are dense in $[0,+\infty)$ the theorem follows again by continuity of $t\mapsto S^-(t)v$.
\end{proof}
We have used at the end of the proof the following result, of which we give a quick proof for completeness:
\begin{lm}
Let $u:M\to \R$ be a continuous function, then the function $t\mapsto S^-(t)u$ is uniformly continuous. 
\end{lm}
\begin{proof}
Using the semigroup property and non--expansiveness, one obtains that for all $0\leqslant s\leqslant t$, $\|S^-(t)u - S^-(s)u\|_\infty \leqslant \|S^-(t-s)u - u\|_\infty$. Therefore, it is enough to prove continuity at $0$. Moreover, again using the non--expansive character, one sees by an approximation argument that it is enough to prove the result for $u$ being a Lipschitz function.

Assume therefore that $u$ is Lipschitz continuous with Lipschitz constant $K>0$. As $L$ is superlinear, there is a constant $C>0$ such that
$$\forall (x,v) \in TM,\quad L(x,v) \geqslant K\|v\|_x -C.$$
It follows that for any absolutely continuous $\gamma : [-t,0]\to M$,
$$\int_{-t}^0 L\big(\gamma(s),\dot\gamma(s)\big) \dd s \geqslant  \int_{-t}^0 K\|\dot\gamma(s)\|_{\gamma(s)} \dd s -tC \geqslant K \dd\big(\gamma(0),\gamma(t)\big) - tC.$$
It follows that if $\gamma(0)=x$, recalling that $u$ is $K$--Lipschitz, 
\begin{multline*}
 \int_{-t}^0 \Big[L\big(\gamma(s),\dot\gamma(s)\big)+\alpha\Big]\dd s +u\big(\gamma(-t)\big)- u(x) \geqslant
 \\
 \geqslant K \dd\big(\gamma(0),\gamma(t)\big) - tC - K \dd\big(\gamma(0),\gamma(t)\big) 
\end{multline*}
and taking an infimum on all curves,
$S^-(t)u(x) - u(x) \geqslant -tC$.

Finally, comparing with a constant curve in the definition of the Lax--Oleinik semigroup one finds that
$$S^-(t)u(x)\leqslant u(x)+\int_{-t}^0 L(x,0)\dd s \leqslant u(x) +t\max_{y\in M} L(y,0).$$
Those two inequalities prove the lemma. 
\end{proof}
One may wonder if there is a relationship between discrete weak KAM solutions and weak KAM solutions. It turns out that the answer is yes and it is closely related to the autonomous aspect of our setting\footnote{The second part of the following Theorem becomes tautological when considering time--dependent $1$--periodic Hamiltonians as the definitions of weak KAM solutions  then coincide.}.
\begin{Th}\label{ewd}
Let $c = h_1$ be the cost function. Then we have $\alpha(0) = c[0]$. Moreover a function $u$ is a discrete weak KAM solution for $c$ if and only if it is a weak KAM solution for $H$.
\end{Th}
The proof heavily relies on Fathi's Theorem \cite{Fa1}:
\begin{Th}[Fathi \cite{Fa1}]\label{cvsg}
Let $v:M\to \R$ be a continuous function. Then $t\mapsto S^-(t)v+t\alpha(0)$ uniformly converges to a weak KAM solution, for $H$,  as $t\to +\infty$.
\end{Th}
\begin{proof}[Proof of Theorem \ref{ewd}]
It is clear that if $u$ is a  weak KAM solution, then $u=S^-(1)u+\alpha(0) = T^-u +\alpha(0)$. It follows that $\alpha(0)=c[0]$ because of the uniqueness in Theorem \ref{weak KAM}. Moreover, any weak KAM solution is a discrete weak KAM solution.

Let now $v$ be a discrete weak KAM solution. It follows that $v=S^-(1)v+\alpha(0)$ and then that 
$$\forall n\in \N,\quad v=S^-(n)v+n\alpha(0).$$
By Fathi's theorem, there exists a weak KAM solution $\tilde v$ such that $S^-(t)v+t\alpha(0) \to \tilde v$ as $t\to +\infty$. It then follows that $v=\tilde v$ is a weak KAM solution.
\end{proof}

If $\alpha \in \R$, denote by $\SSS_\alpha$ the set of subsolutions to \eqref{sta}, or equivalently functions verifying \eqref{subsol}. Note that $u\in \SSS_\alpha$ if and only if $t\mapsto S^-(t)u+t\alpha$ is non--decreasing. Finally, $\SSS$ will denote the special set of critical subsolutions $\SSS_{\alpha(0)}$. As will be seen later, unlike what happens for weak KAM solutions, the set $\SSS_\alpha \subset \S_\alpha$ is much smaller than its discrete analogue. The first point of the next proposition is a hint as to why. We state without proof analogous results to Proposition \ref{propS}:
\begin{pr}\label{PropS}
Let $\alpha\in\R$, the following assertions hold:
\begin{enumerate}[label=(\roman*)]
\item The family of $\alpha$--subsolutions is equi--Lipschitz.
\item The set $\SSS_\alpha$ of $\alpha$--subsolutions is closed (with respect to pointwise and uniform convergence), convex and stable by the Lax--Oleinik semigroup: $S^-(t)(\SSS_\alpha)\subset \SSS_\alpha$ for all $t\geqslant 0$.
\end{enumerate}
\end{pr} 

And here is the characterization of $\alpha(0)$ similar to Proposition \ref{criticalSS}:

\begin{pr}\label{CriticalSS}
The following equality holds: $\alpha(0) = \min\{ \alpha\in \R, \ \ \SSS_\alpha\neq \varnothing\}$.
\end{pr}

We finish this section by discussing the continuous way of adapting the second proof of the discrete weak KAM theorem. Indeed, if $\lambda \in (0,1)$, considering the mapping $u\mapsto S^-(1)(\lambda u)$ does not make much sense from the PDE point of view.

The function $u_\lambda$ constructed in the second proof of Theorem \ref{weak KAM} satisfies $T^-(\lambda u_\lambda) = u_\lambda$. By setting $v_\lambda = \lambda u_\lambda$, the previous equation may be rewritten as 
$$v_\lambda - T^-(v_\lambda) = (\lambda-1)u_\lambda = (1-\lambda^{-1})v_\lambda.  $$

We now follow the intuition that $T^-$stands for the time $1$ of an evolution semigroup, hence $v_\lambda - T^-(v_\lambda)$ can be interpreted as a discrete derivative with respect to time.

Going back to the Hamilton--Jacobi equation, and following the previous analysis, we are looking for a function $u : M\to \R$ such that 
$$\frac{\dd}{\dd t} S^-(t)u _{|t=0} = -\ell u,$$
where we applied the change of variable $\ell =\lambda^{-1}-1$. Hence $\ell>0$ is aimed to go to $0$. Remembering that $(t,x) \mapsto S^-(t)u(x)$ solves the Hamilton--Jacobi equation, we are actually looking for a function $u$ solving
\begin{equation}\label{discountedHJ}
\ell u(x)+H(x,D_xu)=0,\quad x\in M \tag{$\ell$HJ}
\end{equation}
in the viscosity sense. The preceding equation is called the {\it discounted equation} and $\ell$ is called the {\it discount factor}. It turns out this is precisely the method used in \cite{LPV} which is historically the first  paper where weak KAM solutions appear. In this foundational preprint, they prove   the following results, for a wider class of Hamiltonians (in particular, no convexity is required):

\begin{Th}\label{eqdiscounted}
\hspace{2em}
\begin{enumerate}
\item For all $\ell>0$ there exists a unique function $U_\ell : M\to \R$ which is a viscosity solution of \eqref{discountedHJ}.
\item The family $(\ell U_{\ell})_{\ell>0}$ is equi-bounded.
\item The family $( U_{\ell})_{\ell>0}$ is equi-Lipschitz.
\item Given $x_0\in M$ and setting $v_\ell = U_\ell - U_\ell(x_0)$, it follows that the family $( v_{\ell})_{\ell>0}$ is relatively compact.
\item The family $(\ell U_{\ell})_{\ell>0}$ uniformly converges to the constant function $-\alpha(0)$\footnote{In the preprint, this constant is actually denoted by $-\overline{H}(0)$ and $\overline H$ is called the {\it effective Hamiltonian}.} as $\ell\to 0$ and any limit function $v_0$ of the family  $( v_{\ell})_{\ell>0}$, as $\ell\to 0$, is a solution of  \eqref{sta} with right hand side $\alpha(0)$.
\end{enumerate}
\end{Th}

\subsection{The positive Lax--Oleinik semigroup}

As in our discrete setting, Fathi introduced the positive Lax--Oleinik semigroup as follows:

\begin{df}\rm
Let $u : M\to \R$ be any continuous function, for any $t>0$ and $x\in M$ we define:
$$S^+(t)u (x) = \sup_{\substack{\gamma:[0,t]\to M \\ \gamma(0)=x}} u\big(\gamma(t)\big) -\int_{0}^t L\big(\gamma(s),\dot\gamma(s)\big)\dd s= \sup_{y\in M} u(y) - h_t(x,y),$$
where the supremum is taken amongst absolutely continuous curves.
\end{df}

Once again, this semigroup has a natural PDE interpretation. Let us introduce the Hamiltonian: 
$\widecheck H : (x,p)\mapsto H(x,-p)$. One verifies that the associated Lagrangian is given by
$$\forall (x,v)\in TM,\quad \widecheck L(x,v) = \sup_{p\in T^*_x M}  p(v) - \widecheck H(x,p) =  \sup_{p\in T^*_x M}  p(-v) - H(x,p) =L(x,-v).$$
 Therefore  the positive semigroup is written as follows:
 \begin{multline*}
 S^+(t)u (x) = -\inf_{\substack{\gamma:[0,t]\to M \\ \gamma(0)=x}} -u\big(\gamma(t)\big) +\int_{0}^t L\big(\gamma(s),\dot\gamma(s)\big)\dd s \\ 
 =  -\bigg[\inf_{\substack{\check\gamma:[-t,0]\to M \\ \check \gamma(0)=x}} -u\big(\check\gamma(t)\big) +\int_{-t}^0 \widecheck L\big(\check\gamma(s),\dot{\check\gamma}(s)\big)\dd s\bigg]= -\widecheck S(t)(-u)(x),
 \end{multline*}
 where $\widecheck S$ denotes the Lax--Oleinik semigroup associated to $\widecheck H$. Hence the function $(t,x) \mapsto -S^+(t)u(x)$ solves 	an evolutionary Hamilton--Jacobi equation with Hamiltonian $\widecheck H$ and initial data $-u$. 
 
As the Hamiltonian $\widecheck H$ is also Tonelli, it is automatic that the positive semigroup $S^+$ has the same properties as $S^-$ stated in Proposition \ref{prop S}, that we do not recall here. As for the discrete case, a positive weak KAM theorem follows:

\begin{Th}[positive weak KAM]\label{C.weakKAM+}
The critical constant $\alpha(0)$ is the only one for which the equation $u =  S^+(t) u - t\alpha(0)$, for all $t>0$, admits solutions $u:M\to \R$. Moreover the constant $\alpha(0)$ has the following caracterization: 
$$\forall v\in \BB(M,\R),\quad \frac{S^+(t) v }{t} \underset{t\to +\infty}{\longrightarrow}  \alpha(0),$$
 and the convergence is uniform.

\end{Th} 

\subsection{Strict subsolutions, Aubry sets}

As is expected,  given a constant $\alpha \in \R$, a function $u : M\to \R$ is an $\alpha$--subsolution ($u\in  \SSS_\alpha$) if
\begin{equation*}
\forall t>0,\ \forall \gamma : [-t,0]\to M, \quad u\big(\gamma(0)\big) - u\big(\gamma(-t)\big)\leqslant  \int_{-t}^0 \Big[L\big(\gamma(s),\dot\gamma(s)\big)+\alpha\Big]\dd s ,
\end{equation*}
where $\gamma $ ranges in the set of absolutely continuous curves. It can be established that   subsolutions are automatically Lipschitz hence differentiable almost everywhere.

The terminology comes from the fact that subsolutions are indeed viscosity subsolutions to the critical stationary Hamilton--Jacobi equation \eqref{sta}. Whence the following characterizations of subsolutions hold:

\begin{pr}
Let $u : M\to \R$ and $\alpha\in \R$ be a constant. The following are equivalent:
\begin{enumerate}
\item $u\in \SSS_\alpha$ is an $\alpha$--subsolution;
\item the family of functions $\big(S^-(t)u +t\alpha \big)_{t\geqslant 0}$ is non--decreasing;
\item  the family of functions $\big(S^+(t)u -t\alpha \big)_{t\geqslant 0}$ is non--increasing;
\item the function $u$ verifies $H(x,D_x u)\leqslant \alpha$ in the viscosity sense;
\item the function $u$ verifies $H(x,D_x u)\leqslant \alpha$ for almost every $x\in M $ (more precisely at every $x\in M$ such that $u$ is differentiable at $x$).
\end{enumerate}
\end{pr}

We now focus on the particular case of critical subsolutions. The idea is that, because the constant $\alpha(0)$ is a threshold between a world with subsolutions and a world without, there is no critical subsolution where the inequality $H(x,D_x u)\leqslant \alpha(0)$ is everywhere strict. More precisely, the obstruction to having strict inequalities is concentrated on a subset of $M$: the projected Aubry set. This set was introduced for twist maps (that are a particular discrete setting) by Aubry, Le Daeron and Mather \cite{ Aubry,Mather,MaFo}. For Hamiltonian systems, it was later on introduced by Mather (\cite{Mather1}) by dynamical systems means. The present interpretation, using critical subsolutions, is due to Fathi (\cite{Fa,Fa4,Fa3,Fa5}).

A fundamental result on subsolutions is due to Fathi and Siconolfi \cite{FSC1,FS05} (for $C^1$ subsolutions) and was then improved by Bernard (\cite{BernardC11}):

\begin{Th}\label{strictC}
There exists a function $u_0 : M\to \R$ that is $C^{1,1}$ ($C^1$ with Lipschitz differential) and is a critical subsolution ($u_0\in \SSS_{\alpha_0}$). Moreover it verifies the following  fundamental property:

if for some $x\in M$, $H(x,D_xu_0) = \alpha(0)$ then if $u\in \SSS_{\alpha(0)}$ is any other critical subsolution, $u$ is differentiable at $x$ and $D_x u = D_x u_0$. Hence $H(x,D_xu) = \alpha(0)$.
\end{Th}
 
 The function $u_0$ above is a $C^{1,1}$ strict subsolution. The set 
 $$\AA = \{x\in M , \ \ H(x,D_xu_0) = \alpha(0)\},$$
  is called the {\it projected Aubry set}. It is non--empty \big(otherwise $u_0$ would be a subsolution for a constant less that $\alpha(0)$\big). It is also compact. The Aubry set is $\AA^* = \{(x,D_xu_0), \ \ x\in \AA\}\subset T^*M$. Of course, $\AA^*\subset H^{-1}(\{\alpha(0)\})$. This is actually a deep fact. It implies a Theorem of Carneiro (\cite{carneiro})  which we  come back to later. It may seem at first glance that those sets depend on $u_0$ but it is not the case. Indeed, $\AA$ is a set of points where all subsolutions $u$ are differentiable\footnote{This actually suffices to characterize the projected Aubry set as proven in \cite{FSC1}.} and $ H(x,D_xu) = \alpha(0)$. The Aubry set in $TM$ is then naturally $\AA' = \LL^{-1}(\AA^*)$.
  
  As we saw in the discrete setting, points of the projected Aubry set come in whole sequences, giving rise to the Aubry set (subset of $X^\Z$). Moreover, this Aubry set is invariant by the shift on $X^\Z$. Similar phenomena arise in the classical setting, the dynamical systems here being those of $\varphi_L$ the Euler--Lagrange flow and $\varphi_H$ the Hamiltonian flow:
  
  \begin{Th}\label{invA}
  The Aubry set $\AA^*$ is invariant by $\varphi_H$ meaning that for all $t\in \R$, $\varphi_H^t(\AA^*) = \AA^*$. Equivalently, the Aubry set $\AA'$ is invariant by $\varphi_L$ meaning that for all $t\in \R$, $\varphi_L^t(\AA') = \AA'$.
  
  It follows that if $(x,p)\in \AA^*$, $u\in \SSS_{\alpha(0)}$ is a critical subsolution ant $t\in \R$, then 
  $\big(x(t), p(t)\big) = \big(x(t), D_{x(t)}u\big)$ where we set $\big(x(t),p(t)\big) = \varphi_H^t(x,p)$. In particular, we stress that $u$ is automatically differentiable at $x(t)$ for all $t\in \R$.
  
  The curve $\big(x(t), \dot x(t)\big)_{t\in \R}$ is a trajectory of the Euler--Lagrange flow. Moreover, its projection on $M$ calibrates $u$ in the sense that 
  $$\forall s<t,\quad u\big(x(t)\big)-u\big(x(s)\big) = \int_s^t L\big(x(\sigma),\dot x(\sigma)\big)\dd \sigma +(t-s)\alpha(0).$$
  In particular, recalling that $u$ verifies \eqref{subsol}, it follows that 
  \begin{align*}
  \forall t>0,\quad & S^-(t)u(x) = u\big(x(-t)\big) +  \int_{-t}^0 L\big(x(\sigma),\dot x(\sigma)\big)\dd \sigma +t\alpha(0)=u(x)+t\alpha(0); \\
   &  S^+(t)u(x) = u\big(x(t)\big) -  \int_0^t L\big(x(\sigma),\dot x(\sigma)\big)\dd \sigma -t\alpha(0)=u(x)-t\alpha(0).
    \end{align*}

  \end{Th}

From this follows the analogue of Corollary \ref{cst}:

\begin{co}\label{ccst}
Let $x\in M$. The following assertions are equivalent:
\begin{enumerate}[label=(\roman*)]
\item $x\in \AA$,
\item $\forall u\in \SS'_{\alpha(0)},\  \forall t>0,\quad u(x) = S^-(t)u(x)+t\alpha(0)$,
\item $\forall u\in \SS'_{\alpha(0)}, \  \forall t>0,\quad u(x) = S^+(t)u(x)-t\alpha(0)$.
\end{enumerate}
\end{co}

We end this section with a new description of the relation between  the Aubry set for the Hamiltonian and the Aubry set for the time--$1$ action functional $h_1$. For the sake of clarity, we denote them respectively $\AA_H$ and $\AA_{h_1}$.

\begin{Th}
The equality $\AA_H = \AA_{h_1}$ holds.
\end{Th}

\begin{proof}
Recall that if $n>0$, thanks to the choice of cost function, the Lax--Oleinik semigroups are linked by the equalities $T^{-n} = S^-(n)$ and $T^{+n} = S^+(n)$. Moreover, we have seen that both critical constants for $H$ and $h_1$ coincide. It follows that classical subsolutions are discrete subsolutions, $\SSS_{\alpha(0)} \subset \S_{\alpha(0)}$ (the inclusion being very often strict). Indeed, $u\in \SSS_{\alpha(0)}$ if and only if $t\mapsto S^-(t)u+t\alpha(0)$ is non--decreasing for $t\in [0,+\infty)$ while $u\in  \S_{\alpha(0)}$ if and only if the sequence $n\mapsto S^-(n)u+n\alpha(0)$ is non--decreasing for $n\geqslant 0$.

Let now $x\in \AA_{h_1}$ and $u\in \SSS_{\alpha(0)}$. We deduce from the preceding discussion and Corollary \ref{cst} that the sequence $n\mapsto S^-(n)u(x)+n\alpha(0)$ is constant. As the familly $t\mapsto S^-(t)u+t\alpha(0)$ is non--decreasing, it has to be constant. This being true for any classical subsolution, we deduce that $x\in \AA_H$ by Corollary \ref{ccst}. We have established that $\AA_{h_1} \subset \AA_H$.

Let then $x\in \AA_H$ and $u\in \S_{\alpha(0)}$ be a discrete subsolution. We set $u_- = \lim\limits_{n\to +\infty } S^-(n)u+n\alpha(0)$ and $u_+ = \lim\limits_{n\to +\infty } S^+(n)u-n\alpha(0)$ that both exist (the sequences are monotonous).  They are respectively a negative and positive weak KAM solution and verify $u_+\leqslant u \leqslant u_-$. Finally, let us set $v = \lim\limits_{t\to +\infty } S^-(t)u_++t\alpha(0)$ that is a negative weak KAM solution. As classical and discrete weak KAM solutions coincide, we know from Corollary \ref{ccst} that $u_+(x) =  S^-(t)u_+(x)+t\alpha(0)$ for all $t>0$, hence $u_+(x) = v(x)$. 

We will prove that $v=u_-$. Let $\e>0$ and $N>0$ such that $\| S^+(n)u-n\alpha(0)-u_+\|<\e$ for all $n\geqslant N$. By  application of Proposition \ref{comp} and monotonicity of the semigroups, one establishes that $S^-(n)\circ  S^+(n)u \geqslant u$. By taking $n>N$ and using the monotonicity of $S^-$ again, it follows that 
$$S^-(n)u_+ +n\alpha(0) \geqslant S^-(n)\circ S^+(n) u -\e \geqslant u-\e.$$
Then if $m>0$ we obtain that $S^-(n+m)u_+ +(n+m)\alpha(0) \geqslant S^-(m)u +m\alpha(0) - \e$. Finally letting $m\to +\infty$ it follows that $v \geqslant u_--\e$, and as this is true for all $\e>0$ we obtain that $v\geqslant u_-$.

The reverse inequality is easier: as $u_+\leqslant u_-$ then
$$\forall n>0, \quad S^-(n)u_+ +n\alpha(0) \leqslant S^-(n)u_- +n\alpha(0) =u_-,$$
 and we conclude by letting $n\to +\infty$.

Finally, it follows that $u_+(x) = v(x) = u_-(x)$ and as for $n>0$, 
$$u_+\leqslant u\leqslant S^-(n)u+n\alpha(0)\leqslant u_-,$$
 we also have $u(x) =S^-(n)u(x)+n\alpha(0) =  u_-(x)$. Thus $x\in \AA_{h_1}$ and the proof is complete.
\end{proof}

In the preceding proof the functions $u_-$ and $u_+$ form what is called a {\it conjugate pair} (see \cite{Fa3} for the classical definition and \cite{BeBu} for definitions in a discrete setting), a notion that will reappear in this text (as in Remark \ref{conjugate}).
\newpage

\chapter{More (dynamical) characterizations of the Aubry sets}
So far, we have constructed several different Aubry sets and have started to study the way they are related. Even though their initial definition stems from one particular subsolution, we already began to understand that, in the end, they entail informations concerning all subsolutions. In fact they  depend on the cost function $c$ only, and not on the particular subsolution $u_0$. In this section we push further the understanding of the meaning of those Aubry sets. 

Most results of this Chapter are to be found in \cite{Z} by the author for the most abstract and general part.  When it comes to the more regular setting of costs defined on a compact manifold first results are written in \cite{ZJMD} by the author, then generalized in \cite{BeZa} with Bernard.

\section{The Peierls Barrier} 
The aim of this paragraph is to give a characterization of the Aubry set in terms of action along long chains of points. This characterization builds on the following notion:
\begin{df}\label{Peierl}\rm
If $n>0$ is a positive integer and $(x,y)\in X\times X$, let 
$$c_n(x,y) = \inf \left\{ \sum\limits_{i=0}^{n-1} c(x_i,x_{i+1}), \quad  (x_0, x_1,\cdots , x_{n-1}, x_n)\in X^{n+1}, \ x_0 = x,\ x_n=y \right\}.$$

The {\it Peierls barrier} is the function $h : X\times X \to X$ defined by 
$$\forall (x,y)\in X\times X,\quad h(x,y) = \liminf\limits_{n\to +\infty} c_n(x,y)+n\cc.$$
\end{df}
It follows from the previous Definition that  if $u\in \BB(X,\R)$ is a bounded function, then $\Tn u(x) = \inf\limits_{y\in X} u(y) + c_n(y,x)$.

The name Peierls barrier appears in Aubry and Le Daeron's work \cite{Aubry} in the context of Conservative Twist Maps of the Annulus (see the last Chapter of this text). It associates to each rotation number a real number that vanishes if and only if there is an invariant circle with this rotation number (all those notions are defined in the last Chapter). Hence it is a barrier to the existence of invariant circles. Mather then introduced variations and  studied properties of this barrier, still in the context of Twist Maps in \cite{MatherCrit}. He went on to generalize the Peierls barrier to higher dimensional Lagrangian settings in \cite{MatherF}. The Definition just presented is analogous to what is done in the latter reference.

Let us start by giving some properties of this new object:
\begin{pr}\label{peierl-prop}
\hspace{2em}
\begin{enumerate}[label=(\roman*)]
\item $h$ is well defined and continuous;
\item for all $u\in \S$ and $(x,y)\in X$, $u(y)-u(x) \leqslant h(x,y)$, in particular for all $x\in X$, $h(x,x)\geqslant 0$;
\item for all $x,y,z\in X$ and integer $n>0$ the following inequalities hold:
\begin{multline}\label{ineqtriang}
h(x,y)\leqslant h(x,z) + c_n(z,y)+n\cc ;\\  h(x,y)\leqslant  c_n(x,z)+n\cc+h(z,y) ; \\ h(x,y)\leqslant h(x,z)+h(z,y);
\end{multline}
\item for all $x\in X$, the function $h_x = h(x,\cdot)$ is a weak KAM solution and the function $h^x = -h(\cdot ,x)$ is a positive weak KAM solution.
\end{enumerate}

\end{pr}
\begin{proof}
Let $\om$ denote a  modulus of continuity for $c$: a non--decreasing function  from $ [0,+\infty)$ to itself that is continuous at $0$, with $\om(0)=0$, such that 
$$\forall (x,y,x',y')\in X^4,\quad |c(x,y) - c(x',y')| \leqslant \om\big(d(x,x')+d(y,y')\big).$$
 Let $x,x',y,y'$ be points, $n$ an integer, and $x=x_0,x_1,\cdots ,x_{n-1},x_n=y$ such that $c_n(x,y) =  \sum\limits_{i=0}^{n-1} c(x_i,x_{i+1})$ (they exist by compactness of $X$ and continuity of the maps). Then one gets by definition that 
\begin{multline*}
c_n(x',y')-c_n(x,y) \\
 \leqslant c(x',x_1)+\sum\limits_{i=1}^{n-2} c(x_i,x_{i+1})+c(x_{n-1},y')-\sum\limits_{i=0}^{n-1} c(x_i,x_{i+1}) \\
 \leqslant \om\big(d(x,x')+d(y,y')\big).
 \end{multline*} 
 We get the same inequality for $c_n(x,y)-c_n(x',y')$ by the same argument, which proves that the $c_n$ are equicontinuous. Moreover, one checks  readily that $c_{n+1}(x,y) = T^{-n}c(x,\cdot) (y)$. Therefore, combined with the first point of Remark \ref{bounded}, this implies that the family on functions $c_n +n\cc$ is uniformly bounded. This proves (i).
 
 We have already seen in Remark \ref{rem-eg}-(iii) that if $u\in \S$ and $x,y$ are points, then for any integer $n>0$, $u(y)-u(x)\leqslant c_n(x,y)+n\cc$. Taking a liminf yields  (ii).
 
 One has by definition that if $m,n$ are integers and $x,y$ are points, then $c_{n+m}(x,y) = \inf\limits_{Z\in X} c_m(x,Z)+c_n(Z,y)$. Hence if $z$ is a third point, then 
 $$c_{n+m}(x,y)+(n+m)\cc \leqslant  c_m(x,z)+m\cc +c_n(z,y)+n\cc.$$
 Letting $m\to +\infty$ and taking liminf leads to the first inequality of (iii), $n\to +\infty$ and taking liminf to the second one, and both at the same time to the third.
 
 Let $x\in X$. The previous inequalities applied to $n=1$ show that the functions $h_x$ and $h^x$ are subsolutions. Let us  prove (iv) for $h_x$, the rest being similar. Let $y\in X$ and $k_n$ be an extraction such that $h_x (y) = \lim\limits_{n\to+\infty} c_{k_n+1}(x,y) + (k_n+1)\cc$. For each $n$, there exists $x_n\in X$ such that $c_{k_n+1}(x,y) + (k_n+1)\cc = c_{k_n}(x,x_n)+c(x_n,y) + (k_n+1)\cc$ and taking a further extraction, one may assume that $x_n\to \tilde x$ for some $\tilde x \in X$. Taking the limit and by equicontinuity, one concludes that
 $$h_x (y) = \lim\limits_{n\to+\infty} c_{k_n}(x,x_n) + k_n\cc+c(x_n,y )+\cc \geqslant h(x,\tilde x) + c(\tilde x ,y)+\cc.$$
 The function $h_x$ being a subsolution, the last inequality must be an equality and it follows that $ h_x(y) =  T^-h_x(y)+\cc$ which is what was to be proven.
\end{proof}
 We can actually strengthen (ii) in the previous Proposition as follows:
 \begin{pr}\label{strengthen}
 Let $u\in \S$ and $m,n$ two non--negative integers, then 
 $$\forall (x,y)\in X\times X,\quad  \Tn u(y) - T^{+m} u(x) +(n+m)\cc \leqslant h(x,y).$$
  \end{pr}

\begin{proof}
Let $m,n$ be any non--negative integers and $x=x_{-m},\cdots , x_n=y$ be a chain of points. Then by definition of the Lax--Oleinik semigroups, one gets
 $$\Tn u(y) - T^{+m} u (x)  \leqslant \sum_{i=0}^{n-1} c(x_i,x_{i+1}) + u(x_0) -u(x_0) + \sum_{i=-m}^{-1} c(x_i,x_{i+1}) = \sum_{i=-m}^{n-1} c(x_i,x_{i+1}).
$$
This being true for all chains of points it follows that 
$$\Tn u(y) - T^{+m} u (x)+(n+m)\cc \leqslant c_{n+m}(x,y)+(n+m)\cc.$$
 Note that the left hand side is now non--decreasing in both $n$ and $m$ (as $u\in \S$). Therefore to show the statement, we only have to show that the limit when $n,m\to +\infty$ verifies the same inequality. Let $n_k,m_k$ be two diverging increasing sequences such that  $ c_{n_k+m_k}(x,y)+(n_k+m_k)\cc \to h(x,y)$. Then one concludes that 
\begin{multline*}
\Tn u(y) - T^{+m} u (x)+(n+m)\cc \leqslant \\ 
\leqslant \lim_{k\to +\infty} T^{-n_k} u(y) - T^{+m_k} u (x)+(n_k+m_k)\cc\\
\leqslant  \lim_{k\to +\infty}  c_{n_k+m_k}(x,y)+(n_k+m_k)\cc    = h(x,y),
\end{multline*}
which is the result.
\end{proof}
\begin{rem}\label{conjugate}\rm
The previous proposition can actually be stated in a more optimal way by introducing another notion. Given a subsolution $u\in\S$, as noted, the sequences $\Tn u+ n\cc$ and $\TTn u-n\cc$ converge respectively to functions $u^-$ and $u^+$ which are respectively  a negative and a positive weak KAM solution. The result then states that $u^-(y)-u^+(x)\leqslant h(x,y)$. Such pairs are called conjugate pairs,  they coincide on $\AA$ (Corollary \ref{cst}) and this last point characterizes the pair $(u^-,u^+)$ (see proposition \ref{distlike} in the next paragraph).
\end{rem}

We are now ready to establish a characterization of the Aubry sets  much  easier to handle:
\begin{Th}\label{Aubry0}
The following equalities hold:
$$\AA=\{x\in X,\quad h(x,x)=0\},$$
$$
\widehat \AA = \{(x,y)\in X\times X , \quad c(x,y)+\cc+h(y,x)=0\}.
$$
\end{Th}
\begin{proof}[Beginning of the Proof]
We will only prove two inclusions for now. The other ones will arrive after an intermediate Proposition \ref{T+h}.

Let us call $A_1\subset X$, $A_2\subset X\times X$ the sets appearing in the right hand side of the statement.
 If $u\in \S$ and $h(x,x)=0$, by the previous Proposition both sequences $(\Tn u(x)+n\cc)_{n\in \N}$ and $(\TTn u(x)-n\cc)_{n\in \N}$ are constantly equal to $u(x)$. This implies that $x\in \AA$ (see Corollary \ref{cst}) and $A_1\subset \AA$. 
  
 Assume $(x,y)\in A_2$. Then for any $u\in \S$, summing up the inequalities $u(y)-u(x) \leqslant c(x,y)+\cc$ and $u(x)-u(y)\leqslant h(y,x)$ brings that $0\leqslant c(x,y)+\cc+h(y,x)$. As this is an equality, both inequalities were equalities and $(x,y)\in \widehat \AA$. So $A_2 \subset \widehat \AA$.
 
\end{proof}
In order to obtain the reverse inclusions, we will use the following:
\begin{pr}\label{T+h}
Let $x\in X$, then $\TTn h_x(x) -n\cc \to 0$ as $n\to +\infty$.
\end{pr}

\begin{proof}
Recall $\TTn h_x(x) -n\cc $ is a decreasing sequence ($h_x\in \S$). Moreover, as $h_x(y) - \TTn h_x(x) +n\cc \leqslant  h(x,y)$, we see that $\TTn h_x(x) -n\cc \geqslant 0$ for all $n>0$ and the limit meets the same property.

For each $n>0$, let $x=x_0^n, \cdots, x_n^n$ verify $\TTn h_x(x) = h_x(x_n^n)-\sum\limits_{i=0}^{n-1} c(x_i^n,x_{i+1}^n)$. Up to extracting, assume $x_{n_k}^{n_k} \to y$ for some $y$, passing to the limit, by definition of $h$ we obtain that:
$$\lim_{k\to +\infty} T^{+n_k}  h_x(x) - n_k\cc   \leqslant h_x(y)-h(x,y),$$
which proves the Proposition. 
\end{proof}
\begin{proof}[End of the Proof of Theorem \ref{Aubry0}]
Let $x\in \AA$, then $\TTn h_x (x)-n\cc$ is constant, hence it is $0$, so $h(x,x)=0$.

Let $(x,y)\in \widehat \AA$, then $h_y(y) - h_y(x) = c(x,y)+\cc$. But as $y\in \AA$, $h(y,y)=0$ hence $c(x,y)+\cc + h(y,x)=0$.
\end{proof}

Theorem \ref{Aubry0} gives a more concrete characterization of the Aubry set. Pairs $(x,y)$ in the $2$--Aubry set are starting pairs of arbitrarily long loops of points with arbitrarily small cost (for the cost $c+\cc$).

As a final byproduct of the preceding proofs we obtain:
\begin{pr}\label{c=h}
If $(x,y)\in \widehat\AA$ then $c(x,y)+c[0]=h(x,y)$. 

More generally, if $(x_n)_{n\in \Z}\in \widetilde \AA$, then
$$\forall m<n,\quad \sum_{k=m}^{n-1} c(x_k,x_{k+1})+(n-m)c[0] = h(x_m,x_n).$$
\end{pr}

\begin{proof}
As the function $h_x$ is a critical subsolution we obtain that if $(x,y)\in \widehat\AA$, $h_x(y) - h_x(x) = c(x,y)+c[0]$. As $h(x,x)=0$ we get $h(x,y) = c(x,y)+c[0]$.

For the second equality, we sum the equalities $h_{x_m}(x_{k+1}) - h_{x_m}(x_k) = c(x_k,x_{k+1})+c[0]$ to obtain
$$h(x_m,x_n) = h_{x_m}(x_n) - h_{x_m}(x_m) =  \sum_{k=m}^{n-1} c(x_k,x_{k+1})+(n-m)c[0] .$$
\end{proof}

\section{Examples of points in the Aubry sets}

We will now be more specific about the type of points that belong to the Aubry set. Most of them appear as limit points of minimizing chains. The following lemma is most useful:
\begin{lm}\label{chains}
Let $u:X\to \R$ be a weak KAM solution, then for all $x\in X$ there exists an infinite sequence $(x_{-n})_{n\geqslant 0}$ with $x_0=x$ such that
\begin{equation}\label{eqchain}
\forall n\in \N,\quad u(x)=u(x_{-n})+\sum_{i=-n}^{-1} c(x_i,x_{i+1}) +n\cc.
\end{equation}
We call such sequences, calibrating sequences for $u$.
\end{lm}
\begin{proof}
It is an immediate consequence of successive applications of the fact that for all $x\in X$ there is $y\in X$ such that $u(x) = u(y)+c(y,x)+\cc$.
\end{proof}

\begin{pr}\label{alpha}
Let $u : X\to \R$ be a weak KAM solution and $(x_n)_{n\geqslant 0}$ a sequence given by Lemma \ref{chains}, then $\alpha \big((x_{-n})_{n\geqslant 0}\big)\subset \AA$. If  we set $(\xi_{-n}) = \big((x_{-n-1},x_{-n})\big)_{n\geqslant 0}$ the sequence of pairs of successive points of $(x_{-n})_{n\geqslant 0}$, then $\alpha \big((\xi_{-n})_{n\geqslant 0}\big)\subset \widehat\AA$. Here $\alpha$ denotes the $\alpha$--limit sets, that is limits of converging subsequences $x_{-n_k}$ with $n_k\to +\infty$.
\end{pr}
\begin{proof}
Let $y\in \alpha \big((x_{-n})_{n\geqslant 0}\big)$. There exists an extraction $n_k\to +\infty$ such that $x_{-n_k}\to y$. Up to taking a further extraction if necessary, we may assume that $n_{k+1}-n_k \to +\infty $ is an increasing sequence as well.
 
 It follows from \eqref{eqchain} that
 \begin{equation}\label{alphA}
 \forall k\in \N,\quad u(x_{-n_k}) -u(x_{-n_{k+1}})=\sum_{i=-n_{k+1}}^{-n_k-1}c(x_i,x_{i+1})+(n_{k+1}-n_k)\cc. 
 \end{equation}
 Passing to the limit and by continuity of the functions at stake imply that $u(y)-u(y) = 0 \geqslant h(y,y)$, hence $h(y,y)=0$.
 
 The same proof holds for the second part of the proposition. Late us take an extraction 
$n_k\to +\infty$ such that $\xi_{-n_k}$ converges to $(y,y')$ and such  that $n_{k+1}-n_k \to +\infty $ is an increasing sequence as well. Passing to the limit in \eqref{alphA}, we get
 that $0 \geqslant c(y,y')+\cc + h(y',y)$. This means $(y,y')\in \widehat \AA$.
\end{proof}
\begin{rem}\rm
The same results (with same proof) hold for $\omega$--limit sets of analogous sequences $(x_n)_{n\geqslant 0}$ for positive weak KAM solutions.

If $u$ is a subsolution, the same results also hold for $\alpha$ and $\omega$ limit sets of elements of $\widetilde \AA_u$. 
\end{rem}
We may now state a fundamental property of the Aubry set: 
\begin{Th}\label{uniqueness}
Let $u$ and $v$ be respectively a weak KAM solution and a subsolution such that $u_{|\AA} \geqslant  v_{|\AA}$, then $u\geqslant v$.

Let $u$ and $v$ be two weak KAM solutions such that $u_{|\AA} = v_{|\AA}$, then $u=v$. We say $\AA$  is a uniqueness set for the critical equation.

\end{Th}
\begin{proof}
Let $x_0\in X$ and let $(x_{-n})_{n\geqslant 0}$ be a calibrating sequence for $u$ \big(see \eqref{eqchain}\big). It then comes
\begin{align*}
\forall n \in \N,\quad & u(x_0)-u(x_{-n}) = \sum_{i=-n}^{-1} c(x_i,x_{i+1}) +n\cc,\\
&v(x_0)-v(x_{-n}) \leqslant \sum_{i=-n}^{-1} c(x_i,x_{i+1}) +n\cc.
\end{align*}
Subtracting yields $u(x_0)-v(x_0) \geqslant u(x_{-n})-v(x_{-n})$. As limiting points of the sequence $(x_{-n})_{n\geqslant 0}$ are in the projected Aubry set $\AA$, taking a suitable converging subsequence and passing to the limit brings $u(x_0)-v(x_0) \geqslant 0$ since for $u$ and $v$, the same inequality holds on $\AA$. This proves the first result as $x_0$ is arbitrary.

In the second case, by symmetry, the opposite inequality holds, and the result follows as $x_0$ was taken arbitrarily.
 
\end{proof}
Actually, a reciprocal statement can be proven. Being a subsolution on the Aubry set is the only obstruction to the existence of a weak KAM solution with prescribed values on $\AA$.
\begin{pr}\label{distlike}
 Let $f : \AA \to \R$ be a function such that $f(y)-f(x)\leqslant h(x,y)$ for all $x$ and $y$ in $\AA$. Then there exists a weak KAM solution $u$ such that $u_{|\AA} = f$. 
\end{pr}
\begin{proof}
Let us check that the function $u$ defined by $u(x) = \inf\limits_{y\in \AA} f(y) + h(y,x)$ satisfies the requirements. Let $x\in \AA$. Then $f(y)+h(y,x) \geqslant f(x)-h(y,x)+h(y,x)$ by hypothesis. So $u(x) = f(x)$, by taking $y=x$ in the definition of $u$. The fact that $u$ is a weak KAM solution is a consequence of the general fact that an infimum of weak KAM solutions is a weak KAM solution applied to the family of functions $f(y) + h_y$. This fact is proved in the next lemma. 
\end{proof}
\begin{lm}\label{inf}
Let $(u_\alpha)_{\alpha\in A}$ be a family of subsolutions. Set $\overline u = \sup\limits_{\alpha\in A} u_\alpha$ and $\underline u = \inf\limits_{\alpha\in A} u_\alpha$. Then given that those functions are well defined, the following assertions hold:
\begin{enumerate}[label=(\roman*)]
\item The functions $\overline u$ and $\underline u$ are subsolutions.
\item If all the $u_\alpha$ are weak KAM solutions then $\underline u$ is a weak KAM solution.
\end{enumerate}
\end{lm}
\begin{proof}
By definition,
\begin{multline*}
T^- \underline u(x)=\inf_{y\in X} \inf_{\alpha\in A} u_\alpha(y) + c(y,x) \\
= \inf_{\alpha\in A} \inf_{y\in X} u_\alpha(y) + c(y,x) \\
= \inf_{\alpha\in A} T^-u_\alpha(x) \geqslant \inf_{\alpha\in A} u_\alpha(x) - \cc = \underline u(x)-\cc.
\end{multline*}
It was used that two infimums commute and that the $u_\alpha$ are subsolutions. Hence $\underline u\in \S$. Moreover if all the $u_\alpha$ are weak KAM solutions, then the inequality above  turns out to be an equality and $\underline u$ is a weak KAM solution.

In a similar manner, 
\begin{multline*}
T^- \overline u(x)=\inf_{y\in X} \sup_{\alpha\in A} u_\alpha(y) + c(y,x) \\
\geqslant \sup_{\alpha\in A} \inf_{y\in X} u_\alpha(y) + c(y,x) \\
= \sup_{\alpha\in A} T^-u_\alpha(x) \leqslant \sup_{\alpha\in A} u_\alpha(x) - \cc = \overline u(x)-\cc.
\end{multline*}
\end{proof}

\section{Regularity of subsolutions}
In our discrete setting, the Aubry set enjoys the additional feature to be a set where all subsolutions present systematic regularity properties. Away from the Aubry set, this is false; subsolutions may fail to be continuous  as was shown in \cite{Z}:
\begin{Th}\label{reg}
Let $x\in X$ be a non--isolated point. Then $x\in \AA$ if and only if  all subsolutions $u\in \S$ are continuous at $x$.
\end{Th}
\begin{proof}
If $x\in \AA$ then we have seen that $T^+u - \cc \leqslant u \leqslant T^-u +\cc$ and that equalities hold at $x$. As the left and right terms of the inequalities are continuous, then the famous Sandwich theorem implies that $u$ is continuous at $x$.

The hypothesis that $x$ is not isolated is used in the reciprocal (were $x$ isolated, $u$ could not be discontinuous at $x$). Assume now $x\notin \AA$. Let $u_0$ be the continuous subsolution constructed in Theorem \ref{th-strict}. In particular, the inequality $u \leqslant T^-u +\cc$ is now strict at $x$ and both functions are continuous. It follows from the in--between lemma (proved after) that any function $v$ such that $u_0 \leqslant v\leqslant  T^-u_0 +\cc$ will be a subsolution. In particular, it can be constructed discontinuous at $x$, for example, take $u_0=v$ everywhere, except at $x$ and $v(x) = T^- u_0(x)+\cc$.
\end{proof}
\begin{lm}\label{in--between}
Let $u$ be a subsolution and $v : X\to \R$ such that $u\leqslant v \leqslant T^-u+\cc$, then $v\in \S$.
\end{lm}
\begin{proof}
It follows from the monotonicity of the Lax--Oleinik semigroup:
$$u \leqslant v\leqslant  T^-u +\cc \leqslant T^- v+\cc.$$
\end{proof}

\section{More regularity for subsolutions}\label{more}
Of course, if  more precise regularity results on subsolutions are aimed for, some structure has to be added. Until further notice, we will take as base space a compact smooth manifold $M$ and $c$ will be a cost on $M\times M$. Let us recall a fundamental definition first:
\begin{df}\label{semiconcave}\rm
\begin{enumerate}[label=(\roman*)]
\item Let $\Omega \subset \R^n$ be a convex open set. A function $f: \Omega \to \R$ is said to be $K$--{\it semiconcave} if the function $x\mapsto f(x) - K\|x\|^2$ is concave (the norm  used is the Euclidean one). A function is {\it semiconcave}\footnote{The notion we refer to here is sometimes called semiconcave with linear modulus. For more details and proofs that are omitted here see \cite{CaSi00}. Other good references are also the unavoidable \cite{Fa} and the appendix of \cite{FaFi}.} if it is $K$--semiconcave for some $K\in \R$.

\item A function $f:\Omega \to \R$ is {\it locally semiconcave} if each $x\in \Omega$ belongs to a neighborhood $V_x$ such that the restriction $f_{|V_x}$ is semiconcave. 

\item A function $f : M\to \R$ is {\it locally semiconcave} if for all coordinate patch $\varphi : U\subset \R^n \to M$, the function $f\circ \varphi$ is locally semiconcave ($n$ is assumed to be the dimension of $M$).
\end{enumerate}
\end{df}

\begin{rem}\rm
The property of being locally semiconcave is invariant by $C^2$ diffeomorphisms. Therefore, in the previous definition, it is enough  that $f\circ \varphi_i$ be locally semiconcave for $\varphi_i : U_i \to M$ where the  $\varphi_i(U_i)$ form any finite open cover of $M$. Of course, this property is  much easier to establish.
\end{rem}
As concave functions can be characterized as functions whose graphs admit a hyperplane tangent from above  at every point,  the following is implied:  
\begin{pr}\label{superdiff}
A function $f:\Omega \to \R$ is $K$--semiconcave if and only if for all $x\in \Omega$, there exists a linear form $p_x\in \R^{n*}$ such that 
\begin{equation}\label{concaveineq}
\forall y\in \Omega, \quad f(y) \leqslant f(x) +p_x(y-x) +K\|y-x\|^2.
\end{equation}

\end{pr}
The same holds true for locally semiconcave functions, but restricting to a neighborhood of $x$ only. For locally semiconcave functions on a manifold $M$, $p_x\in T_x^*M$ becomes an element of the cotangent fiber at $x$ and the inequality is true in a chart.

We call {\it superdifferential} of a function  $f$, assumed to be semiconcave (resp. locally semiconcave), at $x$ \big(denoted $\partial^+ f (x)$\big) the set of $p_x$ such that \eqref{concaveineq} holds (resp. in a neighborhood of $x$ or in a chart). 

We state without proof:
\begin{pr}
Let $f: M\to \R$ be locally semiconcave and $x\in \R$. Then $\partial^+ f (x)$ is not empty, closed and convex. Moreover, $f$ is differentiable at $x$ if and only if $\partial^+ f (x)$ is a singleton (which then contains only $D_x f$).
\end{pr}

A very easy, though important, property of semiconcave functions is:
\begin{pr}\label{infconc}
If $(f_\alpha)_{\alpha \in A}$ is a family of $K$--semiconcave functions on $U\subset \R^n$ then $\inf\limits_{\alpha\in A} f_\alpha$ is $K$--semiconcave as soon as it is well defined.
\end{pr}
The proof follows bearly  the analogous property of concave functions. Of course, each notion defined previously has a  {\it semiconvex} counterpart which is defined by replacing concave by convex and $-$ signs by $+$. The opposite of  a semiconcave function is then semiconvex and vice versa. A semiconvex function $f$ has a {\it subdifferential} at each point denoted by $\partial^-f(x)$. It coincides with $-\partial^+(-f)(x)$.

Motivated by the above, one defines sub--and superdifferentials for general functions:
\begin{df}\rm
Let $f : U\to\R$ be a function defined on an open set of $\R^n$ the {\it superdifferential } $\partial^+ f(x)$ \big(resp. {\it subdifferential } $\partial^-f(x)$\big) of $f$ at $x\in U$ is the set of $D_x\varphi$ where $\varphi : U\to \R$ is differentiable at $x$ and verifies that $\varphi \geqslant f$ (resp. $\varphi \leqslant f$) with equality at $x$.
\end{df}
\begin{rem}\rm
In the previous definition, the functions $\varphi$ can be taken equivalently $C^1$. Sub--and superdifferentials are convex and closed. Moreover, $f$ is differentiable at $x$ if and only if they are both non--empty. In this case,  $\partial^+ f(x)= \partial^- f(x)=\{D_xf\}$. As a locally semiconcave function has non--empty superdifferentials, we infer that if $f$ is locally semiconcave, then $f$ is differentiable at $x$ if and only if $\partial^-f(x)$ is non--empty.
\end{rem}

 Here is a not so obvious property that explains the nature of some results:
\begin{pr}\label{semcc}
A function $f : M\to \R$ is $C^{1,1}$ (differentiable with Lipschitz differential) if and only if it is both locally semiconcave and locally semiconvex.
\end{pr}

\begin{df}\label{unifsemiconcave}\rm
A family of functions $f_\alpha : M\to \R$ for  $\alpha\in A$ is said equi--locally semiconcave if $M$ can be covered by finitely many open charts  $\varphi_i(U_i)$, where $U_i\subset \R^n$ and if there are constants $K_i$ such that all $f_\alpha\circ \varphi_i$ are $K_i$--semiconcave. 
\end{df}
\setlength{\fboxrule}{2pt}\fbox{
\begin{minipage}{0.9\textwidth}
{\bf Hypothesis}: In the rest of this section, we assume that the families $c(x,\cdot)$ and $c(\cdot , x)$ for $x\in M$ are equi--locally semiconcave. 
\end{minipage}
}

\vspace{2mm}
 It can be checked (as $M$ is compact) that a particular case of this is when the function $c$ is itself locally semiconcave on $M\times M$.

When they exist,  $\partial_1 c(x,y)$ and $\partial_2 c(x,y)$ denote the partial derivatives of $c$ at $(x,y)$ with respect to the first and second variable.

Coming back to discrete weak KAM theory, a refined version of Proposition \ref{propT} (i) then becomes:
\begin{pr}\label{propTbis}
Under the previous hypotheses, the image of $T^-$ (resp. $T^+$) consists of equi--locally semiconcave (resp. equi--locally semiconvex) functions.
\end{pr}
The proof is nothing but a direct application of Proposition \ref{infconc}. Note that the result for $T^+$ involves semiconvexity because of the minus sign in its definition.

As consequences, let us derive some further regularizing properties of the Lax--Oleinik semigroups:

\begin{pr}\label{regv}
Let $v : M\to \R$ be a continuous function and $x_0 \in M$. Let $y_0 \in M$ verify that 
$T^- v(x_0 ) = v(y_0) + c(y_0,x_0)$ \big(resp. $T^+ v(x_0 ) = v(y_0) - c(x_0,y_0)$\big). Then
\begin{itemize}
\item $\partial^+_2 c(y_0,x_0) \subset \partial^+ T^-v(x_0)$ (resp. $-\partial^+_1 c(x_0,y_0) \subset \partial^- T^+ v(x_0)$)\footnote{By $\partial^+_2 c(y_0,x_0)$ we mean the superdifferential of the map $x\mapsto c(y_0,x)$ at $x_0$.}.
\item In particular, if $T^- v$ (resp. $T^+v$) is differentiable at $x_0$ then $\partial_2 c(y_0,x_0)$ (resp. $-\partial^+_1 c(x_0,y_0) $) exists and $D_{x_0} T^- v = \partial_2 c(y_0,x_0)$ (resp. $-\partial^+_1 c(x_0,y_0) \subset D_{x_0} T^+ v$).
\item If $v$ is locally semiconcave (resp. semiconvex) then $D_{y_0} v = -\partial_1 c(y_0,x_0)$ (resp. $D_{y_0} v = \partial_2 c(x_0,y_0)$) and all the previous quantities do exist.
\end{itemize}

\end{pr}
\begin{proof}
We prove half of the results leaving the rest as an exercise.

The first point is a direct consequence of the  inequality
$$\forall x\in M, \quad T^- v(x ) \leqslant  v(y_0) + c(y_0,x),$$
which is an equality for $x=x_0$.

The second point  then stems from the proved inclusion  $\partial^+_2 c(y_0,x_0) \subset \partial^+ T^-v(x_0)$. By hypothesis, the right hand side is a singleton and the left hand side is not empty, hence they coincide and we get the result.

For the last part, note that the function $\varphi : y \mapsto  v(y) + c(y,x_0)$ reaches its minimum at $y_0$. Hence $0\in \partial^- \varphi (y_0)$. But by hypothesis, $\varphi $ is locally semiconcave and it is easily verified that 
$\partial^+ v(y_0) + \partial^+_1 c(y_0,x_0) \subset \partial^+\varphi (y_0)$. We infer that $\varphi$ is differentiable at $y_0$ with $D_{y_0} \varphi= 0$ and that necessarily, $\partial^+ v(y_0) $ and $ \partial^+_1 c(y_0,x_0)$ are singletons. Hence the result.
\end{proof}

\begin{rem}\rm \label{remsc}
The following results were actually proven and used: if $f:M\to \R$ and $g:M\to \R$ are locally semiconcave functions then
\begin{itemize}
\item $f+g$ is differentiable at  some $x\in M$ if and only if both $f$ and $g$ are; 
\item if $f+g$ reaches a local minimum at some $x\in M$, then $f$ and $g$ are differentiable at $x$.
\end{itemize}

\end{rem}

We have now the necessary material to state a more precise version of Theorem \ref{reg}:

\begin{Th}\label{reg2}
Let $x\in \AA$. Then any subsolution $u$ is differentiable at $x$. Moreover, $D_x u$ does not depend on $u$.
\end{Th}

\begin{proof}
We will use the same inequalities as in Theorem \ref{reg}. Let $u\in \S$, then $T^+u-\cc \leqslant u\leqslant T^-u + \cc$. Moreover those inequalities are equalities at $x\in \AA$. This proves that both $\partial^- u(x) \neq \varnothing$ (as $T^+u$ is locally semiconvex) and $\partial^+u (x) \neq \varnothing$ (as $T^-  u$ is locally semiconcave). Hence $u$ is differentiable at $x$. Note that $T^-u$ and $T^+ u $ being subsolutions they are also differentiable at $x$ and the first inequality above  implies that all differentials  are equal at $x$: $D_x u=D_x T^-u = D_x T^+u$.

It remains to compute this differential. Let $(x',x)\in \widehat \AA$. We know that $u(z)\leqslant  u(x') +c(x',z)+\cc$ for all $z\in M$ and equality holds at $z=x$. As $u$ is differentiable at $x$, this implies that the function $u(x') +c(x',\cdot )+\cc$ has a subdifferential at $x$, as it is locally semiconcave, it is differentiable and its differential is $\partial_2 c(x',x)$. Hence we conclude that $D_x u = \partial_2 c(x',x)$ which happens to be independent on $u\in \S$.
\end{proof}
\begin{rem}\rm
A similar proof implies that if $(x,y)\in \widehat \AA$ and $u\in \S$ then $D_x u = -\partial_1 c(x,y)$. But this is not surprising, as $x$ verifies $c(x',x )+c(x,y) = \min\limits_{z\in M}c(x',z )+c(z,y)$.

Let us stress one more time  that $c$ admits partial derivatives on the $2$--Aubry set, as was actually established.
\end{rem}
We now turn to improving Theorem \ref{th-strict}:
\begin{Th}\label{strictC11}
There exists a strict subsolution $u_1$ which is $C^{1,1}$.
\end{Th}
The proof  makes crucial use of Ilmanen's lemma (see \cite{Ilmanen,FaZaIl,BeIl,BeZa}).
\begin{Th}[Ilmanen's lemma]\label{ilm}
Given two functions $f,g : M\to \R$ such that  $f$ is locally semiconvex, $g$ is locally semiconcave and $f\leqslant g$, there exists a function $h$ which is $C^{1,1}$ such that $f\leqslant h\leqslant g$.

Moreover, if $h_0$ is a continuous function such that $f\leqslant h_0\leqslant g$, then $h$ can be constructed arbitrarily close from $h_0$.
\end{Th}

\begin{proof}[Proof of Theorem \ref{strictC11}]
The proof splits into two steps. First we construct $C^{1,1}$ subsolutions, then we explain how to make them strict.

Let us start with a subsolutions $u$. Then we have seen that $g=T^- u$ is a locally semiconcave subsolution, $f=T^+g-\cc$ is a locally semiconvex subsolution and $f\leqslant g$. By Ilmanen's lemma, there exists a $C^{1,1}$ function $h$ such that $f\leqslant h\leqslant g$. But the in between lemma \ref{in--between}, transposed to $T^+$, tells us that $h\in \S$.

Now that we have a general procedure to construct subsolutions, let us see how to make them strict. Let $u_0$ be the continuous strict subsolution given by Theorem \ref{th-strict}. Let us set $\e : M\times M \to \R$ the function defined by $\e(x,y) = c(x,y)+\cc - u_0(y)+u_0(x)$. This function is everywhere non--negative and verifies $\e^{-1}\{0\} = \widehat \AA$ thanks to the strictness property enjoyed by $u_0$. Let now $\e_1$ be a $C^{\infty}$ function such that $0\leqslant \e_1\leqslant \e$ and $\e_1^{-1}\{0\} = \widehat \AA$. Let us finally consider $\tilde c = c-\e_1$. This new cost  still verifies that the marginal functions $\tilde c(x,\cdot)$ and $\tilde c(\cdot , y)$ are locally-uniformly semiconcave. Moreover, $u_0$ is a $\cc$--subsolution\footnote{Even though it can be proven that $\cc$ is the critical constant for $\tilde c$, this fact is not useful in this proof.} for $\tilde c$, indeed 
$u_0(y)-u_0(x) = c(x,y)+\cc - \e \leqslant \tilde c(x,y)+\cc$. The first part of the proof provides  a $\cc$--subsolution $u_1$ for this cost $\tilde c$ (using the semigroups $T^-_{\tilde c}$ and $T^+_{\tilde c}$ associated to $\tilde c$) which is $C^{1,1}$. Let us verify it is strict for $c$:
for $(x,y)\in M\times M$, we compute 
$$u_1(y)-u_1(x) \leqslant \tilde c(x,y)+\cc = c(x,y)-\e_1(x,y) +\cc  \leqslant c(x,y)+\cc,$$
 and this last inequality is strict whenever $(x,y)\notin \widehat \AA$. This completes the proof.
\end{proof}
\begin{rem}\rm
The previous Theorem can be made more precise. Using the fact that $T^+T^- u\leqslant u\leqslant T^-u+\cc$ and  the last assertion  of Ilmanen's lemma, one proves that if $u$ is continuous, then it can be approximated by $C^{1,1}$ strict subsolutions.

Finally, as a nontrivial convex combination of a subsolution with a strict subsolution is strict, one obtains that $C^{1,1}$ strict subsolutions are dense in the set $\S\cap C^0(M,\R)$.

Let us also stress that, as pointed out in \cite{BeZa}, Ilmanen's lemma (Theorem \ref{ilm}) can be recovered from Theorem \ref{strictC11} by considering the cost $c_{f,g}(x,y) = g(y)-f(x)$.
\end{rem}

\section{Graph properties and dynamics on the Aubry set}
Let us begin by mentioning a first general result under the hypotheses of the previous paragraph.
A combination of Theorems \ref{reg2} and \ref{strictC11} gives the following proposition which has a flavor of Mather's Graph Theorem:

\begin{pr}\label{aubry-cot}
There exists a set $\AA^*\subset T^*M$ whose projection is $\AA$ and such that if $(x,p)\in \AA^*$ then any $u\in \S$ is differentiable at $x$ and $D_x u=p$. Moreover the projection $\AA^* \to \AA$ is a bi--Lipschitz homeomorphism.
\end{pr}
Indeed, $\AA^*$ is just the restriction of the graph of $Du_1$ to $\AA$ where $u_1$ is given by Theorem \ref{strictC11}.

In order to define a dynamics on the Aubry set, one would now like, given a point $x_0\in \AA$, to be able to reconstruct the whole sequence $(x_n)_{n\in \Z}$. To this aim, we impose an additional condition on the cost. It was studied in \cite{ZJMD} and previously introduced in the setting of Optimal Transportation in \cite{FaFi}:
\begin{df}\label{twist condition}\rm
\hspace{2em}
\begin{enumerate}[label=(\roman*)]
\item The cost $c$ has the {\it left twist property} if for any $y$, the map $x\mapsto \partial_2c(x,y)$ is injective on its domain of definition\footnote{The cost $c$ being locally Lipschitz, this map is defined almost everywhere.}.
\item The cost $c$ has the {\it right twist property} if for any $x$, the map $y\mapsto \partial_1c(x,y)$ is injective on its domain of definition.
\item The cost $c$ enjoys the {\it twist condition} if it verifies both the left and the right twist properties.
\item We define the  {\it left Legendre transform} $\L_\ell : \D_\ell\subset M\times M \to T^*M$ by $\L_\ell (x,y) = \big(y,\partial_2 c(x,y)\big)$ and the {\it right Legendre transform}  $\L_r : \D_r\subset M\times M \to T^*M$ by $\L_r (x,y) = \big(x,-\partial_1 c(x,y)\big)$, where $\D_\ell$ and $\D_r$ are the sets of full measures on which the definitions make sense. 
\end{enumerate}

\end{df}

Under this twist condition, one gets this second version of Mather's Graph theorem:

\begin{pr}\label{aubry--graph}
Let us assume that $c$ verifies the twist condition. Then both projections $\pi_i : \widehat \AA \to \AA$ are bijections.
\end{pr}
\begin{proof}
We have seen in the proof of Theorem \ref{reg2} and in the subsequent Remark that if $u\in \S$ and $(x_{-1},x_0,x_1)$ are successive points of a sequence $(x_n)_{n\in \Z}$ then $p=D_{x_0} u = \partial_2c(x_{-1},x_0) = -\partial_1c(x_0,x_1)$. It follows from the left twist condition that $\L_\ell$ is injective and that  $x_{-1} =  \pi_1\big(\L_\ell^{-1}(x_0,p)\big)$ is uniquely determined. Similarly, it follows from the right twist condition that $\L_r$ is injective and that  $x_{1} =  \pi_2\big(\L_r^{-1}(x_0,p)\big)$ is uniquely determined.
\end{proof}

\begin{rem}\rm
In Optimal Transportation, similar twist conditions are used to prove existence of optimal transport maps for semiconcave costs. In the corresponding cases, such Optimal transport maps have their graphs included in analogues of the $2$--Aubry set associated to Kantorovitch pairs. See the work of Fathi and Figalli for example \cite{FaFi}.
\end{rem}

\section{Relations to the classical theory}
This section comes back to the classical setting of a Tonelli Lagrangian $L$ defined on the tangent bundle of a closed and compact smooth manifold $M$. 
\subsection{The classical Peierls Barrier}
The Peierls barrier for Lagrangian systems was introduced by Mather in \cite{MatherF}, inspired by the works of Aubry and Le Daeron for twist maps \cite{Aubry}:
\begin{df}\rm
The Peierls barrier is defined by
$$\forall (x,y)\in M,\quad h_L(x,y) = \liminf_{t\to+\infty}h_t(x,y)+t\alpha(0),$$
where $h_t$ is the minimal action functional previously introduced in \eqref{action}.
 \end{df}

It follows from Fathi's theorem on the convergence of the Lax--Oleinik semigroup \cite{Fa1}, that in this autonomous setting, the liminf is actually a limit. Note that this is not necessarily the case in a discrete setting, or a time periodic setting, as shown in \cite{FaMat}.

\begin{pr}\label{Cpeierl-prop}
\begin{enumerate}[label=(\roman*)]
\item $h_L$ is well defined and continuous;
\item for any subsolution $u\in \SSS$ and $(x,y)\in M$, $u(y)-u(x) \leqslant h_L(x,y)$, in particular for all $x\in X$, $h_L(x,x)\geqslant 0$;
\item for all $(x,y,z)\in M^3$ and real number $t>0$ the following inequalities hold:
\begin{multline}\label{ineqtriangC}
h_L(x,y)\leqslant h_L(x,z) + h_t(z,y)+t\alpha(0) ;\\  h(x,y)\leqslant  h_t(x,z)+t\alpha(0)+h_L(z,y) ; \\ h_L(x,y)\leqslant h_L(x,z)+h_L(z,y);
\end{multline}
\item for all $x\in X$, the function $h_x = h_L(x,\cdot)$ is a weak KAM solution and the function $h^x = -h_L(\cdot ,x)$ is a positive weak KAM solution.
\end{enumerate}
\end{pr}
The proof is the same as that of Proposition \ref{peierl-prop}. Actually the links between the classical Peierls barrier and the discrete one is made even clearer by the next Proposition:
\begin{pr}
Let $h$ be the  Peierls barrier associated to the cost function $h_1$. Then $h=h_L$.
\end{pr}
\begin{proof}
Once again, the proof  heavily relies on the convergence of the Lax--Oleinik semigroup for autonomous Tonelli Lagrangians. Indeed, let $x\in M$, and define $v = h_1(x,\cdot)$. Then it follows from the definitions that if $t>1$,  
$$\forall y\in M,\quad h_t(x,y) = S(t-1)v(y).$$
Whence, 
$$\forall y\in M,\quad h_L(x,y) = \liminf_{t\to +\infty} S(t-1)v(y)+t\alpha(0)= \lim_{t\to +\infty} S(t-1)v(y)+t\alpha(0),$$
while
$$\forall y\in M,\quad h(x,y) = \liminf_{n\to +\infty} S(n-1)v(y)+n\alpha(0)= \lim_{n\to +\infty} S(n-1)v(y)+t\alpha(0),$$
and $h=h_L$.
\end{proof}

Of course, the classical Peierls barrier allows to recover the Aubry sets, as in the discrete case:
\begin{Th}\label{AubryC}
The following equalities hold:
$$\AA=\{x\in M,\quad h_L(x,x)=0\},$$
$$
\AA^* = \{(x,p)\in T^*M , \quad x\in \AA,\ \  p=D_x h_x\}.
$$

\end{Th}

Note that the first equality of the previous Theorem is actually the original definition of Mather in \cite{MatherF}.

\subsection{Examples of points in the Aubry set}

We start by reviewing links between weak KAM solutions and the Aubry set. Those results can now be interpreted as consequences of the analogous results proven in the discrete setting but they were historically obtained first by Albert Fathi.

\begin{pr}\label{cchains}
Let $u$ be a weak KAM solution, then for all $x\in M$, there exists a $C^2$ curve $\gamma_x : (-\infty , 0] \to M$ such that $\gamma_x(0)=x$ and 
$$\forall t>0, \quad u(x) = u\big(\gamma_x(-t)\big) + \int_{-t}^0 L\big(\gamma_x(s),\dot\gamma_x(s)\big) \dd s +t\alpha(0).$$
Such curves are called {\it calibrating} for $u$.
\end{pr}
Of course, a similar statement is also valid for positive weak KAM solutions.
Calibrating curves carry  points of the Aubry set in their closure ($\alpha$--limit set):
\begin{pr}\label{Calpha}
Let $u$ be a weak KAM solution, $x\in M$ and $\gamma_x : (-\infty,0]\to M$ be a calibrating curve given by the previous proposition. If $y\in \alpha(\gamma_x)$ then $y\in \AA$. Moreover, if $(y,v)\in \alpha(\gamma_x,\dot\gamma_x)$, then $(y,v)\in \AA'$.
\end{pr}

Of course a similar statement holds for positive weak KAM solutions and we let the reader infer it.

Conversely, knowing a subsolution or weak KAM solution on the Aubry set is rich in consequences:
\begin{Th}\label{Cuniqueness}
Let $u$ and $v$ be respectively a weak KAM solution and a subsolution such that $u_{|\AA} \geqslant  v_{|\AA}$. Then $u\geqslant v$.

Let $u$ and $v$ be two weak KAM solutions such that $u_{|\AA} = v_{|\AA}$. Then $u=v$.
\end{Th}
Note that thanks to the parallels made between discrete and classical theories, this Theorem is weaker than Theorem \ref{uniqueness}, as there are more discrete subsolutions than classical ones.

Finally, let us  recall the converse to this Theorem:

\begin{pr}\label{Cdistlike}
 Let $f : \AA \to \R$ be a function such that $f(y)-f(x)\leqslant h(x,y)$ for all $x$ and $y$ in $\AA$. Then there exists a weak KAM solution $u$ such that $u_{|\AA} = f$. 
\end{pr}

\subsection{Regularity and more regularity of subsolutions}

We review here regularity properties of classical subsolutions and weak KAM solutions. Most results were obtained by Fathi and Siconolfi in two founding papers \cite{FSC1,FS05}. The proofs are much more intricate than for the discrete theory. Note that on a non--empty compact connected smooth manifold (of positive dimension!) there is no isolated point.

\begin{Th}\label{Creg}
Let $u:M\to \R$ be a critical (classical) subsolution. Then $u$ is Lipschitz continuous on $M$. 

Let $x\in M$, then $x\in \AA$ if and only if all critical subsolutions $u\in \SSS$ are differentiable at $x$.
\end{Th}

In section \ref{more} were introduced further assumptions on the underlying space and the cost function we work with. This originates in the following properties of the action functional and Lax--Oleinik semigroup in the classical theory:

\begin{Th}\label{action SC}
For all $t>0$, the minimal action functional $h_t$ is semiconcave on $M\times M$.

Let $u : M\to \R$ be a bounded function, then for all $t>0$ the function $S^-(t)u$ is semiconcave and $S^+(t) u $ is semiconvex.
\end{Th}

 Given that the Lagrangian is $C^2$ and the functions $h_t$ and Lax--Oleinik semigroups are defined by using infimum and supremum, the preceding Theorem may seem a posteriori natural (given results such as Proposition \ref{infconc}). Its consequences are very powerful. 
 
 For instance let us come back to Theorem \ref{strictC}. In the first version, Fathi and Siconolfi prove the existence of $C^1$ subsolutions by carefully studying the regularity of subsolutions on $\AA$ and by a precise combination of smoothing and partitions of unity on $M\setminus \AA$. Patrick Bernard instead has a more global and decisive idea establishing the following regularization result (\cite{BernardC11}):
 \begin{Th}\label{regB}
 Let $u : M\to \R$ be a bounded function and $t>0$. There exists $\varepsilon>0$ such that for all $t'<\varepsilon$, $S^-(t')\circ S^+(t)u$ is $C^{1,1}$.
 \end{Th}
 One should have in mind that a $C^{1,1}$ function is one that is both semiconcave and semiconvex (Proposition \ref{semcc}). So the idea behind the previous Theorem, and what Bernard proves, is that the image of a semiconcave function by $S^+$ stays semiconcave for small times. This is a general version of an older result known as Lasry--Lions regularization (\cite{LL}). 
 
 Then it should not come as a surprise that in our proof of Theorem \ref{strictC11}, which is the discrete version of Bernard's Theorem, we used a composition of both operators $T^-$ and $T^+$.

 \subsection{Graph properties, twist condition and dynamics on the Aubry set}
 
 Let us start by noticing that thanks to the previous analysis, the Aubry set $\AA^*$ introduced in the Lagrangian setting following Theorem \ref{strictC} coincides with the set $\AA^*$ introduced in Proposition \ref{aubry-cot} when applied to the cost function $h_1$. This explains the similar notation. Hence the conclusions of Proposition \ref{aubry-cot} also hold in our Lagrangian setting as they also follow from Bernard's Theorem \ref{strictC}. This is the content of  Mather's Graph Theorem:
 
 \begin{pr}\label{aubry-graph}
The projections $\AA^* \to \AA$ and $\AA' \to \AA$ are  bi--Lipschitz homeomorphisms.
\end{pr}

This Chapter  ends by explaining why the cost $h_1$ associated to a Tonelli Lagrangian satisfies the left and right twist conditions. This is presented in \cite{ZJMD} and more details are given in \cite{Fa,CaSi00}.

\begin{pr}\label{Ltwist}
Let $L : TM \to \R$ be a Tonelli Lagrangian. Then the time--$1$ minimal action functional $h_1 : M\times M \to \R$ satisfies the left and right twist conditions.
\end{pr}

\begin{proof}
 Let  $(x,y)\in M\times M$. Let $\gamma : [0,1] \to M$ verify that
$$h_1(x,y) = \int_0^1 L\big(\gamma(s),\dot \gamma(s)\big) \dd s,$$
with $\gamma(0)=x$ and $\gamma(1)=y$.
Such a curve exists by Tonelli's Theorem. It is $C^2$ as already observed and solves the Euler--Lagrange equation. By standard variational arguments, one shows that $\big(-\partial_v L\big(x,\dot\gamma(0)\big),\partial_v L\big(y,\dot\gamma(1)\big) \big)\in \partial^+h_1(x,y) $. It follows that if $\partial_1 h_1(x,y)$ exists, then  $\partial_1 h_1(x,y)=-\partial_v L\big(x,\dot\gamma(0)\big)$. Remember that the Fenchel transform $\LL$ defined by \eqref{Fenchel} is a diffeomorphism and observe that 
$$\big(x, \partial_v L\big(x,\dot\gamma(0)\big)\big) = \LL\big(x, \dot\gamma(0)\big) = \big(x, - \partial_1 h_1(x,y)\big).$$
We deduce that the preceding equation uniquely determines $\dot\gamma(0)$ and that the minimizing curve $\gamma$ is unique with $\big(\gamma(s),\dot\gamma(s)\big) = \varphi^s_L\big(\gamma(0),\dot\gamma(0)\big)$ for $s\in [0,1]$. Another consequence is that, denoting by $\pi : TM\to M$ the canonical projection, 
$$\pi\circ \varphi_L^1 \circ \LL^{-1}  \big(x, - \partial_1 h_1(x,y)\big) = \pi \circ \varphi_L^1\big(x, \dot\gamma(0)\big) = y.$$
Hence $y\mapsto  - \partial_1 h_1(x,y)$ is injective, and $h_1$ has the right twist property. The proof of the left twist property is exactly the same.

\end{proof}

We may also interpret the left and right Legendre transforms introduced in Definition \ref{twist condition}. Indeed, if $(x,y) \in \mathcal D_r$ then if $(x,v) = \LL^{-1}\circ \LL_r (x,y)$, $v$ is the initial speed of the unique minimizing curve, going from $x$ to $y$ in time $1$. This has  the following consequence relating different Aubry sets:

\begin{pr}\label{eqAubry}
The following equalities hold for the cost $h_1$:
$$\widehat \AA = \left \{ \left(x ,  \pi \circ \varphi_L^1(x,v) \right) , \quad (x,v) \in \AA' \right\},$$
$$\widetilde \AA = \left \{ \left(  \pi \circ \varphi_L^n(x,v) \right)_{n\in \Z} , \quad (x,v) \in \AA' \right\}.$$
Moreover, for all $(x,v)\in \AA'$, $h_1\big(x ,  \pi \circ \varphi_L^1(x,v) \big) = \int_0^1 L \circ \varphi_L^s(x,v) \dd s$. 
\end{pr}

Finally, let us mention a curious fact about $h_1$. It can be established, using the notions of reachable gradient, that $ \partial_1 h_1(x,y)$ exists if and only if there exists a unique minimizing curve, in time $1$, going from $x$ to $y$ (see \cite{ZJMD}). Similarly,  $ \partial_2 h_1(x,y)$ exists if and only if there exists a unique minimizing curve, in time $1$, going from $x$ to $y$. It is therefore obtained that for the particular cost $h_1$, the equality $\mathcal D_\ell = \mathcal D_r$ holds.

\newpage

\chapter{Minimizing Mather measures and the discounted semigroups}
In this part we go back  to the more general setting of a continuous cost $c$ on a compact metric space $X$. Most results were presented and used by the authors and Fathi, Iturriaga, Davini in \cite{DFIZ2} to study convergence of solutions of the discounted equations. Earlier results and the introductions of Mather measures had appeared in Bernard and Buffoni's work \cite{BeBu}. 

The study of the positive counterpart to the discounted equations is new to our knowledge, both in the discrete and in the continuous setting. So are the results concerning degenerate discounted equations in the discrete setting.

\section{Minimizing Mather measures}

The cost $c$ is a continuous function from $X\times X$ to $\R$ and both canonical projections from $X\times X$ to $X$ are denoted $\pi_1$ and $\pi_2$.

\subsection{An optimal transport like approach}

Recall that if $\mu $ is a Borel measure on $ X\times X$ then $\pi_{1*} \mu$ and  $ \pi_{2*}\mu$ are measures on $X$ defined as follows: if $A\subset X$ is a Borel set, then $\pi_{1*} \mu(A) = \mu(A\times X)$ and $\pi_{2*} \mu (A) = \mu(X\times A)$.

\begin{df}\label{closed}\rm
A Borel probability measure $\mu$ on $X\times X$ is said to be {\it closed } if it has equal marginals:
$\pi_{1*} \mu = \pi_{2*}\mu$. We will denote by $\widehat\PP $ the set of closed Borel probability measures on $X\times X$.
\end{df}
The previous condition is equivalent to the following:
\begin{pr}
A probability measure $\mu$ is closed if and only if for any continuous function $f:X\to \R$, $\int_{X\times X} \big(f(y)-f(x)\big) \dd \mu(x,y)=0$.
\end{pr}
\begin{proof}
The proof is left as an exercise but follows these lines: if $\mu$ is closed, then the property of the proposition holds for indicatrix functions of open or closed sets. Hence it holds for simple functions (linear combination of indicatrix functions) and by density, it holds for continuous functions.

The converse is proved by approximating (from above and below) indicator functions by continuous functions.
\end{proof}
 
 The set $\wh\PP$ of closed  probability measures is clearly convex, closed and  compact (for the weak $*$ topology).
 
 Examples of closed measures can be constructed using Birkhoff averages. Indeed, given $(x_1,\cdots, x_n)\in X^n$,  the measure $\mu = \frac1n\sum\limits_{i=1}^n \delta_{(x_i,x_{i+1})}$\footnote{The notation $\delta$ stands for a Dirac mass.}, with the convention that $x_{n+1}=x_1$, is closed. Its marginals are
 $$\pi_{1*}\mu =  \frac1n\sum\limits_{i=1}^n \delta_{x_i} =  \frac1n\sum\limits_{i=1}^n \delta_{x_{i+1}}= \pi_{2*}\mu.$$
 
 Let us now introduce the concept of {\it minimizing measure}. It was first introduced by Mather for twist maps in \cite{MatherMeasure} and studied by Bernard and Buffoni in \cite{BeBu} in a  context similar to the present one, following their earlier works on optimal transportation \cite{BeBu1,BeBu2}.
 \begin{Th} \label{minimizing}
 The following equality holds:
 $$-\cc = \min_{\mu\in \wh\PP} \int_{X\times X} c(x,y) \dd \mu(x,y).$$
  There exists a closed measure realizing the minimum in the previous equality. Moreover, a closed measure realizes this minimum if and only if it is supported on the $2$--Aubry set $\widehat \AA$.
 \end{Th}
 \begin{proof}
 Let $\mu\in \wh\PP$ and $u_0$ a strict continuous subsolution given by Theorem \ref{th-strict}. Then one has
 $$0=\int_{X\times X} \big( u_0(y) - u_0(x)\big)\dd\mu(x,y) \leqslant \int_{X\times X}\big(c(x,y)+\cc \big)\dd \mu(x,y).$$
 This proves that $-\cc \leqslant   \min\limits_{\mu\in \wh\PP} \int_{X\times X} c(x,y) \dd \mu(x,y)$.
 Moreover, as $u_0$ is continuous, one has equality for a measure $\mu$ if and only if $u_0(y)-u_0(x) = c(x,y)+\cc$ for $\mu$--almost--every $(x,y)$ that is if $\mu$ is supported on $\widehat \AA$.
 
 Let us now construct such a measure. We  use Birkhoff averages. Let $f:X\to \R$ be any continuous function, $x\in X$ and for all $n\in \N$, let $x^n_{-n}, \cdots , x_0^n=x$ verify that $\Tn f(x) = f(x_{-n}^n) + \sum\limits_{i=-n}^{-1} c(x_i^n, x_{i+1}^n)$. Define $\mu_n =\frac1n\sum\limits_{i=-n}^{-1} \delta_{(x_i^n,x_{i+1}^n)}$. Finally let $(n_k)_{k\in \N}$ be an extraction such that the $(\mu_{n_k})_{k\in \N}$ converge to a measure $\mu$. As the $\mu_n$ are probability measures, so is $\mu$. 
 
 Let us verify that $\mu$ is closed: this follows by passing to the limit in the inequality
 \begin{multline*}
 \forall g\in C^0(X,\R),\quad \Big|\int_{X\times X}\big(g(y)-g(x)\Big) \dd \mu_n(x,y)\big| \\
 =\frac 1n\Big| \sum_{i=-n}^{-1} g(x_{i+1}^n)-g(x_i^n) \Big|\\
 =\frac 1n|g(x_{-n}^n)-g(x)|
 \leqslant \frac 2n\|g\|_\infty \to 0.
 \end{multline*}

 Let us verify $\mu $ is minimizing: recall that the family $\Tn f +n\cc$ is uniformly bounded. Hence 
 $$\int_{X\times X}\big(c(x,y)+\cc \big)\dd \mu_n(x,y)= \frac1n \big( \Tn f (x_{-n}^n)+n\cc - f(x)\big) \to 0.$$
 This proves that $\int_{X\times X} c(x,y) \dd \mu = -\cc$ and concludes the proof.
 \end{proof}

\begin{df}\rm
We denote by $\widehat\PP_0$ the set of minimizing closed probability measures, that is,  the set of closed probability measures $\mu\in \wh\PP$ such that  $\int c(x,y) \dd \mu = -\cc$. Such a measure $\mu$ is termed a {\it Mather measure}.

We define the Mather set $\widehat \MM\subset X\times X$ by 
$$\widehat \MM =\overline{ \bigcup_{\mu\in \widehat\PP_0} \mathrm {supp}(\mu)},$$
where supp stands for the support of a measure. The projected Mather set is $\MM = \pi_1(\widehat \MM) = \pi_2(\widehat \MM)$.
\end{df}

\begin{rem}\rm
The set 
$\widehat\PP_0$ is clearly itself compact and convex. Moreover by Theorem \ref{minimizing} the Mather set is a subset of the $2$--Aubry set: $\widehat \MM \subset \widehat \AA$.

Finally, the  Mather set is by definition closed, but one can prove that there is no need to take the closure in its definition. Indeed, there exists one minimizing measure $\mu_0$ whose support is the whole of $\widehat \MM$. To construct it, one considers a  sequence $(\mu_n)_{n>0} $ dense in $\widehat\PP_0$ and one then verifies that $\mu_0 = \sum\limits_{n>0} \frac{1}{2^n} \mu_n$ meets all the requirements.
\end{rem}

The proof of Theorem \ref{minimizing} sheds, once more, light on the general principle that long minimizing chains cannot stay too far from the Aubry set (as already seen in Proposition \ref{alpha}). This allows to give a stronger  version of Theorem \ref{uniqueness} and Proposition \ref{distlike}:

\begin{Th}\label{uniquenessM}
\hspace{2em}
\begin{enumerate}
\item Let $u$ and $v$ be respectively a weak KAM solution and a subsolution such that $u_{|\MM} \geqslant  v_{|\MM}$. Then $u\geqslant v$.

Let $u$ and $v$ be two weak KAM solutions such that $u_{|\MM} = v_{|\MM}$. Then $u=v$.
\item Conversely,  let $f : \MM \to \R$ be a function such that $f(y)-f(x)\leqslant h(x,y)$  for all $x$ and $y$ in $\MM$, where $h$ is Peierl's barrier. Then there exists a weak KAM solution $u$ such that $u_{|\MM} = f$. 
\end{enumerate}
\end{Th}
\begin{proof}
\begin{enumerate}
\item
Let $x_0\in X$ and let $(x_{-n})_{n\geqslant 0}$ be a calibrating sequence for $u$. As observed in the proof of Theorem \ref{uniqueness},
 $u(x_0)-v(x_0) \geqslant u(x_{-n})-v(x_{-n})$ for all $n>0$. 
 
 Limiting points of the sequence $(x_{-n})_{n\geqslant 0}$ are not necessarily in $\MM$. However, we prove that there exists  a suitable subsequence converging to a point in $\MM$, allowing to conclude the proof as in Theorem \ref{uniqueness}. Assume by contradiction the contrary. There exists an $\varepsilon >0$ such that $d(x_{-n},\MM) \geqslant \varepsilon$ for all $n\in \N$. Let $F = \{x\in X ,\ \   d(x,\MM) \geqslant \varepsilon\}$, that is a closed set. As in the proof of Proposition \ref{minimizing}, define $\mu_n =\frac1n\sum\limits_{i=-n}^{-1} \delta_{(x_i,x_{i+1})}$. Let finally $(n_k)_{k\in \N}$ be an extraction such that the sequence  $(\mu_{n_k})_{k\in \N}$ converges to a measure $\mu$. By hypothesis all the $\mu_n$ have their support included in $F\times F$, so the same holds for $\mu$. But, as proved in Proposition \ref{minimizing}, $\mu \in \widehat\PP_0$ is a Mather measure, hence the support of $\mu$ is included in $\widehat \MM$, and this is absurd.
 
In the second case, by symmetry, the opposite inequality holds, and the result follows as $x_0$ was taken arbitrary.
\item This point is established exactly as \ref{distlike}, we do not reproduce its proof.
 \end{enumerate}
\end{proof}

\subsection{An ergodic point of view}\label{BernardBuffoni}
 Mather sets were introduced as subsets of $X\times X$. As for Aubry sets, that hides the underlying dynamics. As  Aubry sets may be equivalently defined on $X\times X$ or on $X^\Z$, there are analogous measures defined on $X^\Z$. This is explained in \cite[Paragraph 4.2]{BeBu}. Indeed, denoting $s : (x_n)_{n\in \Z} \mapsto  (x_{n+1})_{n\in \Z}$ the shift operator, given a Borel probability measure $\tilde \mu$ on $X^\Z$ that is invariant by $s$, its push-forward $(\pi_{0,1})_*\tilde\mu$ by the projection $\pi_{0,1} : (x_n)_{n\in \Z} \mapsto (x_0,x_1)$ is a Borel probability measure on $X\times X$ that is closed in the sense of Definition \ref{closed} and such that $\int_{X^\Z}c(x_0,x_1) \dd \tilde\mu \big((x_n)_{n\in \Z}\big) = \int_{X\times X} c(x,y)\dd (\pi_{0,1})_*\tilde\mu$. 

Conversely, if $\mu$ is a Borel closed  probability measure on $X\times X$, Bernard and Buffoni construct, via a disintegration of $\mu$ with respect to the projection $\pi_2 : X\times X \to X$, 	a shift invariant measure $\tilde\mu$ on $X^\Z$ such that $\int_{X^\Z}c(x_0,x_1) \dd \tilde\mu \big((x_n)_{n\in \Z}\big) = \int_{X\times X} c(x,y)\dd\mu$. We therefore derive the following analogues of previous results, either by using the correspondence of Bernard and Buffoni, or by reproducing the proofs in this context. We leave the latter  to the reader.

\begin{df}\rm
Denote $\widetilde \PP$ be the set of shift invariant Borel probability measures on $X^\Z$. This is the set of Borel probability measures $\tilde \mu$ on $X^\Z$ such that $s_*\tilde \mu = \tilde \mu$.
\end{df}
The following result holds:
\begin{pr}
The critical constant is characterized by
$$\inf_{\tilde\mu \in \widetilde\PP}\int_{X^\Z}c(x_0,x_1) \dd \tilde\mu \big((x_n)_{n\in \Z}\big)=\min_{\tilde\mu \in \widetilde\PP}\int_{X^\Z}c(x_0,x_1) \dd \tilde\mu \big((x_n)_{n\in \Z}\big) = -c[0].$$
Moreover, an invariant measure $\tilde \mu \in \widetilde \PP$ is minimizing if and only if it is supported on the Aubry set $\widetilde\AA$. 
\end{pr}

\begin{df}\rm
We define $\widetilde\PP_0\subset \widetilde\PP$ to be the set of shift invariant Borel probability measures $\tilde \mu_0$ on $X^\Z$ such that
$$\int_{X^\Z}c(x_0,x_1) \dd \tilde\mu_0 \big((x_n)_{n\in \Z}\big)=\inf_{\tilde\mu \in \widetilde\PP}\int_{X^\Z}c(x_0,x_1) \dd \tilde\mu \big((x_n)_{n\in \Z}\big) = -c[0].$$
Such measures are also called {\it minimizing} or {\it Mather measures } and the context makes it clear whether a measure is defined on $X^\Z$ or on $X\times X$.

We define the Mather set $\widetilde \MM\subset X^\Z$ by 
$$\widetilde \MM =\overline{ \bigcup_{\tilde\mu\in \widetilde\PP_0} \mathrm {supp}(\tilde\mu)},$$
where supp stands for  the support of a measure.
\end{df}

The set  $\widetilde\PP_0$ is convex and compact. Finally the initial discussion together with Proposition \ref{pr-sequence} yield  that:
\begin{pr}
The following equalities hold: $\MM = \pi_0(\widetilde \MM)$ and $\widehat \MM = \pi_{0,1}(\widetilde\MM)$.
\end{pr}

\section{The discounted equation}

This Chapter ends by returning to the roots, more precisely to the second proof of the weak KAM Theorem \ref{weak KAM}. Recall that if $\lambda \in (0,1)$ then $u_\lambda$ is the unique function such that $u_\lambda = T^-_\lambda u_\lambda = T^- (\lambda u_\lambda)$. We now prove a result first obtained in \cite{DFIZ2}:
\begin{Th}\label{discounted}
There exists  a weak KAM solution $u_1$ such that $u_\lambda +\frac{\cc}{1-\lambda} \to u_1$ where the convergence takes place as $\lambda \to 1$ and is uniform.
\end{Th}
The proof is divided into several steps. It was already shown that as $\lambda \to 1$, $(1-\lambda)u_\lambda \to -\cc$ \big(Remark \ref{bounded} (ii)\big). Actually one gets something more precise:
\begin{pr}\label{pr-bounded}
The family $u_\lambda +\frac{\cc}{1-\lambda}$ is uniformly bounded as $\lambda \to 1$.
\end{pr}

This will be a simple consequence of the following comparison principle:
\begin{lm}\label{comp-princ}
Let $v_1$ be such that $v_1\leqslant T^-_\lambda v_1$ and let $v_2$ verify $v_2\geqslant T^-_\lambda v_2$. Then $v_1\leqslant u_\lambda \leqslant v_2$.
\end{lm}
\begin{proof}
By induction, one has for all $n\in \N$ that $v_1\leqslant T^{-n}_\lambda v_1$ and $v_2 \geqslant  T^{-n}_\lambda v_2$. Both right hand side terms converge to $u_\lambda $ as $n\to +\infty$ (recall $T^-_\lambda$ is a contraction). The results follow by passing to the limit.
\end{proof}
\begin{proof}[Proof of Proposition \ref{pr-bounded}]
Let $u$ be a weak KAM solution. Then adding and subtracting big constants to $u$ provides two weak KAM solutions $\overline u$ and $\underline u$ which are positive and negative respectively and verify $\overline u\geqslant \lambda \overline u$ and $\underline u\leqslant  \lambda \underline u$ . We then obtain that    
$$\forall \lambda \in (0,1),\quad \overline u -\cc =T^-(\overline u) \geqslant T^-(\lambda \overline u) = T^-_\lambda (\overline u).$$
This can be rewritten $\overline u-\frac{\cc}{1-\lambda} \geqslant T^-_\lambda(\overline u-\frac{\cc}{1-\lambda})$. In a same manner,  $\underline u-\frac{\cc}{1-\lambda} \leqslant T^-_\lambda(\underline u-\frac{\cc}{1-\lambda})$.
 Apply the previous lemma to obtain that $  \underline u-\frac{\cc}{1-\lambda}       \leqslant  u_\lambda \leqslant \overline u-\frac{\cc}{1-\lambda}$ which implies the proposition.
\end{proof}

 As the functions $u_\lambda +\frac{\cc}{1-\lambda}$ are equicontinuous and equibounded, thanks to the Arzel\` a--Ascoli Theorem, to prove the convergence it is enough to prove that all converging subsequences have the same limit. We now establish constraints on such limits:
 
 \begin{pr}\label{inegdiscount}
 Let $\mu\in \widehat\PP_0$ be a Mather measure. Assume $u_{\lambda_n} +\frac{\cc}{1-\lambda_n}\to u$ as $n\to +\infty$ for some extraction $\lambda_n\to 1$. Then $\int_X u(x)\dd\pi_{1*}\mu (x) \leqslant 0$.
 \end{pr}
 
 \begin{proof}
 Start from the inequalities $u_\lambda (y) - \lambda u_\lambda(x) \leqslant c(x,y)$ for all pairs $(x,y)$. Integrating with respect to $\mu$ yields
 $$\int_{X\times X} \big(u_\lambda (y) - \lambda u_\lambda(x)\big)\dd \mu(x,y) \leqslant \int_{X\times X}c(x,y)\dd\mu(x,y)=-\cc,$$
 as $\mu $ is minimizing. But since $\mu$ is closed, both marginals are equals and the left hand side is equal to $(1-\lambda)\int_X u_\lambda(x) \dd \pi_{1*}\mu(x)$. Dividing by $(1-\lambda)$ one obtains that  $\int_X \big(u_\lambda(x) + \frac{\cc}{1-\lambda}\big)\dd \pi_{1*}\mu(x)\leqslant 0$. The result now follows taking $\lambda = \lambda_n$ and passing to the limit.
 \end{proof}

Note that as the functions $u_\lambda$ are equicontinuous, any accumulation point $u$ as in the previous Proposition is automatically continuous.

The next step is to identify a reasonable candidate for the limit. This is done in the next Definition:
\begin{df}\label{F^-}\rm
Let $\FF\subset \S\cap C^0(X,\R)$ be the set of continuous subsolutions $u$ verifying the constraint $\int_X u(x)\dd\pi_{1*}\mu (x) \leqslant 0$ for all Mather measures $\mu\in\widehat \PP_0$. 

We define $u_1 = \sup\limits_{u\in \FF} u$ where the supremum is  taken pointwise.
\end{df}
 The set $\FF$ is not empty for it contains negative subsolutions (recall $\S$ or the set of weak KAM solutions are invariant by addition of constants). Restricting to continuous functions is not necessary (see \cite{DFIZ2} for the alternative approach of considering all subsolutions), but it simplifies some proofs. Elements of  $\FF$ are bounded above as they must take at least a non--positive value. Hence $u_1$ is well defined. The idea of taking supremums of solutions or subsolutions in viscosity solutions theory is rather standard, we will see here that it is very useful.
 
 Of course, Proposition \ref{inegdiscount} has a trivial consequence: if $u = \lim\limits_{n\to +\infty} u_{\lambda_n} +\frac{\cc}{1-\lambda_n}$ for some sequence  $\lambda_n\to 1$ then $u\in \FF$ and $u\leqslant u_1$.
 
  In order to establish the full  convergence, we have to prove the reverse inequality. This will be done by constructing some appropriate Mather measures. First we give a representation formula for $u_\lambda$:
 
 \begin{lm}\label{ulambda}
 For any $\lambda\in (0,1)$ and $x\in X$, we have 
 $$u_\lambda(x) = \min_{\substack{ (x_n)_{n\leqslant 0} \\ x_0=x}} \sum_{n\leqslant 0} \lambda^{-n} c(x_{n-1},x_n).$$
 \end{lm}
 
 \begin{proof}
 As $T^-_\lambda$ is a contraction on the set of continuous functions its fixed point is the limit of the iterates starting with any initial function. Taking the $0$ function, one computes that if $k>0$,
 $$T^{-k}_\lambda \bar 0(x) = \min_{x_{-k},\cdots , x_0=x}\sum_{i= -k+1}^0 \lambda^{-i} c(x_{i-1},x_i).$$
 The result follows letting $n\to +\infty$. The fact that all infimums are minimums comes from the usual compactness arguments.
 \end{proof}
 
 \begin{pr}\label{Laplace}
 Let $x\in X$ and for all $\lambda \in (0,1)$ let $(x_n^\lambda)_{n\leqslant 0}$ such that $x_0^\lambda=x$ and $u_\lambda(x) =  \sum\limits_{n\leqslant 0} \lambda^{-n} c(x^\lambda_{n-1},x^\lambda_n)$. Define the probability measure $\mu_\lambda$ by 
 $$\forall f\in C^0(X\times X, \R),\quad \int_{X\times X} f(x,y)\dd\mu_\lambda (x,y) = (1-\lambda) \sum_{n\leqslant 0} \lambda^{-n} f(x^\lambda_{n-1},x^\lambda_n).$$
 Assume finally that for some subsequence $\lambda_n\to 1$ the sequence $(\mu_{\lambda_n})_{n\in \N}$ converges to $\mu$. Then the measure $\mu$ is a Mather measure.
 \end{pr}

\begin{proof}
The multiplicative term $(1-\lambda)$ ensures that the measures $\mu_\lambda$ are probability measures. Hence so is $\mu$. We therefore have to prove that $\mu$ is closed and minimizing.

The fact that $\mu$ is closed does not depend on the particular choice of the sequences $(x_n^\lambda)_{n\leqslant 0}$ and results from the following computation:

Let $f : X\to \R$ be a continuous function. Then
\begin{align*}
\Big|\int_{X\times X} \big(f(y)-f(x)\big) \dd \mu_\lambda(x,y)\Big| &=  (1-\lambda) \Big|\sum_{n\leqslant 0} \lambda^{-n} \big(f(x^\lambda_{n})-f(x^\lambda_{n-1}) \big) \Big|\\
&= (1-\lambda) \Big| f(x)+ \sum_{n\leqslant -1} (\lambda^{-n} - \lambda^{-n-1}) f(x_n^\lambda)\Big|   \\
&\leqslant (1-\lambda) \|f\|_\infty \Big( 1+ (1-\lambda) \sum_{n\leqslant -1} \lambda^{-n-1} \Big) \\
&= 2(1-\lambda) \|f\|_\infty \to 0.
\end{align*}

 On the contrary, the fact that $\mu$ is minimizing depends heavily on the  use of the definition of $(x_n^\lambda)$:

$$(1-\lambda)u_\lambda (x) = (1-\lambda)  \sum_{n\leqslant 0} \lambda^{-n} c(x^\lambda_{n-1},x^\lambda_n) = \int_{X\times X}c(x,y)\dd\mu_\lambda(x,y) .  $$
As $\lambda_n\to 1$, the left hand side goes to $-\cc$ by Remark \ref{bounded} (ii), and the right hand side converges to $\int_{X\times X}c(x,y)\dd\mu(x,y)$.
\end{proof}

We now explain why those measures play a particular role:
\begin{lm}\label{lmcle}
Let $w\in \S$ be a continuous subsolution, then using the previous notation,
$$\forall \lambda \in (0,1), \quad u_\lambda (x) \geqslant w(x) - \int_{X} w(z) \dd \pi_{1*} \mu_\lambda (z).$$  
\end{lm}
\begin{proof}
We start with the definition of $u_\lambda$ and then use that $w$ is a subsolution as follows:
\begin{align*}
u_\lambda (x) =  \sum_{n\leqslant 0} \lambda^{-n} c(x^\lambda_{n-1},x^\lambda_n) 
 &\geqslant  \sum_{n\leqslant 0} \lambda^{-n} \big(w(x^\lambda_n) -w(x^\lambda_{n-1})\big)\\
&=w(x) - \sum_{n\leqslant 0} (\lambda^{-n} - \lambda^{-n+1} )w(x^\lambda_{n-1})
\\
&= w(x) -(1-\lambda)  \sum_{n\leqslant 0} \lambda^{-n}w(x^\lambda_{n-1}) \\
&= w(x) - \int_{X} w(z) \dd \pi_{1*} \mu_\lambda (z).
\end{align*}
\end{proof}

At last, let us conclude:
\begin{proof}[Proof of Theorem \ref{discounted}]
Let $u_{\lambda_n}\to u$ be a converging subsequence, we have already seen that $u\leqslant u_1$ where $u_1$ is given by Definition \ref{F^-}. 

Let now $x\in X$ and for $\lambda\in (0,1)$,  let $(x_n^\lambda)_{n\leqslant 0}$ such that $x_0^\lambda=x$ and $u_\lambda(x) =  \sum\limits_{n\leqslant 0} \lambda^{-n} c(x^\lambda_{n-1},x^\lambda_n)$ and define the probability measure $\mu_\lambda$ as in Proposition \ref{Laplace}. Extracting a further subsequence, assume that the $\mu_{\lambda_n}$ converge to a measure $\mu$ which is then a Mather measure by Proposition \ref{Laplace}. Let $w\in \FF$, applying the previous Lemma \ref{lmcle} we get
$u_\lambda (x) \geqslant w(x) - \int_{X} w(z) \dd \pi_{1*} \mu_\lambda (z)$ and along the subsequence $\lambda_n$ letting $n\to +\infty$ yields (using $w\in \FF$)
$$u(x) \geqslant w(x) - \int_{X} w(z) \dd \pi_{1*} \mu (z) \geqslant w(x).$$

Taking the supremum over $w\in \FF$, we conclude that $u(x) \geqslant u_1(x)$. Hence we have established the convergence.

\end{proof}
As a byproduct of the previous proof and of Proposition \ref{inegdiscount}, we have established that 
\begin{pr}\label{prlimsat}
The limit of the discounted approximation verifies $u_1\in \FF$.
\end{pr}

We continue this paragraph by establishing an alternative formula for the limit function $u_1$.
\begin{pr}\label{discountedbis}
For all $x\in X$, the following equality holds:
$$u_1(x) = \min_{\mu \in \widehat\PP_0} \int_X h(y,x) \dd \pi_{1*} \mu(y),$$
where $u_1$ is the function of Theorem \ref{discounted} and $h$ the Peierls barrier.
\end{pr}

\begin{proof}
We denote $\hat u$ the right hand side. We first claim that $\hat u$ is a subsolution. Indeed, each function $h_y = h(y, \cdot)$ is a subsolution by Proposition \ref{peierl-prop}. Hence, if $m $ is a Borel probablility measure on $X$, so is $h_m$ defined by $h_m(x) =  \int_X h(y,x) \dd  m(y)$ since $\SS$ is closed and convex (see Proposition  \ref{propS}).
Finally, as $\hat u$ is an infimum of functions of this type, it is itself a subsolution by Lemma \ref{inf}.

Next, we establish that $u_1\leqslant \hat u$. Let $u\in \SS$ be a continuous subsolution, we know that $u(x)-u(y)\leqslant h(y,x)$ for all pairs $(x,y)$ (Proposition \ref{peierl-prop}). Let $\mu \in \widehat\PP_0$ be a Mather measure, integrating with respect to $y$ the previous inequality yields
$$ u(x) - \int_X u(y)  \dd \pi_{1*} \mu(y) \leqslant \int_X h(y,x) \dd \pi_{1*} \mu(y).$$
If $u\in \F$ then we conclude that $u(x)\leqslant \int_X h(y,x) \dd \pi_{1*} \mu(y)$. This being valid for all $u\in \F$ and for all $\mu \in \widehat\PP_0$ we obtain the desired inequality $u_1\leqslant \hat u$.

We conclude by proving the reverse inequality. Let $y\in X$, the function $h^y = -h(\cdot ,y)$ is a subsolution (by Proposition \ref{peierl-prop}). Moreover, by definition of $\hat u$, the function $h^y + \hat u(y) \in \F$. In particular, $u_1 \geqslant h^y + \hat u(y)$ and evaluating at $y$ we obtain $u_1(y) \geqslant -h(y,y) + \hat u(y)$. If we specify moreover $y\in \AA$ to be in the projected Aubry set then we have proved that (see Theorem \ref{Aubry0}):
$$\forall y\in \AA,\quad u_1(y) \geqslant \hat u(y).$$
This is enough  to conclude that $u_1 \geqslant \hat u$ everywhere, indeed, $u_1$ is a weak KAM solution and $\hat u\in \SS$ hence Theorem \ref{uniqueness} applies.
\end{proof}

\begin{rem}\label{remergo}\rm
The limit of the family $(u_\lambda)_{\lambda\in (0,1)}$ as $\lambda \to 1$ can be reformulated in terms of Mather measures on $X^\Z$. Indeed, as marginals of such measures are the same as those on $X\times X$ one finds that 
$$\mathcal F = \left\{ u \in \SS\cap C^0(X,\R) , \forall \tilde \mu \in \widetilde \PP_0, \quad \int_{X^\Z}u(x_0) \dd \tilde\mu \big((x_n)_{n\in \Z}\big)\leqslant 0\right\}.$$

And also, for all $x\in X$, the following equality holds:
$$u_1(x) = \min_{\tilde\mu \in\widetilde \PP_0} \int_X h(x_0,x) \dd  \tilde\mu \big((x_n)_{n\in \Z}\big) .$$

\end{rem}

And finally, here is a mild property of $u_1$:
\begin{pr}\label{sature}
There exists a Mather measure $\mu_0\in \widehat\PP_0$ such that  
$$\int_X u_1(x)\dd \pi_{1*}\mu_0(x) = 0.$$

Moreover, it can be imposed that $\mu_0$ is an extremal point of $\widehat\PP_0$.
\end{pr}
\begin{proof}
By Proposition \ref{inegdiscount} the selected function verifies  $u_1\in \FF$ meaning that  $\int_X u_1(x)\dd\pi_{1*}\mu (x) \leqslant 0$ for all Mather measures $\mu\in \widehat\PP_0$. If the result were not true, by compactness of $\widehat\PP_0$ there would be an $\varepsilon>0$ such that  $\int_X u_1(x)\dd\pi_{1*}\mu (x) \leqslant -\varepsilon$ for all Mather measures $\mu\in\widehat \PP_0$. Then the function $u_1+\varepsilon$ would also belong to $\FF$ contradicting the definition of $u_1$ given in \ref{F^-}.

The second assertion is a direct consequence of Choquet's Theorem (\cite{Ph}).  Indeed, it states that if $\mu_0$ is a measure given by the first part of the Proposition, then there exists a probability measure $w$ on $\widehat\PP_0$, supported on the extremal points of $\PP_0$ such that $\mu_0 = \int_{\widehat\PP_0} \mu \dd w(\mu)$. Any measure $\mu_1$ in the support of $w$ has to verify $\int_X u_1(x)\dd\pi_{1*}\mu_1 (x)=0$.
\end{proof}

\begin{rem}\label{ergo}\rm
The previous Proposition holds as well when considering Mather measures as measures on $X^\Z$ thanks to the point of view of Bernard and Buffoni (see Remark \ref{BernardBuffoni}). In this case, denoting by $\widetilde{\PP}_0$ the set of minimizing shift invariant measures, extremal measures 
are the ergodic ones with respect to the action of the shift.
\end{rem}

Before turning to the positive counterpart of those results let us provide a simplistic economical interpretation. As previously, $X$ is the metric space of wine stores in France, $c:X\times X \to \R$ the cost of a 24 hour delivery and $R : X\to \R$ provides the price $R(x)$ of a bottle of Château Rayas\footnote{Château Rayas is definitively the best red wine, and arguably the best white wine, that the author has had the privilege of tasting. They are both of the appellation Châteauneuf du Pape which is the most prestigious of the meridional Rhône valley. Wines made by their owner, Emmanuel Reynaud, have no equal.} in the store $x$. The discount factor plays the role of an interest rate, or of inflation depending on the point of view. If some money $m>0$ is placed in the bank at a daily rate $\lambda^{-1}$, then tomorrow it will be worth $\lambda^{-1}m$. Conversely, if one buys today a bottle of Château Rayas at the price $R(y)$ but only pays it tomorrow, it is considered that the actualized price is $\lambda R(y)$ \big(as this amount of money put in the bank today will buy the bottle tomorrow at price $R(y)$\big). Henceforth taking into account this effect of time, the actualized least price to obtain a bottle of Château Rayas at $x$ tomorrow is $T_\lambda^- R(x) = \inf\limits_{y\in X} \lambda R(y) + c(y,x)$, considering that the transportation will be paid tomorrow at tomorrow's price.

In this context, the function $u_\lambda$, fixed point of $T^-_\lambda$ is called equilibrium state. It is the only price function such that a buyer has not to worry about the time at which he wishes to receive his bottle. It is also the asymptotic price of a bottle for someone willing to wait a very very long time, when the interest rate is at $\lambda$.

\section{Discount for the positive Lax--Oleinik semigroup}
We here address the positive counterpart of the previous results and explore some relations between the obtained limits. This is new to our knowledge. 

Of course, all the constructions and results of the previous section hold for the positive Lax--Oleinik semigroup. If $\lambda\in [0,1)$ we denote by $v_\lambda$ the unique fixed point of the operator $T^+_\lambda : u \mapsto T^+(\lambda u)$ that is a contraction. Similar arguments as in the previous paragraph yield:

\begin{Th}\label{discounted+}
There exists  a positive weak KAM solution $v_1$ such that $v_\lambda - \frac{\cc}{1-\lambda} \to v_1$ where the convergence takes place as $\lambda \to 1$ and is uniform.
\end{Th}
The functions $v_\lambda$ have the following explicit form:
 \begin{lm}\label{vlambda}
 For any $\lambda\in (0,1)$ and $x\in X$, we have 
 $$v_\lambda(x) = -\min_{\substack{ (x_n)_{n\geqslant 0} \\ x_0=x}} \sum_{n\geqslant 0} \lambda^{n} c(x_{n},x_{n+1}).$$
 \end{lm}
The limit $v_1$ has the following form:
\begin{pr}\label{discountedbis+}
Let $\FF^+\subset \S\cap C^0(X,\R)$ be the set of continuous subsolutions $u$ verifying the constraint $\int_X u(x)\dd\pi_{1*}\mu (x) \geqslant 0$ for all Mather measures $\mu\in \widehat\PP_0$. 

We have the formulas $v_1 = \inf\limits_{u\in \FF^+} u$ where the infimum is  taken pointwise. 

The function $v_1$ verifies $v_1\in \FF^+$. 

And finally for all $x\in X$,
$$v_1(x) = \max_{\mu \in \widehat\PP_0} \int_X- h(x,y) \dd \pi_{1*} \mu(y).$$
\end{pr}

As for the negative Lax--Oleinik semigroup (Remark \ref{remergo}) the previous proposition can be stated in terms of Mather measures on $X^\Z$, which we leave to the reader.  

We conclude by asking the following:

\vspace{2mm}
\fbox{
 {\bf Question}: what are the links between $u_1$ and $v_1$?
 }
 \vspace{2mm}
 
  Unfortunately, the answer may seem disappointing, there is, in general no particular link. For example, except in very particular instances, they are not a conjugate pair (as they do not have any reason to coincide on the Mather set $\MM$). They are not even ordered even though the following inequalities hold on the projected Aubry set:
\begin{pr}\label{u<v}
The functions $u_1$ and $v_1$ verify 
$$\forall x\in \AA,\quad u_1(x)\leqslant v_1(x).$$
\end{pr}
\begin{proof}
Let us argue by contradiction assuming that there exists $x_0\in \AA$ such that $v_1(x_0)<u_1(x_0)$. We set $\varepsilon = u_1(x_0) - v_1(x_0) >0$. We will construct a Mather measure $\mu_0$ such that $\int_X v_1(z) \dd \pi_{1*} \mu_0(z)  <0$. This will be our contradiction as $v_1 \in \FF^+$ meaning that $\int_X v_1(z) \dd \pi_{1*} \mu_0(z) \geqslant 0$.

Let $(x_n)_{n\in \Z} \in \widetilde \AA$ be a sequence associated to $x_0$. As $u_1$ and $v_1$ are critical subsolutions, one infers (see Remark \ref{remoublieee}) that 
\begin{align*}
\forall n\geqslant 0, \quad  & u_1(x_0) - u_1(x_{-n}) = \sum_{k=-n}^{-1} c(x_k,x_{k+1}) +n\cc, \\
 & v_1(x_0) - v_1(x_{-n}) = \sum_{k=-n}^{-1} c(x_k,x_{k+1}) +n\cc.
\end{align*}
It follows that $\varepsilon = u_1(x_{-n}) - v_1(x_{-n})$ for all $n\geqslant 0$. By continuity, one finds that $u_1 - v_1$ is constantly equal to $\varepsilon$ on $\overline{\{ x_{-n}, \ n\geqslant 0 \}}$.

Last, arguing as in the proof of Theorem \ref{uniquenessM}, we construct a minimizing Mather measure $\mu_0 \in \widehat\PP_0$ such that the support of $\pi_{1*}\mu_0$ is included in $\overline{\{ x_{-n}, \ n\geqslant 0 \}}$. We conclude, using again that $u_1\in \FF$, that 
$$\int_X v_1(z) \dd \pi_{1*} \mu_0(z) \leqslant \int_X \big( v_1(z) - u_1(z)\big) \dd \pi_{1*} \mu_0(z) = -\varepsilon.$$
%

\end{proof}

The concluding general result here gives a condition for $u_1$ and $v_1$ to be a conjugate pair:

\begin{pr}\label{conjdiscount}
The following assertions are equivalent:
\begin{enumerate}
\item\label{1} The functions $u_1$ and $v_1$ form a conjugate pair,
\item\label{2} $u_{1|\AA} =v_{1|\AA}$,
\item\label{3} $u_1\geqslant v_1$,
\item\label{4} for all Mather measures $ \mu$, the equality $\int_X u_1(x)\dd \pi_{1*}\mu(x) = 0$  holds,
\item\label{5} for all Mather measures $ \mu$, the equality $\int_X v_1(x)\dd \pi_{1*}\mu(x) = 0$  holds,
\item\label{6} there exists a critical subsolution $v\in \SS$ such that for all Mather measures $ \mu$, the equality $\int_X v(x)\dd \pi_{1*}\mu(x) = 0$  holds.
\end{enumerate}

\end{pr}

\begin{proof}
Assertion \eqref{1} being equivalent to \eqref{2} follows from the definition of a conjugate pair as explained in Remark \ref{conjugate}.

If \eqref{2} holds, then \eqref{3} holds as this inequality is always true for a conjugate pair. Reciprocally, if \eqref{3} holds, then by Proposition \ref{u<v}, \eqref{2} is true.

Assertion  \eqref{3} implies \eqref{4} and \eqref{5}. It is an immediate consequence of the fact that $u_1\in \FF$ and $v_1\in \FF^+$.

Then, \eqref{4} or \eqref{5} implies \eqref{6} is straightforward as negative or positive weak KAM solutions are subsolutions.

Let us now establish \eqref{6} implies \eqref{2}. Let $v$ be the subsolution given by the hypothesis and let us denote by $v^-$ and $v^+$ the respective limits of $\Tn v+ n\cc$ and $\TTn v-n\cc$ as $n\to +\infty$. As $v^-_{|\AA} = v^+_{|\AA} = v_{|\AA}$ we obtain respectively a negative and positive weak KAM solution satisfying the hypothesis of (6). The idea of the proof is that there can be at most one such negative weak KAM solution (and similarly, at most one  such positive weak KAM solution).

 To this aim, 
%
let $\mu \in \wh\PP_0$ so that
 $$\int_{X} v^-(x)\dd \pi_{1*}\mu(x)=\int_{X} v^+(x)\dd \pi_{1*}\mu(x)=\int_{X} v(x)\dd \pi_{1*}\mu(x) = 0.$$
 As $v^-\in \FF$, $v^-\leqslant u_1$ and as $u_1\in \FF$ it follows that
 $$\forall \mu \in \wh \PP_0,\quad 0=\int_{X} v^-(x)\dd \pi_{1*}\mu(x)\leqslant \int_{X} u_1(x)\dd \pi_{1*}\mu(x)\leqslant 0.$$
 So $u_1$ itself satisfies the hypothesis of (5). Moreover, combining the previous equalities with $v^-\leqslant u_1$ implies that $v^-$ and $u_1$ coincide on the support of $\mu$ (as both functions are continuous). This being true for all minimizing Mather measures  $\mu \in \wh \PP_0$, we conclude that 
%
%
 $u_{1|\MM} = v^-_{|\MM}$, by Theorem \ref{uniquenessM}, we deduce that $u_1 = v^-$. The same proof yields that $v_1 = v^+$. Hence the pair $(u_1,v_1)$ is a conjugate pair.
\end{proof}

\section{Degenerate discounted equations}\label{sectiondeg}

As an original contribution, let us finish by a generalization of  the discounted convergence results. Instead of modifying the Lax--Oleinik semigroup to make it a contraction, we  perturb it so that it is still a $1$--Lipschitz map. Yet conditions are given in order to  select again a weak KAM solution as the perturbation gets smaller. In this generality, the results of this paragraph are new.

We consider a continuous function $\alpha : X\to \R$ that verifies the following two conditions:
\vspace{2mm}

\fbox{
\begin{minipage}{0.85\textwidth}
\begin{enumerate}
\item[($\alpha1$)] the function $\alpha$ has values in $[0,1)$,
\item[($\alpha2$)] for all minimizing Mather measure $\mu \in \wh\PP_0$, $\int_X \alpha(x)\dd\pi_{1*} \mu(x) >0$.
\end{enumerate}
\end{minipage}
}

\vspace{2mm}
This last property is obviously verified if $\alpha $ is positive on the projected Aubry set $\mathcal A$ (this was the condition of \cite{Zdisc}) or if $\alpha$ is positive on $\mathcal M$.
The problem to be studied is understanding the behavior of functions $u_\lambda : X\to \R$, for $\lambda \in (0,1)$, verifying 
$$\forall x\in X, \quad u_\lambda(x) = T^-\big((1-\lambda\alpha)u_\lambda \big) (x) +c[0],$$
as $\lambda \to 0$. The convergence result is stated later on in this section in Theorem \ref{Thconvdiscdeg}. Therefore, let us denote by $\TTT_\lambda $ the mapping $v\mapsto T^-\big((1-\lambda\alpha)u_\lambda \big)+c[0]$. Just like $T^-$ (see Proposition \ref{propT}), the operators  $\TTT_\lambda $ are $1$--Lipschitz and order preserving.

We start by easy properties in order to get acquainted with the operators:

\begin{pr}\label{minlambda}
Let $v : X\to \R$ be a continuous function, then 
$$\forall x\in X, \exists x_{-1} \in X , \quad \TTT_\lambda v (x) = \big(1-\lambda\alpha(x_{-1})\big)v(x_{-1}) + c(x_{-1},x)+c[0].$$
More generally, for all $n>0$, there is a chain $(x_{-n}, \cdots , x_0=x)$ such that
$$\TTT_\lambda^{-n} v(x) = \beta_{-n} v(x_{-n}) + \sum_{k=-n}^{-1} \beta_{k+1}\big(c(x_{k},x_{k+1})+c[0]\big),$$
where $\beta_k = \prod\limits_{j=k}^{-1} \big(1-\lambda\alpha(x_{j})\big)$, for $-n\leqslant k\leqslant -1$ and $\beta_0 = 1$.
\end{pr}
\begin{proof}
The first point is a direct consequence of compactness and continuity while the second follows from a straightforward induction.
\end{proof}
Beware that the notation $\beta_k$ is misleading as it depends on the chain $(x_{k}, \cdots , x_0)$.
We now address the issue of fixed points of $\TTT_\lambda$.
\begin{df}\rm
 We will say a function $u : X\to \R$ is a $\lambda$--discounted subsolution if $u\leqslant \TTT_\lambda u$ or equivalently
\begin{equation}\label{inegdeg}
\forall (x,y)\in X\times X, \quad u(x) - \big(1-\lambda\alpha(y)\big)u(y) \leqslant c(y,x)+c[0].
\end{equation}

A function $v : X\to \R$ is a $\lambda$--discounted solution if $v =  \TTT_\lambda v$.
\end{df}

 By definition and successive applications of Proposition \ref{minlambda} one gets: 

\begin{pr}\label{caliblambda}
Let $u : X\to \R$ be a $\lambda$--discounted subsolution, then for all $n>0$ and all finite chains $(y_{-n}, \cdots , y_0=x)$,
$$u(y_0) \leqslant \beta_{-n} u(y_{-n}) + \sum_{k=-n}^{-1} \beta_{k+1}\big(c(y_{k},y_{k+1})+c[0]\big),$$
where $\beta_k = \prod\limits_{j=k}^{-1} \big(1-\lambda\alpha(y_{j})\big)$ and $\beta_0 = 1$.

Let $v : X\to \R$ be a $\lambda$--discounted solution. Then for all $x\in X$, there exists an infinite chain $(x_k)_{k\leqslant 0}$ such that $x_0 = x$ and 
$$\forall n>0,\quad v(x) = \beta_{-n} v(x_{-n}) + \sum_{k=-n}^{-1} \beta_{k+1}\big(c(x_{k},x_{k+1})+c[0]\big),$$
where $\beta_k = \prod\limits_{j=k}^{-1} \big(1-\lambda\alpha(x_{j})\big)$ and $\beta_0 = 1$.

\end{pr}
The convention  adopted here is that an empty product has value $1$, so that in the previous notation, the formula also holds for $\beta_0$.

The next result is reminiscent of strong comparison principles in viscosity solutions theory:

\begin{Th}\label{strongCP}
Let $\lambda \in (0,1)$, $u : X\to \R$ be a $\lambda$--discounted subsolution and $v : X\to \R$ be a $\lambda$--discounted solution. Then $u\leqslant v$.
\end{Th}

\begin{proof}
As $u\leqslant \TTT_\lambda u$, it is enough to prove that  $ \TTT_\lambda u \leqslant v$. Hence by Proposition \ref{propT}, one assumes  that $u$ is continuous, without loss of generality.  Then consider a strict subsolution $u_0 : X\to \R$ given by Theorem \ref{th-strict}. Moreover, up to subtracting a big constant, we assume that $u_0$ is negative. For $\varepsilon \in (0,1)$ we define $u_\varepsilon = (1-\varepsilon) u_0 + \varepsilon u$. The function $u_\varepsilon$ is a $\lambda$--discounted subsolution. As a matter of fact, if $(x,y) \in X\times X$,

\begin{multline}\label{subsollambda}
u_\varepsilon(x) - \big(1-\lambda\alpha(y)\big)u_\varepsilon(y) = \\
=\varepsilon \big(u(x) - \big(1-\lambda\alpha(y)\big)u(y)\big) + (1-\varepsilon)\big( u_0(x) - \big(1-\lambda\alpha(y)\big)u_0(y)\big) \\
\leqslant \varepsilon \big(c(y,x)+c[0]\big) + (1-\varepsilon)\big(u_0(x) -u_0(y)\big) \leqslant c(y,x)+c[0],
\end{multline}
where it was used first that $u_0$ is negative and then that it is a critical subsolution.

Let now $x_0 \in X$ such that $u_\varepsilon(x_0) - v(x_0) = \max  (u_\varepsilon -v)$. We aim at proving that $u_\varepsilon(x_0) - v(x_0) \leqslant 0$. Let us argue by contradiction, assuming that $u_\varepsilon(x_0) - v(x_0)>0$. Let $(x_k)_{k\leqslant 0}$ be a chain given by Proposition \ref{caliblambda} for $v$, $(\beta_k)_{k\leqslant 0}$ the associated sequence as defined in the same Proposition \ref{caliblambda}. It follows from both assertions of Proposition \ref{caliblambda} that for all $k< 0$, 
\begin{multline*}
(u_\varepsilon - v)(x_0) = u_\varepsilon(x_0) - \beta_k v(x_k) - \sum_{j=k}^{-1} \beta_{j+1}\big(c(x_{j},x_{j+1})+c[0]\big) \\
\leqslant  \beta_k u_\varepsilon(x_k) +\sum_{j=k}^{-1} \beta_{j+1}\big(c(x_{j},x_{j+1})+c[0]\big) - \beta_k v(x_k) - \sum_{j=k}^{-1} \beta_{j+1}\big(c(x_{j},x_{j+1})+c[0]\big) \\
 =  \beta_k (u_\varepsilon - v)(x_k) \leqslant (u_\varepsilon - v)(x_k),
\end{multline*}
where the last inequality is obtained using the contradiction hypothesis and the inequalities $0<\beta_k\leqslant 1$.
By definition of $x_0$, it follows that all the preceding inequalities are equalities. In particular,  it comes that $\beta_k = 1$ for all $k< 0$ which in turn implies that $\alpha(x_k) = 0$ for all $k\leqslant 0$, by definition of $\beta_k$. Moreover, tracing the inequalities used, it follows that 
$$\forall k\leqslant 0, \quad u_\varepsilon (x_0) =  \beta_k u_\varepsilon(x_k) +\sum_{j=k}^{-1} \beta_{j+1}\big(c(x_{j},x_{j+1})+c[0]\big).$$
Going back to \eqref{subsollambda} and using that, there as well, inequalities are indeed equalities, it follows that $u_0(x_k) - u_0(x_{k+1}) = c(x_k,x_{k+1})+c[0]$ for all $k<0$. By definition of $u_0$ and thanks to its property of being strict, we conclude that $(x_k,x_{k+1}) \in \wh\AA$ for all $k<0$.

Let us now define, for $n>0$ the probability measure $\mu_n = \frac{1}{n} \sum\limits_{k=-n}^{-1} \delta_{(x_k,x_{k+1})} $ (that is supported on $\wh\AA$). Let $\mu$ be an accumulation point of the sequence of probability measures $(\mu_n)_{n>0}$ for some subsequence $(n_i)_{i\geqslant 0}$. Arguing as in the proof of Theorem \ref{minimizing}, we find that the measure $\mu$ is closed. As the $2$-Aubry set is closed, the measure $\mu$ is supported on $\wh\AA$. The last part of Theorem \ref{minimizing} implies that $\mu$ is a Mather minimizing measure.

Finally, using that $\alpha(x_k)=0$ for all $k< 0$ observe that 
$$\int_X \alpha(x)\dd\pi_{1*} \mu(x)  = \lim_{i\to+\infty} \int_X \alpha(x)\dd\pi_{1*} \mu_{n_i}(x)  = \lim_{i\to+\infty}  \frac{1}{n_i} \sum\limits_{k=-n_i}^{-1} \alpha(x_k) = 0,$$
thus contradicting Hypothesis ($\alpha2$). Hence $u_\varepsilon(x_0) - v(x_0) = \max  (u_\varepsilon -v)\leqslant 0$ and $u_\varepsilon \leqslant v$. As this holds for all $\varepsilon \in (0,1)$, letting $\varepsilon \to 1$  proves that $u\leqslant v$.
\end{proof}
The previous proof combines two main ideas. The first one is that subsolutions can be approximated by strict subsolutions, thus forcing interesting phenomena to take place on the Aubry set. This is made possible by the convex structure of our minimization problems. The second idea is to construct illicit Mather measures assuming that subsolutions or solutions do not verify suitable properties. This line of reasoning will be used several times in what follows. 

As $\lambda$--discounted solutions are obviously $\lambda$--discounted subsolutions, the previous Proposition brings as a consequence that there can be at most one $\lambda$--discounted solution. The next existence result shows there is exactly one:

\begin{Th}\label{existencelambda}
For all $\lambda \in (0,1)$ there exists a unique $\lambda$--discounted solution.
\end{Th}
\begin{proof}
Let $\underline u$ be a negative weak KAM solution and $\overline u$ be a positive weak KAM solution. Applying the modified Lax--Oleinik semigroup yields
$$\TTT_\lambda (\underline u) = T^-\big((1-\lambda\alpha)\underline u \big) )+c[0] \geqslant T^-(\underline u)+c[0]= \underline u.
$$
A straightforward induction yields that the sequence $\big(\TTT_\lambda^n(\underline u)\big)_{n\geqslant 0}$ is non--decreasing.

Similarly,  
$$\TTT_\lambda (\overline u) = T^-\big((1-\lambda\alpha)\overline u \big) )+c[0] \leqslant T^-(\overline u)+c[0]= \overline u.
$$
A straightforward induction yields that the sequence $\big(\TTT_\lambda^n(\overline u)\big)_{n\geqslant 0}$ is non--increasing. 

Finally, as $\underline u <\overline u$ it follows that $\TTT_\lambda^n(\underline u) \leqslant \TTT_\lambda^n(\overline u)$ for all $n\geqslant 0$. The sequence $\big(\TTT_\lambda^n(\underline u)\big)_{n\geqslant 0}$ is bounded and non--decreasing, made of equi--continuous functions, hence it converges (uniformly) towards a function $u_\lambda : X\to \R$, verifying $\underline u \leqslant u_\lambda \leqslant \overline u$, that is, by continuity of $\TTT_\lambda$, a $\lambda$--discounted solution.

\end{proof}

\begin{df}\rm\label{defdisc}
For all $\lambda \in (0,1)$, the unique $\lambda$--discounted solution is denoted by  $u_\lambda^\alpha$.
\end{df}

 As a byproduct of the previous proof, it was established:

\begin{co}\label{equilambdabounded}
The family $(u_\lambda^\alpha)_{\lambda \in (0,1)}$ is uniformly bounded and consists of equi--continuous functions.
\end{co}
The last part holds as the $u_\lambda^\alpha$ are in the image of $T^-$ (Proposition \ref{propT}). As the family  $(u_\lambda^\alpha)_{\lambda \in (0,1)}$ is relatively compact, to prove that it converges when $\lambda \to 0$, it is enough to prove there is a unique accumulation point.

The next proposition establishes the crucial property of such accumulation points, similarly to Proposition \ref{inegdiscount}:

\begin{pr}\label{inegdegdisc}
 Let $\mu\in \widehat\PP_0$ be a Mather measure. Assume $u_{\lambda_n}^\alpha\to u$ as $n\to +\infty$ for some extraction $\lambda_n\to 0$. Then $\int_X \alpha(x)u(x)\ \dd\pi_{1*}\mu (x) \leqslant 0$.
\end{pr}
 
 \begin{proof}
 Let us start from the family of inequalities given by \eqref{inegdeg}, applied to the functions $u_\lambda$. Integrating against $\mu$ it is obtained that
 $$0=\int_{X\times X} \big[ c(y,x)+c[0] \big] \dd \mu(y,x) \geqslant \int_{X\times X} \big[u_\lambda^\alpha(x) - \big(1-\lambda\alpha(y)\big)u_\lambda^\alpha(y) \big] \ \dd \mu(y,x).
 $$
 As $\mu$ is closed and $u_\lambda$ continuous, dividing by $\lambda$, we gather  that  
 $$\forall \lambda \in (0,1),\quad \int_X \alpha(y)u_\lambda^\alpha(y)\ \dd\pi_{1*}\mu (y) \leqslant 0.$$
  Passing to the limit along the subsequence $(\lambda_n)_{n\geqslant 0}$, yields the result.
 \end{proof}

Particular Mather measures can then be constructed starting from calibrating chains given by Proposition \ref{caliblambda}. One first needs to establish a crucial property they satisfy:

\begin{pr}\label{epslambda}
There exists $M >0$
such that for all $\lambda\in (0,1)$ and $x_0 \in X$, if $(x_k^\lambda)_{k\leqslant 0}$ is a sequence given by Proposition \ref{caliblambda} applied to $u_\lambda^\alpha$ with $x_0^\lambda = x_0$, then
$$\lambda\sum_{k\leqslant 0}  \prod_{j=k}^{-1} \big(1-\lambda\alpha(x_{j}^\lambda)\big) <M.$$
\end{pr}

\begin{proof}
Let us argue by contradiction assuming the result does not hold. Then there exist a sequence $(\lambda_n)_{n\in \N} \in (0,1)^\N$ and points $x^n_0 \in X$ such that for each integer $n\in \N$ there exists a sequence
$(x_k^n)_{k\leqslant 0}$ given by Proposition \ref{caliblambda} associated to $u_{\lambda_n}^\alpha$  and an integer $N_n>0$  such that 
$\lambda_nC_n =\lambda_n \sum\limits_{k=-N_n}^{-1} \beta_{k+1}^n \to+\infty$,  having adopted the notation
$\beta_k^n =  \prod\limits_{j=k}^{-1} \big(1-\lambda_n\alpha(x_{j}^n)\big) $.
This implies that $N_n \to +\infty$ as $0<\beta_{k}^n \leqslant 1$.

For all integer $n\in \N$, let us define the probability measure on $X\times X$, 
$$\mu_n =C_n^{-1} \sum_{k=-N_n}^{-1} \beta_{k+1}^n \delta_{(x_k^n ,x_{k+1}^n  )}.$$
Up to an extraction, let us assume furthermore that the sequence $\mu_n$ converges to a probability measure $\mu$. We will prove that $\mu $ is a minimizing Mather measure violating condition ($\alpha2$).

{\bf The measure $\mu$ is closed:} let $f : X\to \R$ be a continuous function. We compute, using an Abel transform:
\begin{multline*}
\Big| \int_{X\times X}\big(f(y)-f(x)\big) \dd \mu_n(x,y)  \Big| = C_n^{-1}\Big|  \sum_{k=-N_n}^{-1} \beta_{k+1}^n\big(  f(x_{k+1}^n)  -f(x_k^n)\big) \Big| \\
=C_n^{-1}\Big|  \sum_{k=-N_n}^{-1} (\beta^n_{k} -\beta_{k+1}^n) f(x_{k}^n) -\beta_{-N_n}^n f(x_{-N_n}^n) +\beta_{0}f(x_{0}^n)  \Big| \\
\leqslant C_n^{-1} \Big[ \sum_{k=-N_n}^{-1} (\beta^n_{k+1} -\beta_{k}^n)\| f \|_\infty + 2\|f\|_\infty \Big] \leqslant 4C_n^{-1}\|f\|_\infty.
\end{multline*}
In the previous  chain of inequalities  it was used that the sequences $(\beta_k^n)_{k\leqslant 0}$ are non--decreasing and take values in $[0,1]$. As $C_n\to +\infty$, letting $n\to +\infty$, it is obtained that 
$$\int_{X\times X}\big(f(y)-f(x)\big) \dd \mu(x,y) = \lim_{n\to +\infty} \int_{X\times X}\big(f(y)-f(x)\big) \dd \mu_n(x,y)  = 0.$$ 
Therefore $\mu$ is closed.

{\bf The measure $\mu$ is minimizing:} we use the definition of $\mu_n$ and the property of the sequences $(x_k^n)_{k\leqslant 0}$.
\begin{align*}
 \Big|\int_{X\times X}\big(c(x,y)+c[0]\big) \dd \mu_n(x,y) \Big| & = C_n^{-1} \Big| \sum_{k=-N_n}^{-1} \beta_{k+1}^n\big(c(x_k^n ,x_{k+1}^n  )+c[0]\big)\Big| \\
&= C_n^{-1}\Big|\big( u_{\lambda_n}^\alpha(x_0^n) - \beta_{-N_n}^n u_{\lambda_n}^\alpha(x_{-N_n}^n)\big)\Big| \leqslant 2C_n^{-1} \|u_{\lambda_n}^\alpha\|_\infty.
\end{align*}
Recalling that the family $(u_\lambda)_{\lambda^\alpha \in (0,1)}$ is uniformly bounded (Corollary \ref{equilambdabounded}) letting $n\to +\infty$ it follows that $\int_{X\times X}c(x,y)\  \dd \mu(x,y) = -c[0]$.

{\bf The measure $\mu$ satisfies $\int \alpha\ \dd \mu = 0$:} we use the inequality $\exp(x)\geqslant 1+x$ and the definition of $\beta_k^n$ to estimate
\begin{align*}
\int_{X\times X} \alpha(x)\ \dd \mu_n(x,y) &= C_n^{-1}\sum_{k=-N_n}^{-1} \beta_{k+1}^n \alpha(x_k^n) \\
&  \leqslant C_n^{-1}\sum_{k=-N_n}^{-1} \alpha(x_k^n) \exp\Big(-\lambda_n \sum_{j=k+1}^{-1} \alpha(x_j^n)\Big)  \\
&\leqslant \frac{  \exp(\|\alpha\|_\infty) }{C_n} \sum_{k=-N_n}^{-1} \alpha(x_k^n) \exp\Big(-\lambda_n \sum_{j=k}^{-1} \alpha(x_j^n)\Big) .
\end{align*}
As the $\alpha_k^n$ are non--negative and the function $x\mapsto \exp(-x)$ is decreasing,  the right hand side can be estimated by comparing sum and integral to conclude that 
$$\int_{X\times X} \alpha(x)\ \dd \mu_n(x,y) \leqslant  \frac{ \exp(\|\alpha\|_\infty)}{C_n} \int_0^\infty \exp(-\lambda_n x) \ \dd x = \frac{ \exp(\|\alpha\|_\infty)}{\lambda_nC_n}.$$
As $\lambda_nC_n \to +\infty$, it follows that $\int_{X\times X} \alpha(x)\ \dd \mu(x,y) = 0$. Thus $\mu$ is a Mather measure contradicting ($\alpha 2$) and the result is proved.
\end{proof}

As a Corollary, a refined representation formula comes up for the functions $u_\lambda^\alpha$:

\begin{co}\label{replambda}
Let $\lambda \in (0,1)$ and $x_0\in X$. If $(x_k)_{k\leqslant 0}$ is given by Proposition \ref{caliblambda} applied to $u_{\lambda}^\alpha$, then 
$$u_\lambda^\alpha (x_0) =  \sum_{k=-\infty}^{-1} \beta_{k+1}\big(c(x_{k},x_{k+1})+c[0]\big),$$
with $\beta_k = \prod\limits_{j=k}^{-1} \big(1-\lambda\alpha(x_{j})\big)$ and $\beta_0 = 1$.
\end{co}
\begin{proof}
By Proposition \ref{epslambda}, the sum $\sum \beta_k$ is convergent which implies that $\lim\limits_{k\to -\infty} \beta_k = 0$. As the function $u_\lambda^\alpha$ is bounded, the result follows by simply letting $n\to +\infty$ in the second part of Proposition \ref{caliblambda}.
\end{proof}

Let us  now enter  the convergence part of this section. Motivated by Proposition \ref{inegdegdisc} we give the following definition: 

\begin{df}\rm
Let $\mathcal F_\alpha$ be the set of continuous critical subsolutions $u:X\to \R$ such that 
 $\int_X \alpha(x)u(x)\ \dd\pi_{1*}\mu (x) \leqslant 0$ for all Mather measure $\mu\in \widehat\PP_0$.
\end{df}
Let us first state the main Theorem. The careful reader will notice quite a resemblance with Theorem \ref{discounted} and Proposition \ref{discountedbis}:

\begin{Th}\label{Thconvdiscdeg}
The family of functions $(u_\lambda^\alpha)_{(0,1)}$ uniformly converges as $\lambda \to 0$. Moreover, denoting by $u_0^\alpha$ the limit, the two following formulas hold:
\begin{itemize}
\item for $x_0\in X$, $u_0^\alpha(x_0) = \max\limits_{u\in \mathcal F_\alpha} u(x_0)$;
\item for $x_0\in X$, 
$$u_0^\alpha(x_0) = \min\limits_{\mu\in \widehat\PP_0} \frac{ \int_X \alpha(x)h(x,x_0)\ \dd\pi_{1*}\mu (x) }  {\int_X \alpha(x)\ \dd\pi_{1*}\mu (x)  },$$
where $h : X\times X \to \R$ still denotes Peierls' barrier given by Definition \ref{Peierl}.
\end{itemize}
\end{Th}

The proof of this Theorem is split into several Lemmas resembling what was done for the standard discounted equation.

\begin{df}\rm\label{defmeaslambda}
If $\lambda \in (0,1)$ and $x_0\in X$, we choose a sequence  $(x_k^\lambda)_{k\leqslant 0}$ given by Proposition \ref{caliblambda} applied to $u_{\lambda}^\alpha$ with $x_0 = x_0^\lambda$.  The probability measure $\mu_{x_0}^\lambda$ is defined by:
$$\mu_{x_0}^\lambda  =C_{x_0,\lambda}^{-1} \sum_{k=-\infty}^{-1} \beta_{k+1}^{x_0,\lambda} \delta_{(x_k^\lambda ,x_{k+1}^\lambda  )},$$
where $\beta_k^{x_0,\lambda} =  \prod\limits_{j=k}^{-1} \big(1-\lambda\alpha(x_{j}^\lambda)\big) $ and
$C_{x_0,\lambda} = \sum\limits_{k=-\infty}^{-1} \beta_{k+1}^{x_0,\lambda} $.

\end{df}

The sum defining $C_{x_0,\lambda}$ is indeed finite by Proposition \ref{epslambda}.

\begin{lm}\label{discdegmes}
Let $x_0 \in X$ and $\lambda_n\to 0$ be a sequence such that the family of measures $(\mu_{x_0}^{\lambda_n})_{n\in \N}$ converges to a probability measure $\mu$. Then $\mu $ is a minimizing Mather measure.
\end{lm}

\begin{proof}
We first prove that $C_{x_0,\lambda_n}\to +\infty$. Indeed, for all $n>0$ and $k\leqslant 0$, $\beta_k^{x_0,\lambda_n}\geqslant (1-\lambda_n\|\alpha\|_\infty)^{|k|}$ thus implying that 
$$C_{x_0,\lambda_n}  \geqslant \sum_{j=0}^{+\infty} (1-\lambda_n\|\alpha\|_\infty)^{j}= \frac{1}{\lambda_n\|\alpha\|_\infty} \underset{n\to+\infty}{\longrightarrow}+\infty.
$$
By computations the reader should already be familiar with, from the proof of Proposition \ref{epslambda}, it is proven that $\mu$ is closed.
Let $f : X\to \R$ be a continuous function. We compute using an Abel transform:
\begin{multline*}
\Big| \int_{X\times X}\big(f(y)-f(x)\big) \ \dd \mu_{x_0}^{\lambda_n}(x,y)  \Big| = C_{x_0,\lambda_n}^{-1}\Big|  \sum_{k=-\infty}^{-1} \beta_{k+1}^{x_0,\lambda_n}\big(  f(x_{k+1}^{\lambda_n})  -f(x_k^{\lambda_n})\big) \Big| \\
=C_{x_0,\lambda_n}^{-1}\Big|  \sum_{k=-\infty}^{-1} (\beta^{x_0,\lambda_n}_{k} -\beta_{k+1}^{x_0,\lambda_n}) f(x_{k}^{x_0,\lambda_n})  +\beta_{0}^{x_0,\lambda_n} f(x_{0})  \Big| \\
\leqslant C_{x_0,\lambda_n}^{-1} \Big[ \sum_{k=-\infty}^{-1} (\beta^{x_0,\lambda_n}_{k+1} -\beta_{k}^{x_0,\lambda_n})\|  f\|_\infty + \|f\|_\infty \Big] \leqslant 2C_{x_0,\lambda_n}^{-1}\|f\|_\infty.
\end{multline*}
As $C_{x_0,\lambda_n}\to +\infty$ this proves that $\int_{X\times X}\big(f(y)-f(x)\big)\  \dd \mu(x,y) = 0$.

And then it is established that $\mu$ is minimizing:
\begin{multline*}
 \Big|\int_{X\times X}\big(c(x,y)+c[0]\big) \dd \mu_{x_0}^{\lambda_n}(x,y) \Big|=
 \\ 
 = C_{x_0,\lambda_n}^{-1} \Big| \sum_{k=-\infty}^{-1} \beta_{k+1}^{x_0,\lambda_n}\big(c(x_k^{\lambda_n} ,x_{k+1}^{\lambda_n}  )+c[0]\big)\Big| \\
= C_{x_0,\lambda_n}^{-1}\big| u_{\lambda_n}^\alpha(x_0) \big| \leqslant C_{x_0,\lambda_n}^{-1} \|u_{\lambda_n}^\alpha\|_\infty.
\end{multline*}
Corollary \ref{replambda} was used for the last equality. Recalling that the family $(u_\lambda^\alpha)_{\lambda \in (0,1)}$ is uniformly bounded (Corollary \ref{equilambdabounded}), letting $n\to +\infty$ it follows that 
$$\int_{X\times X}c(x,y)\  \dd \mu(x,y) = -c[0],$$
 thus concluding the proof.

\end{proof}
The next lemma is similar to Lemma \ref{lmcle}:

\begin{lm}\label{lminegdegcle}
Let $x_0\in X$, $\lambda \in(0,1)$ and $w\in \mathcal S$ be a continuous subsolution, then
$$u_\lambda^\alpha(x_0) \geqslant w(x_0) - \lambda  C_{x_0,\lambda} \int_X \alpha(z) w(z) \ \dd \pi_{1*}\mu_{x_0}^\lambda (z).
$$
\end{lm}

\begin{proof}
We use Corollary \ref{replambda} and the fact that $w$ is a critical subsolution:
\begin{multline*}
u_\lambda^\alpha(x_0)  =  \sum_{k=-\infty}^{-1} \beta_{k+1}^{x_0,\lambda}\big(c(x_{k}^\lambda , x_{k+1}^\lambda)+c[0]\big) \\
\geqslant  \sum_{k=-\infty}^{-1} \beta_{k+1}^{x_0,\lambda}\big(w(x_{k+1}^\lambda) - w(x_{k}^\lambda) \big)  \\
=\sum_{k=-\infty}^{-1} (\beta^{x_0,\lambda}_{k} -\beta_{k+1}^{x_0,\lambda}) w(x_{k}^{x_0,\lambda})  +\beta_{0}^{x_0,\lambda}w(x_{0}).
\end{multline*}
The last equality follows by an Abel transform. We now use the definition of $\beta^{x_0,\lambda_n}_{k}$ to compute
\begin{multline*}
\beta^{x_0,\lambda}_{k} -\beta_{k+1}^{x_0,\lambda} =  \prod\limits_{j=k}^{-1} \big(1-\lambda\alpha(x_{j}^\lambda)\big)- \prod\limits_{j=k+1}^{-1} \big(1-\lambda\alpha(x_{j}^\lambda)\big) \\
= -\lambda  \alpha(x_{k}^\lambda)  \prod\limits_{j=k+1}^{-1} \big(1-\lambda\alpha(x_{j}^\lambda)\big)
= -\lambda  \alpha(x_{k}^\lambda) \beta_{k+1}^{x_0,\lambda}.
\end{multline*}
Going back to the previous computation and remembering that $\beta_{0}^{x_0,\lambda}=1$ yields
\begin{multline*}
u_\lambda^\alpha(x_0) \geqslant w(x_0) - \lambda  \sum_{k=-\infty}^{-1} \beta_{k+1}^{x_0,\lambda}    w(x_{k}^{x_0,\lambda})\alpha(x_{k}^\lambda)
\\
= w(x_0) - \lambda  C_{x_0,\lambda}   \int_X \alpha(z) w(z) \ \dd \pi_{1*}\mu_{x_0}^\lambda (z).
\end{multline*}
\end{proof}

The first part of  Theorem \ref{Thconvdiscdeg} is now ready to be proven:

\begin{proof}[Proof of Theorem \ref{Thconvdiscdeg} first formula]
Let $\lambda_n \to 0$ be a sequence such that $(u_{\lambda_n}^\alpha)_{n\in \N}$ converges to a function $v: X\to \R$.  Henceforth the function $v$ is  a weak KAM solution by continuity of the Lax--Oleinik operator. We have also defined for all $x \in X$,  $u_0^\alpha(x) = \max\limits_{u\in \mathcal F_\alpha} u(x)$. The aim here is to  prove that $v=u_0^\alpha$.

By Proposition \ref{inegdegdisc}, $v\in \mathcal F_\alpha$ and therefore $v\leqslant u_0^\alpha$.

Let us now prove the reverse inequality. Let $x_0\in X$. Up to a further extraction, we assume that the sequence of probability measures $(\mu_{x_0}^{\lambda_n})_{n\in \N}$ weakly converges to a measure $\mu$ that is a minimizing Mather measure thanks to Lemma \ref{discdegmes}. If $w\in \mathcal F_\alpha$, by definition, $\int_X \alpha(x)w(x) \ \dd\pi_{1*}\mu(x) \leqslant 0$. Combining with Proposition \ref{epslambda} entails that 
$\limsup\limits_{n\to+\infty} \lambda_n  C_{x_0,\lambda_n} \int_X \alpha(z) w(z) \ \dd \pi_{1*}\mu_{x_0}^{\lambda_n} (z)\leqslant 0$. Plugging into the inequality of Lemma \ref{lminegdegcle}  and letting $n\to +\infty$ gives
$$v(x_0) \geqslant w(x_0) - \limsup_{n\to+\infty} \lambda_n  C_{x_0,\lambda_n} \int_X \alpha(z) w(z) \ \dd \pi_{1*}\mu_{x_0}^{\lambda_n} (z)\geqslant w(x_0).
$$
As this holds for all $w\in \mathcal F_\alpha$ it comes that $v(x_0) \geqslant u_0^\alpha(x_0)$ and being true for all $x_0\in X$,  the first convergence formula is proven.
\end{proof}

 This section ends by establishing the second representation formula for $u_0$. To this aim, we set $\hat u_0^\alpha(x_0) =\min\limits_{\mu\in \widehat\PP_0} \frac{ \int_X \alpha(x)h(x,x_0)\dd\pi_{1*}\mu (x) }  {\int_X \alpha(x)\dd\pi_{1*}\mu (x)  }$ for all $x_0 \in X$, we will prove that $u_0^\alpha = \hat u_0^\alpha$. The proof follows closely that of Proposition \ref{discountedbis}:
 
 \begin{proof}[Proof of Theorem \ref{Thconvdiscdeg} second formula]
 We first claim that $\hat u_0^\alpha$ is a subsolution. Indeed, each function $h_y = h(y, \cdot)$ is a subsolution by Proposition \ref{peierl-prop}. Hence, if $\mu $ is a probablility measure on $X$, so is $h_\mu^\alpha$ defined by $h_\mu^\alpha(x) = \frac{ \int_X \alpha(y)h(y,x) \dd  \mu(y)}{\int_X \alpha(y) \dd  \mu(y)}$ since $\SS$ is closed and convex (see Proposition  \ref{propS}).
Last, as $\hat u_0^\alpha$ is an infimum of functions of this type, it is itself a subsolution by Lemma \ref{inf}.

Next, we establish that $u_0^\alpha\leqslant \hat u_0^\alpha$. Let $u\in \SS$ be a continuous subsolution, we know that $\alpha(y)\big(u(x)-u(y)\big)\leqslant\alpha(y) h(y,x)$ for all pairs $(x,y)$ (Proposition \ref{peierl-prop}). Let $\mu \in \widehat\PP_0$ be a Mather measure, integrating with respect to $y$ the previous inequality yields
$$\left( \int_X \alpha(y) \  \dd \pi_{1*} \mu(y ) \right) u(x) - \int_X \alpha(y) u(y)\  \dd \pi_{1*} \mu(y) \leqslant \int_X \alpha(y)h(y,x)\  \dd \pi_{1*} \mu(y).$$
If $u\in \mathcal F_\alpha$ it follows that $\Big(\int_X \alpha(y) \  \dd \pi_{1*} \mu(y)\Big)u(x)\leqslant \int_X \alpha(y)h(y,x) \ \dd \pi_{1*} \mu(y)$. This being valid for all $u\in \mathcal F_\alpha$ and for all $\mu \in \widehat\PP_0$,  the desired inequality $u_0^\alpha\leqslant \hat u_0^\alpha$ is obtained.

We conclude by proving the reverse inequality. Let $y\in X$, the function $h^y = -h(\cdot ,y)$ is a subsolution (by Proposition \ref{peierl-prop}). Moreover, by definition of $\hat u_0^\alpha$, the function $h^y + \hat u_0^\alpha(y) \in \mathcal F_\alpha$. In particular, $u_0^\alpha \geqslant h^y - \hat u_0^\alpha(y)$ and evaluating at $y$ yields $u_0^\alpha(y) \geqslant -h(y,y) + \hat u_0^\alpha(y)$. If we specify, moreover, $y\in \AA$ to be in the projected Aubry set, leads to the inequalities  (see Theorem \ref{Aubry0})
$$\forall y\in \AA,\quad u_0^\alpha(y) \geqslant \hat u_0^\alpha(y).$$
This is enough  to conclude that $u_0^\alpha \geqslant \hat u_0^\alpha$ everywhere, indeed, $u_0^\alpha$ is a weak KAM solution and $\hat u_0^\alpha\in \SS$ hence Theorem \ref{uniqueness} applies.

 \end{proof}

\section{Comment on the discounted procedure}

As  already mentioned, the introduction of the functions $u_\lambda$ in the second proof of the Weak KAM Theorem \ref{weak KAM} is very natural. Indeed, let us recall a classical fixed point Theorem:

\begin{Th}
Let $C$ be a compact  convex subset of a Fr\' echet vector--space and $f : C \to C$ be a $1$--Lipschitz map. Then $f$ admits a fixed point.
\end{Th}

This result is of course weaker than the Schauder-Tychonoff Theorem but a simple proof goes as follows. Up to conjugating by a translation,  assume that $0\in C$. Then for $\lambda \in (0,1)$, the function $f_\lambda : C\to C$ defined by $f_\lambda(x) = f(\lambda x)$ is well defined and a contraction of a complete metric space. It admits a unique fixed point $x_\lambda\in C$. Consequently, by compactness of $C$ one can  consider a sequence $\lambda_n \to 1$ such that $(x_{\lambda_n})_{n\in \N}$ converges to a point $x^*\in C$. It is then immediate that $x^*$ is a fixed point of $f$. Note that if $0$ is a fixed point of $f$, then $x_\lambda = 0$ for all $\lambda \in (0,1)$.

A natural question is to figure out if in the previous procedure, the whole family $x_\lambda$ always converges. If this were the case, our discounted Theorem \ref{discounted} would be less interesting. However, this is not the case as we now illustrate.

Our example is constructed in $(\R^2 , \| \cdot \|_1)$. More precisely, let us consider the triangle defined by 
$$\mathfrak T = \left\{(x,y)\in \R^2 ,\ \ -\frac12 \leqslant y \leqslant -| x | +\frac12 \right\}.$$

If $\alpha \in (0,1)$, we look for a map $f$ that takes the following form:
$$f(x,y) = \Big(x+\varepsilon(y) , \alpha \big(y+\frac12\big) -\frac12\Big) ,$$ 
where $\varepsilon : \big[-\frac12,\frac12\big]\to \R$ is a map to be determined such that $\varepsilon(-\frac12) = 0$. In this setting, the bottom edge of $\mathfrak T$ is made of fixed points of $f$.

Simple verifications show that $f : \mathfrak T\to \mathfrak T$ is well defined as soon as $|\varepsilon(y) |\leqslant (1-\alpha)(y+\frac12)$. Moreover, it is $1$-Lipschitz if $\varepsilon$ is $(1-\alpha)$-Lipschitz.

If those conditions are verified, an explicit computation shows that for $\lambda\in (0,1)$, denoting by $X_\lambda = (x_\lambda,y_\lambda)$ the unique fixed point of $f_\lambda$, 
$$(x_\lambda,y_\lambda) = \left( \frac{1}{1-\lambda}\varepsilon\Big(\frac{\lambda(\alpha-1)}{2(1-\alpha \lambda)} \Big) , \frac{\alpha -1}{2(1-\alpha \lambda)} \right).$$

By setting $g(\lambda) = \frac{\lambda(\alpha-1)}{2(1-\alpha \lambda)} $, one computes that 
$$g^{-1}(\mu) =  \frac{2\mu}{\alpha - 1 + 2\alpha \mu} .$$
Hence $g$ is a bi--Lipschitz decreasing homeomorphism from $[0,1]$ to $[-1/2,0]$.

Now, define $h : \R \to \R$ by $h(x)= (1-x)\sin\big(\ln(|1-x|)\big)$ for $x\neq 1$ that  extends by continuity with $h(1)=0$. As
$$\forall x\neq 1, \quad h'(x) =  - \sin\big(\ln(|1-x|)\big) - \cos\big(\ln(|1-x|)\big),$$
$h$ is a Lipschitz function. It follows that for $\varepsilon_0>0$ small enough, the function $\varepsilon = \varepsilon_0 h\circ g^{-1}$ is $(1-\alpha)$--Lipschitz on $[-\frac12,0 ]$ and verifies $\varepsilon(-\frac12)=0$. Extend it by $\varepsilon(y) = \varepsilon(0)$ for $y\in [0,\frac12]$.

\begin{figure}[h!]
	\begin{minipage}[b]{0.45\linewidth}
		\centering \includegraphics[scale=0.45]{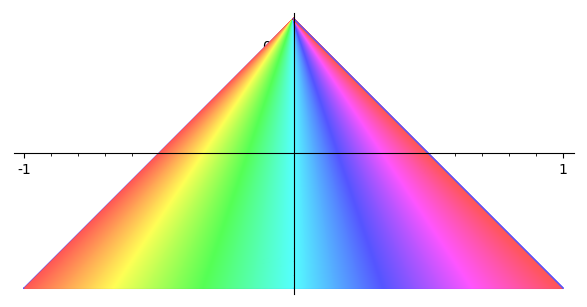}
		\caption{\ The triangle  $\mathfrak T$ filled with rainbow colours. }
	\end{minipage}
	\hfill
	\begin{minipage}[b]{0.48\linewidth}	
		\centering \includegraphics[scale=0.45]{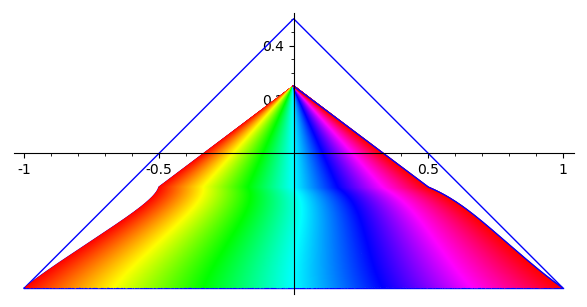}
		\caption{Its image by $f$ for $\alpha = \frac 34$, $\varepsilon_0 = \frac{1}{10}$.}
	\end{minipage}
\end{figure}

For the function $f$ associated to the latter $\varepsilon$, we compute that
$$\forall \lambda\in (0,1),\quad X_\lambda = (x_\lambda,y_\lambda) = \Big( \varepsilon_0\sin\big(\ln(1-\lambda)\big) , \frac{g(\lambda)}{\lambda}\Big).$$

\begin{figure}[h!]
	\begin{minipage}[b]{0.5\linewidth}
		\centering \includegraphics[scale=0.45]{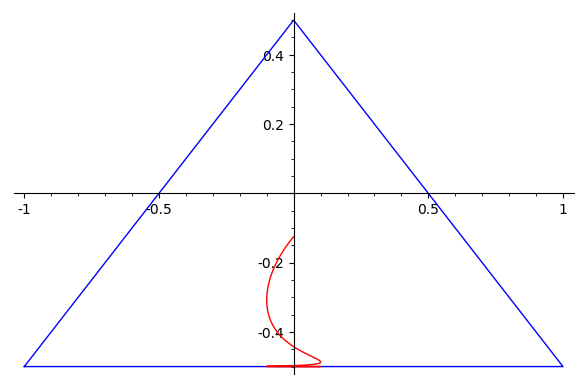}
		\caption{In red, the curve of fixed points $(X_\lambda)_{\lambda \in (0,1)}$. }
	\end{minipage}
	\hfill
	\begin{minipage}[b]{0.48\linewidth}	
		\centering \includegraphics[scale=0.45]{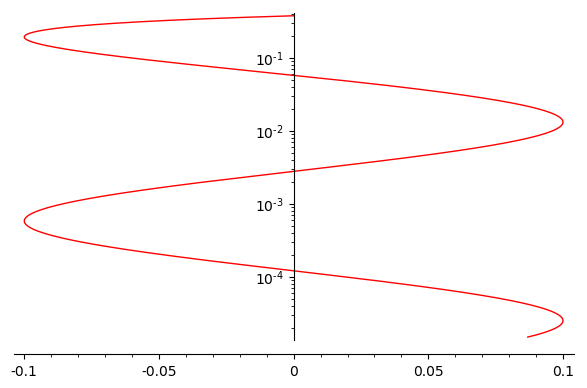}
		\caption{Same curve in vertical logarithmic scale.}
	\end{minipage}
\end{figure}

Clearly $X_\lambda $ diverges as $\lambda \to 1$. Let us also refer to \cite{Ziliotto} for other counterexamples related to the discounted equations with non--convex Hamiltonians.

On the positive side, let us mention another convergence result. We state it in finite dimensions and refer to \cite{Zfixed}\footnote{Since writing \cite{Zfixed}, the author realized that the following result is actually a particular case of previous Theorems of Reich (\cite{Reich}). See also \cite{Kirk,Kirk2} for many further developments.} and references therein for further results. Let us recall that a norm $\| \cdot \|$ on $\R^n$ is called smooth if the function $x \mapsto \|x \|$ is $C^1$ on $\R^n \setminus \{0\}$ or equivalently, if the function  $x \mapsto \|x \|^2$ is $C^1$ on $\R^n $.

\begin{Th}
Let $\| \cdot \|$ be a smooth norm on $\R^n$. Let $C\subset \R^n$ be a compact convex set such that $0\in C$. Finally, let $f : C\to C $ be a $1$--Lipschitz map. For all $\lambda \in (0,1)$ we denote by $X_\lambda\in C$ the unique point such that $X_\lambda = f(\lambda X_\lambda)$. Then the family $(X_\lambda)_{\lambda \in (0,1)}$ converges as $\lambda \to 1$. 
\end{Th}

Smoothness here is used as the unit sphere has a unique tangent linear hyperplane at each of its points. A good exercise is to prove the Theorem in the Euclidean case. If the norm comes from a scalar product $\langle \cdot , \cdot \rangle$, and if $\|x\|=1$, this tangent hyperplane is given by the linear form $\langle x, \cdot \rangle$. In this case, it can be established that as $\lambda \to 0$, the points $X_\lambda$ converge to the orthogonal projection of $0$ on the set of fixed points of $f$.

\section{Relations to the classical theory}
Here again, $L$ is a Tonelli Lagrangian on $TM$ the tangent bundle of a smooth compact manifold $M$ endowed with a Riemannian metric.
\subsection{Minimizing Mather measures} 
\underline{The Ma\~ n\'e point of view}: this first approach was actually introduced after Mather's original one by Ma\~ n\'e in \cite{MClosed,Mane}. Mather then noticed that Ma\~ n\'e's point of view could be reformulated in a more synthetic manner and the following results were definitively written in the present form in \cite{FSC1}. 

\begin{df}\label{closedM}\rm
A Borel probability measure $\mu'$ on $TM$ is termed {\it closed} if it has finite first moment, $\int_{TM} \|v\|_x \dd \mu' (x,v) <+\infty$ and if for all $C^1$ functions $f : M\to \R$,
$$\int_{TM} D_x f (v)\dd \mu' (x,v)=0.$$
We denote by $\PP'$ the set of closed probability measures on $TM$.
\end{df}
The finite first moment condition is there so that the integral is absolutely convergent.
Examples of closed measures can be constructed  of the form $\nu_\gamma = \frac1T \int_0^T\delta_{\textrm{$\big($} \gamma(s),\dot\gamma(s)\textrm{$\big)$}} \dd s$ where $\gamma : [0,T]\to M$ is a  $C^1$ curve such that $\gamma(0)=\gamma(T)$. Indeed, if $f : M\to \R$ is $C^1$, then
$$\int_{TM} D_x f (v)\dd \nu_\gamma (x,v) = \int_0^TD_{\gamma(s)}f\big(\dot\gamma(s)\big) \dd s = f\big(\gamma(T)\big) - f\big(\gamma(0)\big) = 0.$$

Ma\~n\' e's version of Mather measures and Mather's critical value is then:
 \begin{Th}[Ma\~n\' e] \label{minimizingMane}
 The following equality holds:
 $$-\alpha(0) = \min_{\mu'\in \PP'} \int_{TM} L(x,v) \dd \mu'(x,v).$$
 Moreover, a closed measure realizes this minimum if and only if it is supported on the Aubry set $ \AA'$. Last, a minimizing measure is automatically invariant by the Lagrangian flow $\varphi_L$.
 \end{Th}
 
 This justifies the definition of the Mather set:
 \begin{df}\rm
Let us  denote by $\PP_0'$ the set of minimizing closed probability measures, that is, closed probability measures $\mu'$ such that  $\int_{TM} L(x,v) \dd \mu' = -\alpha(0)$. The measure $\mu'$ is then said to be a {\it Mather measure}.

On $T^*M$, we define the set $\PP_0^* = \{ \LL_* \mu' , \ \ \mu' \in \PP_0'\}$.

Let us define the {\it Mather set} $ \MM'\subset TM$ by 
$$ \MM' =\overline{ \bigcup_{\mu'\in \PP_0'} \mathrm {supp}(\mu')},$$
 The {\it projected Mather set} is $\MM = \pi(\MM') $.
 
  Finally the {\it Mather set in 
 $T^*M$} is  
 $$\MM^* = \LL(\MM') = \overline{ \bigcup_{\mu^*\in \PP_0^*} \mathrm {supp}(\mu^*)}.
 $$
\end{df}

It is apparent from these results that $\MM' \subset \AA'\subset \LL^{-1}\big(H^{-1}(\{\alpha(0)\})\big)$ (this is Carneiro's Theorem \cite{carneiro}) and, as for the discrete case, and for the same reasons, $\PP_0'$ is convex and compact and there exists one Mather measure whose support is the whole $\MM'$. Finally, Theorem \ref{uniquenessM}, stating that $\MM$ is a uniqueness set for weak KAM solutions, holds. We do not rewrite it here.

\underline{The Mather point of view}: it is more dynamical in nature, hence reminiscent of the ergodic viewpoint of subsection \ref{BernardBuffoni}. It  also reflects Mather's original definitions as stated in \cite{Mather1} following his results on twist maps from \cite{MatherMeasure}.

\begin{df}\label{invmeasure}\rm
We denote $\PP'_L$ be the set of Borel probability measures on $TM$ invariant by the Lagrangian flow $\varphi_L$.
\end{df}

The historical definition of Mather's critical constant is contained in the next result:

\begin{pr}
The critical constant is characterized by
$$-\alpha(0) = \min_{\mu'\in \PP'_L} \int_{TM} L(x,v) \ \dd \mu'(x,v).$$
Moreover, minimizing measures are automatically closed, hence $\mu'$ is minimizing if and only if $\mu'\in \PP_0'$.
\end{pr}
From a dynamical point of view, Mather's approach is obviously more natural. However, the big drawback is that the condition of being flow invariant depends on the Lagrangian and its flow, in contrast to  the condition of being closed.  This is actually what motivated Ma\~n\'e's change of paradigm as he wanted to study how Mather measures evolve under perturbations of a Lagrangian. It is also very useful as it applies to less regular Lagrangians and Hamiltonians.

 To end this paragraph, let us pursue our systematic approach of highlighting  relationships between objects coming from the classical setting and their analogues coming from the discrete setting for the time--$1$ action functional $h_1$. The main result states that projected Mather sets coincide in both settings, justifying the same notation:
 
 \begin{pr}\label{projMatherset}
 Denoting by $\MM_L$ the projected Mather set associated to $L$ and $\MM_{h_1}$ the projected Mather set associated to its  time--$1$ action functional, the equality $\MM_L = \MM_{h_1}$ holds.
 \end{pr}
 
 \begin{proof}
 Given a point $x\in \AA$, we will denote by $v_x\in T_xM$ the unique vector $v$ such that $(x,v)\in \AA'$ (Theorem \ref{aubry-graph}) and by $y_x = \pi\circ\varphi_L^1(x,v_x)$ the only point such that $(x,y_x)\in \widehat\AA$ (see Proposition \ref{eqAubry}). Then extend the vector--field $x\mapsto v_x$ to a Lipschitz vector--field on $M$ for example by defining $v_x=\LL^{-1}( D_x u)$ where $u $ is a $C^{1,1}$ critical subsolution given by Bernard's Theorem \ref{strictC} or its discrete analogue Theorem \ref{strictC11}.
 
 Let $\mu'\in \PP_0'$ be a classical Mather measure. We associate to it a probability measure on $M\times M$ as follows. If $f : M\times M \to \R$ is a continuous function, then
 \begin{equation}\label{mesure discrete}
 \int_{M\times M} f(x,y)\dd \mu(x,y) = \int_{TM} f\big(x, \pi\circ\varphi_L^1(x,v)\big)\dd \mu'(x,v).
 \end{equation}
 As $\mu'$ has support included in $\AA'$, it follows that $\mu$ has support included in 
 $$\big\{\big(x, \pi\circ\varphi_L^1(x,v)\big), \ \ (x,v)\in \AA' \big\} = \widehat \AA.$$
 Let $g : M\to \R$ be a continuous function, then
  $$\int_{M\times M} \big(g(x)-g(y)\big)\dd \mu(x,y) = \int_{TM} \big(g(x)-g\big( \pi\circ\varphi_L^1(x,v)\big)\big)\dd \mu'(x,v) = 0,$$
  because $\mu'$ is invariant by $\varphi_L^1$   which implies that $ \int_{TM} g\big(\pi(x,v)\big)\dd  \mu' = g\big( \pi\circ\varphi_L^1(x,v)\big)\dd \mu'(x,v)$. Hence $\mu $ is closed.
  
We then    compute the action of $\mu$, remembering Proposition \ref{eqAubry}:
\begin{align*}
\int_{M\times M} h_1(x,y) \dd \mu(x,y) &  =\int_{TM} h_1\big(x, \pi\circ\varphi_L^1(x,v)\big)\dd \mu'(x,v)\\
&= \int_{\AA'} h_1\big(x, \pi\circ\varphi_L^1(x,v)\big)\ \dd \mu'(x,v) \\
&= \int_{\AA'} h_1\big(x, \pi\circ\varphi_L^1(x,v_x)\big)\ \dd \pi_*\mu'(x)\\
&= \int_{\AA'}\int_0^1 L\big(\varphi_L^s(x,v_x)\big)\ \dd s\ \dd \pi_*\mu'(x)\\ 
&= \int_0^1\int_{\AA'} L\big(\varphi_L^s(x,v_x)\big)\ \dd \pi_*\mu'(x)\, \dd s\\ 
&= \int_0^1\int_{\AA'} L(x,v_x)\dd \pi_*\mu'(x)\, \dd s\\ 
&=-\alpha(0).
\end{align*}
The use of the Fubini theorem is justified by the fact that $[0,1]$ and $\AA'$ are compact and $L$ is continuous. It follows that $\mu$ is minimizing hence a discrete Mather measure. Finally, the definition of $\mu$  given by \eqref{mesure discrete} shows that 
$$\textrm{supp}(\mu) = \big\{ \big(x, \pi\circ\varphi_L^1(x,v_x)\big), \ \ (x,v_x)\in \textrm{supp}(\mu') \big\} .$$
It follows that $\pi_1\big(\textrm{supp}(\mu) \big) = \pi\big(\textrm{supp}(\mu')\big)$. That being true for all measures $\mu'\in \PP_0'$ allows to conclude that $\MM_L \subset \MM_{h_1}$. 

Let now    $\mu    \in \widehat\PP_0$ be a minimizing discrete Mather measure on $M\times M$. We define a measure $\mu'_0$ on $TM$ as follows: if $f : TM\to \R$ is bounded and continuous,
$$\int_{TM} f(x,v) \dd \mu_0'(x,v) = \int_{M\times M }  f(x,v_x) \dd \mu(x,y).$$
The measure $\mu_0' $ is not necessarily invariant by the whole Lagrangian flow, but it is $\varphi_L^1$--invariant. Indeed,
\begin{align*}
\int_{TM} f\circ \varphi_L^1(x,v)\ \dd \mu_0'(x,v) =& \int_{M\times M }  f\circ \varphi_L^1(x,v_x)\ \dd \mu(x,y)    \\
=&\int_{\widehat \AA  }  f\circ \varphi_L^1(x,v_x) \ \dd \mu(x,y)    \\
=&       \int_{\widehat \AA }  f(y_x,v_{y_x})\  \dd \mu(x,y)\\
=&       \int_{\widehat \AA }  f(y,v_{y}) \ \dd \mu(x,y)\\
=&       \int_{M\times M }  f(y,v_{y}) \ \dd \mu(x,y)\\
=&       \int_{M\times M }  f(x,v_{x}) \dd \mu(x,y)\\
=&       \int_{TM }  f(x,v) \ \dd \mu'_0(x,v).
\end{align*}
In the previous computation the second equality follows from the fact that $\mu$ is supported in $\widehat \AA$, the third comes from Proposition \ref{eqAubry}, the fourth is a consequence of Proposition \ref{aubry--graph}. Finally the end stems from the fact that $\mu$ is closed, applied to the function $g: x\mapsto f(x,v_x)$. 
It follows that the measure $\mu' = \int_0^1 (\varphi_L^s)_*\mu'_0 \dd s$ is $\varphi_L$--invariant.

 Let us now prove that $\mu'$ is a classical Mather measure.

\begin{align*}
\int_{TM} L(x,v)\ \dd \mu'(x,v) & =   \int_0^1 \int_{TM}L(x,v)\ \dd (\varphi_L^s)_*\mu'_0(x,v)\ \dd s \\
&= \int_0^1 \int_{M\times M}L\circ \varphi_L^s (x,v_x)\ \dd \mu(x,y) \ \dd s\\
&=  \int_{M\times M}\int_0^1L\circ \varphi_L^s (x,v_x) \ \dd s \ \dd \mu(x,y) \\
&=  \int_{\widehat \AA}\int_0^1L\circ \varphi_L^s (x,v_x) \  \dd s\  \dd \mu(x,y) \\
& = \int_{\widehat \AA} h_1 (x, y_x) \dd \mu(x,y) \\
&= \int_{M\times M} h_1 (x, y) \ \dd \mu(x,y) = -\alpha(0).
\end{align*}
Here, it was used that $\mu$ is supported in $\widehat \AA$ and that it is a discrete Mather measure. Hence $\mu'$ is minimizing as desired. Finally, it follows from its definition that 
$$\textrm{supp}(\mu') = \{ \varphi_L^s(x,v_x) , \ \ (x,y_x)\in \textrm{supp}(\mu) , \ \ s\in [0,1] \}.$$
In particular, $\pi_1\big( \textrm{supp}(\mu)\big) \subset   \pi\big(          \textrm{supp}(\mu')\big)$.
As this holds for all discrete Mather measures, it comes  that $\MM_{h_1} \subset \MM_L$. 

This concludes the proof.
 \end{proof}

\begin{rem}\label{perteinfo}\rm
In the previous proof, the construction associating a discrete Mather measure to a classical Mather measure, $\mu' \mapsto \mu$, is injective and the projected supports are the same.

In contrast, the reverse construction $\mu \mapsto \mu'$ may not be injective. Due to the necessity to apply the Lagrangian flow, the support increases and there is a loss of information. In other words, there may be more discrete Mather measures than classical ones; more precisely, the following inclusion holds:
$$\{\pi_* \mu', \ \ \mu'\in \PP_0' \} \subset \{\pi_{1*} \mu, \ \ \mu\in \widehat\PP_0 \},$$
but the inclusion may be strict.  

 For example, as we will see later for twist maps, on a totally periodic circle which does not consist exclusively of fixed points, there is a unique classical Mather measure. In the same time, there are infinitely many discrete Mather measures obtained from averaging Dirac measures on  periodic orbits.
\end{rem}

\subsection{The classical discounted equation}

Recall that for all $\ell>0$ there exists a unique function $U_\ell : M\to \R$, given by Theorem \ref{eqdiscounted}, that solves in the viscosity sense \eqref{discountedHJ} that is, such that $\ell U_\ell(x) + H(x,D_xU_\ell)=0$. The main result on this topic is the convergence of those functions as $\ell\to 0$ proven originally in \cite{DFIZ} (following partial results in \cite{IS}). The result is actually obtained for Hamiltonians that are only assumed to be continuous, coercive and convex:

\begin{Th}\label{discountedC}
There exists a weak KAM solution $U_0$ such that $U_\ell + \frac{\alpha(0)}{\ell}\to U_0$, where the convergence takes place as $\ell\to 0$ and is uniform.
\end{Th}

The next lemma, called {\it  strong comparison principle}, is an analogue to Lemma \ref{comp-princ} and gives some more informations on Theorem \ref{eqdiscounted}:
\begin{lm}\label{Comp-Princ}
Let $\ell>0$ be a contstant, let $v_1:M\to \R$ (resp. $v_2 : M\to \R$) be a viscosity subsolution (resp. supersolution) to \eqref{discountedHJ}. Then $v_1 \leqslant U_\ell\leqslant v_2$.
\end{lm}
The solutions, as in Lemma \ref{ulambda}, are expressed by an explicit formula:
\begin{lm}\label{Uell}
For any $\ell>0$ and $x\in M$,
$$U_\ell(x) = \min_{\substack{\gamma : (-\infty,0] \to M \\ \gamma(0)=x}} \int_{-\infty}^0 e^{\ell s } L\big(\gamma(s), \dot\gamma(s)\big)\ \dd s,$$
where the minimum is taken amongst absolutely continuous curves and is reached by a $C^2$ curve.
\end{lm}
The selected weak KAM solution is then identified as follows 

\begin{Th}\label{formulediscount}

Let $\FF'\subset \SS'$ be the set of classical subsolutions $u$ verifying the constraint $\int_M u(x)\dd\pi_*\mu' (x) \leqslant 0$ for all Mather measures $\mu'\in \PP'_0$. 

The selected weak KAM solution is then $U_0 = \sup\limits_{u\in \FF'} u$ where the supremum is a priori taken pointwise.

Moreover, the following alternative formula holds:
$$U_0(x) = \min_{\mu' \in \PP_0'} \int_M h(y,x) \dd \pi_* \mu'(y),$$
where $h$ is again the Peierls barrier.
\end{Th}

Let us continue with this Proposition:
\begin{pr}\label{satureC}
There exists a Mather measure $\mu_0'\in \PP_0'$ such that  
$$\int_X U_0(x)\dd \pi_*\mu_0'(x) = 0.$$

Moreover, it can be imposed that $\mu_0'$ is ergodic for the Lagrangian flow.
\end{pr}
All those results' proofs follow closely the proofs we gave for their discrete analogue and are to be found in \cite{DFIZ}.

To conclude, as there are more measures in $\widehat\PP_0$ than in $\PP_0'$ one finds that
\begin{pr}\label{inegdiscounted}
The following inequality holds: 
$u_1\leqslant U_0$.
\end{pr}
However, we will provide later an example where  this inequality is strict.

\subsection{Discount for the positive classical L.--O. semigroup}

As is now customary, all results have a ``positive" pendant by reversing time, meaning, by considering the Hamiltonian $\widecheck H$. In this instance,  for $\ell>0$,  define the function $V_\ell$ such that $-V_\ell$ is the only viscosity solution to the equation 
\begin{equation}\label{CdiscountedHJ}
\ell u(x)+\widecheck H(x,D_xu)=0,\quad x\in M \tag{$\ell\widecheck {\rm H}$J}.
\end{equation}

\begin{Th}\label{discounted+C}
There exists  a positive weak KAM solution $V_0$ such that $V_\ell - \frac{\alpha(0)}{\ell} \to V_0$ where the convergence takes place as $\ell \to 0$ and is uniform.
\end{Th}
The functions $V_\ell$ are given by the explicit formula:
 \begin{lm}\label{Vell}
 For any $\ell>0$ and $x\in M$,  
 $$V_\ell(x) =  -\min_{\substack{\gamma : [0,+\infty) \to M \\ \gamma(0)=x}} \int_0^{+\infty} e^{-\ell s } L\big(\gamma(s), \dot\gamma(s)\big)\dd s,$$
where the minimum is taken amongst absolutely continuous curves and is reached by a $C^2$ curve.
 \end{lm}
The limit $V_0$ has the following form:
\begin{pr}\label{discountedbis+C}
Let $\FF^{\prime +}\subset \SSS$ be the set of subsolutions $u$ verifying the constraint $\int_M u(x)\dd\pi_*\mu (x) \geqslant 0$ for all Mather measures $\mu\in \PP'_0$. 

The limit  $V_0$ is expressed as $V_0 = \inf\limits_{u\in \FF^{\prime +}} u$ where the infimum is a priori taken pointwise. And finally for all $x\in M$,
$$V_0(x) = \max_{\mu \in \PP'_0} \int_M- h(x,y) \dd \pi_* \mu(y).$$
\end{pr}

As for the discrete case,  relations do exist between 
$U_0$ and $V_0$ (the proofs are similar hence omitted):
\begin{pr}\label{U<V}
The functions $U_0$ and $V_0$ verify the inequality
 $U_{0|\AA}\leqslant V_{0|\AA}$.
\end{pr}

As far as conditions for $U_0$ and $V_0$ to be a conjugate pair are concerned:

\begin{pr}\label{conjdiscount'}
The following assertions are equivalent:
\begin{enumerate}
\item\label{1'} The functions $U_0$ and $V_0$ form a conjugate pair,
\item\label{2'} $U_{0|\AA} =V_{0|\AA}$,
\item\label{3'} $U_0\geqslant V_0$,
\item\label{4'} for all classical Mather measures $ \mu' \in \PP_0'$, the equality $\int_M U_0(x)\dd \pi_*\mu'(x) = 0$  holds,
\item\label{5'} for all classical Mather measures $ \mu'\in \PP_0'$, the equality $\int_M V_0(x)\dd \pi_*\mu'(x) = 0$  holds,
\item\label{6'} there exists a critical subsolution $v\in \SSS$ such that for all Mather measures $ \mu'$, the equality $\int_M v(x)\dd \pi_*\mu'(x) = 0$  holds.
\end{enumerate}

\end{pr}

\subsection{Some degenerate discounted Hamilton--Jacobi equations}

The corresponding results are inspired by \cite{Zdisc} using also methods and ideas introduced in \cite{ZQJ} where more general problems are studied. Those results  hold as well for less regular Hamiltonians.

In this time--continuous setting, one still considers a continuous function $\delta : M\to \R$ that takes non--negative values and satisfies the condition
$$\forall \mu' \in \PP_0', \quad \int_M \delta(x) \ \dd\pi_* \mu' >0.$$

The degenerate discounted Hamilton--Jacobi equation that here studied is 
\begin{equation}\label{discountedHJdeg}
\ell \delta(x)u(x)+H(x,D_xu)=\alpha(0),\quad x\in M .\tag{$\ell\delta$HJ}
\end{equation}

To be more precise the condition prescribed in \cite{Zdisc} is that $\delta$ is positive on the projected Aubry set. However, the more general case studied in \cite{ZQJ} handles a wider class of perturbations of the critical equation that can be non--linear in $u(x)$. All  results stated below therefore follow from those two references.

The first existence result hereafter  states that our problem is well posed and is, to our knowledge original in this generality:

\begin{Th}\label{existencedeg}
For all $\ell>0$ there exists a unique viscosity solution to \eqref{discountedHJdeg}  denoted by $U_\ell^\delta$.
\end{Th}

The proof of existence uses the next lemma, which is a strong comparison principle. It is also new with such conditions on $\delta$. It  is an analogue to Theorem \ref{strongCP}:
\begin{lm}\label{Comp-Princdeg}
Let $\ell>0$ be a constant, let $v_1:M\to \R$ be a viscosity subsolution to \eqref{discountedHJdeg} and $v_2 : M\to \R$ be a viscosity supersolution to \eqref{discountedHJdeg}. Then $v_1 \leqslant U_\ell\leqslant v_2$.
\end{lm}

The convergence result in this case is:

\begin{Th}\label{discountedCdeg}
There exists a weak KAM solution $U_0^\delta$ such that $U_\ell^\delta\to U_0^\delta$ where the convergence takes place as $\ell\to 0$ and is uniform.
\end{Th}

The solutions, as in Corollary \ref{replambda}, are expressed by an explicit formula:
\begin{lm}\label{Uelldelta}
If $t>0$ and  $\gamma : [-t,0] \to M$ is an absolutely continuous curve, we set $A_\gamma(-t) = -\int_{-t}^0 \delta \circ \gamma (s) \ \dd s$.

For any $\ell>0$ and $x\in M$, if $t>0$ then
\begin{multline*}
U_\ell^\delta(x) = \min_{\substack{\gamma : [-t,0] \to M \\ \gamma(0)=x}}\Big\{ \exp\big( \ell A_\gamma(-t)\big) U_\ell^\delta\big(\gamma(-t)\big) 
\\
+ \int_{-t}^0  \exp\big( \ell A_\gamma(s)\big) \big[L\big(\gamma(s), \dot\gamma(s)\big)+\alpha(0)\big]\ \dd s\Big\},
\end{multline*}
where the minimum is taken amongst absolutely continuous curves and is reached by a Lipschitz curve.

Moreover, there exists a Lipschitz curve $\gamma : (-\infty ,0] \to M$ such that $\gamma(0) = x$ and 
$$U_\ell^\delta(x) = \int_{-\infty}^0  \exp\big( \ell A_\gamma(s)\big) \big[L\big(\gamma(s), \dot\gamma(s)\big)+\alpha(0)\big] \ \dd s. $$
\end{lm}
The selected weak KAM solution is then identified as follows: 

\begin{Th}\label{formulediscountdeg}

Let $\FF'_\delta \subset \SS'$ be the set of classical subsolutions $u$ verifying the constraint $\int_M \delta(x)u(x)\dd\pi_*\mu' (x) \leqslant 0$ for all Mather measures $\mu'\in \PP'_0$. 

The selected weak KAM solution is then $U_0^\delta = \sup\limits_{u\in \FF'_\delta} u$ where the supremum is  taken pointwise.

Moreover, the alternative formula holds:
$$U_0(x) = \min_{\mu' \in \PP_0'} \frac{\int_M \delta(y) h(y,x)\  \dd \pi_* \mu'(y)}{ \int_M \delta(y) \  \dd \pi_* \mu'(y) },$$
where $h$ is again the Peierls barrier.
\end{Th}

\newpage
\newpage
\mbox{}
\newpage

\chapter{A family of examples}

We explore here explicit examples to show how the pair $(u_1,v_1)$ may behave.
Recall that on the one hand $u_1$ is the limit of the solutions to the discounted equations $(u_\lambda)_{\lambda\in (0,1)}$ as $\lambda \to 1$ for the negative Lax--Oleinik semigroup $T^-$. On the other hand, $v_1$ is the limit of the solutions to the discounted equations $(v_\lambda)_{\lambda\in (0,1)}$ as $\lambda \to 1$ for the positive Lax--Oleinik semigroup $T^+$.

 As the examples presented below come from Hamiltonian systems, some familiarity with the classical theory could help the reader. The study of those examples is familiar to specialists of weak KAM theory but we have not found it written in the literature. We believe that the informations they entail is interesting and that they provide counter--examples  to natural  questions. At the end of the Chapter, we also address the question as to whether weak KAM solutions selected by the discounted approximation procedure in the discrete and in the continuous setting coincide.  

The  setting will be the one dimensional torus $\T^1=\R\slash \Z$. 
\vspace{2mm}

\fbox{
\begin{minipage}{0.9\textwidth}
We consider a smooth potential $V : \T^1\to \R$ that attains its maximum at exactly two points $0$ and $X$ and such that $V(0) = V(X)=0$. 
\end{minipage}}
\vspace{2mm}

Consider the Hamiltonian function $H_0(x,p)= \frac12p^2 + V(x)$ defined on $\T^1\times \R$. The associated Lagrangian is then $L_0 : (x,v) \mapsto  \frac12v^2 - V(x)$. The cost function used is the time--$1$ action functional $h_1^0$ associated to $H_0$, as defined by \eqref{action}. By Theorem \ref{ewd}, discrete and classical weak KAM solutions coincide and we will use this fact. Again, some knowledge of classical Hamilton--Jacobi equations can be useful though not necessary to read this Chapter. Moreover, as these examples fall in the scope of Conservative Twist Maps of the annulus, the latter also illustrate results of the following and last chapter of this essay.

Let us  denote by $f^\pm  : x\mapsto \pm\sqrt{-2V(x)}$. The level set $H_0^{-1}(\{0\})$ is the union of the graphs of $f^+$ and $f^-$. Those graphs touch at $ (0,0) $ and $(X,0)$. 

A last assumption on $H_0$ is the following:
\begin{equation}\label{eye}
\boxed{\alpha:= \frac{ \int_0^{X}f^+(x)\dd x}{X} < \frac{ \int_X^{1}f^+(x)\dd x}{1-X} := \beta.}
\end{equation}

\section{\texorpdfstring{The study of $H_0$}{The study of  H0}}
 Let $X_0^0\in [0,X]$ and $X_1^0\in [X,1]$ verify that 
$$\int_0^{X^0_0}f^+(x)\dd x = \int_{X^0_0}^{X}f^+(x)\dd x \  ; \quad \int_X^{X^0_1}f^+(x)\dd x = \int_{X^0_1}^{1}f^+(x)\dd x .$$
The function $u_1^0$ defined by 
$$u_1^0(x) = \begin{cases}
\int_0^x f^+(s)\dd s & \mathrm{if} \ \ 0\leqslant x\leqslant X^0_0; \\
\int_0^{X^0_0} f^+(s)\dd s + \int_{X_0^0}^{x} f^-(s)\dd s & \mathrm{if} \ \ X^0_0\leqslant x\leqslant X; \\
\int_X^x f^+(s)\dd s & \mathrm{if} \ \ X\leqslant x\leqslant X^0_1 ;\\
\int_X^{X^0_1} f^+(s)\dd s + \int_{X^0_1}^{x} f^-(s)\dd s & \mathrm{if} \ \ X^0_1\leqslant x\leqslant 1 
\end{cases}
$$
verifies $(u_1^{0})'(s)\in H_0^{-1}(\{0\})$ at every point where the derivative exists, that is for $s\in \T^1\setminus\{X^0_0,X^0_1\}$. Moreover, it is semiconcave (as seen here by the fact that at $X^0_0$ and $X^0_1$, the left derivative is bigger than the right derivative). This is enough in this context to prove that $u_1^0$ is a viscosity solution of the stationary Hamilton--Jacobi equation $H_0\big(x,(u_1^{0})'(x)\big)=0$\footnote{Indeed, for a semiconcave function, the super condition property only has to be checked at differentiability points. For more general results see \cite{BJ,BJ2}.}.

\begin{figure}[h!]
\begin{center}

  \begin{tikzpicture}[scale = 1.2]
    
   \draw  (0,-3)--(0,3) ;
     \draw  (8,-3)--(8,3) ;
     \draw (0,0) -- (8,0);

\draw (0,0) node[left]{(0,0)} ;
\draw (8,0) node[right]{(1,0)} ;
\draw (1.7,0) node[below]{$X_0^0$} ;
\draw (5.7,0) node[below]{$X_1^0$} ;
\draw (4,-0.2) node[below]{$X$} ;
\draw [color=blue](1,2.5) node {$H_0^{-1}(\{0\})$} ;

  \draw[color=red][line width=2pt][domain=0:1][samples=250] plot (2
   *\x,2* abs{sin(pi*\x/2 r)} );
    \draw[color=red][line width=2pt][domain=1:2][samples=250] plot (2
   *\x,-2* abs{sin(pi*\x/2 r)} );

  \draw[color=red][line width=2pt][domain=2:3][samples=250] plot (2
   *\x,3* abs{sin(pi*\x/2 r)} );
    \draw[color=red][line width=2pt][domain=3:4][samples=250] plot (2
   *\x,-3* abs{sin(pi*\x/2 r)} );

\draw [color=red] [line width=2pt](2,2* abs{sin(pi*1/2 r)})--(2,-2* abs{sin(pi*1/2 r)}) ;
\draw [color=red] [line width=2pt](6,3* abs{sin(pi*1/2 r)})--(6,-3* abs{sin(pi*1/2 r)}) ;
   \draw [color=blue][domain=0:2][samples=250] plot (2
   *\x,2* abs{sin(pi*\x/2 r)} );
    \draw[color=blue][domain=0:2][samples=250] plot (2
   *\x,-2* abs{sin(pi*\x/2 r)} );
   
     \draw[color=blue][domain=2:4][samples=250] plot (2
   *\x,3* abs{sin(pi*\x/2 r)} );
    \draw[color=blue][domain=2:4][samples=250] plot (2
   *\x,-3* abs{sin(pi*\x/2 r)} );

\end{tikzpicture}

\caption{The graph of the superdifferential $\partial^+ u_1^0$ is drawn in red.}
\label{double2}
\end{center}
\end{figure}

%
\begin{itemize}
\item  It means that the critical constant  is $\alpha(0) = 0$. 
\item The  Aubry and Mather sets are included in the graph of $(u_1^{0})'$ and all Hamiltonian trajectories either converge to the fixed point $(0,0)$ or to the fixed point $(X,0)$. It can be easily concluded from this that $\AA = \MM = \{0,X\}$, that classical Mather measures are convex combinations of Dirac measures  $\delta_{(0,0)}$ and $\delta_{(X,0)}$ (on $T\T^1$), and that discrete Mather measures  are convex combinations of Dirac measures, $\delta_{(0,0)}$ and $\delta_{(X,X)}$ (on $\T^1\times \T^1$).  Note that,  in the present context, this illustrates a classical Theorem of Carneiro for autonomous Hamiltonian systems, namely  that Mather measures are supported on the critical energy level (\cite{carneiro}).
\item Finally, as $u_1^0(0) = u_1^0(X) = 0$, one deduces that the function $u_1^0$ is indeed the weak KAM solution selected by the discounted approximation (Theorem \ref{discounted}). By Proposition \ref{conjdiscount}, in this case, by setting $v_1^0$ the positive weak KAM solution selected by the discounted procedure, the pair $(u_1^0,v_1^0)$ is a conjugate pair. Here, one easily computes that $v_1^0 = -u_1^0$.
\end{itemize}

\section{\texorpdfstring{Increasing the cohomology class: $\boxed{c\in [0,\alpha]}$}{Increasing the cohomology class: c in [0,alpha]}}

We now initiate a classical procedure in Aubry--Mather theory: changing cohomology class. This is related to the topology of the underlying space $X=\T^1$.  This procedure is more thoroughly detailed in the final Chapter on Conservative Twist maps of the Annulus. In the Hamiltonian setting, this originates in the work of Mather \cite{Mather1} who noticed that suitably correcting the Lagrangian (or the Hamiltonian) by a closed $1$--form does not modify the Lagrangian minimizers. The resulting objects of Aubry--Mather theory then only depend on the cohomology class of the $1$--form. In the context of $\T^1$, the first cohomology group $H^1(\T^1,\R)$ is isomorphic to $\R$ and if $c\in \R$, a representing $1$--form is the constant form $x\in \T^1\mapsto c$ where here $c$ is identified to the linear form $v\in\R \mapsto cv$.

Let $c\in [0, \alpha]$, we consider the Hamiltonian $H_c : (x,p)\mapsto \frac12(p+c)^2 + V(x)$. The associated Lagrangian is $L_c : (x,v) \mapsto \frac12 (v-c)^2 -\frac12 c^2 -V(x)$. The flow associated to $H_c$ is conjugated (by a vertical translation) to that of $H_0$. The cost function is $h_1^c$, associated to the time--$1$ action functional of $H_c$.
 Let $X_0^c\in [0,X]$ and $X_1^c\in [X,1]$ verify that 
\begin{multline*}
\int_0^{X^c_0}f^+(x)\dd x-cX_0^c = \int_{X^c_0}^{X}f^+(x)\dd x + c(X-X_0^c) \ ; \\
 \int_X^{X^c_1}f^+(x)\dd x - c (X^c_1 -X) = \int_{X^c_1}^{1}f^+(x)\dd x + c(1-X^c_1) .
\end{multline*}
The function $u_1^c$ defined by 
$$u_1^c(x) = \begin{cases}
\int_0^x f^+(s)\dd s -cx & \mathrm{if} \ \ 0\leqslant x\leqslant X^c_0; \\
\int_0^{X^c_0} f^+(s)\dd s + \int_{X_0^c}^{x} f^-(s)\dd s -cx& \mathrm{if} \ \ X^c_0\leqslant x\leqslant X; \\
\int_X^x f^+(s)\dd s -cx & \mathrm{if} \ \ X\leqslant x\leqslant X^c_1 ;\\
\int_X^{X^c_1} f^+(s)\dd s + \int_{X^c_1}^{x} f^-(s)\dd s -cx& \mathrm{if} \ \ X^c_1\leqslant x\leqslant 1; 
\end{cases}
$$
verifies $(u_1^{c})'(s)\in H_c^{-1}(\{0\})$  for $s\in \T^1\setminus\{X^c_0,X^c_1\}$ and  is  semiconcave. As previously, this yields that $u_1^c$ is a viscosity solution of the stationary Hamilton--Jacobi equation $H_c\big(x,(u_1^{c})'(x)\big)=0$. 

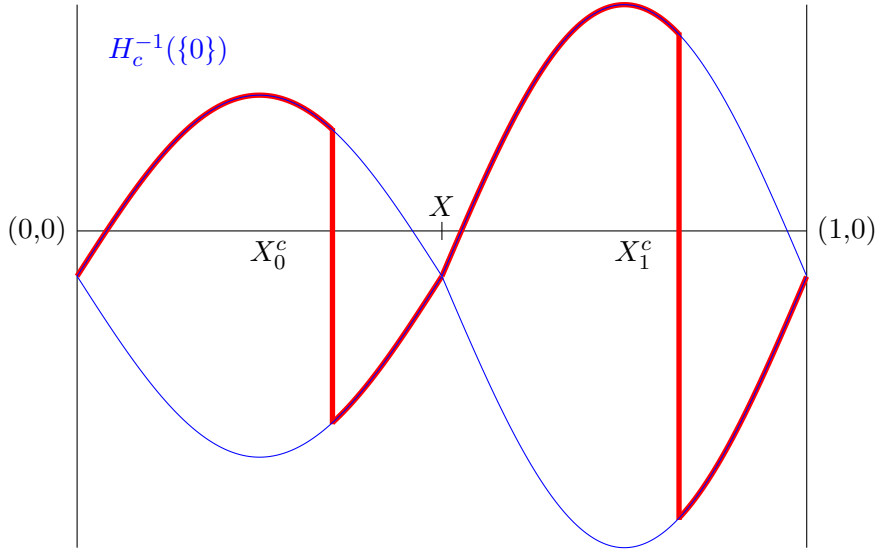
\begin{figure}[h!]
\begin{center}

  \begin{tikzpicture}[scale = 1.2]

   \draw  (0,-3)--(0,3) ;
     \draw  (8,-3)--(8,3) ;
     \draw (0,.5) -- (8,.5);
     \draw (4,.6) -- (4,.4);

\draw (0,0.5) node[left]{(0,0)} ;
\draw (8,0.5) node[right]{(1,0)} ;
\draw (2.1,0.5) node[below]{$X_0^c$} ;
\draw (6.1,0.5) node[below]{$X_1^c$} ;
\draw (4,1) node[below]{$X$} ;
\draw [color=blue](1,2.5) node {$H_c^{-1}(\{0\})$} ;

  \draw[color=red][line width=2pt][domain=0:1.4][samples=250] plot (2
   *\x,2* abs{sin(pi*\x/2 r)} );
    \draw[color=red][line width=2pt][domain=1.4:2][samples=250] plot (2
   *\x,-2* abs{sin(pi*\x/2 r)} );

  \draw[color=red][line width=2pt][domain=2:3.3][samples=250] plot (2
   *\x,3* abs{sin(pi*\x/2 r)} );
    \draw[color=red][line width=2pt][domain=3.3:4][samples=250] plot (2
   *\x,-3* abs{sin(pi*\x/2 r)} );

\draw [color=red] [line width=2pt](2.8,2* abs{sin(pi*.695 r)})--(2.8,-2* abs{sin(pi*.695 r)}) ;
\draw [color=red] [line width=2pt](6.6,3* abs{sin(pi*.645 r)})--(6.6,-3* abs{sin(pi*.645 r)}) ;
  \draw [color=blue][domain=0:2][samples=250] plot (2
   *\x,2* abs{sin(pi*\x/2 r)} );
    \draw[color=blue][domain=0:2][samples=250] plot (2
   *\x,-2* abs{sin(pi*\x/2 r)} );
   
     \draw[color=blue][domain=2:4][samples=250] plot (2
   *\x,3* abs{sin(pi*\x/2 r)} );
    \draw[color=blue][domain=2:4][samples=250] plot (2
   *\x,-3* abs{sin(pi*\x/2 r)} );

\end{tikzpicture}

\caption{The graph of the superdifferential $\partial^+ u_1^c$ is drawn in red.}
\end{center}
\end{figure}

\begin{itemize}
\item It means that the critical constant for the cost function $h_1^c$,  denoted $\alpha(c)$, verifies $\alpha(c) = 0$. 
\item The  Aubry and Mather sets are included in the graph of $(u_1^{c})'$ and all Hamiltonian trajectories either converge to the fixed point $(0,-c)$ or  to the fixed point $(X,-c)$. From there, it can be easily concluded  that $\AA = \MM = \{0,X\}$ (here we drop the subscript $c$ as the sets are independent of it), that classical Mather measures are convex combinations of Dirac measures, $\delta_{(0,0)}$ and $\delta_{(X,0)}$ (on $T\T^1$), and that  discrete Mather measures  are convex combinations of Dirac measures $\delta_{(0,0)}$ and $\delta_{(X,X)}$ (on $\T^1\times \T^1$). 
\item Finally, as $u_1^c(0) = u_1^c(X) = 0$, it is deduced that the function $u_1^c$ is indeed the weak KAM solution selected by the discounted approximation (Theorem \ref{discounted}). 
\item By Proposition \ref{conjdiscount}, in this case again, by setting $v_1^c$ the positive weak KAM solution selected by the discounted procedure, the pair $(u_1^c,v_1^c)$ is a conjugate pair. 
\end{itemize}

Here, one  computes  
$$v_1^c(x) = \begin{cases}
\int_0^x f^-(s)\dd s -cx & \mathrm{if} \ \ 0\leqslant x\leqslant \widecheck X^c_0; \\
\int_0^{\widecheck X^c_0} f^-(s)\dd s + \int_{\widecheck X_0^c}^{x} f^+(s)\dd s -cx& \mathrm{if} \ \ \widecheck X^c_0\leqslant x\leqslant X; \\
\int_X^x f^-(s)\dd s -cx & \mathrm{if} \ \ X\leqslant x\leqslant \widecheck X^c_1 ;\\
\int_X^{\widecheck X^c_1} f^-(s)\dd s + \int_{\widecheck X^c_1}^{x} f^+(s)\dd s -cx& \mathrm{if} \ \ \widecheck X^c_1\leqslant x\leqslant 1 ,
\end{cases}
$$
where $\widecheck X_0^c\in [0,X]$ and $\widecheck X_1^c\in [X,1]$ verify

\begin{multline*}
\int_0^{\widecheck X^c_0}f^-(x)\dd x-c\widecheck X_0^c = \int_{\widecheck X^c_0}^{X}f^-(x)\dd x + c(X-\widecheck X_0^c) \ ;  \\
 \int_X^{\widecheck X^c_1}f^-(x)\dd x - c (\widecheck X^c_1 -X) = \int_{\widecheck X^c_1}^{1}f^-(x)\dd x + c(1-\widecheck X^c_1) .
\end{multline*}

\section{\texorpdfstring{A change of regime: $\boxed{c\in \big( \alpha,  \int_0^{1}f^+(x)\dd x\big)}$}{A change of regime}}

 If $c\in \big(\alpha,  \int_0^{1}f^+(x)\dd x\big)$. It happens that the critical constant is again $0$ but it is not anymore possible to construct a critical subsolution that vanishes both at $0$ and at $X$.
 Let  $X^c\in [X,1]$ verify that 
$$
\int_0^{X^c}f^+(x)\dd x-cX^c = \int_{X^c}^{1}f^+(x)\dd x + c(1-X^c)  
 .
$$
The function $u_1^c$ defined by 
$$u_1^c(x) = \begin{cases}
\int_0^x f^+(s)\dd s -cx & \mathrm{if} \ \ 0\leqslant x\leqslant X^c; \\
\int_0^{X^c} f^+(s)\dd s + \int_{X^c}^{x} f^-(s)\dd s -cx& \mathrm{if} \ \ X^c\leqslant x\leqslant 1.
\end{cases}
$$
verifies $(u_1^{c})'(s)\in H_c^{-1}(\{0\})$  for $s\in \T^1\setminus\{X^c\}$ and  is  semiconcave. As previously, this yields that $u_1^c$ is a viscosity solution of the stationary Hamilton--Jacobi equation $H_c\big(x,(u_1^{c})'(x)\big)=0$. 

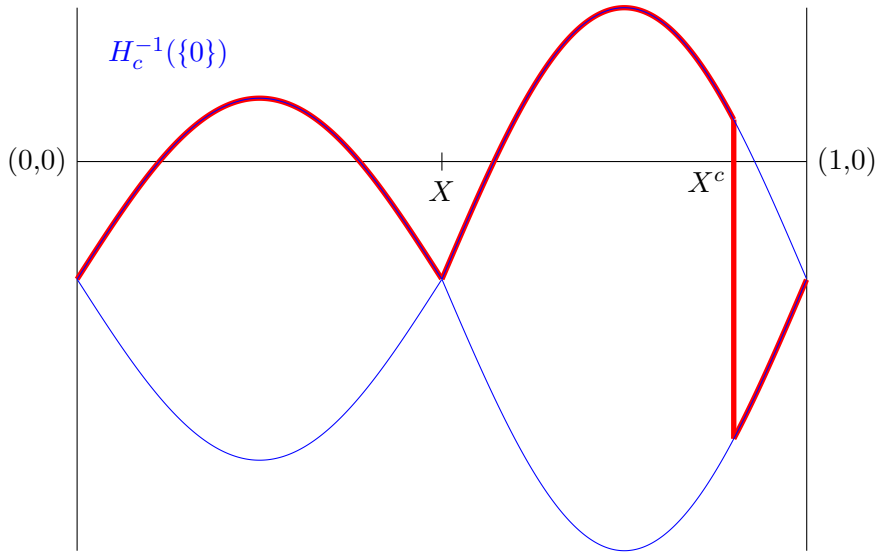
\begin{figure}[h!]
\begin{center}

  \begin{tikzpicture}[scale = 1.2]

   \draw  (0,-3)--(0,3) ;
     \draw  (8,-3)--(8,3) ;
     \draw (0,1.3) -- (8,1.3);
     \draw (4,1.2) -- (4,1.4);

\draw (0,1.3) node[left]{(0,0)} ;
\draw (8,1.3) node[right]{(1,0)} ;
\draw (6.9,1.3) node[below]{$X^c$} ;
\draw (4,1.2) node[below]{$X$} ;
\draw [color=blue](1,2.5) node {$H_c^{-1}(\{0\})$} ;

  \draw[color=red][line width=2pt][domain=0:2][samples=250] plot (2
   *\x,2* abs{sin(pi*\x/2 r)} );

  \draw[color=red][line width=2pt][domain=2:3.6][samples=250] plot (2
   *\x,3* abs{sin(pi*\x/2 r)} );
    \draw[color=red][line width=2pt][domain=3.6:4][samples=250] plot (2
   *\x,-3* abs{sin(pi*\x/2 r)} );

\draw [color=red] [line width=2pt](7.2,3* abs{sin(pi*.8 r)})--(7.2,-3* abs{sin(pi*.8 r)}) ;
  \draw [color=blue][domain=0:2][samples=250] plot (2
   *\x,2* abs{sin(pi*\x/2 r)} );
    \draw[color=blue][domain=0:2][samples=250] plot (2
   *\x,-2* abs{sin(pi*\x/2 r)} );
   
     \draw[color=blue][domain=2:4][samples=250] plot (2
   *\x,3* abs{sin(pi*\x/2 r)} );
    \draw[color=blue][domain=2:4][samples=250] plot (2
   *\x,-3* abs{sin(pi*\x/2 r)} );

\end{tikzpicture}

\caption{The graph of the superdifferential $\partial^+ u_1^c$ is drawn in red.}
\end{center}
\end{figure}

%
\begin{itemize}
\item It means that the critical constant $\alpha(c) = 0$.
\item Here again, it can be established that $\AA = \MM = \{0,X\}$ and  that classical Mather measures are convex combinations of the Dirac measures (on $T\T^1$) that are $\delta_{(0,0)}$ and $\delta_{(X,0)}$, and that discrete Mather measures  are convex combinations of the Dirac measures (on $\T^1\times \T^1$) that are $\delta_{(0,0)}$ and $\delta_{(X,X)}$. Finally,  $u_1^c(0) =0$ and $ u_1^c(X) =\int_0^X f^+(s)\dd s -cX< 0$. 
\item This function $u_1^c$ is indeed the weak KAM solution selected by the discounted approximation (Theorem \ref{discounted}). Roughly speaking, as between $0$ and $X$, $(u_1^c)' = f^+-c$, it is the fastest growing weak KAM solution. 
\item By Proposition \ref{conjdiscount}, in this case, by setting $v_1^c$ the positive weak KAM solution selected by the discounted procedure, the pair $(u_1^c,v_1^c)$ is  NOT a conjugate pair. 
\end{itemize}
Here, one  computes by similar means that 
$$v_1^c(x) = \begin{cases}
\int_X^x f^+(s)\dd s -c(x-X) & \mathrm{if} \ \ 0\leqslant x\leqslant X; \\
\int_X^x f^-(s)\dd s -c(x-X) & \mathrm{if} \ \ X\leqslant x\leqslant \widecheck X^c ;\\
\int_X^{\widecheck X^c} f^-(s)\dd s + \int_{\widecheck X^c}^{x} f^+(s)\dd s -c(x-X)& \mathrm{if} \ \ \widecheck X^c\leqslant x\leqslant 1; 
\end{cases}
$$
where  $\widecheck X^c\in [X,1]$ verifies

$$\int_X^0 f^+(s)\dd s+cX= \int_X^{\widecheck X^c} f^-(s)\dd s + \int_{\widecheck X^c}^{1} f^+(s)\dd s -c(1-X).
$$

Note that in this regime, for $c$ close to  $\int_0^{X}f^+(x)\dd x$ the functions $u_1^c$ and $v_1^c$ are not ordered while, as will become clear next, for $c$ close to $  \int_0^{1}f^+(x)\dd x$ then $u_1^c< v_1^c$.

\section{\texorpdfstring{The limiting case: $\boxed{c_0=  \int_0^{1}f^+(x)\dd x}$}{The limiting case: }}
 In this limit case, $c_0=  \int_0^{1}f^+(x)\dd x$, again the critical constant is  $0$ for example by invoking the continuity of Mather's $\alpha$--function\footnote{This result is a straightforward consequence of Proposition \ref{alphaconv} in next Chapter.} or because we exhibit a weak KAM solution below. 

Indeed the function $u_1^{c_0}$ defined by 
\begin{equation}\label{u_1crit}
\forall x\in [0,1],\quad u_1^{c_0}(x) =
\int_0^x f^+(s)\dd s -{c_0}x,
\end{equation}
 verifies $u_1^{c_0}(0)=u_1^{c_0}(1)=0$ whence to be identified with a function on $\T^1$. It is $C^1$ (even $C^{1,1}$ in agreement with Fathi's result \cite{Fa2}) and a classical solution of  $H_{c_0}\big(x,(u_1^{c_0})'(x)\big)=0$, hence a weak KAM solution. 
\begin{itemize}
\item It means that the critical constant $\alpha({c_0}) = 0$.
In this particular case, all viscosity subsolutions are of the form $u_1^{c_0} + K$, where $K\in \R$ and it can even be proved that all discrete subsolutions are of the same form. 
\item The situation is then different from the previous cases as the projected Aubry set is the whole torus $\AA_{c_0}=\T^1$ and the classical Aubry set $\AA^*_{c_0}$ is the whole graph of $f^+-{c_0}$. 
\item On the contrary, as the Hamiltonian dynamics on the critical level set remains the same as in the previous examples, the invariant measures remain the same and $\MM = \{0,X\}$.

\item
At last,  $u_1^{c_0}(0) =0$ and $ u_1^{c_0}(X) =\int_0^X f^+(s)\dd s -{c_0}X< 0$. The function $u_1^{c_0}$ is indeed the weak KAM solution selected by the discounted approximation (Theorem \ref{discounted}). 
 
\item Setting $v_1^{c_0} = u_1^{c_0} - u_1^{c_0}(X)$ one obtains the weak KAM solution selected by the positive discounted approximation. Here, $u_1^{c_0} < v_1^{c_0}$ and the pair is not conjugated (in this case all negative weak KAM solutions are positive weak KAM solutions hence conjugate pairs are trivial).
 \end{itemize}

\section{\texorpdfstring{Positive rotation numbers: $\boxed{c>  \int_0^{1}f^+(x)\dd x}$}{Positive rotation numbers:}}

Let us now discuss what happens for $c>  \int_0^{1}f^+(x)\dd x$. Again the cost is $h_1^c$ associated to $H_c$. The behavior of weak KAM solutions and minimal trajectories are those of an area preserving twist diffeomorphism. It will be treated more thoroughly in the next Chapter but  some results are briefly used here. 

For $c>  \int_0^{1}f^+(x)\dd x$ the situation is quite similar to the previous one. There exists a unique subsolution up to constants and therefore, up to constants, there exists  a unique weak KAM solution,  be it negative or positive. One such  subsolution is the following: recalling that $V : \T^1 \to \R$ is the potential used in the definition of $H_0$,  if $a>0$, let us denote by $f^+_a : x\mapsto \sqrt{2\big(a-V(x)\big)}$ the function whose graph is the upper part of the level set $H_0^{-1}\{a\}$. There exists a unique $a_c$ such that $\int_0^1  f^+_{a_c}(x)\dd x = c$. A subsolution for $H_c$ (that is also a positive and negative weak KAM solution) is then $u : x\mapsto \int_0^x  f^+_{a_c}(t) \dd t -cx$. Therefore, the critical constant is $a_c = \alpha(c)$.

To each real number $c\in \R$, we associate a rotation number $\rho(c)\in \R$. Its projection to $\T^1$, written $\varrho(c)$, has the property that $\int_{\T^1\times \T^1}(y-x)\dd \mu(x,y) =\varrho(c)$ for all Mather measure $\mu \in \widehat\PP_0$.  The function $\rho : \R\to \R$ is continuous and non--decreasing. It is not uniquely defined, but determined up to an integer. Here we make the choice of setting $\rho(0)=0$.   All the previous cases treated correspond to a vanishing rotation number. In our present case, the Aubry set in $T ^*\T^1 $ is the graph of the function $f_{\alpha(c)}^+$: $\AA^*_c = \big\{ \big(x, f_{\alpha(c)}^+(x)-c\big), \ \ x\in \T^1 \big\}$ and the $2$--Aubry set is (thanks to Proposition \ref{eqAubry})
$$\widehat \AA_c = \left \{ \left(x ,  \pi \circ \varphi_{H_c}^1\big(x,f_{\alpha(c)}^+(x)-c\big) \right) , \quad x \in \T^1 \right\}.$$
Hence, the rotation number property can be rewritten
\begin{multline*}
\int_{\T^1\times \T^1}(y-x)\dd \mu(x,y)=\int_{\T^1\times \T^1}\big( \pi \circ \varphi_{H_c}^1\big(x,f_{\alpha(c)}^+(x)-c\big)-x\big)\dd \mu(x,y)\\
=\int_{\T^1}\big( \pi \circ \varphi_{H_c}^1\big(x,f_{\alpha(c)}^+(x)-c\big)-x\big)\dd\pi_{1*} \mu(x,y)  = \varrho(c),
\end{multline*}

and  $\mu$ being closed yields that for all continuous function $f : \T^1 \to \R$:
\begin{multline*}
\int_{\T^1\times \T^1}\big(f(y)-f(x)\big)\dd \mu(x,y)=\int_{\T^1\times \T^1}\Big(f\big( \pi \circ \varphi_{H_c}^1(x,f_{\alpha(c)}^+(x)-c)\big)-f(x)\Big)\dd \mu(x,y)\\
=\int_{\T^1}\Big( f\big(\pi \circ \varphi_{H_c}^1(x,f_{\alpha(c)}^+(x)-c)\big)-f(x)\Big)\dd\pi_{1*} \mu(x,y)  =0.
\end{multline*}
This means that if we set $\psi_c : x\mapsto \pi \circ \varphi_{H_c}^1\big(x,f_{\alpha(c)}^+(x)-c\big)$ (also called projected dynamics) then $\psi_c$ is a circle diffeomorphism of rotation number $\varrho(c)$ and the measure $\pi_{1*}\mu$ is $\psi_c$--invariant.
 
 We now focus on the case: \textbf{$\rho(c)$ is irrational}. In this case it is known from Poincaré-Denjoy theory that $\psi_c$ is conjugated to a rotation of angle $\rho(c)$ (because $\psi_c$ is $C^2$) and that there exists a unique $\psi_c$--invariant measure. Hence necessarily,
 $$\pi_{1*}\mu =\frac{1}{T_c} \int_0^{T_c} \delta_{ \pi \circ \varphi_{H_c}^s(0,f_{\alpha(c)}^+(0)-c)} \dd s,$$
 where $T_c$ is the smallest positive constant such that $ \varphi_{H_c}^{T_c}(0,f_{\alpha(c)}^+(0)-c)=(0,f_{\alpha(c)}^+(0)-c)$, as the latter is $\psi_c$--invariant\footnote{It can be proven that $\rho(c) = T_c^{-1}$.}. Lifting things up to $\T^1\times \T^1$, we find that
 $$\mu =\frac{1}{T_c} \int_0^{T_c} \delta_{( \pi \circ \varphi_{H_c}^s(0,f_{\alpha(c)}^+(0)-c), \pi \circ \varphi_{H_c}^{s+1}\textrm{$\big($}0,f_{\alpha(c)}^+(0)-c)\textrm{$\big)$}} \dd s.$$
As a conclusion  $\widehat\PP_0 = \{\mu_c\}$ where $\mu_c$ is the previous measure and similarly, $ \PP_0^* = \{\mu_c^*\}$ where 
$$\mu_c^* = \frac{1}{T_c} \int_0^{T_c} \delta_{ \varphi_{H_c}^s(0,f_{\alpha(c)}^+(0)-c)} \dd s.$$
Note that $ \PP_0^*$ is always a singleton for non $0$ rotation numbers in this $1$ dimensional autonomous setting. However, for rational rotation numbers, in the discrete setting, there are many Mather measures supported on the various periodic orbits.

\subsection{\texorpdfstring{Non--continuity of $u_1^c$ with respect to $c$}{Non--continuity of u1c with respect to c}}

We now aim at studying the behavior of $\mu_c$ as $c\to c_0$:

\begin{pr}\label{mesureCV}
Assume that $V''(0)V''(X)\neq 0$. Let $(c_n)_{n> 0}$ be a decreasing sequence converging to $c_0$ such that $\rho(c_n)$ is irrational for all $n>0$, then 
\begin{multline*}
\mu_{c_n} \underset{n\to +\infty}{\longrightarrow} \\
 \frac{\big(\sqrt{-V''(0)}\big)^{-1}}{\big(\sqrt{-V''(0)}\big)^{-1}+\big(\sqrt{-V''(X)}\big)^{-1}}\delta_{(0,0)} 
 +  \frac{\big(\sqrt{-V''(X)}\big)^{-1}}{\big(\sqrt{-V''(0)}\big)^{-1}+\big(\sqrt{-V''(X)}\big)^{-1}}\delta_{(X,X)}.\end{multline*}
\end{pr}

\begin{proof}
As the set of probability measures on $X\times X$ is compact, to prove the result one just needs to prove that any converging subsequence of $(\mu_{c_n})_{n>0}$ has the announced limit. Hence without loss of generality, let us assume that  $(\mu_{c_n})_{n>0}$ is a converging sequence. We now prove it converges to $\frac{\big(\sqrt{-V''(0)}\big)^{-1}}{\big(\sqrt{-V''(0)}\big)^{-1}+\big(\sqrt{-V''(X)}\big)^{-1}}\delta_{(0,0)} +  \frac{\big(\sqrt{-V''(X)}\big)^{-1}}{\big(\sqrt{-V''(0)}\big)^{-1}+\big(\sqrt{-V''(X)}\big)^{-1}}\delta_{(X,X)}$.

  We know from \cite{Mather1} for example, or Proposition \ref{alphaconv} below, that the function $c\mapsto \alpha(c)$ is convex and continuous. Hence the limit of the sequence $(\mu_{c_n})_{n>0}$ is a Mather measure, therefore of the form $\beta_0 \delta_{(0,0)} + \beta_X \delta_{(X,X)}$ with $0\leqslant \beta_0,\beta_X\leqslant 1$ and $\beta_0+\beta_X = 1$.
  
With this information at hand, we will in fact only study the measures $\pi_{1*} \mu_{c_n}$ and prove they converge to $\frac{\big(\sqrt{-V''(0)}\big)^{-1}}{\big(\sqrt{-V''(0)}\big)^{-1}+\big(\sqrt{-V''(X)}\big)^{-1}}\delta_{0} +  \frac{\big(\sqrt{-V''(X)}\big)^{-1}}{\big(\sqrt{-V''(0)}\big)^{-1}+\big(\sqrt{-V''(X)}\big)^{-1}}\delta_{X}$.  

Let $\gamma_n : t\mapsto  \pi \circ \varphi_{H_{c_n}}^t(0,f_{\alpha(c_n)}^+(0)-c_n)$. By looking at the Hamiltonian equations \eqref{Hamilton} and recalling that $H_{c_n}$ is constant on a Hamiltonian trajectory, one finds that
\begin{equation}\label{projdyn}
\forall t\in \R,\quad \dot\gamma_n(t) = \sqrt{2\big(\alpha(c_n) - V\big( \gamma_n(t)\big)\big)}.
\end{equation}
The coefficients $\beta_0$ and $\beta_X$ are proportional to the relative amount of time that the trajectory $\gamma_n $ stays respectively in a neighborhood of $0$ and $X$, as $n\to +\infty$.

Until the end of this proof,  let us  no longer think of points on the circle $\T^1$ but by lifting to $\R$, but keeping the same notations. Hence the function $V$ is now a $1$-periodic function on $\R$. Integrating \eqref{projdyn}, one computes that if $x<y$, the time it takes  $\gamma_n$ to go from $x$ to $y$ is 
\begin{equation}\label{time}
t_{x,y} = \int_x^y \frac{\dd s}{\sqrt{2\big(\alpha(c_n)-V(s)\big)}}.
\end{equation}
In particular, $T_{c_n} =  \int_0^1 \frac{\dd s}{\sqrt{2\textrm{$\big($}\alpha(c_n)-V(s)\textrm{$\big)$}}}$.

Let $0<\varepsilon <\max\{-V''(0),-V''(X)\}$. Let us consider $\eta>0$ such that,
$$\left[|x|<\eta \right]  \Longrightarrow   \left[ \left|V(x) -V''(0)\frac{x^2}{2}\right|<\varepsilon\frac{x^2}{2}\right] ,$$
 and
  $$\left[|x-X|<\eta\right] \Longrightarrow \left[\left |V(x)-V''(X) \frac{(x-X)^2}{2} \right |<\varepsilon\frac{(x-X)^2}{2} \right].$$ 
  We now split the integral defining $T_{c_n}$ into the $4$ following pieces:
\begin{multline}
T_{c_n}=  \int_0^1 \frac{\dd s}{\sqrt{2\left(\alpha(c_n)-V(s)\right)}} \\
=\underbrace{ \int_{-\eta}^{\eta} \frac{\dd s}{\sqrt{2\left(\alpha(c_n)-V(s)\right)}}  }_{\tcircle{1}}  + \underbrace{\int_{X-\eta}^{X+\eta} \frac{\dd s}{\sqrt{2\left(\alpha(c_n)-V(s)\right)}}  }_{\tcircle{2}} \\
+\underbrace{\int_\eta^{X-\eta} \frac{\dd s}{\sqrt{2\left(\alpha(c_n)-V(s)\right)}}   }_{\tcircle{3}}+ \underbrace{  \int_{X+\eta}^{1-\eta} \frac{\dd s}{\sqrt{2\left(\alpha(c_n)-V(s)\right)}}  }_{\tcircle{4}}.
\end{multline}
Let $M>0$ be a constant independent of $n$ such that $|{\tcircle{3}}|+|{\tcircle{4}}|<M$ for all $n>0$. Such an $M$ exists as the denominators appearing in the integrals are uniformly positive. 

Let us now study and estimate ${\tcircle{1}}$. From the definition of $\eta$ the following inequalities are infered:
\begin{multline} 
\underbrace{ \int_{-\eta}^{\eta} \dfrac{\dd s}{\sqrt{2\left(\alpha(c_n)+\big(\varepsilon-V''(0)\big)\dfrac{s^2}{2}\right)}} }_{\tcircle{5}} 
\\
\leqslant  \int_{-\eta}^{\eta} \dfrac{\dd s}{\sqrt{2\left(\alpha(c_n)-V(s)\right)}}
\\
\leqslant \underbrace{  \int_{-\eta}^{\eta} \dfrac{\dd s}{\sqrt{2\left(\alpha(c_n)+\big(-V''(0)-\varepsilon\big)\dfrac{s^2}{2}\right)}} }_{\tcircle{6}} .
\end{multline}
Terms  {\tcircle{5}} and  \tcircle{6} can be integrated explicitly, let us deal with  \tcircle{5}.
\begin{align*}
\int_{-\eta}^{\eta} & \frac{\dd s}{\sqrt{2\left(\alpha(c_n)+\big(\varepsilon-V''(0)\big)\dfrac{s^2}{2}\right)}} 
=
\frac{ \sqrt{\frac{2\alpha(c_n) }{\varepsilon-V''(0)}      }}{ \sqrt{2\alpha(c_n)  }}     \int_{-\eta\sqrt{\frac{\varepsilon-V''(0)}{2\alpha(c_n)}     }}^{\eta\sqrt{\frac{\varepsilon-V''(0)}{2\alpha(c_n)}}}
 \dfrac{    \dd t}{\sqrt{1+t^2}}     \\    
 &= \frac{1}{\sqrt{\varepsilon - V''(0)}}\left[             \ln\left(  \sqrt{1+t^2}+t                   \right)           \right]_{-\eta\sqrt{\frac{\varepsilon-V''(0)}{2\alpha(c_n)}     }}^{\eta\sqrt{\frac{\varepsilon-V''(0)}{2\alpha(c_n)}}}        \\
 &=  \frac{1}{\sqrt{\varepsilon - V''(0)}}    \ln\left[  \dfrac{  \sqrt{  \dfrac{ \big(\varepsilon -V''(0)\big)\eta^2    }{        2\alpha(c_n)    }+1           }              +\eta\sqrt{   \dfrac{  \varepsilon-V''(0)      }{  2\alpha(c_n)     }        }                      }{                \sqrt{  \dfrac{ \big(\varepsilon -V''(0)\big)\eta^2    }{        2\alpha(c_n)    }+1           }              -\eta\sqrt{   \dfrac{  \varepsilon-V''(0)      }{  2\alpha(c_n)     }        }                           }                                                   \right] \\
 &=  \frac{1}{\sqrt{\varepsilon - V''(0)}}    \ln\left[    \left(            \sqrt{  \dfrac{ \big(\varepsilon -V''(0)\big)\eta^2    }{        2\alpha(c_n)    }+1           }              +\eta\sqrt{   \dfrac{  \varepsilon-V''(0)      }{  2\alpha(c_n)     }        }             \right)  ^2    \right]       \\
 &=  \frac{1}{\sqrt{\varepsilon - V''(0)}} \left\{  -\ln(2)-\ln\big(\alpha(c_n)\big) +\ln\left[    \sqrt{    \big(\varepsilon -V''(0)\big)\eta^2 +2\alpha(c_n)    }          +\eta \sqrt{\varepsilon-V''(0)  }               \right]           \right\}      \\
 &\underset{n\to +\infty}{\sim}             \dfrac{-\ln\big(\alpha(c_n)\big)}{\sqrt{\varepsilon-V''(0)}}.
\end{align*}
The last relation uses the continuity of $\alpha$ and the consecutive  limit: $\lim\limits_{n\to +\infty} \alpha(c_n) = 0$.

The same computation for \tcircle{6} yields
$$ \int_{-\eta}^{\eta} \dfrac{\dd s}{\sqrt{2\left(\alpha(c_n)+\big(-V''(0)-\varepsilon\big)\dfrac{s^2}{2}\right)}} \underset{n\to +\infty}{\sim}             \dfrac{-\ln\big(\alpha(c_n)\big)}{\sqrt{-\varepsilon-V''(0)}}.$$
As for  \tcircle{2} the same strategy is adopted:

\begin{multline} 
\underbrace{ \int_{X-\eta}^{X+\eta} \dfrac{\dd s}{\sqrt{2\left(\alpha(c_n)+\big(\varepsilon-V''(X)\big)\dfrac{s^2}{2}\right)}} }_{\tcircle{7}} 
\\
\leqslant  \int_{X-\eta}^{X+\eta} \dfrac{\dd s}{\sqrt{2\left(\alpha(c_n)-V(s)\right)}}
\\
\leqslant \underbrace{  \int_{X-\eta}^{X+\eta} \dfrac{\dd s}{\sqrt{2\left(\alpha(c_n)+\big(-V''(X)-\varepsilon\big)\dfrac{s^2}{2}\right)}} }_{\tcircle{8}} .
\end{multline}

Similar computations yield
$$ \int_{X-\eta}^{X+\eta} \dfrac{\dd s}{\sqrt{2\left(\alpha(c_n)+\big(\varepsilon-V''(X)\big)\dfrac{s^2}{2}\right)}} \underset{n\to +\infty}{\sim}             \dfrac{-\ln\big(\alpha(c_n)\big)}{\sqrt{\varepsilon-V''(X)}},$$
and
$$ \int_{X-\eta}^{X+\eta} \dfrac{\dd s}{\sqrt{2\left(\alpha(c_n)+\big(-V''(X)-\varepsilon\big)\dfrac{s^2}{2}\right)}} \underset{n\to +\infty}{\sim}             \dfrac{-\ln\big(\alpha(c_n)\big)}{\sqrt{-\varepsilon-V''(X)}}.$$

Let now $f: \R \to \R$ be a continuous $1$--periodic function that is constant on $[-\eta,\eta]$ and on $[X-\eta, X+\eta]$, with $0\leqslant \min\big( f(0),f(X)\big)$. We know that
$$\int_{[0,1]} f(s) \dd \pi_{1*} \mu_{c_n}(s) \underset{n\to +\infty}{\longrightarrow}   \beta_0f(0)+\beta_Xf(X). $$
Gathering the previous computations we infer that
$$\frac{f(0)      \tcircle{1} + f(X) \tcircle{2} -\|f\|_\infty M    }{T_{c_n}} \leqslant  \int_{[0,1]} f(s) \dd \pi_{1*} \mu_{c_n}(s)    \leqslant \dfrac{ f(0)      \tcircle{1} + f(X)\tcircle{2} +\|f\|_\infty M}{T_{c_n}}.$$ 
And letting $n\to +\infty$ one discovers that
\begin{multline*}\dfrac{     \dfrac{f(0)}{\sqrt{ \varepsilon-V''(0)    } }    + \dfrac{f(X)}{\sqrt{ \varepsilon-V''(X)    } }       }{   \dfrac{1}{  \sqrt{- \varepsilon-V''(0)} } + \dfrac{1}{  \sqrt{- \varepsilon-V''(X)} }        }   
\\
\leqslant  \beta_0f(0)+\beta_Xf(X)     \leqslant \\       
   \dfrac{     \dfrac{f(0)}{\sqrt{- \varepsilon-V''(0)    } }    + \dfrac{f(X)}{\sqrt{- \varepsilon-V''(X)    } }       }{   \dfrac{1}{  \sqrt{ \varepsilon-V''(0)} } + \dfrac{1}{  \sqrt{-\varepsilon-V''(X)} }        } 
.
\end{multline*}
This being true for all $\varepsilon>0$ and all non--negative $f(0)$ and $f(X)$,  the lemma is proved.
\end{proof}

At last, we can deduce the following:
\begin{pr}
Assume that $H_0 : (x,p)\mapsto \frac 12 p^2 +V(x)$ is a Hamiltonian such that $V : \T^1\to \R$ is smooth, non--positive, and verifies $V^{-1}\{0\} = \{0,X\}$ for some $X\in \T^1\setminus \{0\}$. Assume moreover that $V''(0)\neq V''(X)$ are both negative and again
$$\alpha:= \frac{ \int_0^{X}f^+(x)\dd x}{X} < \frac{ \int_X^{1}f^+(x)\dd x}{1-X} := \beta.$$
 Then, using the previous notations, if   $(c_n)_{n> 0}$ is a decreasing sequence converging to $c_0$ such that $\rho(c_n)$ is irrational for all $n>0$, the family of functions $(u_1^{c_n})_{n>0}$ does not converge to $u_1^{c_0}$.
\end{pr}

\begin{proof}
Using that a limit of weak KAM solutions is a weak KAM solution and the previous Proposition \ref{mesureCV}, one proves that  $(u_1^{c_n})_{n>0}$ converges to the unique weak KAM solution $u$, for  $H_{c_0}$ such that  
$$\frac{\big(\sqrt{-V''(0)}\big)^{-1}}{\big(\sqrt{-V''(0)}\big)^{-1}+\big(\sqrt{-V''(X)}\big)^{-1}}u(0) +  \frac{\big(\sqrt{-V''(X)}\big)^{-1}}{\big(\sqrt{-V''(0)}\big)^{-1}+\big(\sqrt{-V''(X)}\big)^{-1}}u(X)=0.$$
 Recall that at cohomology $c_0$, there is a unique weak KAM solution, up to constants. Moreover, the condition $\alpha<\beta$ ensures that it is not possible to have $u(0)=u(X)=0$. Hence $u(0)u(X)<0$. On the contrary, as $u_1^{c_0}$  is given by the formula \eqref{u_1crit}:
$$\forall x\in [0,1],\quad u_1^{c_0}(x) =
\int_0^x f^+(s)\dd s -{c_0}x,$$
we have $u_1^{c_0}(0)u_1^{c_0}(X)=0$.

\end{proof}

\subsection{\texorpdfstring{A situation where $u_1^c \neq U_0^c$}{A situation where}}

We come back to Proposition \ref{inegdiscounted}. More precisely, we answer by the negative the natural question: does the discounted procedure select the same weak KAM solution in the discrete setting and in the continuous setting?

We now focus our attention on the unique real number $c_{\frac12} \in \R$ such that $\rho(c_{\frac12})=\frac12$. Note that $c_{\frac12}>0$ and that $\AA^*_{c_{\frac12}} \subset H^{-1}(\{\alpha(c_{\frac12})\}$. It is actually the upper connected component of this level set of $H$. Moreover, it can be characterized  as follows:
\begin{pr}\label{TFI}
Let $\Phi_{H_{c_{\frac12}}} : \R^2 \to \R^2$ denote the  lift of $\varphi_{H_{c_{\frac12}}}$ that fixes the point $(0,-{c_{\frac12}})$ and $P : \R^2 \to \T^1 \times \R$ the canonical projection. Then 
$$\AA^*_{c_{\frac12}} = P\Big( \big \{(x,p) \in \R^2 , \quad \Phi_{H_{c_{\frac12}}}^2(x,p) = (x+1,p) \big \} \Big).
$$
\end{pr}

\begin{proof}[Sketch of Proof]
There are many possible ways to tackle this Proposition according to the property of the system used. The proof is essentially given in \cite[Proposition 15]{Arnaud3} and very much related to \cite[Proposition 2]{AABZ} that proves a  version of the result, in arbitrary dimension, by using $C^0$--integrability in a neighborhood of $\AA^*_{c_{\textrm{$\frac12$}}}$. It can also be deduced from the Implicit function Theorem, proving that the right hand side of the above equality is a manifold.

Let us sketch a proof using a stronger integrability, reminiscent of the Arnol'd--Liouville Theorem \cite{Duist}. Let us call $B$ the right hand side of the equality to prove. A first step is  that $B$ is a (potentially partial) graph over $\T^1$. We distinguish three cases remembering that Hamiltonian orbits are included in level sets of $H$. 
\begin{itemize}
\item if $H(x,p)< \max V=0$ then the orbit $\{\varphi_{H_{c_{\frac12}}}^s(x,p), \quad s\in \R\}$ projects on an interval strictly included in $\T^1$ hence $(x,p)\notin B$.
\item  if $H(x,p)\geqslant  \max V$ with $p\leqslant -c_{\frac 12}$. Let $(\tilde x ,p)\in \R^2$ such that $P(\tilde x ,p)=(x,p)$. Then if we define for $s\in \R$, $\Phi_{H_{c_{\frac12}}}^s(\tilde x,p)=(\tilde x_s,p_s)$, one computes that $s\mapsto \tilde x_s$ is non--increasing, hence $(x,p)\notin B$.
\item  if $H(x,p)\geqslant  \max V$ with $p\geqslant -c_{\frac 12}$, meaning that $p\geqslant \sqrt{ -2 V( x )}-c_{\frac 12}$. We fix here $\tilde x\in \R$ and let $p$ vary. Looking at the Hamiltonian equations and more precisely computing the time $t_p$ such that $\Phi_{H_{c_{\frac12}}}^{t_p}(\tilde x,p)=(\tilde x+1,p+1)$ with equation \eqref{time}, it can be seen that $p\mapsto t_p$ is decreasing. Hence there is at most one $p$ such that $t_p = 2$.
\end{itemize}

Now recall that $\AA^*_{c_\frac 12} = \big\{ \big(x, f_{\alpha(c_{\frac 12})}^+(x)-c_{\frac 12}\big), \ \ x\in \T^1 \big\}$. We denote by $\mathfrak A$ the lift of $\AA^*_{c_\frac 12}$ to $\R^2$, that is invariant under $\Phi_{H_{c_{\frac12}}}$. We define the map $g : \R \to \R$ by 
$$\forall x\in \R, \quad \Phi_{H_{c_{\frac12}}}^1 \big(x, f_{\alpha(c_{\frac 12})}^+(x)-c_{\frac 12}\big) = \big(g(x), f_{\alpha(c_{\frac 12})}^+\circ g (x)-c_{\frac 12}\big).$$
The function $g$ is the lift of a circle diffeomorphism and its rotation number here is $\rho(c_{\frac12}) = \frac 12$. It follows from Poincaré's theory of rotation numbers that there exists a real number $x_0\in \R$ such that $g^2(x_0) = x_0+1$. Let now $x\in \R$ be any real number. There is a time $t\in \R$ such that  $ \big(x, f_{\alpha(c_{\frac 12})}^+(x)-c_{\frac 12}\big)=\Phi_{H_{c_{\frac12}}}^t\big(x_0, f_{\alpha(c_{\frac 12})}^+(x_0)-c_{\frac 12}\big)$.
It follows that 
\begin{multline*}
\Phi_{H_{c_{\frac12}}}^2\big(x, f_{\alpha(c_{\frac 12})}^+(x)-c_{\frac 12}\big) = \Phi^2\circ\Phi_{H_{c_{\frac12}}}^t\big(x_0, f_{\alpha(c_{\frac 12})}^+(x_0)-c_{\frac 12}\big) \\
=\Phi^t\circ\Phi_{H_{c_{\frac12}}}^2\big(x_0, f_{\alpha(c_{\frac 12})}^+(x_0)-c_{\frac 12}\big)  \\
=\Phi^t\circ\Phi_{H_{c_{\frac12}}}^2\big(x_0+1, f_{\alpha(c_{\frac 12})}^+(x_0)-c_{\frac 12}\big) =\big(x+1, f_{\alpha(c_{\frac 12})}^+(x)-c_{\frac 12}\big).
\end{multline*}

Thus it has been proven that $\mathfrak A\subset B$ and that the right hand side is a partial graph while the left hand side is a full graph. Hence both terms are equal.
\end{proof}
Let us define the map $h : \T^1 \to \R$ by 
$$\forall x\in \T^1, \quad \varphi_{H_{c_{\frac12}}}^1 \big(x, f_{\alpha(c_{\frac 12})}^+(x)-c_{\frac 12}\big) = \big(h(x), f_{\alpha(c_{\frac 12})}^+\circ h (x)-c_{\frac 12}\big).$$
The function $h$ has $g:\R \to \R$ as a lift.
 It follows from the previous result and Theorem \ref{minimizing} that for all $x\in \T^1$, the measure $\mu_x = \frac12\Big(\delta_{\textrm{$\big(x,h(x)\big)$}} + \delta_{\textrm{$\big(h(x),x\big)$}}\Big)$ is a discrete Mather measure.
 
 \begin{lm}
 Assume that $u_1^{c_{\frac12}} = U_0^{c_\frac12}$, then for all $x\in \T^1$, $u_1^{c_{\frac12}} (x) = -u_1^{c_{\frac12}}\circ g(x)$. 
 \end{lm}

\begin{proof}
From the definitions of  $u_1^{c_{\frac12}}$ and $U_0^{c_\frac12}$, and from the previous discussions result that
\begin{equation}\label{contdiscrete}
\forall x\in \T^1, \quad \int_{\T^1} u_1^{c_{\frac12}}(s) \dd \pi_{1*} \mu_x(s) = \frac 12\Big(u_1^{c_{\frac12}}(x) + u_1^{c_{\frac12}}\circ g(x) \Big) \leqslant 0.
\end{equation}
Moreover, using Proposition \ref{satureC} and the fact that there is a single classical Mather measure, we obtain that 
$$\int_{\T^1} u_1^{c_{\frac12}}(x) \dd\pi_* \mu^*_{c_{\frac12}}(x) = 0.
$$
Recalling the definition of $\mu^*_{c_{\frac12}}$ we find:
\begin{multline*}
0=\int_{\T^1} u_1^{c_{\frac12}}(x) \dd\pi_* \mu^*_{c_{\frac12}}(x) =\int_0^2  u_1^{c_{\frac12}}\Big( \pi \circ \varphi_{H_{c_{\frac12}}}^s \big(0, f_{\alpha(c_{\frac 12})}^+(0)-c_{\frac 12}\big)\Big) \dd s \\
=\int_0^1 \Big[  u_1^{c_{\frac12}}\Big( \pi \circ \varphi_{H_{c_{\frac12}}}^s \big(0, f_{\alpha(c_{\frac 12})}^+(0)-c_{\frac 12}\big)\Big) + u_1^{c_{\frac12}}\Big( \pi \circ \varphi_{H_{c_{\frac12}}}^{s+1} \big(0, f_{\alpha(c_{\frac 12})}^+(0)-c_{\frac 12}\big)\Big)     \Big] \dd s \\
=\int_0^1 2 \int_{\T^1}  u_1^{c_{\frac12}} (x)  \dd \pi_{1*} \mu_{    \pi \circ \varphi_{H_{c_{\frac12}}}^s \textrm{$\big(0, f_{\alpha(c_{\frac 12})}^+(0)-c_{\frac 12}\big)$}   }(x) \dd s\leqslant 0.
\end{multline*}
It follows that all inequalities in \eqref{contdiscrete} are equalities, hence the result.

\end{proof}

It is deduced that under the hypotheses of the previous Lemma, $u_1^{c_{\frac12}}(x)$ and $u_1^{c_{\frac12}}\circ g(x)$ must have opposite signs for all $x$. An example in which it is not the case is provided in \cite[Appendix A.2.]{AZ}. We give below a different simple situation where this cannot happen:

\begin{pr}
Let $V : \T^1 \to \R$ be a non--constant $\frac12$--periodic function. Then for the associated Hamiltonian $H_{c_{\frac12}}$, it holds $u_1^{c_{\frac12}} \neq U_0^{c_\frac12}$.
\end{pr}

\begin{proof}
Indeed, in this case, both functions $u_1^{c_{\frac12}}$  and $ U_0^{c_\frac12}$ are also $\frac 12$--periodic (this follows from the uniqueness of the solutions to the discounted equations that hence must be $\frac12$--periodic). Consequently if they coincide, the previous lemma tells us that they must be identically $0$. This is clearly not the case.
 \end{proof}

\section{Concluding example}

Due to the simple structure of the Aubry set and of the set of minimizing measures in all  previous examples, one can check that $u_1^c$ and $v_1^c$ form a conjugate pair ``up to a constant''. More precisely, the modified pair $\big(u_1^c - u_1^c(0) , v_1^c - v_1^c(0)\big)$ is a conjugate pair in all the previous examples. One may wonder if this is always the case.

As a matter of fact, the answer is again negative. We propose here a slightly (but not too much) sophisticated example shedding light on this fact. The computations are not carried on fully and left to the reader. We hope the previous examples give enough insight to make what follows quite straightforward.

Once again we consider  that $H_0 : (x,p)\mapsto \frac 12 p^2 +V(x)$ is a Hamiltonian such that $V : \T^1\to \R$ is smooth, non--positive, and verifies $V^{-1}\{0\} = \{0,X_1,X_2\}$ for some $0<X_1<X_2<1\in \T^1$. Assume again that 
$$\underbrace{ \frac{ \int_0^{X_1}f^+(x)\dd x}{X_1} }_{=\alpha}< \underbrace{\frac{ \int_{X_1}^{X_2}f^+(x)\dd x}{X_2-X_1} }_{= \beta}<\underbrace{\frac{ \int_{X_2}^{1}f^+(x)\dd x}{1-X_2}}_{= \gamma}.$$
If $c\in \R$ we again denote by $H_c : (x,p)\mapsto \frac12(p+c)^2 + V(x)$ and  respectively by $u_1^c$ and $v_1^c$ the negative and positive weak KAM solutions selected by the discounted procedure for the time--$1$ minimal action functional associated to $H_c$. Reasoning as in the previous sections, one checks that for $\alpha<c<\min(\beta,c_0)$, the function $u_1^c$ verifies 
$u_1^c(0)=u_1^c(X_2) = 0$ and $u_1^c(X_1)<0$ while  the function $v_1^c$ verifies 
$v_1^c(X_1)=v_1^c(X_2) = 0$ and $v_1^c(0)>0$. Hence $u_1^c$ and $v_1^c$ are not conjugated ``up to a constant''.

\newpage
\newpage
\mbox{}
\newpage

\chapter{Twist maps}
The results of discrete weak KAM theory will now be applied to the particular and founding case of Exact Conservative Twist Maps of the annulus. Excellent surveys on the subject are \cite{Atwist,Bang,MaFo} and we refer to those references for classical results that we leave without proof. The aim of this section is to present some of the results of \cite{AZ,AZ2,AZ3} obtained in collaboration with Marie--Claude Arnaud, shedding light on the structure of weak KAM solutions and minimizing orbits for exact conservative twist maps. Before giving the precise definition of an exact  conservative twist map, let us emphasize  that the examples of the previous section are closely related to such transformations. Indeed, for a Tonelli Hamiltonian on $T^*\T^1$, the Hamiltonian flow $\varphi^s_H$ is an exact conservative twist map for small times $s>0$. Hence the time $1$ map $\varphi^1_H$ is a composition of a finite number of conservative twist maps.
 
 \section{Definitions and variational structure}
 
 In the rest of this section $\T^1 = \R \slash \Z$ is the circle, the $2$--dimensional annulus is denoted by $\A = T^*\T^1 = \T^1 \times \R$.  The points of that annulus are denoted by   $(\theta,r)\in \A$. Throughout this section we will often consider objects coming from $\A$ lifted to $\R^2$, its universal cover, or from $\T^1$, lifted to $\R$. When done so, a $\sim$ will be added to the original name. For example if $g : \T^1 \to \R$ is any function then $\tilde g  :\R \to \R$ is the lift of $g$. 
 
 When dealing with products, $\T^1\times \T^1$, $\T^1 \times \R$, or $\R \times \R$, the notations $\pi_1$ and $\pi_2$ stand  for the projections on the first and second variable. 
 \subsection{Definition and Birkhoff's theorem}
 
 \begin{df}\label{twistmap}\rm
 An exact conservative twist map of the annulus (abbreviated ECTM) is a $C^1$--diffeomorphism $f : \A \to \A$ such that 
 \begin{enumerate}
 \item $f$ is isotopic to the identity map.
 \item $f$ is exact symplectic: by denoting $f(\theta,r) = \big(\Theta(\theta,r),R(\theta,r)\big)$ the $1$--form $R\dd \Theta - r\dd\theta$ is exact.
 \item $f$ twists verticals to the right:  if $\tilde f =(\widetilde \Theta , R): \R^2 \to \R^2 $ is a lift of $f$ to the universal cover of $\A$ then for all $\tilde \theta \in \R$, the map $r \mapsto  \widetilde \Theta(\tilde \theta ,r)$ is an increasing $C^1$--diffeomorphism of $\R$.
  \end{enumerate}

 \end{df}
\begin{rem}\rm\label{genfunrem}
\hspace{2em}
\begin{enumerate}
\item The first property means that there is a continuous path of diffeomorphisms of the annulus $(f_s)_{s\in [0,1]}$ such that $f_0$ is the identity map and $f_1=f$.  In other topological words, $f$ preserves both ends of the annulus in the sense that uniformly in $\theta$, $\lim\limits_{r\to +\infty} R(\theta,r) = +\infty $ and  $\lim\limits_{r\to -\infty} R(\theta,r) = -\infty $. 
\item The second point means that there is a function $S : \A \to \R$, called {\it generating function}, such that 
\begin{equation}\label{defS}
\dd S = R\dd \Theta -r\dd \theta=R\big( \frac{\partial \Theta }{\partial \theta }\dd \theta +\frac{\partial \Theta }{\partial r } \dd r \big)-r\dd \theta.
\end{equation}

The generating function is defined up to a constant.

Another way of formulating this is to say that if we denote by $\lambda = r\dd \theta$ (the $1$--form  called Liouville form) then $f^*\lambda - \lambda$ is exact:

This has two major implications.
\begin{enumerate}
\item[(i)]   (proved by using Stokes' furmula)  if $\mathcal C$ is a $C^1$ essential circle, meaning an injective $C^1$ closed curve going around the cylinder, then the algebraic area between $\mathcal C$ and $f(\mathcal C)$ is $0$. This means that points are not globally shifted up or down by $f$ in the annulus and that it is worth looking for invariant compact sets.
\item[(ii)] The second is that $f$ preserves the canonical symplectic $2$--form which is here the Lebesgue area form. Indeed, as this symplectic form is  $\dd \lambda = \dd r \wedge \dd \theta$, one finds that 
$$0 = \dd \dd S = \dd R \wedge \dd \Theta - \dd r \wedge \dd \theta = f^*( \dd r \wedge \dd \theta)- \dd r \wedge \dd \theta.$$
\end{enumerate}
\item 
The last point  implies that for $\tilde\theta_0\in \R$ and $\ww\Theta_0\in \R$ there is a unique $r\in \R $ such that $\ww\Theta(\tilde \theta_0,r) = \ww \Theta_0$. It means that $(\tilde \theta,r)\mapsto \big(\tilde \theta, \ww \Theta(\tilde\theta,r)\big)$ is a $C^1$--diffeomorphism that plays the role of Legendre transform and  $(\tilde \theta, \ww \Theta)$ will serve as coordinates of $\R^2$. The latter powerful idea is behind all the variational structure of twist maps.

The twist condition implies that 
 \begin{equation}\label{droite}
 \forall (\theta_0,r_0)\in \A, \quad \frac{\partial \Theta }{\partial r } (\theta_0,r_0)>0.
 \end{equation}
 In fact, it is sometimes replaced by the stronger uniform condition that there exists $\varepsilon>0$ such that 
$ \frac{\partial \Theta }{\partial r } (\theta_0,r_0)>\varepsilon$ for all $ (\theta_0,r_0)\in \A$.
\end{enumerate}
\end{rem}

One can at this stage state an important dynamical property of ECTM on invariant curves (see \cite{Birk}):

\begin{Th}[Birkhoff]\label{birk}
Let $\CC\subset \A$ be a continuous essential circle invariant by $f$, i.e. a continuous embedding of $\T^1$ that is not homotopic to a point and such that $f(\CC)=\CC$. Then there exists a Lipschitz function $g : \T^1\to \R$ such that $\CC=\big\{ \big( \theta , g(\theta)\big),\ \ \theta\in \T^1\big\} $ is the graph of $g$.
\end{Th}

This important Theorem uses all properties of the ECTM. Indeed in \cite[Proposition 5.13]{LeC} are examples of non conservative twist maps leaving invariant essential circles that are not graphs. It was generalized a few decades later to higher dimensional settings by Arnaud \cite{ArBir}.

\subsection{The generating function, properties and consequences}

Until the end of this section, we choose once and for all a generating function $S: \A \to \R$ (see Remark \ref{genfunrem}),   $\s : \R^2 \to \R$ is the lift of $S$ and let us choose  $\f = (\widetilde \Theta, R) : \R^2 \to \R^2$ a lift  of $f$. By the twist condition (Definition \ref{twistmap} point $3$), the map $\L : (\tilde \theta, r) \mapsto \big(\tilde \theta,\widetilde\Theta(\tilde\theta,r)\big)$ is a $C^1$--diffeomorphsim of $\R^2$ (see the links with Definition \ref{twist condition}). In the sequel,  $(\tilde \theta, \widetilde \Theta)$ are systematically used as coordinates, meaning that  the function $\ww S \circ \L^{-1}$ is considered instead of $\ww S$. For readability issues, it is still written $\ww S( \tilde \theta, \ww \Theta)$.

The properties of $f$ translate into the following features of $S$:
\begin{pr}\label{genfun}
The function $S $, through $\ww S$, verifies the following:
\begin{enumerate}
\item The function $\ww S $ is $C^2$ and periodic:
$$\forall (\tilde \theta, \ww \Theta) \in \R^2, \quad \ww S (\tilde \theta+1, \ww \Theta+1) = \ww S  (\tilde \theta, \ww \Theta).$$
\item For all $(\tilde \theta, \ww \Theta, r, R)\in \r^4$, 
\begin{equation}\label{vartwist}
\tilde f ( \tilde \theta,r) =( \ww\Theta,R) \iff
\begin{cases}
r= - \dfrac{\partial \ww S}{\partial \tilde \theta}(\tilde \theta, \ww \Theta),
\vspace{1mm}\\
R=\dfrac{\partial \ww S}{\partial \ww \Theta} (\tilde \theta, \ww \Theta).
\end{cases}
\end{equation}
\item The twist condition translates as follows:
\begin{itemize}
\item for $\tilde \theta\in \R$ fixed, the map $\widetilde \Theta \mapsto \frac{\partial \ww S}{\partial \tilde \theta}(\tilde \theta, \ww \Theta)$ is a decreasing $C^1$--diffeomorphism of $\R$;
\item for $\widetilde \Theta\in \R$ fixed, the map $\tilde \theta \mapsto \frac{\partial \ww S}{\partial \widetilde \Theta}(\tilde \theta, \ww \Theta)$ is a decreasing $C^1$--diffeomorphism of $\R$.

\end{itemize}

\end{enumerate}
\end{pr}

Let us comment on the previous proposition.

\begin{rem}\rm \label{genfunbis}
\hspace{2em}
\begin{enumerate}
\item It can actually be established that given a function $\ww S$ satisfying the three points of the previous proposition the associated function $\tilde f$ is a lift of an ECTM.
\item  The second point of Proposition \ref{genfun} can be directly read on equation \eqref{defS}.
\item The two items of point $3$ are equivalent. The first one results from the fact that the map $ \Theta \mapsto -\frac{\partial \ww S}{\partial \tilde \theta}(\tilde \theta, \ww \Theta)$ is the inverse of the map 
$r\mapsto  \ww \Theta(\tilde \theta,r)$.

The second one is an emanation of the fact that if $f$ is an ECTM that twists verticals to the right then $f^{-1}$ is an ECTM that twists  verticals to the left. A direct consequence of any of these two facts is that
\begin{equation}\label{derneg}
\forall (\tilde \theta_0, \ww \Theta_0) \in \R^2, \quad \frac{\partial^2 \ww S}{\partial \tilde \theta \partial \ww\Theta} (\tilde \theta_0, \ww \Theta_0)<0.
\end{equation}
Another important aftermath is that $\ww S$ is superlinear, meaning that 
\begin{equation}\label{Tsuperlinear}\lim_{|\tilde \theta - \ww \Theta| \to +\infty } \frac{\ww S (\tilde \theta, \ww \Theta)}{|\tilde \theta - \ww \Theta|} = +\infty.
\end{equation}
\end{enumerate}
\item It is proven in \cite{MatherP} (see also \cite{MatherMeasure,Bang}) that finite compositions of ECTM also possess a similar generating function. Hence all variational results relying solely on the generating function and its minimizers also apply to such finite compositions of ECTM.  
\end{rem}

Let us now define the notion of minimizing chains and sequences.
\begin{df}\label{minimi}\rm
\begin{enumerate}
\item Let $a<b$ be two integers such that $b-a>1$ and $(\tilde\theta_i)_{a\leqslant i \leqslant b}\in \R^{b-a+1}$. The chain $(\tilde\theta_i)_{a\leqslant i \leqslant b}$ is termed {\it minimizing} if
$$\sum_{i=a}^{b-1}\ww S( \tilde\theta_i,\tilde\theta_{i+1}) = \min_{\substack{\{x_i\in \R, \ a\leqslant i \leqslant b\} \\ x_a = \tilde\theta_a , x_b=\tilde\theta_b}}   \sum_{i=a}^{b-1} \ww S( x_i,x_{i+1}).$$
\item An infinite sequence  $(\tilde\theta_i)_{-\infty < i \leqslant b}$ or  $(\tilde\theta_i)_{a\leqslant i<+\infty }$ or  $(\tilde\theta_i)_{i\in \Z}$ is minimizing if all its finite subchains are minimizing.
\end{enumerate}

\end{df}

Note that if $(\tilde\theta_i)_{a\leqslant i \leqslant b}\in \R^{b-a+1}$ is minimizing between $\tilde\theta_a$ and $\tilde\theta_b$ and if $a\leqslant a'<b' -1<b'\leqslant b$ then the subchain  $(\tilde\theta_i)_{a'\leqslant i \leqslant b'}\in \R^{b'-a'+1}$ is minimizing between $\tilde\theta_{a'}$ and $\tilde\theta_{b'}$ which justifies the second definition.
 
 The first order necessary condition for a chain to be  minimizing  translates as follows: 
 let  $(\tilde\theta_i)_{i\in I}$ be a minimizing chain or sequence and $i_0\in I$ not be an extremity of $I$, then 
 \begin{equation}\label{dEL}
\frac{\partial \ww S}{\partial \ww\Theta}(\tilde\theta_{i_0-1},\tilde\theta_{i_0}) + \frac{\partial \ww S}{\partial \tilde\theta}(\tilde\theta_{i_0},\tilde\theta_{i_0+1}) = 0.
 \end{equation}
 
 This, together with equation \eqref{vartwist}, yields the following proposition:
 
 \begin{pr}\label{critmin}
 Let  $(\tilde\theta_i)_{i\in I}$ be a minimizing chain or sequence and for all $i\in I$, let 
 $$r_i = -   \frac{\partial \ww S}{\partial \tilde\theta}(\tilde\theta_{i},\tilde\theta_{i+1})=\frac{\partial\ww S}{\partial \ww\Theta}(\tilde\theta_{i-1},\tilde\theta_{i}),$$
 (where only  the well defined term is taken if $i$ is an extremity of $I$). Then $(\tilde\theta_i,r_i)_{i\in I}$ is a piece of orbit of $\tilde f$ in the sense that for all $i\in I$ such that $i+1\in I$, then $\tilde f(\tilde\theta_i,r_i)= (\tilde\theta_{i+1},r_{i+1})$. 
 
 In particular, a minimizing chain is uniquely determined by two consecutive terms $(\tilde\theta_{i},\tilde\theta_{i+1})$.
 \end{pr}

 Let us then introduce the notion of crossing and its consequences:
 
 \begin{df}\label{crossing}\rm
Two chains  $(\tilde\theta_i)_{i\in I}$ and $(\tilde\theta'_i)_{i\in I}$ are said to cross
 \begin{itemize}
 \item at some index $i_0\in I$ if $\tilde\theta_{i_0} = \tilde\theta'_{i_0}$;
 \item between two indices $i_0\in I$ and $i_0+1\in I$ if $(\tilde\theta_{i_0}  -\tilde\theta'_{i_0})(\tilde\theta_{i_0+1} - \tilde\theta'_{i_0+1})<0$.
 \end{itemize}

 \end{df}

\begin{rem}\label{remcross}\rm
If $(\tilde\theta_i)_{i\in I}$ and $(\tilde\theta'_i)_{i\in I}$ are distinct minimizing chains that cross at some index $i_0\in I$ that is an interior index, then the twist condition implies that
\begin{equation}\label{crossbis}
(\tilde\theta_{i_0-1}  -\tilde\theta'_{i_0-1})(\tilde\theta_{i_0+1} - \tilde\theta'_{i_0+1})<0.
\end{equation}
Indeed, let $r_{i_0} = -   \frac{\partial \ww S}{\partial \tilde\theta}(\tilde\theta_{i_0},\tilde\theta_{i_0+1})$ and $r'_{i_0} = -   \frac{\partial \ww S}{\partial \tilde\theta}(\tilde\theta_{i_0}',\tilde\theta_{i_0+1}')$. Necessarily $r_{i_0}\neq r_{i_0}'$ because otherwise  $(\tilde\theta_i)_{i\in I}$ and $(\tilde\theta'_i)_{i\in I}$ are the projections of the same orbit of $\tilde f$ (Proposition \ref{critmin}). Assume then for example that $r_{i_0} > r'_{i_0}$. We deduce from the twist condition that $ \tilde\theta_{i_0+1}>\tilde\theta'_{i_0+1}$ and that $\tilde\theta_{i_0-1}<\tilde\theta_{i_0-1}'$.

Inequality \eqref{crossbis} is often taken as the definition of crossing at $i_0$.
\end{rem}
 
One can then state Aubry \& Le Daeron's Fundamental Lemma:
 \begin{lm}\label{aubryfl}

  Let  $(\tilde\theta  ,\tilde\theta', \widetilde\Theta , \widetilde\Theta')\in \R^4$ such that  $(\tilde\theta  -\tilde\theta')(\widetilde\Theta - \widetilde\Theta')<0$. Then 
 $$\ww S(\tilde\theta,\widetilde\Theta)+\ww S (\tilde\theta',\widetilde\Theta')> \ww S(\tilde\theta,\widetilde\Theta')+\ww S(\tilde\theta',\widetilde\Theta).$$
 \end{lm}

\begin{proof}
It follows from the chain of equalities (in which we use the notations $\ww\Theta_t =  t \ww\Theta'+(1-t)\ww\Theta$ and $\tilde \theta_t = t\tilde\theta' + (1-t)\tilde\theta$),
\begin{multline*}
\ww S(\tilde\theta,\widetilde\Theta') - \ww S(\tilde\theta,\widetilde\Theta) +\ww S(\tilde\theta',\widetilde\Theta)-\ww S (\tilde\theta',\widetilde\Theta')   \\
=  \int_0^1 \Big[ \frac{\partial \ww S}{\partial \ww\Theta} (\tilde \theta,  \ww\Theta_t) -  \frac{\partial \ww S}{\partial \ww\Theta} (\tilde \theta',  \ww\Theta_t)\Big]\dd t\cdot (\ww\Theta'-\ww\Theta) \\
= \int_0^1 \int_0^1 \frac{\partial^2 \ww S}{\partial \tilde \theta \partial \ww\Theta} (\tilde\theta_s,\ww\Theta_t) \dd s \dd t \cdot (\tilde \theta-\tilde\theta') (\ww\Theta'-\ww\Theta) <0.
\end{multline*}
The last inequality is a consequence of the crossing hypothesis and of \eqref{derneg}.
\end{proof}

 We may now state Aubry's Non--crossing Lemma:
 \begin{pr}\label{noncrossing}
 Let $(\tilde\theta_i)_{i\in I}$ and $(\tilde\theta'_i)_{i\in I}$ be two distinct minimizing chains. Then one of the following holds
 \begin{itemize}
 \item  $(\tilde\theta_i)_{i\in I}$ and $(\tilde\theta'_i)_{i\in I}$ don't cross,
 \item  $(\tilde\theta_i)_{i\in I}$ and $(\tilde\theta'_i)_{i\in I}$ cross exactly once,
 \item   $I=[a,b]$ is a finite interval and $(\tilde\theta_i)_{i\in I}$ and $(\tilde\theta'_i)_{i\in I}$ cross exactly twice, at $a$ and $b$.
 \end{itemize}
 In the last case, both $(\tilde\theta_i)_{i\in I}$ and $(\tilde\theta'_i)_{i\in I}$ are maximal in the sense that neither one is a strict subchain of a minimizing chain.
 \end{pr}
 
 \begin{proof}
 To prove this proposition, let us assume that $(\tilde\theta_i)_{i\in I}$ and $(\tilde\theta'_i)_{i\in I}$ cross (at least) twice and that one of those crossing is not an extremity of $I$. There are several cases to deal with, we only cover two of them and let the other ones as an exercise as the ideas are the same.
 
 \underline{First case}: there are $\alpha<\beta$ such that $(\tilde\theta_i)_{i\in I}$ and $(\tilde\theta'_i)_{i\in I}$ cross between $\alpha$ and $\alpha+1$ and between $\beta$ and $\beta+1$. Then define two chains $(\hat\theta_i)_{i\in [\alpha, \beta+1]}$ and $(\hat\theta'_i)_{i\in [\alpha, \beta+1]}$ as follows:
 $$\hat\theta_i = 
 \begin{cases}
\tilde\theta_i' \ \ \textrm{if} \ \ i\in [\alpha+1,\beta]\\
\tilde\theta_i\ \ \textrm{if} \ \ i\in \{\alpha,\beta+1\};
 \end{cases}
$$
  $$\hat\theta'_i = 
 \begin{cases}
\tilde\theta_i \ \ \textrm{if} \ \ i\in [\alpha+1,\beta]\\
\tilde\theta_i'\ \ \textrm{if} \ \ i\in \{\alpha,\beta+1\}.
 \end{cases}
$$
Note that  $(\tilde\theta_i)_{i\in [\alpha, \beta+1]}$ and  $(\hat\theta_i)_{i\in [\alpha, \beta+1]}$ have same endpoints and so do  $(\tilde\theta'_i)_{i\in [\alpha, \beta+1]}$ and $(\hat\theta'_i)_{i\in [\alpha, \beta+1]}$ hence
$$\sum_{i=\alpha}^{\beta}\ww S( \tilde\theta_i,\tilde\theta_{i+1})\leqslant  \sum_{i=\alpha}^{\beta}\ww S( \hat\theta_i,\hat\theta_{i+1}); \quad \textrm{and} \quad 
\sum_{i=\alpha}^{\beta}\ww S( \tilde\theta'_i,\tilde\theta'_{i+1}) \leqslant   \sum_{i=\alpha}^{\beta}\ww S( \hat\theta'_i,\hat\theta'_{i+1})  .
$$
 Moreover
\begin{multline}
\sum_{i=\alpha}^{\beta}\ww S( \tilde\theta_i,\tilde\theta_{i+1})+\sum_{i=\alpha}^{\beta}\ww S( \tilde\theta'_i,\tilde\theta'_{i+1}) -\sum_{i=\alpha}^{\beta}\ww S( \hat\theta_i,\hat\theta_{i+1})-\sum_{i=\alpha}^{\beta}\ww S( \hat\theta'_i,\hat\theta'_{i+1}) \\
= \ww S( \tilde\theta_\alpha,\tilde\theta_{\alpha+1})+\ww S( \tilde\theta_\beta,\tilde\theta_{\beta+1})+ \ww S( \tilde\theta'_\alpha,\tilde\theta'_{\alpha+1})+\ww S( \tilde\theta'_\beta,\tilde\theta'_{\beta+1}) \\
- \ww S( \tilde\theta_\alpha,\tilde\theta'_{\alpha+1})-\ww S( \tilde\theta_\beta,\tilde\theta'_{\beta+1})- \ww S( \tilde\theta'_\alpha,\tilde\theta_{\alpha+1})-\ww S( \tilde\theta'_\beta,\tilde\theta_{\beta+1})>0 ;
\end{multline}
where the last inequality is obtained by two applications of Aubry's Fundamental Lemma \ref{aubryfl}.
This contradicts either the fact that $(\tilde\theta_i)_{i\in [\alpha, \beta+1]}$ is minimizing or that $(\tilde\theta'_i)_{i\in [\alpha, \beta+1]}$ is minimizing.

\underline{Second case}: there are  $\alpha<\beta$ such that $(\tilde\theta_i)_{i\in I}$ and $(\tilde\theta'_i)_{i\in I}$ cross at $\alpha$ and at $\beta$ and such that $\alpha-1\in I$.

 Then define two chains $(\hat\theta_i)_{i\in [\alpha-1, \beta]}$ and $(\hat\theta'_i)_{i\in [\alpha-1, \beta]}$ as follows:
 $$\hat\theta_i = 
 \begin{cases}
\tilde\theta_i' \ \ \textrm{if} \ \ i\in [\alpha,\beta]\\
\tilde\theta_i\ \ \textrm{if} \ \ i\in \{\alpha-1,\alpha\};
 \end{cases}
$$
  $$\hat\theta'_i = 
 \begin{cases}
\tilde\theta_i \ \ \textrm{if} \ \ i\in [\alpha,\beta]\\
\tilde\theta_i'\ \ \textrm{if} \ \ i\in \{\alpha-1,\alpha\}.
 \end{cases}
$$
We purposely insist on the fact that $\hat\theta_\alpha=\hat\theta'_\alpha=\tilde\theta_\alpha=\tilde\theta'_\alpha$.
Note that  $(\tilde\theta_i)_{i\in [\alpha-1, \beta]}$ and  $(\hat\theta_i)_{i\in [\alpha-1, \beta]}$ have same endpoints and so do  $(\tilde\theta'_i)_{i\in [\alpha-1, \beta]}$ and $(\hat\theta'_i)_{i\in [\alpha-1, \beta]}$. Hence
$$\sum_{i=\alpha-1}^{\beta-1}\ww S( \tilde\theta_i,\tilde\theta_{i+1})\leqslant  \sum_{i=\alpha-1}^{\beta-1}\ww S( \hat\theta_i,\hat\theta_{i+1}); \quad \textrm{and} \quad 
\sum_{i=\alpha-1}^{\beta-1}\ww S( \tilde\theta'_i,\tilde\theta'_{i+1}) \leqslant   \sum_{i=\alpha-1}^{\beta-1}\ww S( \hat\theta'_i,\hat\theta'_{i+1})  .
$$

Moreover the definitions of  $(\hat\theta_i)_{i\in [\alpha-1, \beta]}$ and $(\hat\theta'_i)_{i\in [\alpha-1, \beta]}$ yield
\begin{equation*}
\sum_{i=\alpha}^{\beta}\ww S( \tilde\theta_i,\tilde\theta_{i+1})+\sum_{i=\alpha}^{\beta}\ww S( \tilde\theta'_i,\tilde\theta'_{i+1}) -\sum_{i=\alpha}^{\beta}\ww S( \hat\theta_i,\hat\theta_{i+1})-\sum_{i=\alpha}^{\beta}\ww S( \hat\theta'_i,\hat\theta'_{i+1}) =0.
\end{equation*}
It follows that both inequalities above are equalities and that  $(\tilde\theta_i)_{i\in [\alpha-1, \beta]}$ and  $(\hat\theta_i)_{i\in [\alpha-1, \beta]}$ are both minimizing. As they coincide for both indices $\alpha-1$ and $\alpha$, they are equal (see Proposition \ref{critmin}). The same argument shows that  $(\tilde\theta'_i)_{i\in [\alpha-1, \beta]}$ and  $(\hat\theta'_i)_{i\in [\alpha-1, \beta]}$ are equal, and finally we have proved that  $(\tilde\theta_i)_{i\in [\alpha-1, \beta]}$ and  $(\tilde\theta'_i)_{i\in [\alpha-1, \beta]}$ are equal. This in turn implies that $(\tilde\theta_i)_{i\in I} = (\tilde\theta'_i)_{i\in I}$  which is a contradiction. 

The last assertion of the Proposition remains to be proven. The previous argument could be adapted here. Let us though propose another one. Assume by contradiction that $(\tilde\theta_i)_{i\in [a,b]}$ and $(\tilde\theta'_i)_{i\in [a-1, b]}$ are minimizing and verify $\tilde\theta_a=\tilde\theta'_a$, $\tilde\theta_b=\tilde\theta'_b$ (the other cases are treated the same way). Then by the minimization hypothesis it is  inferred that 
 $$ \sum_{i=a}^{b-1} \ww S (\tilde\theta_i,\tilde\theta_{i+1})=\sum_{i=a}^{b-1} \ww S ( \tilde\theta'_i,\tilde\theta'_{i+1}).$$
Moreover, note that the chain $(\tilde \theta_{a-1}', \tilde \theta_{a} ,\tilde \theta_{a+1})$ is not minimizing, otherwise there would be equality $\tilde \theta_{a+1}=\tilde \theta_{a+1}'$ contradicting the beginning of the Proposition. It follows there exists $\hat \theta \in \R$ such that 
$$\ww S(\tilde \theta_{a-1}', \hat \theta)+\ww S (\hat \theta ,\tilde \theta_{a+1})<\ww S(\tilde \theta_{a-1}', \tilde \theta_{a})+\ww S (\tilde \theta_{a} ,\tilde \theta_{a+1}).$$

The chain $(\hat\theta_i)_{i\in [a-1, b]}$ is then defined as  follows:
 $$\hat\theta_i = 
 \begin{cases}
 \tilde\theta_{a-1}' \ \ \textrm{if} \ \ i=a-1\\
\hat \theta \ \ \textrm{if} \ \ i=a\\
\tilde\theta_i\ \ \textrm{if} \ \ i\in [a+1,b];
 \end{cases}
$$

inducing that 
\begin{multline*}
\sum_{i=a-1}^{b-1} \ww S ( \hat\theta_i,\hat\theta_{i+1}) = \ww S(\tilde \theta_{a-1}', \hat \theta)+\ww S (\hat \theta ,\tilde \theta_{a+1})+\sum_{i=a+1}^{b-1} \ww S (\tilde\theta_i,\tilde\theta_{i+1}) \\
< \ww S(\tilde \theta_{a-1}', \tilde \theta_{a})+\ww S (\tilde \theta_{a} ,\tilde \theta_{a+1})+\sum_{i=a+1}^{b-1} \ww S (\tilde\theta_i,\tilde\theta_{i+1}) 
=\sum_{i=a-1}^{b-1} \ww S ( \tilde\theta'_i,\tilde\theta'_{i+1}).
\end{multline*}
This contradicts the fact that $(\tilde\theta'_i)_{i\in [a-1, b]}$ is minimizing.
  \end{proof}
  The previous non--crossing Proposition can be enforced.
  \begin{df}\rm
  When $I$ is suitably infinite, one says that two chains $(\tilde\theta_i)_{i\in I}$ and $(\tilde\theta'_i)_{i\in I}$ are respectively
  \begin{itemize}
  \item $\alpha$--asymptotic if $\lim\limits_{i\to -\infty} |\tilde\theta_i-\tilde\theta'_i| = 0$; we then say $(\tilde\theta_i)_{i\in I}$ and $(\tilde\theta'_i)_{i\in I}$ cross at $-\infty$;
  \item $\omega$--asymptotic if $\lim\limits_{i\to +\infty} |\tilde\theta_i-\tilde\theta'_i| = 0$; we then say $(\tilde\theta_i)_{i\in I}$ and $(\tilde\theta'_i)_{i\in I}$ cross at $+\infty$.
  \end{itemize}
  \end{df}
  Using this terminology, we state without proof (see for instance \cite[Lemma 3.9]{Bang}) the following strengthening of Proposition \ref{noncrossing}:
  \begin{pr}\label{cross-asymptotic}
Let $(\tilde\theta_i)_{i\in I}$ and $(\tilde\theta'_i)_{i\in I}$ be two distinct minimizing chains. Assume furthermore that the sequence $|\tilde\theta_{i+1} - \tilde\theta_i |$ is bounded. Then $(\tilde\theta_i)_{i\in I}$ and $(\tilde\theta'_i)_{i\in I}$ cross at most once, except possibly at both ends of $I$.

 In the latter case, both chains are maximal minimizing chains.
  \end{pr}

\begin{rem}\rm
A consequence of our forthcoming analysis  will be that the hypothesis concerning the boundedness of $|\tilde\theta_{i+1} - \tilde\theta_i |$ is in fact automatically verified.
  \end{rem}

 This section ends with a fundamental property of minimizing sequences. It is by no means a direct consequence of the previous stated facts. The proof is quite tricky and we refer the interested reader to  \cite[Theorem 3.15]{Bang} for details. 
 
 \begin{Th}[Aubry--Mather] \label{AM}
 Let $(\tilde\theta_i)_{i\in \Z}$ be a minimizing sequence. Then the following hold:
 \begin{enumerate}
 \item For all $(a,b)\in \Z^2$, $(\tilde\theta_i)_{i\in \Z}$ and $(\tilde\theta_{i-a}+b)_{i\in \Z}$ do not cross.
 \item There exists a homeomorphism  $\tilde g : \R \to \R$  verifying $\tilde g(x+1) =\tilde  g(x)+1$ for all $x\in \R $ and such that\footnote{$g$ is the lift of an orientation preserving circle homeomorphism.}  
 $$\forall i\in \Z, \quad \tilde g(\tilde\theta_i) = \tilde\theta_{i+1}.$$
 \item It follows then from Poincaré theory that there exists a real number $\rho\in \R$ called {\it  rotation number} such that 
 $$\forall i\in \Z, \quad |\tilde\theta_i - \tilde\theta_0 - i\rho|<1.$$
 In particular, $\lim\limits_{|i|\to +\infty} \dfrac{\tilde\theta_i}{i} = \rho$.
 \end{enumerate}

 \end{Th}

 \section{Examples and Moser's Theorem}
 \subsection{Notions of integrability}
 
 Our first family of examples are called {\it integrable ECTMs}. They are of the following form:
 \begin{ex}\label{Tintegrable}\rm 
 Let $\rho : \R \to \R$ be an increasing diffeomorphism. Then the map $f_\rho : (\theta,r) \mapsto (\theta+\rho(r),r)$ is an ECTM.
 \end{ex}
 
 Such maps have a particularly simple dynamics. Indeed, for all $r\in \R$ if we denote by $\CC_r = \{(\theta,r),\ \ \theta \in \T^1\}$ the canonical circle of height $r$, then $\CC_r$ is invariant by $f_\rho$ and the dynamics of $f_\rho$ restricted to $\CC_r$ is a rotation of angle $\rho(r)$ (mod $1$). One checks that $f_\rho $ is the time $1$ map of the Tonelli Hamiltonian flow of $H_\rho : (\theta,r)\mapsto \varrho(r)$ where $\varrho$ is a primitive of $\rho$.
 
 By extension, let us introduce several classes of ECTM exhibiting similar dynamical features.

\begin{df}\rm
Let $f : \A \to \A$ be an ECTM. We say $f$ is integrable if there exists a $C^1$ area preserving diffeomorphism $\Phi : \A \to \A$ and $\rho : \R \to \R$ such that $f = \Phi^{-1} \circ \f_\rho \circ \Phi$. 
\end{df}
\begin{rem}\rm
In the case of an integrable ECTM, the function $\rho $ will automatically be an increasing diffeomorphism. However, note that the previous definition can be  extended trivially to any area preserving transformation of the annulus.

In the previous definition, the transformation $f$ is the time $1$ map of the Hamiltonian flow\footnote{This is a general fact and proves that the group of Hamiltonian diffeomorphisms is a normal subgroup of the group of symplectomorphisms (see \cite[Exercise 7 page 471]{Audin}).} of $H_\rho \circ \Phi$. 
\end{rem}
 
 The following notions are obtained by considering twist maps  for which the annulus $\A$ is foliated by invariant circles. Bare in mind that by Birkhoff's Theorem \ref{birk}, such circles are automatically Lipschitz graphs.  The integrable case corresponds to a classical foliation. The next notions are obtained by weakening the regularity of the invariant foliation. We therefore start by defining what are those non--regular foliations.
 
 \begin{df}\label{C0fol}\rm
 A continuous foliation $\FF = \{\FF_c, \ c\in \R\}$, (otherwise called lamination) of $\A$ by graphs is defined through   a continuous function $\eta : (\theta,c) \mapsto \eta_c(\theta)$ from $\A$ to $\R$ such that
 \begin{itemize}
 \item for all $\theta \in \R$, the map $c \mapsto \eta_c(\theta)$ is a (increasing) homeomorphism of $\R$,
 \item for all $c\in \R$, \  $\int_{\T^1} \eta_c(\theta)\dd \theta = c$.
 \end{itemize}
For $c\in \R$ the circle $\FF_c = \big\{\big(\theta, \eta_c(\theta)\big), \ \ \theta\in \T^1\big\}$ is the leaf of the continuous foliation at cohomology $c$. It then follows that $\A $ is the disjoint union of the leaves of the foliation.

Reciprocally, any such function $\eta : (\theta,c) \mapsto \eta_c(\theta)$ from $\A$ to $\R$ defines a continuous foliation of $\A$ by graphs.
 \end{df}

 \begin{rem}\rm
 The second point  in the definition of $\eta$ is a normalization  condition. It could be dropped to give an equivalent notion. However, it is so convenient we prefer to include it directly in the definition. 
 \end{rem}
 
 For an intermediate regularity, arises naturally  the notion of Lipschitz foliation:

  \begin{df}\label{Lipfol}\rm
 A Lipschitz  foliation $\FF = \{\FF_c, \ c\in \R\}$  of $\A$ by graphs is the data of a continuous  foliation with associated function $\eta : (\theta,c) \mapsto \eta_c(\theta)$ from $\A$ to $\R$ such that
$$\exists K>0, \forall (c_1,c_2)\in \R^2,\forall \theta \in \T^1, \quad \frac1K |c_1-c_2| \leqslant |\eta_{c_1}(\theta) - \eta_{c_2}(\theta) | \leqslant K|c_1-c_2|.
$$
 \end{df}

It is now at grasp to define weaker notions of integrability for twist maps:

\begin{df}\rm\label{integrabletwist}
\hspace{2em}
\begin{itemize}
\item An ECTM $f$ of the annulus is $C^0$--integrable if there exists a continuous foliation,   $\FF = \{\FF_c, \ c\in \R\}$,  of $\A$ by graphs, such that for all $c\in \R$, $f(\FF_c)=\FF_c$.
\item  An ECTM $f$ of the annulus is Lipschitz integrable if there exists a Lipschitz foliation,   $\FF = \{\FF_c, \ c\in \R\}$,  of $\A$ by graphs, such that for all $c\in \R$, $f(\FF_c)=\FF_c$.
\end{itemize}
\end{df}
  
  Those various notions of integrability each have dynamical consequences on the underlying ECTM. Such results are presented without proof in the last Section \ref{lastlast}.
Research on those notions is quite frustrating though, starting from the fact that it is still conjectural whether those three notions of integrability are different.  For instance, there are no known examples of ECTM that are $C^0$--integrable but not integrable (meaning that all the leaves are smooth).  
 
 On the bright side, it is proved in \cite{AZ,AZ2} that  continuous foliations by graphs that are invariant by an ECTM  must satisfy some particular properties. This is used to exhibit foliations that  cannot  be invariant by an ECTM.
 
 \begin{Th}\label{folpasinv}
 Let $\FF = \{\FF_c, \ c\in \R\}$ be the foliation associated to the function $\eta_c(\theta) = c+\varepsilon(c)\cos(2\pi \theta)$, where $\varepsilon : \R \to \R$ is a non $C^1$, Lipschitz, function with Lipschitz constant less than $(2\pi)^{-1}$. Then $\FF$ is not invariant by any ECTM.
 \end{Th}

Though the full proof of this Theorem goes beyond the scope of the present text, some explanations will be provided at the very end, in Section \ref{lastlast}.

\subsection{The standard family}

According to Remark \ref{genfunbis} it is very easy to construct twist maps with no particular property.
However, let us mention a historically very important family of examples.

\begin{ex}\rm \label{standardfamily}
For all $\varepsilon \in \R$ let us define the map $F_\varepsilon : \A \to \A$ by

$$\forall (\theta,r)\in \A, \quad F_\varepsilon(\theta,r) = \Big( \theta + r - \frac{\varepsilon}{2\pi} \sin(2\pi \theta) , r- \frac{\varepsilon}{2\pi} \sin(2\pi \theta)\Big).$$
\end{ex}

This family serves as a test for the state--of--the--art  research on twist maps.  On certain aspects, the picture is not glorious. For instance, a conjecture of Sinai is that $F_{\varepsilon}$ has positive metric entropy for all $\varepsilon \neq 0$ (with respect to the Lebesgue measure) even though nobody knows how to prove it even for a single parameter. More on such questions is discussed in works of Berger and Berger, Turaev \cite{Ber1,Ber2}. A picture of the dynamics for some arbitrary value of $\varepsilon$ is presented in the introduction of \cite{Gole} and clearly shows that such a dynamics is very rich. On the one hand, for $\varepsilon = 0$ the map $F_0$ is the most basic example among  integrable ECTMs. For $\varepsilon$ small, KAM theory applies and many invariant circles with diophantine rotation numbers persist (see \cite{Her1,Her2} and references therein). On the other hand, a theorem of Mather \cite{Mather43} states that for $\varepsilon >\frac43$, $F_\varepsilon$ has no invariant essential circle. 

What happens to the circles when they disappear is a challenging question: what is their regularity, the dynamics on them at the last parameter...
After they disappear  Aubry--Mather theory provides an answer  as to what they become.

\subsection{General twist maps and a Theorem of Moser}

Amongst other examples and open questions, the nature of possible invariant circles of twist maps is still not well understood. In the recent \cite{avila}, the authors construct a $C^1$ ECTM having an invariant circle that is not everywhere differentiable with a minimal (irrational) restricted dynamics on it\footnote{Recall that a homeomorphism of the circle having an irrational rotation number is either minimal (all orbits are dense) or not. In the first case, it is conjugated to the irrational rotation. In the second case, it is only semi--conjugated to the irrational rotation as there are wandering intervals. We then  speak of a Denjoy counterexample. A Theorem of Denjoy states that Denjoy counterexamples cannot be of class $C^2$.}.
This answers partially a question of Arnaud \cite{A11} asking whether such $C^2$ maps exist. 

Concerning invariant essential circles on which the dynamics is irrational and not minimal (that of a Denjoy counterexample) examples have been constructed in \cite{Her1} where a $C^1$ invariant circle is constructed, and in \cite{A13} where a non--differentiable invariant curve is constructed (thus answering a question of Mather). By Denjoy's Theorem, such an invariant circle cannot be $C^2$.

Last, concerning the inverse problem, in \cite{Arnaud3} are provided examples of essential circles that are graphs of Lipschitz maps but cannot be invariant by any ECTM.

 Before turning back to weak KAM theory, let us mention an important Theorem of Moser \cite{Moser} stating that any regular ECTM can be represented by a Hamiltonian function. The drawback is that the latter is not autonomous (i.e. it is time--dependent).
 
 \begin{Th}[Moser]\label{moser}
 Let $f: \A \to \A$ be an ECTM. There exists a $C^2$ time--dependent Hamiltonian $H : \R \times \A \to \R$ such that 
 \begin{itemize}
 \item 
 for all $t \in \R$, $ H(t ,\cdot, \cdot ) : \A \to \R$ is a Tonelli Hamiltonian;
 \item for all $(t, \theta, r) \in \R\times \A$, then $ H(t+1,\theta,r) =    H(t,\theta,r)$;
\item the ECTM $ f$ is the time $1$ map $\varphi_0^1$ of the Hamiltonian flow of $ H$ generated by the equations
\begin{equation}\label{tHamilton}\begin{cases}
\dot{\theta}(t) = \partial_r  H\big(t,\theta(t),r(t)\big), \\
\dot{r}(t) = -\partial_{\theta} H\big(t,\theta(t),r(t)\big).
\end{cases}
\end{equation}
 \end{itemize}
 \end{Th}
 
 Note that in Moser's original article, only the case of a $C^\infty$ ECTM is treated. However, it is stated in the paper that the proof adapts to less regular functions. Moreover, it can be checked that Moser's construction allows to interpolate between Identity and $f$ by twist maps using the family $(\varphi_0^t)_{t\in [0,1]}$. The generating functions $(\widetilde S_t)_{t\in (0,1]}$ of those twist maps are given by the relations
 $$\tilde \varphi_0^t (\tilde \theta,r) = (\widetilde\Theta , R) \ \iff  \widetilde S_t(\tilde\theta, \widetilde\Theta) = \int_0^t \widetilde L\big(s, \gamma_{( \tilde \theta,r)}(s) , \dot \gamma_{( \tilde \theta,r)}(s) \big),$$
 where the Lagrangian function $L : \R\times \A \to \R$ is defined as previously by 
 $$\forall (t,\theta, v)\in \R \times \A , \quad L(t,\theta, v) = \sup_{r\in \R} rv - H(t,\theta,r),$$
 and the curve $\gamma_{( \tilde \theta,r)}$ is defined by
 $$\forall s\in \R, \quad \gamma_{( \tilde \theta,r)}(s) = \pi_1 \circ \tilde \varphi_0^s (\tilde \theta,r).$$
 Here, the generating function is given by the Lagrangian (minimal) action.
 
 \section{Weak KAM for twist maps}
 We now turn to the study of the Lax--Oleinik semigroup and weak KAM solutions for an ECTM, $f: \A \to \R$. As was already  apparent in the previous Chapter dedicated to examples, there is actually a one parameter family of semigroups, indexed by the cohomology. The underlying metric space here is $X = \T^1$ that is obviously compact. We will precisely build a $1$--parameter family of cost functions using that the first cohomology group $H^1(\T^1,\R)$ is isomorphic to $\R$. This appears in the early works of Mather on twist maps that later led him to define the $\alpha$ function in a higher dimensional Lagrangian setting (\cite{Mather1}). The latter idea was also adapted in \cite{ZJMD} on more general metric spaces.  
 
As previously, $S$ is a generating function of $f$,   $\s : \R^2 \to \R$ is the lift of $S$ and we choose  $\f = (\widetilde \Theta, R) : \R^2 \to \R^2$ a lift  of $f$.  The canonical projection is denoted by $\pi : \R\to \T^1$.

\begin{df}\rm \label{coutc}
Let $c\in \R$ that we will refer to as a cohomology class. The cost function $S^c : \T^1 \times \T^1 \to \R$ is defined by
$$\forall (\theta,\Theta)\in  \T^1 \times \T^1, \quad S^c(\theta, \Theta) = \inf_{\substack{ \pi(\tilde \theta) = \theta \\ \pi(\ww\Theta) = \Theta   } } \ww S(\tilde \theta, \ww \Theta) + c(\tilde \theta - \ww\Theta).$$
\end{df}
\begin{rem}\rm\label{remSc}
\hspace{2em}
\begin{itemize}
\item In the previous definition, using the translation invariance of $\ww S$, it is also possible to fix a $\tilde \theta \in [-1,1]$ such that $\pi(\tilde \theta) = \theta$ and take the infimum solely on $\ww \Theta$.
\item As $\ww S$ is superlinear \eqref{Tsuperlinear} it is easily seen that the previous infimum is actually a minimum. Moreover, if as asserted, we restrict to $\tilde \theta \in [-1,1]$ then there exists $K>0$ (depending on $c$) such that the minimum on $\ww\Theta$ can be restricted to $\ww\Theta\in [-K,K]$.
\item Given $\theta \in \T^1$ and $\tilde \theta \in \R$ such that $\pi(\tilde \theta) = \theta$, the derivative $\partial_2 S^c ( \theta , \Theta)$ exists if and only if there exists a unique $\ww \Theta_0$ realizing the minimum  $\ S^c ( \theta , \Theta) =  \ww S(\tilde \theta, \ww \Theta_0) + c(\tilde \theta - \ww\Theta_0)$. In such case, $\partial_2 S^c ( \theta , \Theta) = \partial_2 S (\tilde  \theta , \ww\Theta_0) -c$.
\item  Similarly, given $\Theta \in \T^1$ and $\widetilde \Theta \in \R$ such that $\pi(\widetilde \Theta) = \Theta$, the derivative $\partial_1 S^c ( \theta , \Theta)$ exists if and only if there exists a unique $\tilde \theta_0$ realizing the minimum  $\ S^c ( \theta , \Theta) =  \ww S(\tilde \theta_0, \ww \Theta) + c(\tilde \theta_0 - \ww\Theta)$. In such case, $\partial_1 S^c ( \theta , \Theta) = \partial_1 \ww S (\tilde  \theta_0 , \ww\Theta) +c$.
\item The previous points and the fact that $f$ is a twist map yields that $S^c$ verifies the left and right twist conditions as defined in Definition \ref{twist condition} (see also \cite[Proposition 6.4]{ZJMD}).
\item From the previous points, the same strange regularity property  observed for costs coming from Tonelli Lagrangians (see discussion following Proposition \ref{eqAubry}) is brought to light:  given $(\theta, \Theta)\in \T^1\times \T^1$, the following are equivalent
\begin{itemize}
\item $\partial_1 S^c ( \theta , \Theta)$ exists;
\item $\partial_2 S^c ( \theta , \Theta)$ exists.
\end{itemize}
\end{itemize}
\end{rem}
With those facts and bearing in mind that an infimum of a compact family of $C^2$ functions is semiconcave, it follows that 
\begin{pr}\label{Scsemic}
The function $S^c$ is locally semiconcave. More precisely, any of its lifts to $\R^2$ is semiconcave.
\end{pr}

In particular, if $S^c$ admits any partial derivative at points $(\theta, \Theta)$ then  $S^c$ is differentiable at $( \theta , \Theta)$.

From the previous points and the semiconcavity of $S^c$ (see the proof of Proposition \ref{regv} and the following Remark \ref{remsc}), it follows that:
\begin{pr}\label{prSc}
If $(\theta_k)_{k\in I}$ is a minimizing sequence or chain (with at least $3$ points) for $S^c$ then all the derivatives $\partial_i S^c(\theta_k, \theta_{k+1})$, $i\in \{1,2\}$ and $(k,k+1)\in I\times I$, exist. If $k_0\in I$ and $\tilde\theta_{k_0}$ is a lift of $\theta_{k_0}$ then there exists a unique (minimizing) chain  $(\tilde\theta_k)_{k\in I}$ such that
$$\forall k<k', \quad \sum_{\ell = k}^{k'-1} S^c(\theta_{\ell},\theta_{\ell+1} )= \sum_{\ell = k}^{k'-1} \ww S(\tilde\theta_{\ell},\tilde\theta_{\ell+1}) +c(\tilde\theta_k-\tilde\theta_{k'}).$$
As is now customary, setting then 
$$r_\ell = -\partial_1  \ww S(\tilde\theta_{\ell},\tilde\theta_{\ell+1})  = \partial_2  \ww S(\tilde\theta_{\ell-1},\tilde\theta_{\ell})= c  -\partial_1   S^c(\theta_{\ell},\theta_{\ell+1})  =c+ \partial_2   S^c(\theta_{\ell-1},\theta_{\ell}),$$
then  $(\theta_k,r_k)_{k\in I}$ is a piece of orbit of $f$ and $(\tilde\theta_k,r_k)_{k\in I}$ is a piece of orbit of $\tilde f$.
\end{pr}

We will now focus our attention on the negative Lax--Oleinik semigroups associated to $S^c$ and gather results previously proven.

\begin{df}\rm \label{LOtwist}
Given $c\in \R$, we define the operator $T^c$ which, to a bounded function $u : \T^1 \to \R$, associates the function 
$$T^c u : \Theta \in \T^1 \mapsto \inf_{\theta \in \T^1} u(\theta) + S^c(\theta ,\Theta).$$
Equivalently, recalling that two infimums commute, if $\tilde u : \R \to \R$  is the  lift of $u$ then the lift of $T^c u$ is given by the relation
$$\forall \ww\Theta \in \R, \quad \ww T^c \tilde u (\ww\Theta) = \inf_{\tilde \theta \in \R} \tilde u ( \tilde \theta) + \ww S(\tilde \theta, \ww \Theta) + c(\tilde \theta - \ww\Theta).$$
\end{df}
 The weak KAM Theorem then states:

\begin{Th}\label{KAMft}
For all $c\in \R$ there exists a unique constant $\alpha(c)$ for which the equation 
\begin{equation}\label{wkc}
u = T^c u + \alpha(c)
\end{equation}
admits solutions $u^c : \T^1 \to \R$. Such a solution (or its lift $\tilde u_c : \R \to \R$) will be called a {\it weak KAM solution at cohomology $c$}.
\end{Th}
The function $\alpha : \R \to \R$ is called {\it Mather's $\alpha$ function}.
\begin{pr}\label{alphaconv}
Mather's $\alpha$ function is convex and superlinear.
\end{pr}
\begin{proof}
One easily reconstructs from the definition of $ S^c$ and Proposition \ref{criticalSS} that if $c\in \R$, then $\alpha(c)$ is the least constant $a\in \R$ such that there exists a $1$--periodic function $\tilde v : \R \to \R$ verifying 
$$\forall (\tilde \theta, \tilde \theta') \in \R \times \R, \quad \tilde v(\tilde \theta') - \tilde v(\tilde \theta) \leqslant \ww S( \tilde \theta, \tilde \theta') + c(\tilde\theta - \tilde \theta') +a.$$
If now $c_1$ and $c_2$ are real numbers and $\tilde u_1$ and $\tilde u_2$ are lifts of weak KAM solutions at the corresponding cohomology classes, if $t\in (0,1)$ one infers that, setting $\tilde u_t = t\tilde u_1 + (1-t) \tilde u_2$,
$$\forall (\tilde \theta, \tilde \theta') \in \R \times \R, \quad \tilde u_t(\tilde \theta') - \tilde u_t(\tilde \theta) \leqslant \ww S( \tilde \theta, \tilde \theta') + (tc_1+(1-t)c_2)(\tilde\theta - \tilde \theta') +t\alpha(c_1)+(1-t)\alpha(c_2).$$
It follows that $\alpha(tc_1+(1-t)c_2)\leqslant t\alpha(c_1)+(1-t)\alpha(c_2)$ and the convexity is proved.

For superlinearity, let $k>0$ be an integer. If $c\in \R$, let $u_c : \T^1 \to \R$ be a weak KAM solution at cohomology $c$. If $\tilde u_c$ is its lift, by periodicity, we infer that
$$0=\tilde u_c(k)- \tilde u_c (0) \leqslant \ww S(0,k) -ck+ \alpha(c),$$
$$0=\tilde u_c(-k)- \tilde u_c (0) \leqslant \ww S(0,-k) +ck+ \alpha(c).$$
From those inequalities, we infer that, setting $C_k = \max \big(\ww S(0,k), \ww S(0,-k)\big)$,
$$\forall c\in \R, \quad \alpha(c) \geqslant k|c| - C_k,$$
from which it follows that $\lim\limits_{|c|\to +\infty} \dfrac{\alpha(c)}{|c|} = +\infty$.

\end{proof}

 Informations obtained from Proposition \ref{regv}, Proposition \ref{critmin} and Remark \ref{remSc} are gathered in the next Theorem (keeping in mind that a weak KAM solution is locally semiconcave):
 
 \begin{Th}\label{twistKAM}
 Let $c\in \R$, $u_c : \T^1 \to \R$ be a weak KAM solution at cohomology $c$ and $\tilde u_c : \R \to \R$ its lift.
 \begin{enumerate}
 \item For all $\theta_0 \in \T^1$,  there exists $( \theta_k^c)_{k\leqslant 0}$ such that  $ \theta_0= \theta_0^c$ and
$$\forall k<0 , \ \  u_c ( \theta_0) =   u_c ( \theta_k^c) + \sum_{i=k}^{-1} S^c( \theta_i^c, \theta_{i+1}^c)  +|k|\alpha(c).$$

 \item The previous sequence may not be unique but it is uniquely determined by $(\theta^c_{-1}, \theta_0)$. Moreover, setting for $k\leqslant 0$, $r_k =c+ \partial_2 S^c (\theta_{k-1}^c, \theta_k^c)$ (that exists) then $( \theta_k^c,r_k)_{k\leqslant 0}$ is a piece of orbit of $f$.
 
 \item For all $k<0$, $u_c$ is derivable at $\theta_k^c$ and $c+u_c'(\theta_k^c) = r_k$. Moreover, $r_0 \in \partial^+ u_c(\theta_0)$ and $u_c$ is derivable at $\theta_0$ if and only if the sequence $( \theta_k^c)_{k\leqslant 0}$ is unique.
 
 \item It follows that $f^{-1} \left(  \overline{\G(c +u'_c)}  \right) \subset \G(c +u'_c)$ where $\G(c +u'_c) $ is the set of $\big(\theta, c+u'_c(\theta)\big)$ for $ \theta \in \T^1$ such that $u'_c(\theta)$ exists.\footnote{In the previous inclusion, the fact that we can take a closure on the left hand side is because a limit of calibrating sequences for $u_c$ is still calibrating.}
  Moreover, if  $(\theta_0,r_0) \in \overline{\G(c +u'_c)} $ and  for $k\leqslant 0$, $(\theta_k,r_k) = f^k(\theta_0,r_0)$, then 
 $$\forall k<0 , \ \  u_c ( \theta_0) =   u_c ( \theta_k) + \sum_{i=k}^{-1} S^c( \theta_i, \theta_{i+1})  +|k|\alpha(c).$$

 \item Given a sequence $( \theta_k^c)_{k\leqslant 0}$ as above and $\tilde \theta_0 \in \R $ a lift of $\theta_0$, there   exists a unique $(\tilde \theta_k^c)_{k\leqslant 0}$ that projects on $ ( \theta_k^c)_{k\leqslant 0}$ and such that 

 $$\forall k<0 , \ \ \tilde u_c (\tilde \theta_0) =  \tilde u_c (\tilde \theta_k^c) + \sum_{i=k}^{-1} \ww S(\tilde \theta_i^c,\tilde \theta_{i+1}^c) + c(\tilde \theta_k^c - \tilde \theta_0^c) +|k|\alpha(c).$$
 
 \item With the previous notations for $k\leqslant 0$, $r_k = \partial_2 \ww S (\tilde \theta_{k-1}^c, \tilde\theta_k^c)$ and $(\tilde \theta_k^c,r_k)_{k\leqslant 0}$ is a piece of orbit of $\tilde f$.

 \item For all $k<0$, $\tilde u_c$ is derivable at $\tilde \theta_k^c$ and $c+\tilde u_c'(\tilde\theta_k^c) = r_k$. Moreover, $r_0 \in \partial^+ \tilde u_c(\tilde \theta_0)$ and $\tilde u_c$ is derivable at $\tilde\theta_0$ if and only if the sequence $(\tilde \theta_k^c)_{k\leqslant 0}$ is unique.
 
 \item It follows that $\tilde f^{-1} \left(  \overline{\G(c +\tilde u'_c)}  \right) \subset \G(c +\tilde u'_c)$ where $\G(c +\tilde u'_c) $ is the set of $\big(\theta, c+\tilde u'_c(\theta)\big)$ for $ \tilde\theta \in \R$ such that $\tilde u'_c(\tilde\theta)$ exists.
 Finally, if  $(\tilde\theta_0,r_0) \in \overline{\G(c +\tilde u'_c)} $ and  for $k\leqslant 0$, $(\tilde\theta_k,r_k) = \tilde f^k(\theta_0,r_0)$, then 
 
  $$\forall k<0 , \ \ \tilde u_c (\tilde \theta_0) =  \tilde u_c (\tilde \theta_k) + \sum_{i=k}^{-1} \ww S(\tilde \theta_i,\tilde \theta_{i+1}) + c(\tilde \theta_k - \tilde \theta_0) +|k|\alpha(c).$$
 \end{enumerate}

 \end{Th}
 
 \begin{rem}\label{derivsemcon}\rm
Being in dimension $1$, concave functions have very strong derivability properties that semiconcave functions inherit from. It follows that if $v : \T^1\to \R$ is a semiconcave function, then it admits at all $\theta \in \T^1$ a left derivative $v'_-(\theta)$ and a right derivative $v'_+(\theta)$ that verify $v'_-(\theta)\geqslant v'_+(\theta)$. Its superdifferential at $\theta$ is then the segment $\partial^+ v(\theta)=[v'_+(\theta),v'_-(\theta)]$. Finally, using the previous notation, if $c\in \R$ then 
 $$ \overline{\G(c +v')} = \bigcup_{\theta \in \R} \{\theta\} \times \{c+v'_+(\theta),c+v'_-(\theta)\}.$$
 \end{rem}
 We introduce the notion of full pseudograph that is the graph of the superdifferential of a semiconcave function and appears in various weak KAM related works (\cite{BerPseudo,APseudo}):
 
 \begin{df}\label{pg}\rm
 Let $v : \T^1 \to \R$ be a semiconcave function. Then its full pseudograph is
 $$\PG(v') = \bigcup_{\theta\in \T^1} \{\theta\} \times \partial^+ v(\theta)=\bigcup_{\theta\in \T^1} \{\theta\} \times[v'_+(\theta),v'_-(\theta)].$$ 
 A related notation: for $c\in \R$,
 $$\PG(c+v') = \bigcup_{\theta\in \T^1} \{\theta\} \times \big(c+\partial^+ v(\theta)\big)=\bigcup_{\theta\in \T^1} \{\theta\} \times[c+v'_+(\theta),c+v'_-(\theta)].$$ 
We will use similar notations for semiconcave functions on $\R$.
 \end{df}
 A beautiful theorem due to Marie-Claude Arnaud (using weak KAM methods) yields: 
 
 \begin{pr}\label{pgGraph}
 Let $v : \T^1 \to \R$ be a semiconcave function. Then its full pseudograph $\PG(v')$ is a Lipschitz manifold. In the present case it is a Lipschitz essential circle (meaning it separates the annulus in two unbounded connected components).
 \end{pr}

 We are now ready to state a first result on the interplay between Lax--Oleinik and the non--crossing lemma:
 
 \begin{lm}\label{orderWKAM}
 Let $v : \T^1\to \R$ be a continuous function, $c\in \R$ and $\tilde\theta_1<\tilde\theta_2$ two real numbers.
 \begin{enumerate}
 \item Assume that $\tilde\theta'_1$ and $\tilde\theta'_2$ verify for $i\in \{1,2\}$,
 $$  \widetilde T^c \tilde v(\tilde\theta_i)=\min_{\tilde\theta'\in\R} \big(\tilde v(\tilde\theta')+\ww S(\tilde\theta',\tilde \theta_i)+c(\tilde\theta'-\tilde\theta_i)\big)=\tilde v(\tilde\theta'_i)+\ww S(\tilde\theta'_i,\tilde \theta_i)+c(\tilde\theta'_i-\tilde\theta_i).$$
 
Then $\tilde \theta'_1\leqslant \tilde \theta'_2$.

\item If moreover $v$ is semiconcave, then $\tilde\theta'_1 <\tilde \theta'_2$.
\end{enumerate}
 \end{lm}
 
 \begin{proof}
 Let us argue by contradiction. Then by Proposition \ref{aubryfl} the following holds:
 \begin{multline*}
 \tilde v(\tilde\theta'_1)+\ww S(\tilde\theta'_1,\tilde \theta_1)+c(\tilde\theta'_1-\tilde\theta_1)+\tilde v(\tilde\theta'_2)+\ww S(\tilde\theta'_2,\tilde \theta_2)+c(\tilde\theta'_2-\tilde\theta_2) > \\
  > \tilde v(\tilde\theta'_2)+\ww S(\tilde\theta'_2,\tilde \theta_1)+c(\tilde\theta'_2-\tilde\theta_1)+\tilde v(\tilde\theta'_1)+\ww S(\tilde\theta'_1,\tilde \theta_2)+c(\tilde\theta'_1-\tilde\theta_2).
 \end{multline*}
We infer that at least one of the two inequalities 
$$ \tilde v(\tilde\theta'_1)+\ww S(\tilde\theta'_1,\tilde \theta_1)+c(\tilde\theta'_1-\tilde\theta_1)> \tilde v(\tilde\theta'_2)+\ww S(\tilde\theta'_2,\tilde \theta_1)+c(\tilde\theta'_2-\tilde\theta_1),
$$
$$\tilde v(\tilde\theta'_2)+\ww S(\tilde\theta'_2,\tilde \theta_2)+c(\tilde\theta'_2-\tilde\theta_2)>\tilde v(\tilde\theta'_1)+\ww S(\tilde\theta'_1,\tilde \theta_2)+c(\tilde\theta'_1-\tilde\theta_2),
$$
is valid that is a contradiction.

To prove the second item, we recall that thanks to Proposition \ref{regv}, if $\tilde v$ is semiconcave, then it is derivable both at $\tilde \theta'_1$ and $\tilde \theta'_2$ and $\pi_1 \circ \tilde f\big( \tilde \theta'_i , c+\tilde v' (\tilde \theta_i')\big)  = \tilde \theta_i$. As $\tilde \theta_1 \neq \tilde \theta_2$, necessarily $\tilde \theta'_1 \neq \tilde \theta'_2$.

 \end{proof}
 
 An interesting Corollary, reminiscent of Bernard's Theorem \ref{regB} and that will be needed later is:
 
 \begin{co}\label{graph+}
 Let $c\in \R$ and $u_c : \T^1 \to \R$ be a weak KAM solution at cohomology $c$. Then $T^{c+} u_c$ is a $C^1$ function (where $T^{c+} $ is the positive Lax--Oleinik semigroup given by Definition \ref{dfT+}, associated to $S^c$) and $\G\big(c + (T^{c+} u_c)'\big) = f^{-1} \big( \PG ( c+u'_c)\big) $.
 \end{co}
 
 \begin{proof}
 We start by proving that  $f^{-1} \big( \PG ( c+u'_c)\big) $ is the graph of a continuous function. By Proposition \ref{pgGraph}, there exists a Lipschitz embedding $\gamma : \T^1 \to \a$ such that $\gamma(\T^1) =  \PG ( c+u'_c)$. Denote by $\tilde \gamma : \R \to \R \times\R $ a lift of $\gamma $ and  set $\tilde \gamma = (\tilde \gamma_1,\tilde \gamma_2)$ the coordinates of $\tilde \gamma$. As $u'_c$ is semiconcave, up to reversing the time parametrization of $\gamma$ we may assume that $\tilde\gamma_1$ is non--decreasing and it follows that $\tilde \gamma_2$ is decreasing on intervals where $\tilde \gamma_1$ is constant. We now define for all $t\in \R$, $\big(\Gamma_1(t), \Gamma_2(t)\big) = f^{-1}\big(\tilde \gamma_1(t),\tilde \gamma_2(t)\big)$. Let us establish that $\Gamma_1$ is increasing that will imply our point.
 
 Let $t_1<t_2$.
 \begin{itemize}
 \item Assume for a start that $\tilde\gamma_1(t_1) =\tilde\gamma_1(t_2)$. It follows that   $\tilde\gamma_2(t_1) >\tilde\gamma_2(t_2)$ and by the twist condition, we deduce that $\Gamma_1(t_1)<\Gamma_1(t_2)$.
 \item For the remaining case, $\tilde\gamma_1(t_1) <\tilde\gamma_1(t_2)$ , define now $T_1 = \max\big(t\geqslant t_1, \ \tilde\gamma_1(t_1)=\tilde\gamma_1(t)\big)$ and $T_2 = \min\big(t\leqslant t_2, \ \tilde\gamma_1(t_2)=\tilde\gamma_1(t)\big)$. It follows that $t_1 \leqslant T_1<T_2\leqslant t_2$ and that
  \begin{equation}\label{deriv+-}
 \tilde \gamma_2(T_1) = c+u'_{c+}\big(\tilde\gamma_1(t_1)\big), \quad \tilde \gamma_2(T_2) = c+u'_{c-}\big(\tilde\gamma_1(t_2)\big).
 \end{equation}
 
 As  $\tilde\gamma_1(t_1)=\tilde\gamma_1(T_1)$ and $\tilde\gamma_1(t_2)=\tilde\gamma_1(T_2)$ we obtain from the first case that $\Gamma_1(t_1)\leqslant \Gamma_1(T_1)$ and $\Gamma_1(T_2)\leqslant \Gamma_1(t_2)$.
 
 Then by \eqref{deriv+-} and Theorem \ref{twistKAM}, we deduce that 
 for $i\in \{1,2\}$,
 $$  \widetilde T^c \tilde u_c\big(\tilde\gamma_1(t_i)\big)=\tilde u_c\big(\Gamma_1(T_i)\big)+\ww S\big(\Gamma_1(T_i), \tilde\gamma_1(t_i) \big)+c\big(\Gamma_1(T_i)- \tilde\gamma_1(t_i) \big).$$
 By Lemma \ref{orderWKAM} we obtain that $\Gamma_1(T_1)<\Gamma_1(T_2)$ and finally
 $$\Gamma_1(t_1)\leqslant \Gamma_1(T_1)<\Gamma_1(T_2)\leqslant \Gamma_1(t_2).$$
 \end{itemize}
 
 We now turn to the interpretation in terms of positive Lax--Oleinik semigroup. To this end, we use  the analogues for $T^{c+}$ of the results established for $T^c$, without proofs. Note that $T^{c+}u_c$ is a semiconvex function.  Let $\theta \in \T^1$ and $\Theta \in \T^1$ such that $T^{c+}u_c ( \theta) = u_c(\Theta) - S^c ( \theta,\Theta)$. Then $S^c$ is differentiable at 
$( \theta,\Theta)$. By setting $R = c+ \partial_2 S^c ( \theta,\Theta)$ and $r = c- \partial_1 S^c ( \theta,\Theta)$, 
\begin{itemize}
\item $f(\theta,r) = (\Theta,R)$,
\item $(\Theta,R) \in \PG ( c+u'_c)$,
\item $r-c \in \partial^- T^{c+}u_c(\theta)$.
\end{itemize}
As we have established that $f^{-1} \big( \PG ( c+u'_c)\big) $ is the graph of a continuous function, there is a unique $(\Theta,R) \in \PG ( c+u'_c)$ such that $\pi_1\circ f^{-1}(\Theta,R) = \theta$. It follows that $\Theta$ \big(realizing equality in the definition of $T^{c+}u_c ( \theta)$\big) is unique and that $ T^{c+}u_c$ is derivable at $\theta$. As  this holds for all $\theta$ and by semiconvexity, $ T^{c+}u_c$ is indeed $C^1$. Finally, as $\G\big(c + (T^{c+} u_c)'\big) \subset f^{-1} \big( \PG ( c+u'_c)\big) $ and since both sets are graphs, they are equal.
 
 \end{proof}

 \begin{df}\rm
 Given $c\in \R$ we will denote  by
 \begin{itemize}
  \item $\AA_c \subset \T^1$ the projected Aubry set,
  \item $\widehat \AA_c \subset \T^1 \times \T^1$ the $2$-Aubry set,
  \item $\widetilde \AA_c \subset (\T^1)^\Z$ the Aubry set,
  \end{itemize}
  all three associated to the cost  $\widetilde S_c$. 
  
  We similarly denote by $\AA^*_c\subset \a = \T^1 \times \R$ the set given by  Proposition \ref{aubry-cot} associated to the cost  $\widetilde S_c$ that is also refered to as Aubry set.
 
 We will denote by   $\AAA_c$ and $\AAA^*_c$ the lifts of $\AA_c$ and $\AA_c^*$ to respectively $\R$ and $\R \times \R$ that we will also refer to as projected Aubry set and Aubry set.
 \end{df}
 
 As sequences in $\widetilde \AA_c$ are minimizing for $S^c$ we may apply Proposition \ref{prSc} to obtain
 \begin{pr}\label{PrAc}
 Let $(\theta_i)_{i\in \Z} \in \widetilde \AA_c$ and $\tilde \theta_0\in \R$ a lift of $\theta_0$. Then there exists a unique $(\tilde\theta_i)_{i\in \Z}\in \R^\Z$ such that 
 $$\forall k<k', \quad \sum_{\ell = k}^{k'-1} S^c(\theta_{\ell},\theta_{\ell+1} )= \sum_{\ell = k}^{k'-1} \ww S(\tilde\theta_{\ell},\tilde\theta_{\ell+1}) +c(\tilde\theta_k-\tilde\theta_{k'}).$$
More precisely, if $k\in \Z$, $\tilde\theta_k = \pi_1\circ \tilde f^k(\tilde\theta_0,r_0)$ where 
$$r_0 = -\partial_1  \ww S(\tilde\theta_{0},\tilde\theta_{1})  = \partial_2  \ww S(\tilde\theta_{-1},\tilde\theta_{0})= c  -\partial_1   S^c(\theta_{0},\theta_{1})  =c+ \partial_2   S^c(\theta_{-1},\theta_{0}).$$

 \end{pr}

\begin{df}\rm\label{Aubrylift}
We denote by $\ww\AAA_c \subset \R^\Z$ the set of sequences  $(\tilde\theta_i)_{i\in \Z}\in \R^\Z$ given by the previous proposition.

We denote by $\wh\AAA_c \subset \R^2$ the set of pairs $(\tilde\theta_0,\tilde\theta_1)$ for  $(\tilde\theta_i)_{i\in \Z}\in \ww\AAA_c$. 
\end{df}

 \begin{rem}\rm\label{bilipschitz}
\hspace{2em}
 \begin{itemize}
 \item All canonical projections from respectively  $\widehat \AA_c $, $\widetilde \AA_c$ and $\AA^*_c$ to $\AA_c$ are bi--Lipschitz homeomorphisms.
  \item All canonical projections from respectively  $\widehat \AAA_c $, $\widetilde \AAA_c$ and $\AAA^*_c$ to $\AAA_c$ are bi--Lipschitz homeomorphisms.

 \item The sets  $\AAA_c$ and $\AAA^*_c$ are respectively invariant by horizontal translations $\tilde \theta \mapsto \tilde \theta+1$ and $(\tilde \theta ,r) \mapsto ( \tilde \theta+1,r)$.
 \item The set  $( 0,c)+ \AA^*_c = \{(  \theta , c+r) , \ \ (  \theta , r) \in  \AA^*_c \}$ is invariant by $f$ and   the set $( 0,c)+ \AAA^*_c = \{( \tilde \theta , c+r) , \ \ ( \tilde \theta , r) \in  \AAA^*_c \}$ by $\tilde f$. 
 \end{itemize}

 \end{rem}

This last point is proved using that elements in the projected Aubry sets come in minimizing sequences that calibrate weak KAM solutions. Hence it is possible to apply Remark \ref{remSc} and Theorem \ref{twistKAM}.

 We derive the following consequence (that will be improved later in Corollary \ref{rotnummin1}):
 
 \begin{co}\label{rotationnumber}
 Let $c\in \R$. There exists $\rho(c) \in \R$ such that for all $ u_c : \T^1 \to \R $ weak KAM solution at cohomology $c$, if $(\tilde \theta_k)_{k\leqslant 0} \in \R^{\Z_-}$ calibrates $\tilde u_c$ then 
 $$\forall k\leqslant 0, \quad |\tilde\theta_k - \tilde \theta_0 - k\rho(c) | <2.$$
 \end{co}

\begin{proof}
Let $(x_k)_{k\in \Z} \in \widetilde \AA_c$ that hence calibrates $u_c$. Let $\tilde x_0\in \R$ such that $\tilde x_0 \leqslant \tilde \theta_0 < \tilde x_0 +1$. Finally let $(\tilde x_k)_{k\in \Z}$ be the only sequence that projects on $(x_k)_{k\in \Z} $ and calibrates $\tilde u_c$. Thanks to Theorem \ref{AM}, there exists  a real number number $\rho$ that is independent on $u_c$ such that 
 $$\forall i\in \Z, \quad |\tilde x_i - \tilde x_0 - i\rho|<1.$$
 Moreover, by periodicity, the sequences $(\tilde x_k+1)_{k\in \Z}$ also calibrates $\tilde u_c$.
 
 If $\tilde \theta_0 = \tilde x_0$, then 
$\tilde \theta_k = \tilde x_k$ for all $k\leqslant 0$. Indeed recall that $\tilde u'_c ( \tilde x_0)$ exists and then  a calibrating sequence starting at $\tilde x_0$ is unique (see Theorem \ref{twistKAM}).

In the remaining case, by applying Lemma \ref{orderWKAM} and a straightforward induction, one finds that 
$$\forall k\leqslant 0, \quad \tilde x_k < \tilde \theta_k < \tilde x_k +1,$$
and the result follows.

As all sequences $(x_k')_{k\in \Z} \in \widetilde \AA_c$ calibrate $u_c$, it follows that the initial $\rho $ does not depend on the  initial choice of   $(x_k)_{k\in \Z} \in \widetilde \AA_c$ (by the previous argument). Finally, as  $(x_k)_{k\in \Z} \in \widetilde \AA_c$ calibrates any other weak KAM solution at cohomology $c$, the real number $\rho$ does only depend on $c$, independently of the initially chosen weak KAM solution.
\end{proof}

\section{Mather measures}

Recall that $\wh\PP$ is the set of closed measures on $\T^1\times \T^1$ (Definition \ref{closed}). Then if $c\in \R$, Theorem \ref{minimizing} stipulates that 
$$-\alpha(c) =  \min_{\mu\in \wh\PP} \int_{\T^1\times \T^1} S_c(\theta,\theta')\  \dd \mu(\theta,\theta').$$
Moreover, minimizing Mather measures are those $\mu\in\wh \PP$ whose support is included in $\wh\AA_c$. We will denote by $ \wh\PP_c$ the set of such Mather measures at cohomology $c$. We aim at obtaining analogous notions involving a cost that does not depend on $c$. 

If $\theta \in \T^1$ and $\tilde \theta' \in \R$  denote by $\theta + \tilde \theta' = \pi(\tilde \theta + \tilde \theta' ) \in \T^1$ where $\tilde \theta\in \R$ is any lift of $\theta$. Of course, $\theta + \tilde \theta' $ does not depend on the choice of $\tilde \theta$. 

\begin{df}\rm \label{closedA}
\hspace{2em}
\begin{itemize}
\item Let $\tau : \a \to \T^1$ be defined by $(\theta ,r) \mapsto \theta+r$.
\item We say a Borel probability measure $\mu$ on $\a = \T^1 \times \R$ is closed if it has finite first moment,
  $\int_\a |r| \ \dd \mu(\theta,r) <+\infty$ and if $\tau_* \mu = \pi_{1*} \mu$ meaning that for any continuous function $g : \T^1 \to \R$,
  $$  \int_\a g(\theta+r) \ \dd \mu(\theta,r) = \int_\a g(\theta) \ \dd \mu(\theta,r).$$
    The set of closed probability measures on $\a$ is denoted by $\PP^*$.
\item Given a closed probability measure $\mu \in \PP^*$ we define its rotation number, $\rho(\mu) = \int_\a r\  \dd \mu(\theta,r)$.
\item We define $S^* : \a\to \R$ by $S^* (\theta,r) = \ww S(\tilde \theta, \tilde\theta+r)$ where $\tilde \theta\in \R$ is any lift of $\theta$. Of course, the result does not depend on the choice of $\tilde \theta$ by Proposition \ref{genfun}.
\end{itemize}

\end{df}

\begin{pr}\label{MinMeasAn}
The following holds
\begin{equation}\label{minMeasA}
-\alpha(c) =  \min_{\mu^*\in \PP^*} \int_{\a} \big[ S^*(\theta,r)-cr \big] \ \dd \mu^*(\theta,r).
\end{equation}
Moreover, a closed measure is minimizing if and only if it is supported on the set of pairs
$(\theta,\delta)\in \a $ such that $\theta\in \AA_c$ and 
$$ \delta =\pi_1\circ \tilde f\left( (\pi_{1|\AAA_c^*})^{-1}(\tilde\theta)+(0,c)\right) - \tilde\theta = \pi_1\circ \tilde f(\tilde\theta, c+r_\theta)-\tilde\theta ,$$
 where $\tilde\theta \in \R$ is any lift  of $\theta$ and $r_\theta\in \R$ is the unique  real number such that $(\tilde\theta, r_\theta) \in \AAA_c^*$.

\end{pr}

\begin{proof}
Let $ u : \T^1 \to \R$ be a continuous subsolution for $S^c$ that is strict outside of $\wh \AA_c$ (Theorem \ref{th-strict}), meaning that 
$$\forall (\theta,\theta')\in \T^1 \times \T^1,\quad u(\theta')-u(\theta)\leqslant S^c (\theta,\theta') +\alpha(c) ,$$
with strict inequality as soon as $(\theta,\theta')\notin\wh \AA_c$.
By definition of $S^c$ it follows that if $\tilde u$ is a lift of $u$,
$$\forall (\tilde\theta,\tilde\theta')\in \R \times \R,\quad \tilde u(\tilde\theta')-\tilde u(\tilde\theta)\leqslant \ww S (\tilde\theta,\tilde\theta')+c(\tilde\theta-\tilde\theta') +\alpha(c) .$$  

This can in turn be written as follows:

$$\forall (\theta,\delta)\in \T^1 \times \R,\quad  u(\theta+\delta)-u(\theta)\leqslant  S^* (\theta,\delta)-c\delta +\alpha(c) .$$  
Integrating the previous inequalities against a closed measure $\mu^*\in \PP^*$ yields
$$0=  \int_{\a}u(\theta+\delta)\ \dd \mu^*(\theta,\delta)- \int_{\a} u(\theta) \ \dd \mu^*(\theta,\delta)\leqslant  \int_{\a} \big[ S^*(\theta,\delta)-c\delta+\alpha(c) \big] \ \dd \mu^*(\theta,\delta)$$
and $-\alpha(c) \leqslant  \int_{\a} \big[ S^*(\theta,\delta)-c\delta \big] \dd \mu^*(\theta,\delta)$.

Moreover, equality holds if and only if $\mu^* $ is supported  on pairs $(\theta,\delta)$ such that $u(\theta+\delta)-u(\theta) =  S^* (\theta,\delta)-c\delta +\alpha(c) $. As
$$u(\theta+\delta)-u(\theta)\leqslant S^c(\theta, \theta+\delta)+ \alpha(c) \leqslant  S^* (\theta,\delta)-c\delta +\alpha(c), $$
we deduce that for such $(\theta,\delta) \in \mathrm{supp}(\mu^*)$, 
$$u(\theta+\delta)-u(\theta)= S^c(\theta, \theta+\delta)+ \alpha(c).$$
In turn, we deduce that $(\theta, \theta+\delta)\in \wh \AA_c$, in particular $S^c$ is differentiable at $(\theta, \theta+\delta)$. Then
$$S^c(\theta, \theta+\delta)+ \alpha(c) = S^* (\theta,\delta)-c\delta +\alpha(c) =\ww S(\tilde \theta, \tilde \theta+\delta) -c\delta +\alpha(c),$$
where $\tilde \theta$ is a lift of $\theta$. As weak KAM solutions are calibrated by points of the $2$-Aubry set and derivable on the Aubry set, Theorem \ref{twistKAM} (see also Proposition \ref{PrAc}) gives that 
$$\tilde\theta +r = \pi_1\circ \tilde f (\tilde \theta, -\partial_1 \ww S\big(\tilde \theta, \tilde \theta+r) \big) $$
and 
$$\big(\tilde \theta, -c-\partial_1 \ww S(\tilde \theta, \tilde \theta+r)\big) = \big(\tilde \theta, -\partial_1  S^c( \theta,  \theta+r)\big) \in \AAA_c^*.
$$ 
With the notation of the current Proposition, this is rewritten $r_\theta =  -\partial_1  S^c( \theta,  \theta+r)$.

It remains to prove that such a closed measure realizing equality exists. To that aim let us start from a minimizing Mather measure $\hat\mu_c$ such that $-\alpha(c) = \int_{\T^1\times \T^1} S^c(\theta,\theta')\  \dd \hat\mu_c(\theta,\theta')$, that is henceforth supported on $\wh \AA_c$ by Theorem \ref{minimizing}.
If $\tilde\theta \in \AAA_c$ and  $\ww\Theta\in \AAA_c$  is the unique element such that $(\tilde\theta,\ww\Theta) \in \wh\AAA_c$, we denote $R(\tilde\theta)= \ww\Theta - \tilde\theta$. By periodicity, it is immediate that $R(\tilde\theta)$ only depends on $\theta = \pi(\tilde\theta) \in \AA_c$ hence we will also refer to it as $R(\theta)$.

 Recall that $\tilde\theta \mapsto \ww\Theta$ is the biLipschitz homoeomorphism $\pi_2\circ (\pi_{1|\wh\AAA_c})^{-1} : \AAA_c \to \AAA_c$, as stated in Remark \ref{bilipschitz}, hence $\theta \mapsto R(\theta)$ is Lipschitz.
 
 We now define a Borel probability measure $\mu^*_c$ on $\T^1\times \R$ by setting, for any continuous function $G : \T^1\times \R \to \R$,
 \begin{equation}\label{mesassociated}
 \int_{\T^1\times \R  } G(\theta,r) \ \dd \mu^*_c(\theta,r) = \int_{\wh\AA_c } G\big(\theta,R(\theta)\big) \ \dd \hat\mu_c(\theta,\theta').
 \end{equation}
 As the $R(\theta)$ are uniformly bounded, the measure $\mu^*_c$ is compactly supported hence $\int_\a |r| \ \dd \mu^*_c(\theta,r) <+\infty$. Let us verify it is closed: let $g: \T^1 \to \R$ be a continuous function, then
 \begin{multline*}
 \int_{\T^1\times \R  } \big(g(\theta+r)-g(\theta)\big) \ \dd \mu^*_c(\theta,r) \\
 = \int_{\wh\AA_c } \big(g\big(\theta+R(\theta)\big)-g(\theta)\big) \ \dd \hat\mu_c(\theta,\theta') 
 =\int_{\wh\AA_c } \big(g(\theta')-g(\theta)\big) \ \dd \hat\mu_c(\theta,\theta') \\
 = \int_{\T^1 \times \T^1} \big(g(\theta')-g(\theta)\big) \ \dd \hat\mu_c(\theta,\theta')=0,
\end{multline*}
  where was used that if $(\theta,\theta') \in \wh\AA_c$ then $\theta' = \theta+R(\theta)$ and that $\hat\mu_c\in \wh\PP$ is closed. Hence $\mu^*_c \in \PP^*$ is a closed Borel probability measure on $\a$.
  
  To conclude, by definition, $\mu^*$ is defined on pairs $\big(\theta,R(\theta)\big)$ where $\theta \in \AA_c$ and 
  $$R(\theta) = \pi_1\circ \tilde f\left( (\pi_{1|\AAA_c^*})^{-1}(\tilde\theta)+(0,c)\right) $$
  hence it verifies $-\alpha(c) =  \int_{\a} \big[ S^*(\theta,r)-cr \big] \ \dd \mu_c^*(\theta,r)$ by the first part of the proof.
\end{proof}

\begin{rem}\rm
\begin{itemize}
\item We will also call minimizing Mather measures (at cohomology $c$) on $\a$ closed measures $\mu^*_c$ verifying  $-\alpha(c) =  \int_{\a} \big[ S^*(\theta,r)-cr \big] \ \dd \mu_c^*(\theta,r)$.
\item As seen in the previous proof, if $\hat \mu_c$ is a minimizing Mather measure on $\T^1 \times \T^1$ (hence supported on $\wh \AA_c$) then there is a closed Mather measure $\mu^*_c$ on $\a$ that is naturally associated \big(see \eqref{mesassociated}\big).
\item Reciprocally, if $\mu^*_c$ is a closed Mather measure on $\a$, then it is verified the same way that $\hat\mu_c = \big((\pi_{|\wh\AA_c})^{-1}\circ \pi_1\big)^* \mu_c^*$ is a closed Mather measure on $\T^1 \times \T^1$.
\item The two mappings $\hat\mu_c \mapsto \mu^*_c$ and $\mu^*_c \mapsto \hat\mu_c$ are inverses of one another. 
\end{itemize}
\end{rem}

We now relate two notions of rotation number, hence clarifying the terminology.

\begin{pr}\label{RotNum}
Let $c\in \R $ be a cohomology class and $\mu^*_c$ be a minimizing Mather measure on $\a$. Then $\rho(\mu^*_c) = \rho(c)$ where the first rotation number is provided by Definition \ref{closedA} and the second by Corollary \ref{rotationnumber}.
\end{pr}

\begin{proof}
By definition, $\rho(\mu^*_c) = \int_{\a} r   \ \dd \mu^*_c(\theta,r)$. By Proposition \ref{MinMeasAn} and using notations therein,
$\rho(\mu^*_c) = \int_{\AA_c\times \R} R(\theta)   \ \dd \mu^*_c(\theta,r) 
$. As $\mu^*_c$ is closed and supported on pairs of the form $\big(\theta, R(\theta)\big)$,  we derive that 
$$\rho(\mu^*_c) = \int_{\AA_c\times \R} R\big(\theta+R(\theta)\big)   \ \dd \mu^*_c(\theta,r)  =  \int_{\AA_c\times \R} R_2(\theta)   \ \dd \mu^*_c(\theta,r).$$
 And by induction it follows that for all positive integer $n>0$, 

\begin{equation}\label{ergoeq}
\rho(\mu^*_c)  =  \int_{\AA_c\times \R} R_n(\theta)   \ \dd \mu^*_c(\theta,r) =  \int_{\AA_c\times \R} \frac{1}{n}\sum_{k=0}^{n-1} R_k(\theta)   \ \dd \mu^*_c(\theta,r),
\end{equation}
 where $R_n : \AA_c \to \R$ verifies the induction relation $R_{n+1}(\theta) =R_n\big(\theta + R(\theta)\big) $ and $R_1 = R$.

Let $\theta_0\in \AA_c$, $\tilde\theta_0$ a lift and $(\tilde\theta_k)_{k\in \Z} \in \ww\AAA_c$ the associated sequence given by Definition \ref{Aubrylift} and $(\theta_k)_{k\in \Z} \in \ww\AA_c$ the sequence of projections. For all $k\in \Z$ it then holds that $R(\tilde\theta_k) = \tilde \theta_{k+1}-\tilde\theta_k$. One then readily verifies that $R_n(\theta_0) =R_n(\tilde\theta_0)=\tilde\theta_{n+1}-\tilde\theta_n  $ so that 
$\sum\limits_{k=0}^{n-1} R_k(\theta_0)= \tilde\theta_n - \tilde\theta_0$. By Theorem \ref{AM} and Corollary \ref{rotationnumber} it is infered that $\frac{1}{n}\sum\limits_{k=0}^{n-1} R_k(\theta)$ uniformly converges to $\rho(c)$ as $n\to +\infty$. It is finally deduced from \eqref{ergo} that $\rho(\mu^*_c) = \rho(c)$ as desired.
\end{proof}

\begin{rem}\rm\label{remPoinc}
The idea behind the previous proof is that all the dynamics of $f$ restricted to the Aubry set $\AA^*_c$ can be translated  by projecting to a dynamics on $\AA_c$. This dynamics is then the restriction of a circle orientation preserving diffeomorphism (by generalizations of Theorem \ref{AM}, see \cite{Bang}). Then,  the push forward of any minimizing closed measures on $\T^1$ is invariant by this circle diffeomorphism. Hence the result turns out to be an emanation of Birkhoff's ergodic theorem with this point of view.

Another fact that is apparent from the previous proofs is that if $(\theta,r)\in\A$ is in the support of a minimizing measure $\mu^*_c$ as above and if $\tilde \theta \in \R$ is a lift of $\theta$, then there exists a unique minimizing sequence $(\tilde\theta_k)_{k\in \Z} \in \ww\AAA_c $ such that $\tilde \theta_0 = \tilde \theta$ and $\tilde \theta +r = \tilde \theta_1$. Moreover this sequence only depends on the measure $\mu^*_c$ (meaning it can be recovered without knowing $c$).
\end{rem}

Now introducing Mather's $\beta$ function:

\begin{df}\label{beta}
Mather's $\beta$ function  is defined by 
$$\forall \rho_0\in \R, \quad \beta(\rho_0) =- \inf_{\substack{\mu^* \in \PP^* \\ \rho(\mu^*) = \rho_0}} \int_{ \a} S^*(\theta,r) \ \dd\mu^*(\theta,r)= \sup_{\substack{\mu^* \in \PP^* \\ \rho(\mu^*) = \rho_0}} -\int_{ \a} S^*(\theta,r) \ \dd\mu^*(\theta,r) .$$
\end{df}

Note that for all $\rho_0\in \R$ the set of closed measures on $\a$ with rotation number $\rho_0$ is not empty. One may for instance consider the pull back on the circle $\T^1 \times \{\rho_0\}$ of the Lebesgue measure on the circle $\T^1$. Moreover, it is immediate that the rotation number function $\rho : \PP^* \to \R$ is linear. Arguing as in the first part of Proposition \ref{alphaconv} immediately yields:

\begin{pr}
The function $\beta$ takes values in $\R \cup \{+\infty\}$ and is convex.
\end{pr}

Actually this can be improved by the following Theorem of Mather (\cite{Mather1}):

\begin{Th}\label{beta+}
The function $\beta$ is finite--valued, convex and superlinear. Moreover, $\alpha$ and $\beta$ are convex dual one another meaning that 
$$\forall \rho_0 \in \R, \quad \beta(\rho_0)  = \max_{c\in \R} \rho_0 c - \alpha(c),$$
$$\forall c_0 \in \R, \quad \alpha(c_0)  = \max_{\rho\in \R} \rho c_0 - \beta(\rho).$$

\end{Th}

\begin{proof}
We start from the measure characterization of $\alpha$ (given by \eqref{minMeasA} in Proposition \ref{MinMeasAn})
 that we rewrite in separating measures according to their rotation number:
\begin{multline*}
\alpha(c) =  \max_{\mu^*\in \wh\PP^*} -  \int_{\a} \big[ S^*(\theta,r)-cr \big] \ \dd \mu^*(\theta,r) \\
= \max_{\varrho\in \R} \sup_{\substack{ \mu^*\in \wh\PP^* \\ \rho(\mu^*) = \varrho}} c\varrho -  \int_{ \a} S^*(\theta,r) \ \dd\mu^*(\theta,r) =  \max_{\varrho\in \R}  c\varrho - \beta(\varrho).
\end{multline*}
We recognize here the Fenchel dual of $\beta$: $\alpha = \beta^*$  (see \cite{evans,Rock}). By basic properties of the Fenchel dual, as $\beta^*$ is everywhere finite, it follows that $\beta$ is superlinear. Finally, as $\beta $ is  also convex, then $\beta = \beta^{**} = \alpha^*$ as was to be proved. 
\end{proof}

We deduce from properties of the Fenchel transform (\cite{evans, Rock}), together with Corollary \ref{rotationnumber} and Proposition \ref{RotNum}:

\begin{Th}\label{alphaC1}
The following relations, given $c_0$ and $\rho_0$ real numbers, are equivalent:
\begin{itemize}
\item $c_0 \in \partial^- \beta ( \rho_0)$; 
\item $\beta(\rho_0) + \alpha(c_0)=  c_0\rho_0 $;  
\item $\rho_0 \in \partial^- \alpha(c_0)$;
\item there exists a minimizing measure $\mu^*_{c_0}\in \PP^*$ at cohomology $c_0$ with rotation number $\rho_0$,
\item all minimizing measures $\mu^*_{c_0}\in \PP^*$ at cohomology $c_0$ have rotation number $\rho_0$,
\item there exists a minimizing infinite chain, calibrating a weak KAM solution at cohomology $c_0$ that has rotation number 
$\rho_0$,
\item all minimizing infinite chains, calibrating a weak KAM solution at cohomology $c_0$ have rotation number 
$\rho_0$.
\end{itemize}

As for all $c_0\in \R$, the subdifferential $\partial^-(c_0)$ is a singleton, we deduce that the function $\alpha$ is $C^1$ and the function $\beta$ is strictly convex. The function $c\mapsto \rho(c) = \alpha'(c)$ is continuous surjective and non--decreasing.

\end{Th}
 These last results were first published by Mather in \cite{MatherMeasure} where he attributes them to Aubry.

\section{Order properties of weak KAM solutions}

Results in this section appeared in Arnaud--Zavidovique's works \cite{AZ,AZ3}. Some similar statements also can be found in Zhang's work \cite{Zhang} and related results in the setting of the torus $\T^2$ in \cite{CX}. Interestingly, non variational versions of some among these results date back to Katznelson and Ornstein \cite{KO}.

Our first result is that weak KAM solutions and their full pseudographs (Definition \ref{pg}) are vertically ordered with respect to the rotation number.

 \begin{pr}\label{order}
 Let $c<c'$ be two cohomology classes such that $\rho(c)<\rho(c')$. Let $u_c : \T^1 \to \R$ be a weak KAM solution at cohomology $c$ and  $u_{c'} : \T^1 \to \R$ be a weak KAM solution at cohomology $c'$. If $\theta \in \T^1$ and $r,r'$ are such that $(\theta,r) \in \PG(c+u'_c)$ and $(\theta,r') \in \PG(c'+u'_{c'})$, then $r<r'$.
 
 In particular, if $u_c$ and $u_{c'}$ are derivable at $\theta$ then $c+u'_c(\theta) < c'+ u'_{c'}(\theta)$ and the function $\tilde \theta \mapsto (\tilde u_{c'} - \tilde u_c)(\tilde \theta) + (c'-c)\tilde \theta$ is increasing.
 \end{pr}
 
 \begin{proof}
 As $u_c$ and $u_{c'}$ are semiconcave, it is enough to prove the following about right and left derivatives:  $c'+u'_{c' +}(\theta) > c+u'_{c -}(\theta)$. Set $r_0 = c+u'_{c -}(\theta)$ and $r'_0 =  c'+u'_{c' +}(\theta) $. Let $\tilde \theta_0 = \tilde \theta_0' \in \R$ be a lift of $\theta$. For integers $n<0$, we define $(\tilde \theta_n , r_n) = \tilde f^n (\tilde\theta_0, r_0)$ and  $(\tilde \theta'_n , r'_n) = \tilde f^n (\tilde\theta'_0, r'_0)$. By Theorem \ref{twistKAM} the sequence $(\tilde\theta_n)_{n\leqslant 0}$ (resp. $(\tilde\theta'_n)_{n\leqslant 0}$) calibrate $\tilde u_c$ with cohomology $c$ (resp. $\tilde u_{c'}$ with cohomology $c'$). Hence both sequences are minimizing and by Corollary \ref{rotationnumber} verify $\lim\limits_{n\to -\infty} \frac{\tilde \theta_n}{n} = \rho(c)$ and $\lim\limits_{n\to -\infty} \frac{\tilde \theta'_n}{n} = \rho(c')$. It follows, as $\rho(c) \neq \rho(c')$, that $r_0 \neq r_0'$. 
 
 We now argue by contradiction and assume that $r_0'<r_0$. As  $\tilde f$ twists verticals to the right, $\tilde f^{-1} $ twists verticals to the left implying that $\tilde \theta_{-1} <\tilde \theta'_{-1}$. As $\rho(c')>\rho(c)$ it follows that for large $n<0$,  $\frac{\tilde \theta'_n}{n} > \frac{\tilde \theta_n}{n} $ and then, for large $n<0$, $\theta'_n <\theta_n$. We deduce that the sequences $(\tilde\theta_n)_{n\leqslant 0}$ and $(\tilde\theta'_n)_{n\leqslant 0}$
cross at least twice, once at $0$ and then once at some negative integer or between two consecutive ones. This contradicts Proposition \ref{noncrossing}.

 \end{proof}
 
 With a similar flavor, here is a result on the actions of Lax-Oleinik semigroups (Definition \ref{LOtwist}):
 
 \begin{lm}\label{Tincr}
Let  $c_1<c_2$  be two real numbers. Let  $v_1,v_2 : \T^1\to \R$ be continuous   functions.

If the function $\theta \mapsto (\tilde v_2-\tilde v_1)(\tilde\theta) + (c_2-c_1)\tilde\theta$ is non--decreasing, then so is the function $\tilde\theta \mapsto (\widetilde{ T}^{c_2}\tilde v_2-\widetilde T^{c_1}\tilde v_1)(\tilde\theta) + (c_2-c_1)\tilde\theta$.
\end{lm}

\begin{proof}
Let $\tilde\theta < \tilde\theta'$ be two real numbers. By definition of the operators $T^{c_i}$ and  $\ww T^{c_i}$ there exist $\tilde\theta_2'$ and $\tilde\theta_1$ such that
$$\widetilde T^{c_2}\tilde v_2(\tilde\theta') = \tilde v_2(\tilde\theta'_2) + \ww S(\tilde\theta'_2,\tilde\theta') + c_2(\tilde\theta'_2-\tilde\theta'),$$
$$\widetilde T^{c_1}\tilde v_1(\tilde\theta) = \tilde v_1(\tilde\theta_1) +\ww S(\tilde\theta_1,\tilde\theta) + c_1(\tilde\theta_1-\tilde\theta).$$
There are two cases to consider:
\begin{itemize}
\item if $\tilde\theta_2' <\tilde \theta_1$ we use Aubry \& Le Daeron's fundamental Lemma \ref{aubryfl} to obtain
\begin{align*}
\widetilde T^{c_2}\tilde v_2(\tilde\theta') +\widetilde T^{c_1}\tilde v_1(\tilde\theta) &= \tilde v_2(\tilde\theta'_2) + \ww S(\tilde\theta'_2,\tilde\theta') + c_2(\tilde\theta'_2-\tilde\theta')\\
&\quad\quad+ \tilde v_1(\tilde\theta_1) + \ww S(\tilde\theta_1,\tilde\theta) + c_1(\tilde\theta_1-\tilde\theta) \\
&> \tilde v_2(\tilde\theta'_2) +\ww  S(\tilde\theta'_2,\tilde\theta) + c_2(\tilde\theta'_2-\tilde\theta')
\\
&\quad \quad + \tilde v_1(\tilde\theta_1) + \ww S(\tilde\theta_1,\tilde\theta') + c_1(\tilde\theta_1-\tilde\theta) \\
& \geqslant  \widetilde T^{c_2}\tilde v_2(\tilde\theta) +\widetilde T^{c_1}\tilde v_1(\tilde\theta')+ (c_2-c_1) (\tilde\theta-\tilde\theta').
\end{align*}
After rearranging the terms, this reads
$$ \widetilde T^{c_2}\tilde v_2(\tilde\theta')-\widetilde T^{c_1}\tilde v_1(\tilde\theta') +(c_2-c_1)\tilde\theta' >  \widetilde T^{c_2}\tilde v_2(\tilde\theta)-\widetilde T^{c_1}\tilde v_1(\tilde\theta) +(c_2-c_1)\tilde\theta.$$

\item if $\tilde\theta_2' \geqslant\tilde \theta_1$ we use the hypothesis on $\tilde\theta \mapsto (\tilde v_2-\tilde v_1)(\tilde\theta) + (c_2-c_1)\tilde\theta$ to show that 
$ \tilde v_2(\tilde\theta'_2) + \tilde v_1(\tilde\theta_1) \geqslant \tilde v_2(\tilde\theta_1) + \tilde v_1(\tilde\theta'_2)+ (c_2-c_1) (\tilde\theta_1-\tilde\theta'_2) $ and then
\begin{align*}
\widetilde T^{c_2}\tilde v_2(\tilde\theta') +\widetilde T^{c_1}\tilde v_1(\tilde\theta) &= \tilde v_2(\tilde\theta'_2) +\ww S(\tilde\theta'_2,\tilde\theta') + c_2(\tilde\theta'_2-\tilde\theta') \\
&\quad\quad + \tilde v_1(\tilde\theta_1) + \ww S(\tilde\theta_1,\tilde\theta) + c_1(\tilde\theta_1-\tilde\theta) \\
&\geqslant  \tilde v_2(\tilde\theta_1) + \ww S(\tilde\theta'_2,\tilde\theta') + c_2(\tilde\theta_1-\tilde\theta')\\
&\quad \quad + \tilde v_1(\tilde\theta'_2) + \ww S(\tilde\theta_1,\tilde\theta) + c_1(\tilde\theta'_2-\tilde\theta) \\
&\geqslant     \widetilde T^{c_2}\tilde v_2(\tilde\theta) +\widetilde T^{c_1}\tilde v_1(\tilde\theta')+ (c_2-c_1) (\tilde\theta-\tilde\theta').
\end{align*}
As before, this gives the result after rearranging terms.
\end{itemize}

\end{proof}

We now state our main result of the section. The rest will be devoted to providing elements of its proof. We chose to present the  parts which best illustrate discrete weak KAM theory and the first Chapters of this text.

\begin{Th}\label{mainAZ}
There exists a function $u : \T^1\times \R \to \R$ that is locally Lipschitz and that verifies the following properties:
\begin{enumerate}
\item For all $c\in \R$, $u(0,c) = 0$.
\item For all $c\in \R$, the function $u_c = u(\cdot, c) : \T^1 \to \R$ is a weak KAM solution at cohomology $c$.
\item If $c<c'$, the function $\tilde \theta \mapsto (\tilde u_{c'} - \tilde u_c)(\tilde \theta) + (c'-c)\tilde \theta$ is non--decreasing.
\end{enumerate}
If $u:\T^1 \to \R$ is any continuous function verifying  properties 1. and 2. above, then
\begin{enumerate}[a.]
\item The map $c\mapsto \PG (c+u'_c)$ is continuous for the Hausdorff topology.
\item The entire annulus is filled by pseudographs: $\a = \bigcup\limits_{c\in \R} \PG (c+u'_c)$. 
\end{enumerate}
 
\end{Th}

The whole proof of this Theorem is quite long and can be found in \cite{AZ,AZ3}\footnote{In those references a result of Mather is used:  the $\beta$ function is derivable on $\R \setminus \Q$. We will give here a strategy of proof that avoids using this result that we will recover later.}. We  give the main steps and ideas bellow. 

The first point in the Theorem, $u(0,c) =0$, is a normalization condition and is easily enforced as the set of  weak KAM solutions (at any cohomology class) is invariant by addition of constants.

Let us continue by noticing that a function verifying the first three points of the Theorem is automatically locally Lipschitz. Indeed,  weak KAM solutions are locally uniformly in $c$ equiLipschitz in $\theta$ (as the costs $S^c$ are). Moreover, if $c\in \R$, then $\tilde u_c ( 1) = \tilde u_c ( 0)=0$. Then if $\theta \in \T^1$, $0\leqslant \tilde\theta \leqslant  1$ is a lift, and $c<c'$,
\begin{multline*}
0 = (\tilde u_{c'} - \tilde u_c)(0) + (c'-c)\times 0 \\
  \leqslant (\tilde u_{c'} - \tilde u_c)(\tilde \theta) + (c'-c)\tilde \theta \\
\leqslant (\tilde u_{c'} - \tilde u_c)(1) + (c'-c) \times 1 = c'-c.
\end{multline*}
The result follows from 
$$c-c' \leqslant -\tilde \theta (c'-c) \leqslant  ( u_{c'} -  u_c)( \theta) \leqslant (1-\tilde\theta) (c'-c) \leqslant c'-c.$$

Point $a.$ follows from a more general fact, stated in the next Proposition. As we did not find its proof in the literature we provide it. Note that it is valid in any dimension $N$. In this setting, superdifferentials are linear forms on $\R^N$:

\begin{pr}\label{hausdorff-dif}
Let $N>0$ be an integer, $(v_n)_{n\in \N}$ a sequence of (real valued) equi-semiconcave functions defined on $\T^N$. Assume that the sequence $(v_n)_{n\in \N}$ uniformly converges to a function  $v : \T^N \to \R$ (that is semiconcave). Then $\big(\PG(v_n)\big)_{n\in\N} $ converges to $\PG(v)$ for the Hausdorff distance.
\end{pr}

\begin{proof}
By hypothesis there exists a constant $K>0$ such that all $\tilde v_n - K\| \cdot \|^2 : \R^N \to \R$ are strictly concave (where $\| \cdot \|$ denotes the Euclidean norm and $\tilde v_n$ the lift of $v_n$ to $\R^N$). It follows that $\tilde v - K\| \cdot \|^2 : \R^N \to \R$ is also  concave and let us assume it is strictly concave up to taking a slightly larger $K$. Let $O\subset \R^N$ be a relatively compact, convex, open set. We will show that, restricted to $O$, $\big(\PG(\tilde v_{n|O}) \big)_{n\in \N}$ converges to $\PG(u_{|O})$ for the Hausdorff distance, which implies the result.

\begin{itemize}
\item Let $X\in O$ and let $k_n$ be an increasing sequence of integers and  $(x_{k_n},p_{k_n}) \in O\times \R^N$ such that $p_{k_n} \in \partial^+ \tilde v_{k_n} (x_{k_n})$ for all $n\in \N$ and $x_{k_n}\to X$. Assume moreover that $p_{k_n} \to p$ as $n\to +\infty$. We will prove that $p\in \partial^+ \tilde v(X)$. By hypothesis, for all $n\in \N$, $0\in \partial^+ w_{k_n}$ where $w_{k_n} : x\mapsto \tilde v_{k_n}(x) - K\|x-x_{k_n}\|^2 - p_{k_n}(x)$ is defined on $O$. It follows that $x_{k_n}$ is a (strict) maximum point of the (strictly) concave function  $ w_{k_n}$:
$$\forall y\in O, \quad w_{k_n}(y) \leqslant w_{k_n}(x_{k_n}).$$
As $(w_{k_n})_{n\in \N}$ uniformly converges to the function $w : x\mapsto \tilde v (x) -K\|x-X\|^2 -p(x)$ we can fix $y\in O$ and pass to the limit in the previous inequality to find that 
$ w(y) \leqslant w(X)$. As this is true for all $y\in O$, $X$ is a maximum point of $w$ which proves, going back to $\tilde v$ that $p\in \partial^+ \tilde v(X)$.

\item Let now $X \in O$ and $p \in \partial^+ \tilde v (X)$. We now define $\varpi : x\mapsto \tilde v (x) -K\|x-X\|^2 -p(x)$ that has a strict maximum at $X$. Let  $\varpi_n : x\mapsto \tilde v_n (x) -K\|x-X\|^2 -p(x)$  for all integer $n$, then $\varpi_n$ converges to $\varpi$ uniformly on $O$. Let $\varepsilon >0$ small enough such that $B(X,2\varepsilon) \subset O$. Let $\eta>0$ such that  if $\|x-X\| = \varepsilon$ then $\varpi(x)< \varpi(X)-\eta$. Let $n_0$ such that for all $n>n_0$, $\| \varpi_n - \varpi \|_{\infty,O} < \eta /3$.
If $n>n_0$ and $x\in O$ such that $\|x-X\| = \varepsilon$ it follows that 
$$\varpi_n(x) <\varpi(x)+\frac{\eta}{3} < \varpi(X) - \frac{2\eta}{3} < \varpi_n(X) - \frac{\eta}{3}.$$
It follows that $\varpi_n$ admits a local (hence global) maximum in the ball $B(X,\varepsilon)$ denoted $x_n$ and at which we infer that $0\in \partial^+\varpi_n(x_n)$.

Using this argument applied to a decreasing sequence $\varepsilon_k\to 0$ it is easy to construct a sequence $x'_n$ that converges to $X$ defined for $n>n_1$ large enough, and such that $0\in \partial^+\varpi_n(x'_n)$ for all integer $n>n_1$. Going back to the initial functions, we conclude that  $$\forall n>n_1,\quad p+2K(x_n'-X) =p_n\in \partial^+ \tilde v_n(x_n').$$
 Clearly, $p_n \to p$.

\end{itemize}
The result follows from the two previous points.
\end{proof}

We now give elements of the construction of a continuous choice of weak KAM solutions with respect to the cohomology class:

\begin{proof}[Elements of Proof]
The first fundamental step is:
\begin{lm}\label{extr}
Let $\rho_0 \in \R$, we denote by $[a,b] = [\beta'_-(\rho_0) , \beta'_+(\rho_0)] = \rho^{-1}(\{\rho_0\})$ (by Theorem \ref{alphaC1}). 
Then there exists a unique weak KAM solution at cohomology $a$ (resp. $b$), denoted $u_a$ (resp. $u_b$) such that $u_a(0) = 0$ (resp. $u_b(0)=0$).

Let $(c_n)_{n\in \mathbb N}$ be a sequence converging to $a$ (resp. converging to $b$).  Let $(v_n)_{n\in \N}$ be a sequence of functions on $\T^1$ such that $v_n$ is a weak KAM solution at cohomology $c_n$ for all $n$ verifying $v_n(0)=0$. 
Then the sequence $(v_n)_{n\in \N}$ uniformly converges towards  $u_a$ (resp. $u_b$).
\end{lm}

\begin{proof}
We will prove the result for $a$, the rest being similar.  Let us first consider an increasing sequence $(c_n)_{n\in \N}$ that converges to $a$. It follows from the hypothesis that $\rho(c_n ) <\rho_0$ for all $n\in \N$.

Let $(v_n)_{n\in \N}$ be a sequence of functions on $\T^1$ such that $v_n$ is a weak KAM solution at cohomology $c_n$ for all $n$ verifying $v_n(0)=0$. The functions $v_n$ are equiLipschitz (by Proposition \ref{propT} and its proof, since the $S^{c_n}$ are) hence by Ascoli's Theorem, the sequence $(v_n)_{n\in \N}$ is relatively compact.  As all limit points of $(v_n)_{n\in \N}$ are weak KAM solutions at cohomology $a$ \footnote{It can be verified with the use of Remark \ref{remSc} that $c\mapsto S^c$ is continuous.}  that vanish at $0$, it is enough to prove that there is only one such function. To this end, let $u_a : \T^1\to \R$ be a weak KAM solution at cohomology $a$ such that $u_a(0)=0$. Let $(k_n)_{n\in \N}$ be an increasing sequence such that $(v_{k_n})_{n\in \N}$ uniformly converges to a function $w : \T^1 \to \R$ that is hence a weak KAM solution at cohomology $a$ such that $w(0)=0$. Finally, let $\DD\subset \T^1$ be a full measure set such that all $(v_n)_{n\in \N}$, $u_a$, $w$ are derivable on $\DD$ and $\ww \DD \subset \R$ its lift. By Proposition \ref{order}, for all $n \in \N$, $s\in \DD$, $c_n+  v'_n(s) < a+u'_a(s)$. By Proposition \ref{hausdorff-dif}, for all $s\in \DD$, $w'(s) = \lim\limits_{n\to+\infty} v'_{k_n}(s)$ hence $a+w'(s) \leqslant a+u_a'(s)$. Integrating, it follows that
$$\forall x\in [0,1], \quad \tilde u_a(x) = \int_{\ww\DD \cap [0,x] } \tilde u'_a(s)\ \dd s \geqslant \int_{\ww\DD \cap [0,x] } \tilde w'(s) \ \dd s = \tilde w(x).$$
For $x=1$, $\tilde u_a(1) = \tilde w(1) = 0$. It implies that equality $\tilde u_a'(s) = \tilde w'(s)$ holds for almost every $s\in [0,1]$ and then,  integrating as above, that $\tilde u_a(x) = \tilde w(x)$ for all $x\in [0,1]$. We have thus proved that $u_a = w$ and the uniqueness of $u_a$ follows.

If now $(c_n)_{n\in \N}$ is any sequence converging to $a$ and $(v_n)_{n\in \N}$ is a sequence of  weak KAM solutions at cohomology $c_n$ for all $n$ verifying $v_n(0)=0$ then again, the sequence $(v_n)_{n\in \N}$ is relatively compact and as all its limit points are weak KAM solutions at cohomology $a$ vanishing at $0$, by what has been proved, the sequence $(v_n)_{n\in \N}$ converges to $u_a$.

\end{proof}

Keeping the notations of the Lemma, a straightforward corollary of the preceding Lemma and Proposition \ref{order} is, 

\begin{co}\label{monotonyExt}
Let $c\in \R$ and $v_c : \T^1\to \R$ be a weak KAM solution at cohomology $c$. Then
\begin{itemize}
\item $t\mapsto (\tilde u_a - \tilde v_c)(t) + (a-c)t $ is
\begin{itemize}
\item non--increasing if $c>a$,
\item increasing if $c<a$.
\end{itemize}
\item $t\mapsto (\tilde u_b -\tilde  v_c)(t) + (b-c)t $ is
\begin{itemize}
\item non--decreasing if $c<b$,
\item decreasing if $c>b$.
\end{itemize}

\end{itemize}

\end{co}

We denote by $\mathcal I = \{\beta'_-(\rho_0) , \beta'_+(\rho_0), \ \ \rho_0 \in \R\} \subset \R$ then the previous Lemma tells  that any choice of weak KAM solutions $u_c$ at cohomology $c$ for $c\in \R$ such that $u_c(0)=0$ for all $c\in \R$ is automatically continuous on $\mathcal I$. Note that $\mathcal I$ is closed as it can alternatively be defined by 
$\mathcal I = \overline{\{ \beta'(r), \ \ r\in \R\ \mathrm{ such\  that}\  \beta'(r) \ \mathrm{exists} \} }$. 

The rest of the construction consists in focusing on its complement, $\R \setminus \mathcal I$, that is on intervals of the form $\big(\beta'_-(\rho_0) , \beta'_+(\rho_0)\big)=(a,b)$ that are not empty (that is where $\beta$ is not derivable at $\rho_0$\footnote{We will prove later that such $\rho_0$ are necessarily rational.}). If this is the case, we use in a crucial way that the Aubry set $\AA_c$ does not depend on $c\in \big(\beta'_-(\rho_0) , \beta'_+(\rho_0)\big)$ by a result of Massart (\cite[Proposition 6]{Massart}). We hence denote it $\AA$ until the end of the construction. Let us then interpolate linearly between $u_a$ and $u_b$ restricted to $\AA$ and use Proposition \ref{distlike}
 	and Theorem \ref{uniqueness}. If $t\in (0,1)$ let $c_t = at +(1-t) b $ and $v_t = tu_a + (1-t)u_b$. As $\alpha(c_t) = t\alpha(a) +(1-t)\alpha(b)$ by duality between $\alpha$ and $\beta$, it follows that for all $t\in (0,1)$, $v_t$ is a critical subsolution for $S^{c_t}$. Hence by   Proposition \ref{distlike}
 	and Theorem \ref{uniqueness} there exists a unique weak KAM solution at cohomology $c_t$, that we note $u_{c_t}  = u_t$ for short, such that $u_{t|\AA} = v_{t|\AA}$. The continuity of $t\mapsto u_t$ is a consequence of uniqueness and Ascoli's Theorem. 
	
	Finally, item $3$ of the Theorem \ref{mainAZ} follows from successive applications of Lemma \ref{Tincr} and the alternative characterization: $u_t = \lim\limits_{n\to +\infty} (T^{c_t})^n v_t$.
	
	Again, point $ a.$ follows from Proposition \ref{hausdorff-dif}. The last thing to prove in the Theorem  is $b.$ It uses in a crucial way 2 dimensional topology and Jordan's curve Theorem. It is quite technical and we refer the interested reader to \cite{AZ,AZ3}.
\end{proof}

\section{Structure of infinite minimizing chains}
 
 We continue exploring the results of \cite{AZ,AZ3} by studying infinite minimizing chains (indexed by non--positive numbers) of the twist map $\tilde f : \R^2 \to \R^2$. As we will see, all such minimizing chains calibrate a weak KAM solution, hence fall within the scope of the previous sections.  The present results generalize classical Aubry-Mather theory as the latter studies full orbits that are minimizing. After obtaining the results, the authors discovered that many had already appeared in Bangert's \cite{BangR}. In the latter, Bangert studies Busemann functions that are weak KAM solutions on the universal cover $\R$. We adopt the alternative approach to study and use weak KAM solutions on $\T^1$, as previously defined. Consequently most proofs differ slightly from \cite{BangR}.  
 
 We start by a consequence of Theorem \ref{mainAZ} that is a generalization of  Corollary \ref{rotationnumber}.
 
 \begin{pr}\label{rotnummin}
 Let $(\tilde \theta_k)_{k\leqslant 0}\in \R^{\Z_-}$ be a minimizing chain. Then there exists $\rho \in \R$ such that 
  $$\forall k\leqslant 0, \quad |\tilde\theta_k - \tilde \theta_0 - k\rho | <2.$$

 \end{pr}

\begin{proof}
We denote as usual for all $k\leqslant 0$, $r_k =   \partial_2 \ww S ( \tilde \theta_{k-1},\tilde \theta_k)$ in such a way that $(\tilde \theta_k,r_k)_{k\leqslant 0}$ is a piece of orbit of $\tilde f$. By Theorem \ref{mainAZ} there exists $c\in \R$ and a weak KAM solution at cohomology $c$, $u_c : \T^1 \to \R $ such that $(\tilde \theta_0, r_0) \in \PG(c+\tilde u'_c)$. The result is proved with $\rho = \rho(c)$. Let $r_0^\pm = c+\tilde u_{c\pm}'(\tilde \theta_0)$ (so that $r_0^+ \leqslant r_0^-$). Finally, for $k\geqslant 0$, we set $(\tilde\theta_k^\pm, r_k^\pm) = \tilde f^k(\tilde\theta_0,r_0^\pm)$. It follows from Theorem \ref{twistKAM}, and the following Remark \ref{derivsemcon}, that both chains $(\tilde \theta_k^+)_{k\leqslant 0}$ and $(\tilde \theta_k^-)_{k\leqslant 0}$ calibrate $\tilde u_c$ hence are minimizing. 

If $r_0$ coincides with either $r_0^+$ or $r_0^-$ then the result is a particular case of Corollary \ref{rotationnumber}. We now assume otherwise which translates to $r_0^+ < r_0<r_0^-$. As the chains $(\tilde \theta_k)_{k\leqslant 0}$ and $(\tilde \theta_k^+)_{k\leqslant 0}$ (resp. $(\tilde \theta_k)_{k\leqslant 0}$ and $(\tilde \theta_k^-)_{k\leqslant 0}$) cross at $k=0$, they cannot cross anywhere else. We deduce from the twist hypothesis that $\tilde \theta_{-1}^- < \tilde \theta_{-1} < \tilde \theta_{-1}^+$. Hence we conclude from the non--crossing that 
$ \tilde \theta_{k}^- < \tilde \theta_{k} < \tilde \theta_{k}^+$ for all $k<0$. Finally, from Corollary \ref{rotationnumber} we conclude
$$\forall k<0, \quad k\rho(c)-2 <  \tilde \theta_{k}^--\tilde\theta_0 <\tilde \theta_{k}-\tilde\theta_0 <\tilde \theta_{k}^+-\tilde\theta_0 <k\rho(c)+2.$$
\end{proof}
\begin{rem}\rm\label{autoB}
Note that the previous proposition clearly implies that the hypothesis in Proposition \ref{cross-asymptotic}, that $|\tilde \theta_{i+1} - \tilde \theta_i|$ is bounded, is  automatically verified by any minimizing chain, hence can be dropped.
\end{rem}

We anticipate on the two next sections by stating a result that will be proved later in the text:

\begin{Th}
\label{calibratemin}
 Let $(\tilde \theta_k)_{k\leqslant 0}\in \R^{\Z_-}$ be a minimizing chain. Then there exists a cohomology class $c\in \R$ and a weak KAM solution $u_c:\T^1\to \R$ at cohomology $c$ such that
 $(\tilde \theta_k)_{k\leqslant 0}$ calibrates $\tilde u_c$. 
\end{Th}

As a consequence, we  already state an improvement of Proposition \ref{rotnummin}. The proof of the next result uses ideas from Aubry--Mather
 theory (\cite{Bang}) that play a central role in the study of minimizing chains  with rational rotation numbers.
 
 \begin{co}\label{rotnummin1}
 Let $(\tilde \theta_k)_{k\leqslant 0}\in \R^{\Z_-}$ be a minimizing chain. Then there exists $\rho \in \R$ such that 
 \begin{itemize}
 \item for all pairs of integers $(p,q)$ with $q<0$ such that  $p/q<\rho$, then  $\tilde\theta_q - \tilde \theta_0 < p$,
 \item for all pairs of integers $(p,q)$ with $q<0$ such that  $p/q>\rho$, then  $\tilde\theta_q - \tilde \theta_0 > p.$
 
 \end{itemize}
In particular,
  $$\forall k\leqslant 0, \quad |\tilde\theta_k - \tilde \theta_0 - k\rho | <1.$$
 \end{co}

\begin{proof}
Of course, $\rho$ is the same as the one given by Proposition \ref{rotnummin} and is also $\rho(c)$ where $c$ is given by the previous Theorem. 

Let us prove the first point, the second is similar. By contradiction, assume the existence of an integer $p$ and $q<0$ such that $ \tilde\theta_q - \tilde \theta_0 \geqslant p > q\rho$. By the previous Theorem \ref{calibratemin}, there exists $c\in \R$ and $u_c : \T^1 \to \R$ such that  $(\tilde \theta_k)_{k\leqslant 0}$ calibrates $\tilde u _c$.
\begin{itemize}
\item
We first exclude the equality $\tilde\theta_q = \tilde \theta_0+p$. Indeed,  $\tilde u_c$ is derivable at $\tilde\theta_q$ by Theorem \ref{twistKAM}. Hence assuming by contradiction that $\tilde\theta_q = \tilde \theta_0+p$, by periodicity, $\tilde u_c$ is also derivable at $\tilde \theta_0$ with $\tilde u_c'(\tilde\theta_0) =\tilde u_c'(\tilde\theta_q)$.   With the usual notations, $r_k = c+ \tilde u'_c(\tilde\theta_k)$, recalling that $(\tilde \theta_k,r_k)_{k\leqslant 0}$ is a piece of orbit of $\tilde f$, it follows that 
$$\forall k\leqslant 0,\quad \tilde\theta_{k+q} = \tilde\theta_k +p.$$
By induction, we deduce that $\tilde\theta_{nq} = \tilde\theta_0+np$ for all $n>0$ and finally, dividing by $nq$ and letting $n\to +\infty$ we conclude that $\rho = \frac{p}{q}$. This contradicts the right inequality  $p > q\rho$.
\item We are left with the hypothesis that $\tilde\theta_q - \tilde \theta_0 > p > q\rho$. By periodicity, the chain  $(\tilde \theta_{q+k}-p)_{k\leqslant 0}$ also calibrates $\tilde u_c$. By Lemma \ref{orderWKAM} we deduce that $\tilde\theta_{k+q}-p >\tilde\theta_k$ for all $k\leqslant 0$. By induction we readily obtain that for all $k\leqslant 0$, the sequence $(\tilde\theta_{k+nq}-np)_{n\geqslant 0}$ is increasing. Applying for $k=0$ and dividing by $nq<0$ yields $\dfrac{\tilde\theta_{nq}}{nq} - \dfrac{p}{q} <\dfrac{\tilde\theta_0}{nq}$. Finally, letting $n\to+\infty$ entails $\rho \leqslant \dfrac{p}{q}$ that is again a contradiction.
\end{itemize}

To prove the final statement, let us argue by contradiction and assume that there exists $q<0$ such that $ |\tilde\theta_q - \tilde \theta_0 - q\rho | \geqslant 1$. Then one of the following holds:
$$\exists p\in \Z , \quad \tilde\theta_q - \tilde \theta_0 \leqslant p < q\rho,$$
$$\exists p\in \Z , \quad \tilde\theta_q - \tilde \theta_0 \geqslant p > q\rho,$$
both impossible by what was proved above.
\end{proof}

 We now turn to making Proposition \ref{cross-asymptotic} more precise.
 
 \begin{pr}\label{mincal}
 Let $c\in \R$ be a cohomology class. Let $(\tilde \theta_i)_{i\leqslant 0}$ and  $(\tilde \theta_i^c)_{i\leqslant 0}$ be two minimizing chains such that 
 \begin{itemize}
 \item $\tilde \theta_0 = \tilde \theta_0^c$,
 \item $(\tilde \theta_i)_{i\leqslant 0}$ and  $(\tilde \theta_i^c)_{i\leqslant 0}$ are $\alpha$-asymptotic,
 \item there exists a weak KAM solution $u_c : \T^1 \to \R$ at cohomology $c$, such that $(\tilde \theta_i^c)_{i\leqslant 0}$ calibrates $\tilde u_c$.
 \end{itemize}
Then $(\tilde \theta_i)_{i\leqslant 0}$ also calibrates $\tilde u_c$.
 \end{pr}

\begin{proof}
We argue by contradiction assuming that there exist $n_0<0$ and $\varepsilon >0$ such that 
  $$  \tilde u_c (\tilde \theta_0) <  \tilde u_c (\tilde \theta_{n_0}) + \sum_{i=n_0}^{-1} \ww S(\tilde \theta_i,\tilde \theta_{i+1}) + c(\tilde \theta_{n_0} - \tilde \theta_0) +|n_0|\alpha(c)-\varepsilon.$$
  As for $n<n_0$,
    $$  \tilde u_c (\tilde \theta_{n_0}) \leqslant  \tilde u_c (\tilde \theta_{n}) + \sum_{i=n}^{n_0-1} \ww S(\tilde \theta_i,\tilde \theta_{i+1}) + c(\tilde \theta_n - \tilde \theta_{n_0}) +|n_0-n|\alpha(c),$$ 
     by summing the two previous inequalities, we deduce that 
$$  \forall n\leqslant n_0, \quad \tilde u_c (\tilde \theta_0) <  \tilde u_c (\tilde \theta_{n}) + \sum_{i=n}^{-1} \ww S(\tilde \theta_i,\tilde \theta_{i+1}) + c(\tilde \theta_n - \tilde \theta_0) +|n|\alpha(c)-\varepsilon.$$
By uniform continuity of $\ww S$ on compact sets and Remark \ref{autoB} there exists $n_1<n_0$ such that 
$$\forall i<n_1, \quad | \ww S(\tilde \theta_i,\tilde \theta_{i+1})- \ww S(\tilde \theta^c_i,\tilde \theta_{i+1})| < \frac{\varepsilon}{3},$$
and $|\tilde u_c (\tilde \theta_{n})-\tilde u_c (\tilde \theta^c_{n}) + c(\tilde \theta_{n}-\tilde \theta^c_{n})|<\varepsilon /3$.
It follows that for $n<n_1$,
\begin{multline*}
 \ww S(\tilde \theta_n,\tilde \theta^c_{n+1})+ \sum_{i=n_0}^{-1} \ww S(\tilde \theta^c_i,\tilde \theta^c_{i+1}) \leqslant
\sum_{i=n}^{-1} \ww S(\tilde \theta^c_i,\tilde \theta^c_{i+1})+\frac{\varepsilon}{3} \\
=\tilde u_c (\tilde \theta^c_0) -  \tilde u_c (\tilde \theta^c_{n})  + c(\tilde \theta^c_0 - \tilde \theta^c_n) -|n|\alpha(c)+\frac{\varepsilon}{3} \\ 
 \leqslant \tilde u_c (\tilde \theta_0) -  \tilde u_c (\tilde \theta_{n})  + c(\tilde \theta_0 - \tilde \theta_n) -|n|\alpha(c) +\frac{2\varepsilon}{3} \\
 \leqslant \sum_{i=n}^{-1} \ww S(\tilde \theta_i,\tilde \theta_{i+1})-\frac{\varepsilon}{3}.\end{multline*}
This contradicts the fact that $(\tilde\theta_i)_{i\in [n,0]}$ is minimizing.

\end{proof}
We are now ready to show strong results on infinite minimizing chains according to their rotation number. As we will see, the nature of those results depends strongly on its rationality. The next two sections, are devoted to recalling facts about homeomorphisms of the circle and about Aubry sets of twist maps. Many of those results are enunciated and proved in \cite{Bang} as well as in many surveys about Poincar\' e and Denjoy theories (\cite{wal}).

\subsection{Irrational rotation number}

We assume in this paragraph that $\rho_0\in \R\setminus \mathbb Q$.  Let $c_0 \in \R $ such that $\rho(c_0)=\rho_0$, which exists thanks to Theorems \ref{beta+} and  \ref{alphaC1}. We recall without proof (\cite{wal}):

\begin{Th}\label{poinc}
Let $\tilde g : \R \to \R$ be the lift of an orientation preserving circle homeomorphism $g : \T^1 \to \T^1$ of rotation number $\rho_0 \in \R\setminus \Q$. Then there exists a (unique up to an integer) non--decreasing $\tilde \varphi \to \R \to \R$ such that
$$\forall x\in \R, \quad \tilde \varphi(x+1) = \tilde \varphi(x)+1,$$
that semi--conjugates $\tilde g$ to the translation by $\rho_0$:

$$\forall x\in \R , \quad \tilde\varphi\circ \tilde g (x) = \varphi(x) +\rho_0.$$
There is then an alternative:
\begin{enumerate}
\item Either $\tilde \varphi$ is increasing in which case $\tilde g$ is conjugated to an irrational rotation and all its orbits are recurrent.
\item Either $\tilde \varphi$ is not increasing, in which case the set of recurrent orbits $g$ is a Cantor set denoted by $\mathcal K\subset \T^1$ that lifts to $\ww{ \mathcal K}\subset \R$. If $(\tilde a,\tilde b)$ is a connected component of $\R\setminus     \ww {\mathcal K}$ and $(a,b)\in \T^1$ its projection, the images $g^k(a,b)$ are mutually disjoint as $k\in \Z$. In particular,
$$\sum_{k\in \Z}\big( \tilde g^k(\tilde b) - \tilde g^k(\tilde a)\big) \leqslant 1.$$
\end{enumerate}
In all cases there is a unique $g$--invariant probability measure that has full support in the first case and support $\mathcal K$ in the second case. Finally, if $x_0\in \T^1$ then the $\alpha$--limit set of the orbit of $x_0$ is $\T^1$ in the first case and $\mathcal K$ in the second one.
\end{Th}
We then recall related results on twist maps, the proof of which are omitted, but can be found in \cite{Bang, MaFo, MatherMeasure} and a glimpse of which appears in Remark \ref{remPoinc}:  
\begin{Th}\label{AMT}
Let $\mu^*_{\rho_0}$ be a Mather minimizing measure verifying \eqref{minMeasA} for $c_0$. Then by Theorem \ref{alphaC1}, $\mu^*_{\rho_0}$ is minimizing for all $c\in \rho^{-1}(\{\rho_0\})$. In particular its support verifies
$$\forall c\in \rho^{-1}(\{\rho_0\}), \quad  \pi_1\big({\rm supp}( \mu^*_{\rho_0}) \big)\subset \AA_c.$$
Let $\theta_0^{\rho_0}\in {\rm supp}( \mu^*_{\rho_0})$ and let $\tilde \theta_0^{\rho_0} \in \R$ a lift and $(\tilde \theta_k^{\rho_0})_{k\in \Z} \in \ww\AAA_{c_0}$ the associated minimizing sequence. As seen in the proof of Proposition \ref{RotNum} this sequence only depends on  $\mu^*_{\rho_0}$ hence $(\tilde \theta_k^{\rho_0})_{k\in \Z} \in \ww\AAA_{c}$ for all $c\in \rho^{-1}(\{\rho_0\})$. 

Let $\tilde g : \R \to \R$ the associated map given by Theorem \ref{AM} and $g : \T^1 \to \T^1$ its projection\footnote{The results of Mather proven in \cite{Bang} actually show that the maps $\tilde g$ and $g$ can be chosen only depending on $\rho_0$ and that the measure $\mu^*_{\rho_0}$ is unique.}. Then $\pi_{1*} \mu^*_{\rho_0}$ is $g$--invariant. In particular $\pi_1\big({\rm supp}( \mu^*_{\rho_0}) \big)$ is either $\T^1$ or the $g$--invariant Cantor set $\mathcal K$. Moreover, $\tilde g$ only depends on $\mu^*_{\rho_0}$.
\end{Th}
The main result of this section is then
\begin{Th}\label{mainirr}
Let $(\tilde \theta_k)_{k\leqslant 0}  $ be a minimizing chain of rotation number $\rho_0 \in \R\setminus \Q$. Let $u_{c_0} : \T^1 \to \R$ be a weak KAM solution at cohomology $c_0 \in \rho^{-1}(\{\rho_0\})$. Then $(\tilde \theta_k)_{k\leqslant 0}  $ calibrates $\tilde u_{c_0}$.
\end{Th}
\begin{proof}
As always, $r_k =   \partial_2 \ww S ( \tilde \theta_{k-1},\tilde \theta_k)$. 
Arguing as in the proof of Proposition \ref{rotnummin}, there exists $c\in \R$ and $u_c : \T^1 \to \R$ a weak KAM solution at cohomology $c$ such that $(\tilde \theta_0, r_0) \in \PG(c+\tilde u'_c)$ and $\rho_0 = \rho(c)$. We adopt the same notations as in Proposition \ref{rotnummin} setting $r_0^\pm = c+\tilde u_{c\pm}'(\tilde \theta_0)$ (so that $r_0^+ \leqslant r_0^-$) and for $k\leqslant 0$,  $(\tilde\theta_k^\pm, r_k^\pm) = \tilde f^k(\tilde\theta_0,r_0^\pm)$.

Finally, let $y^-_0\in \pi^{-1} \circ \pi_1\big({\rm supp}( \mu^*_{\rho_0}) \big)\subset \R$ (resp. $y^+_0$) be the biggest element of $\pi^{-1} \circ \pi_1\big({\rm supp}( \mu^*_{\rho_0}) \big) $ such that $y^-_0 \leqslant \tilde\theta_0$ \big(resp. smallest element of $\pi^{-1} \circ \pi_1\big({\rm supp}( \mu^*_{\rho_0}) \big) $ such that $y^+_0 \geqslant \tilde\theta_0$\big). Let then $(y^-_k)_{k\in \Z} \in \ww\AAA_c$ and $(y^+_k)_{k\in \Z}\in \ww\AAA_c$ the associated minimizing sequences that calibrate $\tilde u_c$ (by Theorem \ref{AMT}). As in the proof of Proposition \ref{rotnummin}, for all $k\leqslant 0$,
$$y_k^- \leqslant \tilde \theta_k^- \leqslant \tilde\theta_k \leqslant \tilde\theta_k^+\leqslant y_k^+.$$
By Theorem \ref{AMT}, $y_k^- - y_k^+ \to 0$ as $k\to -\infty$ hence all the present sequences are $\alpha$--asymptotic.

Let now $(\tilde\theta^{c_0}_k)_{k\leqslant 0}$ be a calibrating chain for $\tilde u_{c_0}$ with $\tilde\theta^{c_0}_0 = \tilde\theta_0$ (that exists thanks to  Theorem \ref{twistKAM}). As $(y^\pm_k)_{k\in \Z} $ both calibrate $\tilde u_{c_0}$ the same reasoning as above yields that $(\tilde\theta^{c_0}_k)_{k\leqslant 0}$ and $(y^\pm_k)_{k\in \Z} $ are all $\alpha$--asymptotic. Finally, Proposition \ref{mincal} applies and $(\tilde \theta_k)_{k\leqslant 0}  $  does indeed calibrate $\tilde u_{c_0}$.
\end{proof}
As a Corollary, we obtain uniqueness of weak KAM solutions up to constants with irrational rotation number, and a Theorem due to Mather (\cite{MaDiff}) and Bangert (\cite{BangR}). A different proof also can be found in \cite{BerAIF}.
\begin{co}\label{TheoMather}
If $\rho_0 \in \R\setminus \mathbb Q$ then $\rho^{-1}(\{\rho_0\}) =\{c_0\}$ is a singleton, meaning that $\beta$ is derivable at $\rho_0$. Moreover, if $u : \T^1\to \R$ and $v : \T^1 \to \R$ are weak KAM solutions at cohomology $c_0$ then $u-v$ is constant. 
\end{co}
\begin{proof}
We will prove both statements at once. Let $\{c_0,c_1\} \in \rho^{-1}(\{\rho_0\})$ and let $u_0 : \T^1 \to \R$ be a weak KAM solution at cohomology $c_0$ and $u_1: \T^1 \to \R$ be a weak KAM solution at cohomology $c_1$. Let $D\subset \T^1$ be the set of points where both $u_0$ and $u_1$ are derivable. It is of full Lebesgue measure as its complement is countable. Let $\theta_0 \in D$ and $\tilde \theta_0$ a lift. By Theorem \ref{twistKAM}, there exist unique chains $(\tilde\theta_k^i)_{k\leqslant 0}$ for $i\in \{0,1\}$ calibrating respectively $\tilde u_0$ and $\tilde u_1$ and such that $\tilde\theta_0^i = \tilde \theta_0$ for $i\in \{0,1\}$.

Moreover, $c_i+u'_i(\theta_0) = \partial_2 \ww S ( \tilde\theta^i_{-1} , \tilde \theta^i_0)$. Applying twice Theorem \ref{mainirr} we discover that $(\tilde\theta_k^0)_{k\leqslant 0}$ calibrates $\tilde u_1$ and $(\tilde\theta_k^1)_{k\leqslant 0}$ calibrates $\tilde u_0$. Hence by uniqueness, $(\tilde\theta_k^0)_{k\leqslant 0}=(\tilde\theta_k^1)_{k\leqslant 0}$ and $c_0+u'_0(\theta_0) = c_1+u'_1(\theta_0)$.

Integrating on $\T^1$ it follows that 
$$c_0 = \int_{\T^1} \big(c_0 +u'_0(s) \big) \dd s =   \int_{\T^1} \big(c_1 +u'_1(s) \big) \dd s = c_1.$$
As a conclusion, since  $u_0$ and $u_1$ are Lipschitz with almost everywhere equal derivatives, $u_1 - u_0$ is a constant function.
\end{proof}
As there exists a weak KAM solution $u_{c_0}$ associated to the irrational rotation number $\rho_0$, we know that for every $\tilde\theta_0 \in \R$ there exists a minimizing chain $(\tilde\theta_k)_{k\leqslant 0}$ starting at $\tilde\theta_0 $ with rotation number $\rho_0$ (one calibrating $u_{c_0}$). This next and last corollary on the contrary states that there are not too many such minimizing chains:

\begin{co}\label{fewirr}
Let $\rho_0 \in \R\setminus \mathbb Q$. For all $\tilde\theta \in \R$ there exists at most one minimizing chain $(\tilde\theta_k)_{k\leqslant 0}$ with rotation number $\rho_0$ such that $(\tilde\theta_0 , r_0) \in \tilde f(V_{\tilde\theta})$ where $(\tilde\theta_k,r_k)_{k\leqslant 0}$ is the backward $\tilde f$--orbit associated to $(\tilde\theta_k)_{k\leqslant 0}$ and $V_{\tilde\theta} = \{(\tilde\theta , r), \ r\in \R\}\subset \R\times \R$ is the vertical above $\tilde\theta$.
\end{co}
\begin{proof}
Using the notations of the Corollary, it was just established that, for such a minimizing chain, $(\tilde\theta_0,r_0) \in \PG(c_0 +\tilde u'_{c_0})$ where $\rho^{-1}(\{\rho_0\}) = \{c_0\} $ and $ \PG(c_0 +\tilde u'_{c_0})$ is the unique pseudograph associated to $c_0$.

By Corollary \ref{graph+}, $f^{-1}\big( \PG(c_0 +\tilde u'_{c_0})\big)$ is a graph, meaning that $\tilde f^{-1}\big( \PG(c_0 + \tilde u'_{c_0})\big)\cap V_{\tilde\theta}$ is a singleton and consequently, so is $\PG(c_0 +\tilde  u'_{c_0})\cap \tilde f(V_{\tilde\theta})$. The result follows.
\end{proof}

\subsection{Rational rotation number}
We assume in this paragraph that $\rho_0 = \frac{p}{q}$ is rational, written in irreducible form with $q<0$. We denote by $[a,b] = \rho^{-1}(\{\rho_0\})$. Let us  first recall some classical results from Aubry--Mather theory (\cite{Bang}):
\begin{Th}\label{AMTR}
Let $\mu^*_{\rho_0}$ be a Mather minimizing measure verifying \eqref{minMeasA} for some $c_0\in [a,b]$. Then by Theorem \ref{alphaC1}, $\mu^*_{\rho_0}$ is minimizing for all $c\in \rho^{-1}(\{\rho_0\})$. In particular its support verifies
$$\forall c\in \rho^{-1}(\{\rho_0\}), \quad  \pi_1\big({\rm supp}( \mu^*_{\rho_0}) \big)\subset \AA_c.$$

Let $\theta_0^{\rho_0}\in {\rm supp}( \mu^*_{\rho_0})$ and let $\tilde \theta_0^{\rho_0} \in \R$ a lift and $(\tilde \theta_k^{\rho_0})_{k\in \Z} \in \ww\AAA_{c_0}$ the associated minimizing sequence. As seen in the proof of Proposition \ref{RotNum} this sequence only depends on  $\mu^*_{\rho_0}$ hence $(\tilde \theta_k^{\rho_0})_{k\in \Z} \in \ww\AAA_{c}$ for all $c\in \rho^{-1}(\{\rho_0\})$.

 It verifies 
\begin{equation}\label{typepqeq}
\forall k\in \Z, \quad \tilde \theta_{k+q}^{\rho_0} = \tilde \theta_k^{\rho_0}+p,
\end{equation}
and hence projects to a $q$--periodic sequence on $\T^1$. 

Reciprocally, if a minimizing sequence $ (\tilde \theta_k)_{k\in \Z} \in \R^\Z$ satisfies \eqref{typepqeq}, then it is in any Aubry set $\ww\AAA_c$ for $c\in [a,b]$ and the measure 
$$\mu^* = \frac{1}{q} \sum_{k=0}^{q-1} \delta_{( \theta_k, \tilde \theta_{k+1}-\tilde \theta_k)},$$
 verifies that $\mu^*$ is minimizing (satisfies \eqref{minMeasA}) for all $c\in [a,b]$.
\end{Th}

A notion that stems from the previous Theorem is:
\begin{df}\rm\label{typepq}
A sequence $(\tilde \theta_k)_{k\in \Z}\in \R^\Z $ that verifies
 $$\forall k\in \Z, \quad \tilde \theta_{k+q} = \tilde \theta_k+p,$$
 is said to be of type $(p,q)$.
\end{df}

It follows from the previous Theorem that it makes sense to denote the projected Mather set $\MM_{\rho_0}$ as $\MM_c$ does not depend on the choice of $c\in \rho^{-1}(\{\rho_0\})$. And  $\mathfrak M_{\rho_0}$ denotes its lift to $\R$ as well.

We recall that by Lemma \ref{extr}, there exists a unique weak KAM solution at cohomology $a$ (resp. $b$) that vanishes at $0\in \T^1$, and that we denote $u_a : \T^1 \to \R$ (resp. $u_b : \T^1 \to \R$). The main result of this section can be stated as follows:

\begin{Th}\label{mainIrr}
Let $(\tilde\theta_k)_{k\leqslant 0}$ be a minimizing chain with rotation number $\rho_0 = \frac pq$. Assume that $(\tilde\theta_k)_{k\leqslant 0}$ is not of type $(p,q)$. Then $\tilde \theta_0 \neq \tilde\theta_{q}-p$ and one of the following assertions holds:
\begin{enumerate}
\item $\tilde \theta_0 < \tilde\theta_{q}-p$ and $(\tilde\theta_k)_{k\leqslant 0}$ calibrates $\tilde u_a$,
\item $\tilde \theta_0 > \tilde\theta_{q}-p$ and $(\tilde\theta_k)_{k\leqslant 0}$ calibrates $\tilde u_b$.
\end{enumerate}

\end{Th}
The proof of this Theorem will be done in several steps and provides more precise details about the behavior of minimizing chains and their links with weak KAM solutions.

We start by the following non--crossing lemma for minimizing chains with rotation number $\rho_0=\frac{p}{q}$. It is reminiscent  of Corollary \ref{rotnummin1}:

\begin{pr}\label{non-crossPQ}
Let $(\tilde\theta_k)_{k\leqslant 0}$ be a minimizing chain with rotation number $\rho_0 = \frac pq$. Assume that $(\tilde\theta_k)_{k\leqslant 0}$ is not of type $(p,q)$. Then the minimizing  chains $(\tilde\theta_k)_{k\leqslant 0}$ and $(\tilde\theta_{k+q}-p)_{k\leqslant 0}$ do not cross.

Moreover,
\begin{itemize}
\item if $\tilde \theta_0 < \tilde\theta_{q}-p$ then $(\tilde\theta_k)_{k\leqslant 0}$ calibrates $\tilde u_a$,
\item if $\tilde \theta_0 > \tilde\theta_{q}-p$ then $(\tilde\theta_k)_{k\leqslant 0}$ calibrates $\tilde u_b$.
\end{itemize}
\end{pr}
\begin{proof}

As $(\tilde\theta_k)_{k\leqslant 0}$ and $(\tilde\theta_{k+q}-p)_{k\leqslant 0}$ cross at most once, there exists $n_0 \leqslant 0$ such that one of the following holds
$$\forall k<n_0, \quad \tilde\theta_k< \tilde\theta_{k+q}-p,$$
$$\forall k<n_0, \quad \tilde\theta_k> \tilde\theta_{k+q}-p.$$
Let us deal with the first case, the second being similar. By induction, it follows that for all $k<n_0$, the sequence $(\tilde\theta_{k+mq}-mp)_{m\geqslant 0}$ is increasing.

Let us introduce some notations. Once more, for all $k\leqslant 0$, we set $r_k =   \partial_2 \ww S ( \tilde \theta_{k-1},\tilde \theta_k)$ in such a way that $(\tilde \theta_k,r_k)_{k\leqslant 0}$ is a piece of orbit of $\tilde f$. Let $c\in \R$ such that $(\tilde\theta_0,r_0) \in \PG(c+\tilde u'_c)$ for some weak KAM solution $u_c : \T^1 \to \R$ at cohomology $c$. Then as in Proposition \ref{rotnummin}, $\rho(c) = \frac pq$ and $c\in [a,b]$. Moreover, $\tilde\theta_0 \notin \mathfrak M_{\rho_0}$. Indeed, if it were the case,  as $\tilde u_c$ is derivable  on $\mathfrak M_{\rho_0}$,  we would have $(\tilde \theta_k,r_k)_{k\leqslant 0} = \big(\tilde f^k(\tilde \theta_0,\tilde u_c'(\tilde\theta_0)+c)\big)_{k\leqslant 0}$
 that is of type $(p,q)$ by Theorem \ref{AMTR}.

It is then denoted $y_0^- = \max\{y\in \mathfrak M_{\rho_0} ,\  y<\tilde\theta_0\}$ and $y_0^+ = \min\{y\in \mathfrak M_{\rho_0} , \ y>\tilde\theta_0$\}. Finally, let $(y_k^-)_{k\in \Z}$ and $(y_k^+)_{k\in \Z}$ the associated minimizing sequences, that are of type $(p,q)$. By Theorem \ref{AMTR} the sequences $(y_k^\pm)_{k\in \Z}$ calibrate $\tilde u_c$. 
First note that $y_0^-< \tilde\theta_0<y_0^+$ because $ \mathfrak M_{\rho_0}$ is closed. 
Moreover, arguing as in Proposition \ref{rotnummin} yields that 
$$\forall k\leqslant 0, \quad y_k^-< \tilde\theta_k<y_k^+.$$
As 
\begin{equation}\label{bangbang}
\forall m\geqslant 0, \quad y_k^-=y^-_{k+mq} -mp< \tilde\theta_{k+mq}-mp<y^+_{k+mq} -mp = y_k^+, 
\end{equation}
we deduce that for all $k\in \Z$, the  sequence $(\tilde\theta_{k+mq}-mp)_{m\geqslant -k/q}$ is bounded, increasing for $m$ large enough, hence it converges to some $y_k$. Moreover, the sequence $(y_k)_{k\in \Z}$ is minimizing as a limit of minimizing chains and verifies
$$\forall k\in \Z, \quad y_{k+q}-p = \lim_{m\to +\infty } \tilde\theta_{k+mq+q}-mp -p = \lim_{m\to +\infty } \tilde\theta_{k+(m+1)q}-(m+1)p = y_k.$$ 
It follows from Theorem \ref{AMTR}  that $y_k \in \mathfrak M_{\rho_0}$ and that $(y_k)_{k\in \Z}$ calibrates $\tilde u_c$. 
By \eqref{bangbang}, $\tilde\theta_k < y_k \leqslant y_k^+$. If $ y_k < y_k^+$ it follows from Lemma \ref{orderWKAM} that $\tilde\theta_0 < y_0 < y_0^+$ thus contradicting the definition of $y_0^+$. From \eqref{bangbang} comes that $(\tilde\theta_k)_{k\leqslant 0}$, $(\tilde\theta_{k+q}-p)_{k\leqslant 0}$ and $(y_k^+)_{k\in \Z}$ are $\alpha$--asymptotic. It is deduced from Proposition \ref{noncrossing} that $(\tilde\theta_k)_{k\leqslant 0}$ and $(\tilde\theta_{k+q}-p)_{k\leqslant 0}$ cannot cross (as $(\tilde\theta_{k+q}-p)_{k\leqslant 0}$ is a strict subchain of a minimizing chain).

To finish the proof, let $(\tilde \theta^a_k)_{k\leqslant 0}$ be the minimizing chain calibrating $\tilde u_a$, with $ \tilde \theta^a_0 = \tilde \theta_0$ and $\partial_2 \ww S ( \tilde \theta^a_{-1},\tilde \theta^a_0) = a+\tilde u'_{a+}(\tilde\theta_0)$. As $r_0-c \in \partial^+\tilde u_c(\tilde\theta_0)$, it follows from Corollary \ref{monotonyExt} that $r_0 \geqslant  a+\tilde u'_{a+}(\tilde\theta_0)$. Either there is equality, in which case  $(\tilde \theta^a_k)_{k\leqslant 0} =(\tilde \theta_k)_{k\leqslant 0} $ calibrates $\tilde u_a$, or the inequality is strict and the twist condition entails that $ \tilde \theta^a_{-1} > \tilde \theta_{-1}$. As the two minimizing chains do not cross anymore, and using that  $(y_k^-)_{k\in \Z}$ calibrates $\tilde u_a$, it follows that
$$\forall k\leqslant 0, \quad \tilde\theta_k < \tilde\theta_k^a < y_k^+.$$
Hence $(\tilde \theta^a_k)_{k\leqslant 0}$ and $(\tilde \theta_k)_{k\leqslant 0}$ are $\alpha$--asymptotic and by Proposition \ref{mincal}, $(\tilde \theta_k)_{k\leqslant 0}$ calibrates $\tilde u_a$.
\end{proof}

A Corollary of this proof is:
\begin{co}\label{coR+-}
Let $(\tilde\theta_k)_{k\leqslant 0}$ be a minimizing chain with rotation number $\rho_0 = \frac pq$. Assume that $(\tilde\theta_k)_{k\leqslant 0}$ is not of type $(p,q)$ (Definition \ref{typepq}). Let $(y_k^\pm)_{k\in \Z}$ be the closest orbits of $\mathfrak M_{\rho_0}$ such that $y_k^-< \tilde\theta_k<y_k^+$ for $k\leqslant 0$ as in the proof of Proposition \ref{non-crossPQ}. Then 
\begin{itemize}
\item either $\tilde \theta_0 < \tilde\theta_{q}-p$, $(\tilde\theta_k)_{k\leqslant 0}$  and $(y_k^+)_{k\leqslant 0}$ are $\alpha$--asymptotic,
\item either $\tilde \theta_0 > \tilde\theta_{q}-p$, $(\tilde\theta_k)_{k\leqslant 0}$  and $(y_k^-)_{k\leqslant 0}$ are $\alpha$--asymptotic.
\end{itemize}
\end{co}

As another consequence, we derive that minimizing chains with rotation number $p/q$ are quite rare, same as was established for irrational rotation numbers in Corollary \ref{fewirr}:

\begin{Th}\label{atmost2}
Let $\ww\Theta\in \R$ and $V_{\ww\Theta} = \{\ww\Theta\} \times \R \subset \R\times \R$. Then there are at most two minimizing half orbits of $\tilde f$, $(\tilde\theta_k,r_k)_{k\leqslant 0}$ such that $(\tilde\theta_0,r_0) \in \tilde f( V_{\ww\Theta})$ and any corresponding minimizing chain $(\tilde\theta_k)_{k\leqslant 0}$ has rotation number $\rho_0=\frac{p}{q}$.
\end{Th}
\begin{proof}
It was just settled in Proposition \ref{non-crossPQ} that the chain $(\tilde\theta_k)_{k\leqslant 0}$ calibrates $\tilde u_a$ or $\tilde u_b$. Moreover, by Theorem \ref{twistKAM}, $(\tilde\theta_0,r_0) \in \PG(a+\tilde u'_a) \cup \PG(b+\tilde u'_b)$. Finally, $\tilde f^{-1}(\tilde\theta_0,r_0) \in \tilde f^{-1}\big(\PG(a+\tilde u'_a)\big) \cup \tilde f^{-1}\big( \PG(b+\tilde u'_b)\big)$. By Corollary \ref{graph+}, both  $f^{-1}\big(\PG(a+\tilde u'_a)\big) $  and $ \tilde f^{-1}\big( \PG(b+\tilde u'_b)\big)$ are graphs, hence $ f^{-1}\big(\PG(a+\tilde u'_a)\big) \cup \tilde f^{-1}\big( \PG(b+\tilde u'_b)\big)$ and $V_{\ww\Theta}$ intersect in $1$ or $2$ points. It follows that so do $\PG(a+\tilde u'_a) \cup \PG(b+\tilde u'_b)$ and $ \tilde f( V_{\ww\Theta})$ and the result follows.
\end{proof}

Note that it is now fully established that all minimizing chains  $(\tilde\theta_k)_{k\leqslant 0}$ calibrate a weak KAM solution thanks to Theorem \ref{mainirr} and Proposition \ref{non-crossPQ}. We are therefore allowed to use Corollary \ref{rotnummin1}.

Follows an existence result of orbits displaying behaviors as in Corollary \ref{coR+-}. The proof is inspired by a similar result for bi--infinite minimizing orbits in \cite{Bang}. It shows that on each vertical above points that are not in the projected Mather set $\MM_{\frac pq}$ there are at least two initial points of minimizing half orbits with rotation number $p/q$:
\begin{pr}\label{alphaAsymptotic}
Assume that $\mathfrak M_{\rho_0}\neq \R$. Let $(y_0^-,y_0^+)$ be a connected component of $\R\setminus\mathfrak M_{\rho_0}$, $(y_k^\pm)_{k\in \Z}$ be the associated minimizing orbits of type $(p,q)$. Then for all $\tilde\theta_0\in (y_0^-,y_0^+)$, there exists two minimizing chains $(\tilde\theta_k^\pm)_{k\leqslant 0}$ such that $\tilde\theta_0^\pm = \tilde\theta_0$ and such that $(\tilde\theta_k^+)_{k\leqslant 0}$ is $\alpha$--asymptotic to $(y_k^+)_{k\leqslant 0}$ (resp. $(\tilde\theta_k^-)_{k\leqslant 0}$ is $\alpha$--asymptotic to $(y_k^-)_{k\leqslant 0}$).
\end{pr}
\begin{proof}
Let us prove the existence of  $(\tilde\theta_k^-)_{k\leqslant 0}$, the other being obtained by reversing all inequalities.

Let $(\rho_n)_{n>0}$ be a decreasing sequence converging to $p/q$. For all $n>0$, let $(\tilde\theta^n_k)_{k\leqslant 0}$ be a minimizing chain of rotation number $\rho_n$ such that $\tilde\theta^n_k = \tilde\theta_0$ (for example a chain calibrating any weak KAM solution at a cohomology $c_n$ verifying $\rho(c_n) = \rho_n$). By Corollary \ref{rotnummin1}, 
$$\forall n>0, \quad \tilde\theta_q^n < \tilde \theta_0+p.$$ 
By the same Corollary \ref{rotnummin1} or Proposition \ref{rotnummin}, for all $k\leqslant 0$, fixed, the sequence $(\tilde\theta^n_k)_{n>0}$ is bounded. Up to a diagonal extraction, we may therefore assume that for all $k\leqslant 0$, the sequence $(\tilde\theta^n_k)_{n>0}$ converges to some $\tilde\theta_k^-$. Clearly, the chain $(\tilde\theta_k^-)_{k\leqslant 0}$ is minimizing, as a limit of minimizing chains. Moreover, $\tilde\theta_0^-=\tilde\theta_0$ and  $(\tilde\theta_k^-)_{k\leqslant 0}$ has rotation number $p/q$ by passing to the limit in the inequalities provided by Proposition \ref{rotnummin}. Finally, $\tilde\theta_q^- \leqslant \tilde \theta_0+p$. By Proposition \ref{non-crossPQ}, this cannot be an equality, hence $\tilde\theta_q^- < \tilde \theta_0+p$ and by Corollary \ref{coR+-}, $(\tilde\theta_k^-)_{k\leqslant 0}$ is $\alpha$--asymptotic to $(y_k^-)_{k\leqslant 0}$.
\end{proof}
As a consequence, we recover a result of Mather \cite{MaDiff} and Bangert \cite{BangR}:
\begin{Th}\label{diffR}
The following alternative holds:
\begin{enumerate}
\item the Mather set $\mathfrak M_{\frac pq}= \R$ in which case there is an invariant graph $\mathcal C_{\frac pq}\subset \A$  on which the dynamics of $f$ is $q$--periodic, there exists a unique weak KAM solution (up to constants) associated to the rotation number $\frac pq$ and the $\beta$ function is derivable at $\frac pq$;
\item the Mather set $\mathfrak M_{\frac pq}\neq \R$ and  the $\beta$ function is not derivable at $\frac pq$.

\end{enumerate}

\end{Th}
\begin{proof}
\begin{enumerate}
\item If  $\mathfrak M_{\frac pq}= \R$, we infer from Proposition \ref{aubry-cot}  that if $c\in \R$ verifies $\rho(c) = \frac pq$, and if $u_c$ is a weak KAM solution at cohomology $c$, then $u_c$ is derivable on $\T^1$ (hence $C^1$ and even $C^{1,1}$ by Birkhoff's Theorem \ref{birk}) and $c+u'_c$ does not depend on $c$, nor $u_c$, but only on the  orbits of type $(p,q)$, hence $\beta$ is derivable at $\frac pq$ with $c=\beta'(p/q)$ and $u_c$ is unique up to constants.

\item  Assume now that $\mathfrak M_{\frac pq}\neq  \R$. Let $[a,b] = \rho^{-1}(\{p/q\})$. Let $u_a$ be the unique weak KAM solution at cohomology $a$ that vanishes at $0$ and $u_b$ be the unique weak KAM solution at cohomology $b$ that vanishes at $0$. Let $\tilde\theta_0 \in \R \setminus \mathfrak M_{\frac pq}$ such that  both $\tilde u_a$ and $\tilde u_b$ are derivable at $\tilde\theta_0$ (that exists as $ \R\setminus \mathfrak M_{\frac pq}$ is a non--empty open set and $\tilde u_a$ and $\tilde u_b$ are derivable except on a countable set). Let finally 
$(\tilde\theta_k^\pm)_{k\leqslant 0}$ be the sequences given by the previous Proposition \ref{alphaAsymptotic}. Then clearly, those two sequences are different, hence $\tilde\theta^+_{-1} \neq \tilde\theta^-_{-1}$ (more precisely, $\tilde\theta^-_{-1} < \tilde\theta^+_{-1}$). It follows from Proposition \ref{non-crossPQ} and Theorem \ref{twistKAM} that 
$$a+\tilde u'_a(\tilde\theta_0) = \partial_2\ww S(\tilde\theta_{-1}^+, \tilde\theta_0) \neq   \partial_2\ww S(\tilde\theta_{-1}^-, \tilde\theta_0) = b+\tilde u'_b(\tilde\theta_0).$$
We deduce from Corollary \ref{monotonyExt} that $\tilde\theta \mapsto (\tilde u_b-\tilde u_a)(\tilde\theta) + (b-a)\tilde\theta$ is non--decreasing and non--constant. Finally, integrating  inequality $b+\tilde u_b'(\tilde\theta)  \geqslant a+\tilde u_a'(\tilde\theta) $, that holds  almost--everywhere, between $0$ and $1$ and remembering that it is not an equality almost--everywhere, yields $b>a$. As $[a,b] = \partial^- \beta (p/q)$ we have proven the result.
\end{enumerate}
\end{proof}
The previous proof actually implies the more precise result:

\begin{pr}\label{strictaudessus}
Assume $\mathcal M_{\frac pq}\neq \T^1$, let $[a,b] = \rho^{-1}(\{p/q\})$.  Let $u_a$ be the unique weak KAM solution at cohomology $a$ that vanishes at $0$ and $u_b$ be the unique weak KAM solution at cohomology $b$ that vanishes at $0$. Then for all 
$ \theta \in \T^1\setminus \mathcal M_{\frac pq}$ where both $u_a$ and $u_b$ are derivable, 
$$a+u'_a(\theta) < b+u'_b(\theta).$$
\end{pr}

As a consequence we deduce:
\begin{pr}\label{sens-parcour}
Assume that $\mathfrak M_{\rho_0}\neq \R$. Let $(y_0^-,y_0^+)$ be a connected component of $\R\setminus\mathfrak M_{\rho_0}$, $(y_k^\pm)_{k\in \Z}$ be the associated minimizing orbits of type $(p,q)$. Then 
\begin{enumerate}
\item all  minimizing chains $(\tilde\theta_k^a)_{k\leqslant 0}$ calibrating $\tilde u_a$, with $\tilde \theta_0^a \in (y_0^-,y_0^+)$ are $\alpha$--asymptotic to $(y_k^+)_{k\leqslant 0}$,
\item all  minimizing chains $(\tilde\theta_k^b)_{k\leqslant 0}$ calibrating $\tilde u_b$, with $\tilde \theta_0^b \in (y_0^-,y_0^+)$ are $\alpha$--asymptotic to $(y_k^-)_{k\leqslant 0}$.
\end{enumerate}
\end{pr}

\begin{proof}
Let us prove the first point. 
Let $(\tilde\theta_k^a)_{k\leqslant 0}$  be calibrating  $\tilde u_a$. If $\tilde\theta_0^a$ is a point of derivability of $\tilde u_a$, let $(\tilde\theta_k^+)_{k\leqslant 0}$ be given by Proposition \ref{alphaAsymptotic} that is  $\alpha$--asymptotic to $(y_k^+)_{k\leqslant 0}$ and  such that $\tilde\theta_0^+ = \tilde\theta^a_0$. By Proposition \ref{non-crossPQ} and Corollary \ref{coR+-} $(\tilde\theta_k^+)_{k\leqslant 0}$ calibrates $\tilde u_a$ and as the latter is derivable at $ \tilde\theta^a_0$, such a minimizing chain is unique. Hence  $(\tilde\theta_k^+)_{k\leqslant 0}=(\tilde\theta_k^a)_{k\leqslant 0} $ is indeed  $\alpha$--asymptotic to $(y_k^+)_{k\leqslant 0}$.

Let us now assume  $\tilde\theta_0^a$ is not a point of derivability of $\tilde u_a$. Let $(\ww\Theta^n_0)_{n\geqslant 0}$ be an increasing sequence converging to  $\tilde\theta_0^a$ and made of derivability points of $\tilde u_a$ (that is Lipschitz hence derivable almost everywhere). For all $n$, we denote by  $(\ww\Theta^n_k)_{k\leqslant 0}$ the unique corresponding minimizing chain calibrating $\tilde u _a$. By the beginning of this proof, for $n$ fixed, $(\ww\Theta^n_k)_{k\leqslant 0}$ is  $\alpha$--asymptotic to $(y_k^+)_{k\leqslant 0}$ and verifies $\ww\Theta^n_q>\ww\Theta_0^n +p$ by Corollary \ref{coR+-}. Moreover, setting $r_k^n = a+\tilde u_a'(\ww\Theta^n_k)$, $(\ww\Theta^n_k,r_k^n)_{k\leqslant 0}$ is a piece of orbit of $\tilde f$ by Theorem \ref{twistKAM}. It follows that as $n\to +\infty$, $(\ww\Theta^n_k,r_k^n)_{n\leqslant 0}$ converges to $(\tilde\theta_k^{a-}, r_k^-)$ such that $(\tilde\theta_k^{a-}, r_k^-)_{k\leqslant 0}$ is a piece of orbit of $\tilde f$, $\tilde\theta_0^{a-}=\tilde\theta_0^{a}$ and $r_0^- = a+u'_{a-}(\tilde\theta_0^{a})$. Moreover, $(\tilde\theta_k^{a-})_{k\leqslant 0}$ calibrates $\tilde u_a$. By passing to the limit in the corresponding inequalities for $\ww\Theta^n$ we gather
$\tilde\theta_q^{a-} \geqslant \tilde\theta_0^{a-}+p$ and equality is excluded by Proposition \ref{non-crossPQ}, hence $\tilde\theta_q^{a-} > \tilde\theta_0^{a-}+p$ and $(\tilde\theta_k^{a-})_{k\leqslant 0}$ is $\alpha$-asymptotic to  $(y_k^+)_{k\leqslant 0}$. 

Coming back to the chain  $(\tilde\theta_k^a)_{k\leqslant 0}$, setting $r_0^a = \partial_2 \ww S(\tilde\theta_{-1}^a,\tilde\theta_0^a)$, as $r_0^a-a \in \partial^+ \tilde u_a ( \tilde\theta_0^a)$ and by semiconcavity of  $\tilde u_a$, it follows that $r_0^a \leqslant r_0^-$. Then, by the twist condition, $\tilde\theta_{-1}^a \geqslant \tilde\theta_{-1}^{a-}$. Last, applying again that calibrating chains do not cross away from the origin (Lemma \ref{orderWKAM}) gives $\tilde\theta_{k}^a \geqslant \tilde\theta_{k}^{a-}$ for all $k\geqslant 0$. We then apply to $k=q$ in order to conclude that 
$$\tilde\theta_q^{a}\geqslant \tilde\theta_q^{a-} > \tilde\theta_0^{a-}+p=\tilde\theta_0^{a}+p.$$
The result now follows from Proposition \ref{non-crossPQ}.
\end{proof}
The proof of Theorem \ref{mainIrr} is now fully completed.

To clarify the picture  the following result makes  Proposition \ref{strictaudessus} more precise. As a matter of fact, it implies that the full pseudographs $\PG(a+u'_a)$ and $\PG(b+u'_b)$ only intersect on the Mather set.

\begin{pr}\label{pendulelike}
Assume that $\mathfrak M_{\rho_0}\neq \R$. Let $(y_0^-,y_0^+)$ be a connected component of $\R\setminus\mathfrak M_{\rho_0}$. Then restricted to $(y_0^-,y_0^+)$,   $\PG(a+\tilde u'_a)$ is strictly under $\PG(b+\tilde u'_b)$, in the sense that if $\tilde\theta_0 \in (y_0^-,y_0^+)$ and $r_a$ is such that $(\tilde\theta_0,r_a) \in \PG(a+\tilde u'_a)$ and $r_b$ is such that $(\tilde\theta_0,r_b) \in \PG(b+\tilde u'_b)$, then $r_a<r_b$.
\end{pr}
 \begin{proof}
 Let $(y_k^{\pm})_{k\in \Z}$ be the minimizing orbits of rotation number $\rho_0$ that are of type $(p,q)$. Then as the sequence $(y_k^+-y_k^-)_{k\in \Z}$ is positive valued and $|q|$--periodic, there exists $\varepsilon>0 $ such that $y_k^--y_k^+>\varepsilon$ for all integer $k\in\Z$. 
 
 Let $\tilde \theta_0 \in (y_0^-, y_0^+)$ and let us assume by contradiction that there exist two minimizing chains $(\tilde\theta_k^a)_{k\leqslant 0}$ and $(\tilde\theta_k^b)_{k\leqslant 0}$ calibrating respectively $\tilde u_a$ and $\tilde u_b$, such that $\tilde \theta_0^a = \tilde \theta^b_0 = \tilde\theta_0 \in (y_0^-,y_0^+)$ and verifying $r_0^a \geqslant r_0^b$, where $r_0^a =  \partial_2 \ww S(\tilde\theta_{-1}^a,\tilde\theta_0^a)$ and $r_0^b= \partial_2 \ww S(\tilde\theta_{-1}^b,\tilde\theta_0^b)$. First note that $r_0^a \neq r_0^b$ as otherwise both minimizing chains would be equal and $(\tilde\theta_k^a)_{k\leqslant 0}$ is $\alpha$--asymptotic to $(y_k^+)_{k\leqslant 0}$ while $(\tilde\theta_k^b)_{k\leqslant 0}$ is $\alpha$--asymptotic to $(y_k^-)_{k\leqslant 0}$. As now  $r_0^a >r_0^b$, it follows from the twist condition that $\tilde\theta^a_{-1} < \tilde \theta^b_{-1}$. But for $n<0$ large enough in absolute value, $|\tilde\theta_n^a - y^+_n|<\frac \varepsilon 2$ and  $|\tilde\theta_n^b - y^-_n|<\frac \varepsilon 2$ thus implying that $ \tilde\theta_n^a>\tilde\theta_n^b$. Hence the two minimizing chains cross at $0$ and somewhere between $-1$ and $n$, that is absurd.
 
 The previous result applies in particular to the minimizing and calibrating chains verifying $r_0^a = a+\tilde u'_{a-}(\tilde \theta_0)$ and $r_0^b = b+\tilde u'_{b+}(\tilde\theta_0)$ therefore proving the Proposition.
 \end{proof}

Now that we have a pretty good idea of how are organized the extreme pseudographs $\PG(a+u'_a)$ and $\PG(b+u'_b)$ and their respective calibrating chains, the next results help understanding the looks of more general weak KAM solutions at any cohomology class $c\in (a,b)$.

\begin{pr}\label{ucderiv}
Assume that $\mathfrak M_{\rho_0}\neq \R$. Let $(y_0^-,y_0^+)$ be a connected component of $\R\setminus\mathfrak M_{\rho_0}$. Let $c\in (a,b)$, $v_c : \T^1 \to \R$  a weak KAM solution at cohomology $c$ and $\tilde \theta_0 \in (y_0^-,y_0^+)$ a point of derivability of $\tilde v_c $ the lift of $v_c$. Then one of the following holds:
\begin{enumerate}
\item $\tilde u_a$ is derivable at $\tilde\theta_0$ and $a+\tilde u'_a (\tilde\theta_0 ) = c+\tilde v'_c (\tilde\theta_0 )$;
\item $\tilde u_b$ is derivable at $\tilde\theta_0$ and $b+\tilde u'_b (\tilde\theta_0 ) = c+\tilde v'_c (\tilde\theta_0 )$.
\end{enumerate}

\end{pr}

\begin{proof}
Let $(\tilde\theta^c_k)_{k\leqslant 0}$ be the unique minimizing chain calibrating $\tilde v_c$ such that $\tilde\theta^c_0 = \tilde\theta_0$. By Theorem \ref{mainIrr}, either  $\tilde \theta_0^c < \tilde\theta_{q}^c-p$ and $(\tilde\theta_k^c)_{k\leqslant 0}$ calibrates $\tilde u_a$,
either $\tilde \theta_0^c > \tilde\theta_{q}^c-p$ and $(\tilde\theta_k^c)_{k\leqslant 0}$ calibrates $\tilde u_b$.

Let us consider the first case and prove that $\tilde u_a$ is derivable at $\tilde\theta_0$. Recall that as $\tilde v_c$ is a semiconcave function that is derivable  at $\tilde \theta_0$, then $\tilde v'_c$ is continuous at $\tilde\theta_0$. Let $\mathcal D\subset \R$ be a set, the complement of which is countable, such that all functions $\tilde u_a$, $\tilde u_b$ and $\tilde v_c$ are derivable on $\mathcal D$.  By continuity of $\tilde v_c'$ at $\tilde \theta_0$, there exists $\varepsilon>0$ such that for all $\ww \Theta_0 \in \mathcal N \cap ( \tilde\theta_0 - \varepsilon, \tilde\theta_0+\varepsilon) \subset (y_0^-,y_0^+)$,
$$ \pi_1\circ \tilde f^q \big( \ww\Theta_0, c+\tilde v'_c(\ww\Theta_0)\big) -p >\ww\Theta_0,$$
where we use that $\tilde\theta_q^c = \pi_1\circ \tilde f^q \big( \tilde\theta_0, c+\tilde v'_c(\tilde\theta_0)\big)$. It follows that, setting $\ww\Theta_k =  \pi_1\circ \tilde f^k \big( \ww\Theta_0, c+\tilde v'_c(\ww\Theta_0)\big)$ for all $k\leqslant 0$, the chain $(\ww\Theta_k)_{k\leqslant 0}$ is the unique calibrating chain for $\tilde v_c$ starting at $\ww\Theta_0$ and that it also calibrates $\tilde u_a$, again using Theorem \ref{mainIrr}. As, by definition of $\mathcal D$, such a calibrating chain is also unique for $\tilde u_a$ we uncover that 
$$a+\tilde u'_a ( \ww\Theta_0) = c+\tilde v'_c(\ww\Theta_0) = \partial_2 \ww S ( \ww\Theta_{-1} , \ww\Theta_0).$$
Applying the preceding equality to an increasing sequence $(\ww\Theta_0^n)_{n\geqslant 0}$ of points in $\mathcal N \cap ( \tilde\theta_0 - \varepsilon, \tilde\theta_0+\varepsilon)$ converging to $\tilde \theta_0$ gives $a+\tilde u'_{a-}(\tilde\theta_0) = c+\tilde u'_c(\tilde\theta_0)$. Similarly, taking a decreasing sequence   $(\ww\Theta_0^n)_{n\geqslant 0}$ of points in $\mathcal N \cap ( \tilde\theta_0 - \varepsilon, \tilde\theta_0+\varepsilon)$ converging to $\tilde \theta_0$ gives $a+\tilde u'_{a+}(\tilde\theta_0) = c+\tilde u'_c(\tilde\theta_0)$. Those two equalities provide the desired result.
\end{proof}

We deduce from Proposition \ref{pendulelike}, Proposition \ref{ucderiv}  and from semiconcavity, that weak KAM solutions' pseudographs can jump downward only once from   $\PG(b+\tilde u'_b)$ to $\PG(a+\tilde u'_a)$ on each connected component of $\R\setminus\mathfrak M_{\rho_0}$.

\begin{Th}\label{KAMfaiblependule}
Assume that $\mathfrak M_{\rho_0}\neq \R$. Let $(y_0^-,y_0^+)$ be a connected component of $\R\setminus\mathfrak M_{\rho_0}$. Let $c\in (a,b)$, $v_c : \T^1 \to \R$  a weak KAM solution at cohomology $c$. Then there exists $\ww\Theta \in [y_0^-,y_0^+]$ such that 
\begin{enumerate}
\item $c+\tilde v'_c(\tilde\theta) = b+\tilde u'_b(\tilde\theta)$ for almost every $\tilde\theta\in (y_0^-,\ww\Theta)$,
\item $c+\tilde v'_c(\tilde\theta) = a+\tilde u'_a(\tilde\theta)$ for almost every $\tilde\theta\in (\ww\Theta,y_0^+)$,
\item$c+\tilde v'_{c-}(\ww\Theta) = b+\tilde u'_{b-}(\ww\Theta) $,
\item$c+\tilde v'_{c+}(\ww\Theta) = a+\tilde u'_{a+}(\ww\Theta) $.
\end{enumerate}
 In particular, 
 $$\forall \tilde \theta \in (y_0^-,\ww\Theta),\quad \tilde v_c(\tilde\theta) = \tilde v_c(y_0^-) +\int_{y_0^-}^{\tilde\theta} \tilde u'_b(s) \ \dd s+(b-c)(\tilde\theta - y_0^-);$$
  $$\forall \tilde \theta \in (\ww\Theta,y_0^+),\quad \tilde v_c(\tilde\theta) = \tilde v_c(y_0^+) +\int_{y_0^+}^{\tilde\theta} \tilde u'_a(s) \ \dd s+(a-c)(\tilde\theta - y_0^+) .$$
\end{Th}

\section{A glimpse into the world of weakly integrable twist maps}\label{lastlast}

We wish to give an account on some results originally published in \cite{AZ,AZ2} by Arnaud--Zavidovique.  We will only state them as the proofs go far beyond the scope of this memoir. The understanding of weakly integrable twist maps (Definition \ref{integrabletwist}) is a frustrating task. Indeed, as was already pointed out, there is no known example of $C^0$--integrable twist map with a non $C^1$ invariant circle. The notion appears in various historic works (even if not explicitly defined). In the study of the Hopf conjecture about Riemannian tori without conjugate points, before its definitive answer by Burago and Ivanov in \cite{BI}, it was proved by Heber (\cite{Heber}) that such geodesic flows are $C^0$--integrable. This result was then generalized to exact symplectic twist maps (\cite{CS}) and at last to more general flows of Tonelli Hamiltonians (\cite{AABZ} and also \cite{ArMin}) and twist maps in higher dimension (\cite{A}). Finally, on more general surfaces, let us mention the work \cite{MS}.

The philosophy of our results is to show that weak forms of integrability have strong dynamical implications and that further properties of the underlying foliations can be obtained. The first result completely characterizes $C^0$--integrable twist maps in terms of the function $u : \T^1 \times \R \to \R$ provided by Theorem \ref{mainAZ}.
\begin{Th}\label{AZC0}
There is equivalence between
\begin{enumerate}
\item the map $f$ is $C^0$--integrable,
\item the function $u$ is $C^1$.
\end{enumerate}
Moreover in either case the function $u$ is unique and 
\begin{itemize}
\item for each $c\in \R$, the graph $\mathcal G(c+u'_c)$ is a leaf of the invariant foliation $\mathcal F$,
\item the map $h_c : \theta\mapsto \theta + \frac{\partial u}{\partial c} (\theta,c)$ is a semi--conjugation between the projected dynamics $g_c : \theta \mapsto \pi_1\circ f \big (\theta , c+\frac{\partial u}{\partial \theta} (\theta,c)\big)$ and the rotation $R_{\rho(c)} : \theta \mapsto \theta + \rho(c)$ meaning that $h_c\circ g_c = R_{\rho(c)}\circ h_c$.
\end{itemize}
\end{Th}

The striking fact in the previous Theorem is the regularity with respect to $c$. At irrational rotation numbers, Poincaré--Denjoy theory gives that a semi--conjugation to the corresponding rotation is unique and regularity at such cohomology classes is not surprising. It is not the case at cohomology classes with a rational rotation number and the proof actually gives that the function $h_c$ is $C^{k-1}$ if $f$ is $C^k$. Moreover, previous works of Arnaud \cite{Arnaud3} yield that at such a cohomology class $c$, the invariant circle is also $C^k$ (that also follows from the Implicit Function Theorem) and the restricted dynamics is completely periodic and conjugated to a rotation (cf. Theorem \ref{diffR}). 

The main (hypothetical) feature of a $C^0$--integrable twist map that would not be integrable is the presence of invariant circles with irrational rotation number and a restricted dynamics that is one of a Denjoy counterexample. This means that if $c$ is the corresponding cohomology class, the map $h_c$ is not a homeomorphism.  This is excluded in the case of Lipschitz--integrable twist maps (see Definition \ref{Lipfol}) by the next Theorem.

\begin{Th}\label{AZLip}
Assume that the Exact Conservative Twist Map is Lipschitz integrable. Then there exists an exact area preserving homeomorphism $\Phi $ of $ \T^1\times \R$, which is $C^1$ in the variable $\theta$, such that 
$$\forall (x,c) \in \T^1 \times \R, \quad \Phi \circ f \circ \Phi^{-1} (x,c) = (x + \rho(c) ,c).$$
Moreover, in this case, $\rho : \R \to \R$ is a bi--Lipschitz homeomorphism, all the leaves of the invariant foliation are $C^1$ and the restricted dynamics on each leaf is $C^1$--conjugated to a rotation.

The function $\Phi$ is implicitly defined by the relation
$$\Phi\Big( \theta, c+\frac{\partial u}{\partial \theta}(\theta,c)\Big) = \Big( \theta + \frac{\partial u}{\partial c}(\theta,c),c\Big).$$
\end{Th}

In the previous Theorem, the area preserving homeomorphism $\Phi$ maps the foliation $\mathcal F$ invariant by $f$ to the the standard foliation $\mathcal F^0$ consisting of the obvious circles $ \mathcal F^0_c = \{(\theta,c) , \ \theta\in \T^1\}$. A natural problem is therefore  to find which foliations by graphs are homeomorphic  to the standard foliation by an exact area preserving homeomorphism. On this matter, we provide the following characterization.

\begin{Th}\label{folfol}
Let $\mathcal F$ be a foliation of $\T^1 \times \R$ by graphs of functions $\theta \mapsto \eta_c ( \theta)$ such that for $c\in \R$, $\int_{\T^1}\eta_c(\theta)\ \dd \theta = c$. Then $\mathcal F$ is homeomorphic to the standard foliation $\mathcal F^0$ by an exact area preserving homeomorphism if and only if there exists a $C^1$ function $u : \T^1 \times \R \to \R$ such that
\begin{itemize}
\item $u(0,c) = 0$ for all $c\in \R$,
\item $\eta_c(\theta) = c + \frac{\partial u}{\partial \theta}(\theta,c)$ for all $(\theta,c) \in \T^1 \times \R$,
\item for all $c\in \R$, the map $\theta \mapsto \theta + \frac{\partial u}{\partial c}(\theta,c)$ is a homeomorphism of $\T^1$.

\end{itemize}
Again, an area preserving homeomorphism $\Phi$ sending $\mathcal F$ to $\mathcal F^0$ is implicitly defined by 
 the relation
$$\Phi\Big( \theta, c+\frac{\partial u}{\partial \theta}(\theta,c)\Big) = \Big( \theta + \frac{\partial u}{\partial c}(\theta,c),c\Big).$$

\end{Th}

As a conclusion, the previous Theorem allows to explain Theorem \ref{folpasinv}. In the case of the foliation given by the functions $\eta_c(\theta) = c+\varepsilon(c)\cos(2\pi \theta)$, for a function $\varepsilon : \R \to \R$ which is Lipschitz, non $C^1$, with a small enough Lipschitz constant, the foliation is Lipschitz in the sense of Definition \ref{Lipfol}. Moreover, were this foliation straightened by an exact area preserving homeomorphism the associated function $u$ would be given by 
$$\forall (\theta,c)\in \T^1\times \R, \quad u(\theta,c) = \frac{\varepsilon(c)}{2\pi}\sin(2\pi\theta). $$
This last function is clearly not $C^1$, thus violating the conclusion of Theorem \ref{folfol}.

We conclude that the foliation given by $\eta$ cannot be invariant by an ECTM.

\newpage
\newpage
\mbox{}
\newpage

\bibliography{synthese}

\begin{thebibliography}{100}

\bibitem{IsNe}
{\sc E.~S. Al-Aidarous, E.~O. Alzahrani, H.~Ishii, and A.~M.~M. Younas}, {\em
  Asymptotic analysis for the eikonal equation with the dynamical boundary
  conditions}, Math. Nachr., 287 (2014), pp.~1563--1588.

\bibitem{A}
{\sc M.~Arcostanzo}, {\em The {$C^0$} integrability of symplectic twist maps
  without conjugate points}, Ergodic Theory Dynam. Systems, 41 (2021),
  pp.~48--65.

\bibitem{AABZ}
{\sc M.~Arcostanzo, M.-C. Arnaud, P.~Bolle, and M.~Zavidovique}, {\em Tonelli
  {H}amiltonians without conjugate points and {$C^0$}--integrability}, Math.
  Z., 280 (2015), pp.~165--194.

\bibitem{Argreen}
{\sc M.-C. Arnaud}, {\em Fibr\'{e}s de {G}reen et r\'{e}gularit\'{e} des
  graphes {$C^0$}-lagrangiens invariants par un flot de {T}onelli}, Ann. Henri
  Poincar\'{e}, 9 (2008), pp.~881--926.

\bibitem{Arnaud3}
\leavevmode\vrule height 2pt depth -1.6pt width 23pt, {\em Three results on the
  regularity of the curves that are invariant by an exact symplectic twist
  map}, Publ. Math. Inst. Hautes \'{E}tudes Sci.,  (2009), pp.~1--17.

\bibitem{ArBir}
\leavevmode\vrule height 2pt depth -1.6pt width 23pt, {\em On a theorem due to
  {B}irkhoff}, Geom. Funct. Anal., 20 (2010), pp.~1307--1316.

\bibitem{A11}
\leavevmode\vrule height 2pt depth -1.6pt width 23pt, {\em A nondifferentiable
  essential irrational invariant curve for a {$C^1$} symplectic twist map}, J.
  Mod. Dyn., 5 (2011), pp.~583--591.

\bibitem{ArMin}
{\sc M.-C. Arnaud}, {\em A particular minimization property implies
  {$C^0$}-integrability}, J. Differential Equations, 250 (2011),
  pp.~2389--2401.

\bibitem{APseudo}
{\sc M.-C. Arnaud}, {\em Pseudographs and the {L}ax-{O}leinik semi-group: a
  geometric and dynamical interpretation}, Nonlinearity, 24 (2011), pp.~71--78.

\bibitem{ArLy}
\leavevmode\vrule height 2pt depth -1.6pt width 23pt, {\em Green bundles,
  {L}yapunov exponents and regularity along the supports of the minimizing
  measures}, Ann. Inst. H. Poincar\'{e} Anal. Non Lin\'{e}aire, 29 (2012),
  pp.~989--1007.

\bibitem{A13}
\leavevmode\vrule height 2pt depth -1.6pt width 23pt, {\em Boundaries of
  instability zones for symplectic twist maps}, J. Inst. Math. Jussieu, 13
  (2014), pp.~19--41.

\bibitem{Atwist}
\leavevmode\vrule height 2pt depth -1.6pt width 23pt, {\em Hyperbolicity for
  conservative twist maps of the 2-dimensional annulus}, Publ. Mat. Urug., 16
  (2016), pp.~1--39.

\bibitem{ArSu}
{\sc M.-C. Arnaud and X.~Su}, {\em {On the $C^1$ and $C^2$-convergence to weak
  {K.A.M.} solutions}}.
\newblock working paper or preprint, Feb. 2019.

\bibitem{AZ}
{\sc M.-C. Arnaud and M.~Zavidovique}, {\em On the transversal dependence of
  weak {K.A.M.} solutions for symplectic twist maps}, arXiv, 1809.02372 (2018).

\bibitem{AZ2}
{\sc M.-C. Arnaud and M.~Zavidovique}, {\em Actions of symplectic
  homeomorphisms/diffeomorphisms on foliations by curves in dimension 2},
  Ergodic Theory Dynam. Systems, 43 (2023), pp.~794--826.

\bibitem{AZ3}
{\sc M.-C. Arnaud and M.~Zavidovique}, {\em Weak {KAM} solutions and minimizing
  orbits of twist maps}, Transactions of the AMS,  (2023).

\bibitem{Arnold}
{\sc V.~I. Arnol'~d}, {\em Instability of dynamical systems with many degrees
  of freedom}, Dokl. Akad. Nauk SSSR, 156 (1964), pp.~9--12.

\bibitem{Aubry}
{\sc S.~Aubry and P.~Y. Le~Daeron}, {\em The discrete {F}renkel-{K}ontorova
  model and its extensions. {I}. {E}xact results for the ground-states}, Phys.
  D, 8 (1983), pp.~381--422.

\bibitem{Audin}
{\sc M.~Audin and M.~Damian}, {\em Th\'{e}orie de {M}orse et homologie de
  {F}loer}, Savoirs Actuels (Les Ulis). [Current Scholarship (Les Ulis)], EDP
  Sciences, Les Ulis; CNRS \'{E}ditions, Paris, 2010.

\bibitem{avila}
{\sc A.~Avila and B.~Fayad}, {\em Non-differentiable irrational curves for
  {$C^1$} twist map}, Ergodic Theory Dynam. Systems, 42 (2022), pp.~491--499.

\bibitem{Bang}
{\sc V.~Bangert}, {\em Mather sets for twist maps and geodesics on tori}, in
  Dynamics reported, {V}ol. 1, vol.~1 of Dynam. Report. Ser. Dynam. Systems
  Appl., Wiley, Chichester, 1988, pp.~1--56.

\bibitem{BangR}
\leavevmode\vrule height 2pt depth -1.6pt width 23pt, {\em Geodesic rays,
  {B}usemann functions and monotone twist maps}, Calc. Var. Partial
  Differential Equations, 2 (1994), pp.~49--63.

\bibitem{BaCa}
{\sc M.~Bardi and I.~Capuzzo-Dolcetta}, {\em Optimal control and viscosity
  solutions of {H}amilton-{J}acobi-{B}ellman equations}, Systems \& Control:
  Foundations \& Applications, Birkh\"{a}user Boston, Inc., Boston, MA, 1997.
\newblock With appendices by Maurizio Falcone and Pierpaolo Soravia.

\bibitem{barles}
{\sc G.~Barles}, {\em Solutions de viscosit\'e des \'equations de
  {H}amilton-{J}acobi}, vol.~17 of Math\'ematiques \& Applications (Berlin)
  [Mathematics \& Applications], Springer-Verlag, Paris, 1994.

\bibitem{basou}
{\sc G.~Barles and P.~E. Souganidis}, {\em Convergence of approximation schemes
  for fully nonlinear second order equations}, Asymptotic Anal., 4 (1991),
  pp.~271--283.

\bibitem{BS}
{\sc G.~Barles and P.~E. Souganidis}, {\em On the large time behavior of
  solutions of {H}amilton-{J}acobi equations}, SIAM J. Math. Anal., 31 (2000),
  pp.~925--939 (electronic).

\bibitem{BS2}
{\sc G.~Barles and P.~E. Souganidis}, {\em Some counterexamples on the
  asymptotic behavior of the solutions of {H}amilton-{J}acobi equations}, C. R.
  Acad. Sci. Paris S\'{e}r. I Math., 330 (2000), pp.~963--968.

\bibitem{BJ}
{\sc E.~N. Barron and R.~Jensen}, {\em Semicontinuous viscosity solutions for
  {H}amilton-{J}acobi equations with convex {H}amiltonians}, Comm. Partial
  Differential Equations, 15 (1990), pp.~1713--1742.

\bibitem{BJ2}
\leavevmode\vrule height 2pt depth -1.6pt width 23pt, {\em Optimal control and
  semicontinuous viscosity solutions}, Proc. Amer. Math. Soc., 113 (1991),
  pp.~397--402.

\bibitem{Ber2}
{\sc P.~Berger and D.~Turaev}, {\em On {H}erman's positive entropy conjecture},
  Adv. Math., 349 (2019), pp.~1234--1288.

\bibitem{Ber1}
{\sc P.~Berger and J.-C. Yoccoz}, {\em Strong regularity}, Soci\'{e}t\'{e}
  Math\'{e}matique de France, Paris, 2019.
\newblock Ast\'{e}risque No. 410 (2019) (2019).

\bibitem{BerAIF}
{\sc P.~Bernard}, {\em Connecting orbits of time dependent {L}agrangian
  systems}, Ann. Inst. Fourier (Grenoble), 52 (2002), pp.~1533--1568.

\bibitem{BernardC11}
\leavevmode\vrule height 2pt depth -1.6pt width 23pt, {\em Existence of
  {$C^{1,1}$} critical sub-solutions of the {H}amilton-{J}acobi equation on
  compact manifolds}, Ann. Sci. \'Ecole Norm. Sup. (4), 40 (2007),
  pp.~445--452.

\bibitem{BerPseudo}
\leavevmode\vrule height 2pt depth -1.6pt width 23pt, {\em The dynamics of
  pseudographs in convex {H}amiltonian systems}, J. Amer. Math. Soc., 21
  (2008), pp.~615--669.

\bibitem{Ber}
\leavevmode\vrule height 2pt depth -1.6pt width 23pt, {\em Large normally
  hyperbolic cylinders in a priori stable {H}amiltonian systems}, Ann. Henri
  Poincar\'{e}, 11 (2010), pp.~929--942.

\bibitem{BeIl}
\leavevmode\vrule height 2pt depth -1.6pt width 23pt, {\em Lasry-{L}ions
  regularization and a lemma of {I}lmanen}, Rend. Semin. Mat. Univ. Padova, 124
  (2010), pp.~221--229.

\bibitem{BerARMA}
\leavevmode\vrule height 2pt depth -1.6pt width 23pt, {\em On the number of
  {M}ather measures of {L}agrangian systems}, Arch. Ration. Mech. Anal., 197
  (2010), pp.~1011--1031.

\bibitem{BeBu2}
{\sc P.~Bernard and B.~Buffoni}, {\em The {M}onge problem for supercritical
  {M}a\~{n}\'{e} potentials on compact manifolds}, Adv. Math., 207 (2006),
  pp.~691--706.

\bibitem{BeBu1}
\leavevmode\vrule height 2pt depth -1.6pt width 23pt, {\em Optimal mass
  transportation and {M}ather theory}, J. Eur. Math. Soc. (JEMS), 9 (2007),
  pp.~85--121.

\bibitem{BeBu}
\leavevmode\vrule height 2pt depth -1.6pt width 23pt, {\em Weak {KAM} pairs and
  {M}onge-{K}antorovich duality}, in Asymptotic analysis and
  singularities---elliptic and parabolic {PDE}s and related problems, vol.~47
  of Adv. Stud. Pure Math., Math. Soc. Japan, Tokyo, 2007, pp.~397--420.

\bibitem{BeCo}
{\sc P.~Bernard and G.~Contreras}, {\em A generic property of families of
  {L}agrangian systems}, Ann. of Math. (2), 167 (2008), pp.~1099--1108.

\bibitem{BKZ}
{\sc P.~Bernard, V.~Kaloshin, and K.~Zhang}, {\em Arnold diffusion in arbitrary
  degrees of freedom and normally hyperbolic invariant cylinders}, Acta Math.,
  217 (2016), pp.~1--79.

\bibitem{BR}
{\sc P.~Bernard and J.-M. Roquejoffre}, {\em Convergence to time-periodic
  solutions in time-periodic {H}amilton-{J}acobi equations on the circle},
  Comm. Partial Differential Equations, 29 (2004), pp.~457--469.

\bibitem{BS3}
{\sc P.~Bernard and S.~Suhr}, {\em Lyapounov functions of closed cone fields:
  from {C}onley theory to time functions}, Comm. Math. Phys., 359 (2018),
  pp.~467--498.

\bibitem{BS4}
\leavevmode\vrule height 2pt depth -1.6pt width 23pt, {\em Cauchy and uniform
  temporal functions of globally hyperbolic cone fields}, Proc. Amer. Math.
  Soc., 148 (2020), pp.~4951--4966.

\bibitem{BeZa}
{\sc P.~Bernard and M.~Zavidovique}, {\em Regularization of subsolutions in
  discrete weak {KAM} theory}, Canad. J. Math., 65 (2013), pp.~740--756.

\bibitem{BF2}
{\sc O.~Bernardi and A.~Florio}, {\em A {C}onley-type decomposition of the
  strong chain recurrent set}, Ergodic Theory Dynam. Systems, 39 (2019),
  pp.~1261--1274.

\bibitem{BF}
\leavevmode\vrule height 2pt depth -1.6pt width 23pt, {\em Existence of
  {L}ipschitz continuous {L}yapunov functions strict outside the strong chain
  recurrent set}, Dyn. Syst., 34 (2019), pp.~71--92.

\bibitem{BFW}
{\sc O.~Bernardi, A.~Florio, and J.~Wiseman}, {\em The generalized recurrent
  set, explosions and {L}yapunov functions}, J. Dynam. Differential Equations,
  32 (2020), pp.~1797--1817.

\bibitem{Birk}
{\sc G.~D. Birkhoff}, {\em Sur quelques courbes ferm\'{e}es remarquables},
  Bull. Soc. Math. France, 60 (1932), pp.~1--26.

\bibitem{BFEZ}
{\sc A.~Bouillard, E.~Faou, and M.~Zavidovique}, {\em Fast weak-{KAM}
  integrators for separable {H}amiltonian systems}, Math. Comp., 85 (2016),
  pp.~85--117.

\bibitem{BI}
{\sc D.~Burago and S.~Ivanov}, {\em Riemannian tori without conjugate points
  are flat}, Geom. Funct. Anal., 4 (1994), pp.~259--269.

\bibitem{CGMT}
{\sc F.~Cagnetti, D.~Gomes, H.~Mitake, and H.~V. Tran}, {\em A new method for
  large time behavior of degenerate viscous {H}amilton-{J}acobi equations with
  convex {H}amiltonians}, Ann. Inst. H. Poincar\'{e} Anal. Non Lin\'{e}aire, 32
  (2015), pp.~183--200.

\bibitem{CamLey}
{\sc F.~Camilli, O.~Ley, and P.~Loreti}, {\em Homogenization of monotone
  systems of {H}amilton-{J}acobi equations}, ESAIM Control Optim. Calc. Var.,
  16 (2010), pp.~58--76.

\bibitem{Ley}
{\sc F.~Camilli, O.~Ley, P.~Loreti, and V.~D. Nguyen}, {\em Large time behavior
  of weakly coupled systems of first-order {H}amilton-{J}acobi equations},
  NoDEA Nonlinear Differential Equations Appl., 19 (2012), pp.~719--749.

\bibitem{CanCh}
{\sc P.~Cannarsa and W.~Cheng}, {\em Generalized characteristics and
  {L}ax-{O}leinik operators: global theory}, Calc. Var. Partial Differential
  Equations, 56 (2017), pp.~Paper No. 125, 31.

\bibitem{CCF}
{\sc P.~Cannarsa, W.~Cheng, and A.~Fathi}, {\em On the topology of the set of
  singularities of a solution to the {H}amilton-{J}acobi equation}, C. R. Math.
  Acad. Sci. Paris, 355 (2017), pp.~176--180.

\bibitem{CCF2}
\leavevmode\vrule height 2pt depth -1.6pt width 23pt, {\em Singularities of
  solutions of time dependent {H}amilton-{J}acobi equations. {A}pplications to
  {R}iemannian geometry}, Publ. Math. Inst. Hautes \'{E}tudes Sci., 133 (2021),
  pp.~327--366.

\bibitem{CCLW}
{\sc P.~Cannarsa, W.~Cheng, L.~Jin, K.~Wang, and J.~Yan}, {\em Herglotz'
  variational principle and {L}ax-{O}leinik evolution}, J. Math. Pures Appl.
  (9), 141 (2020), pp.~99--136.

\bibitem{CaSi00}
{\sc P.~Cannarsa and C.~Sinestrari}, {\em Semiconcave functions,
  {H}amilton-{J}acobi equations, and optimal control}, Progress in Nonlinear
  Differential Equations and their Applications, 58, Birkh\"auser Boston Inc.,
  Boston, MA, 2004.

\bibitem{CP}
{\sc P.~Cardaliaguet and A.~Porretta}, {\em Long time behavior of the master
  equation in mean field game theory}, Anal. PDE, 12 (2019), pp.~1397--1453.

\bibitem{carneiro}
{\sc M.~J.~D. Carneiro}, {\em On minimizing measures of the action of
  autonomous {L}agrangians}, Nonlinearity, 8 (1995), pp.~1077--1085.

\bibitem{Chen}
{\sc Q.~Chen}, {\em Convergence of solutions of {H}amilton–{J}acobi equations
  depending nonlinearly on the unknown function}, Advances in Calculus of
  Variations,  (2021).

\bibitem{CCIZ}
{\sc Q.~Chen, W.~Cheng, H.~Ishii, and K.~Zhao}, {\em Vanishing contact
  structure problem and convergence of the viscosity solutions}, Comm. Partial
  Differential Equations, 44 (2019), pp.~801--836.

\bibitem{ZQJ}
{\sc Q.~Chen, A.~Fathi, M.~Zavidovique, and J.~Zhang}, {\em Convergence of the
  solutions of the nonlinear discounted hamilton-jacobi equation: The central
  role of mather measures}, 2023.

\bibitem{CX}
{\sc C.-Q. Cheng and J.~Xue}, {\em Order property and modulus of continuity of
  weak {KAM} solutions}, Calc. Var. Partial Differential Equations, 57 (2018),
  pp.~Paper No. 65, 27.

\bibitem{CS}
{\sc J.~Cheng and Y.~Sun}, {\em A necessary and sufficient condition for a
  twist map being integrable}, Sci. China Ser. A, 39 (1996), pp.~709--717.

\bibitem{clarke}
{\sc F.~Clarke}, {\em Functional analysis, calculus of variations and optimal
  control}, vol.~264 of Graduate Texts in Mathematics, Springer, London, 2013.

\bibitem{Contreras}
{\sc G.~Contreras}, {\em Ground states are generically a periodic orbit},
  Invent. Math., 205 (2016), pp.~383--412.

\bibitem{CFR}
{\sc G.~Contreras, A.~Figalli, and L.~Rifford}, {\em Generic hyperbolicity of
  {A}ubry sets on surfaces}, Invent. Math., 200 (2015), pp.~201--261.

\bibitem{CIS}
{\sc G.~Contreras, R.~Iturriaga, and A.~Siconolfi}, {\em Homogenization on
  arbitrary manifolds}, Calc. Var. Partial Differential Equations, 52 (2015),
  pp.~237--252.

\bibitem{CrIsLi}
{\sc M.~G. Crandall, H.~Ishii, and P.-L. Lions}, {\em Uniqueness of viscosity
  solutions of {H}amilton-{J}acobi equations revisited}, J. Math. Soc. Japan,
  39 (1987), pp.~581--596.

\bibitem{CrIsLiUG}
\leavevmode\vrule height 2pt depth -1.6pt width 23pt, {\em User's guide to
  viscosity solutions of second order partial differential equations}, Bull.
  Amer. Math. Soc. (N.S.), 27 (1992), pp.~1--67.

\bibitem{CrLi}
{\sc M.~G. Crandall and P.-L. Lions}, {\em Viscosity solutions of
  {H}amilton-{J}acobi equations}, Trans. Amer. Math. Soc., 277 (1983),
  pp.~1--42.

\bibitem{DFIZ2}
{\sc A.~Davini, A.~Fathi, R.~Iturriaga, and M.~Zavidovique}, {\em Convergence
  of the solutions of the discounted equation: the discrete case}, Math. Z.,
  284 (2016), pp.~1021--1034.

\bibitem{DFIZ}
\leavevmode\vrule height 2pt depth -1.6pt width 23pt, {\em Convergence of the
  solutions of the discounted {H}amilton-{J}acobi equation: convergence of the
  discounted solutions}, Invent. Math., 206 (2016), pp.~29--55.

\bibitem{DS}
{\sc A.~Davini and A.~Siconolfi}, {\em A generalized dynamical approach to the
  large time behavior of solutions of {H}amilton-{J}acobi equations}, SIAM J.
  Math. Anal., 38 (2006), pp.~478--502 (electronic).

\bibitem{DSZ}
{\sc A.~Davini, A.~Siconolfi, and M.~Zavidovique}, {\em Random {L}ax-{O}leinik
  semigroups for {H}amilton-{J}acobi systems}, J. Math. Pures Appl. (9), 120
  (2018), pp.~294--333.

\bibitem{DW}
{\sc A.~Davini and L.~Wang}, {\em On the vanishing discount problem from the
  negative direction}, Discrete \& Continuous Dynamical Systems - A, 41 (2021),
  p.~2377.

\bibitem{DZSys}
{\sc A.~Davini and M.~Zavidovique}, {\em Aubry sets for weakly coupled systems
  of {H}amilton-{J}acobi equations}, SIAM J. Math. Anal., 46 (2014),
  pp.~3361--3389.

\bibitem{DZJDE}
\leavevmode\vrule height 2pt depth -1.6pt width 23pt, {\em On the (non)
  existence of viscosity solutions of multi-time {H}amilton-{J}acobi
  equations}, J. Differential Equations, 258 (2015), pp.~362--378.

\bibitem{DZdiscsys}
\leavevmode\vrule height 2pt depth -1.6pt width 23pt, {\em Convergence of the
  solutions of discounted {H}amilton-{J}acobi systems}, Adv. Calc. Var., 14
  (2021), pp.~193--206.

\bibitem{dug}
{\sc J.~Dugundji}, {\em Topology}, Allyn and Bacon Inc., Boston, Mass., 1966.

\bibitem{Duist}
{\sc J.~J. Duistermaat}, {\em On global action-angle coordinates}, Comm. Pure
  Appl. Math., 33 (1980), pp.~687--706.

\bibitem{Ev}
{\sc L.~C. Evans}, {\em The perturbed test function method for viscosity
  solutions of nonlinear {PDE}}, Proc. Roy. Soc. Edinburgh Sect. A, 111 (1989),
  pp.~359--375.

\bibitem{evadj}
\leavevmode\vrule height 2pt depth -1.6pt width 23pt, {\em Adjoint and
  compensated compactness methods for {H}amilton-{J}acobi {PDE}}, Arch. Ration.
  Mech. Anal., 197 (2010), pp.~1053--1088.

\bibitem{evans}
\leavevmode\vrule height 2pt depth -1.6pt width 23pt, {\em Partial differential
  equations}, American Mathematical Society, Providence, R.I., 2010.

\bibitem{Fa3}
{\sc A.~Fathi}, {\em Solutions {KAM} faibles conjugu\'ees et barri\`eres de
  {P}eierls}, C. R. Acad. Sci. Paris S\'er. I Math., 325 (1997), pp.~649--652.

\bibitem{Fa4}
\leavevmode\vrule height 2pt depth -1.6pt width 23pt, {\em Th\'eor\`eme {KAM}
  faible et th\'eorie de {M}ather sur les syst\`emes lagrangiens}, C. R. Acad.
  Sci. Paris S\'er. I Math., 324 (1997), pp.~1043--1046.

\bibitem{Fa5}
\leavevmode\vrule height 2pt depth -1.6pt width 23pt, {\em Orbites
  h\'et\'eroclines et ensemble de {P}eierls}, C. R. Acad. Sci. Paris S\'er. I
  Math., 326 (1998), pp.~1213--1216.

\bibitem{Fa1}
\leavevmode\vrule height 2pt depth -1.6pt width 23pt, {\em Sur la convergence
  du semi-groupe de {L}ax-{O}leinik}, C. R. Acad. Sci. Paris S\'er. I Math.,
  327 (1998), pp.~267--270.

\bibitem{Fa2}
\leavevmode\vrule height 2pt depth -1.6pt width 23pt, {\em Regularity of {$C\sp
  1$} solutions of the {H}amilton-{J}acobi equation}, Ann. Fac. Sci. Toulouse
  Math. (6), 12 (2003), pp.~479--516.

\bibitem{Fa}
\leavevmode\vrule height 2pt depth -1.6pt width 23pt, {\em Weak {KAM} {T}heorem
  in {L}agrangian {D}ynamics, preliminary version 10, {L}yon}.
\newblock unpublished, June 15 2008.

\bibitem{FaFi}
{\sc A.~Fathi and A.~Figalli}, {\em Optimal transportation on non-compact
  manifolds}, Israel J. Math., 175 (2010), pp.~1--59.

\bibitem{FFR}
{\sc A.~Fathi, A.~Figalli, and L.~Rifford}, {\em On the {H}ausdorff dimension
  of the {M}ather quotient}, Comm. Pure Appl. Math., 62 (2009), pp.~445--500.

\bibitem{FaMat}
{\sc A.~Fathi and J.~N. Mather}, {\em Failure of convergence of the
  {L}ax-{O}leinik semi-group in the time-periodic case}, Bull. Soc. Math.
  France, 128 (2000), pp.~473--483.

\bibitem{FaPa}
{\sc A.~Fathi and P.~Pageault}, {\em Aubry-{M}ather theory for homeomorphisms},
  Ergodic Theory Dynam. Systems, 35 (2015), pp.~1187--1207.

\bibitem{FaPa2}
\leavevmode\vrule height 2pt depth -1.6pt width 23pt, {\em Smoothing {L}yapunov
  functions}, Trans. Amer. Math. Soc., 371 (2019), pp.~1677--1700.

\bibitem{FSC1}
{\sc A.~Fathi and A.~Siconolfi}, {\em Existence of {$C^1$} critical
  subsolutions of the {H}amilton-{J}acobi equation}, Invent. Math., 155 (2004),
  pp.~363--388.

\bibitem{FS05}
\leavevmode\vrule height 2pt depth -1.6pt width 23pt, {\em P{DE} aspects of
  {A}ubry-{M}ather theory for quasiconvex {H}amiltonians}, Calc. Var. Partial
  Differential Equations, 22 (2005), pp.~185--228.

\bibitem{FScone}
\leavevmode\vrule height 2pt depth -1.6pt width 23pt, {\em On smooth time
  functions}, Math. Proc. Cambridge Philos. Soc., 152 (2012), pp.~303--339.

\bibitem{FaZaIl}
{\sc A.~Fathi and M.~Zavidovique}, {\em Ilmanen's lemma on insertion of
  {$C^{1,1}$} functions}, Rend. Semin. Mat. Univ. Padova, 124 (2010),
  pp.~203--219.

\bibitem{FR1}
{\sc A.~Figalli and L.~Rifford}, {\em Closing {A}ubry sets {I}}, Comm. Pure
  Appl. Math., 68 (2015), pp.~210--285.

\bibitem{FR2}
\leavevmode\vrule height 2pt depth -1.6pt width 23pt, {\em Closing {A}ubry sets
  {II}}, Comm. Pure Appl. Math., 68 (2015), pp.~345--412.

\bibitem{GLT}
{\sc E.~Garibaldi, A.~O. Lopes, and P.~Thieullen}, {\em On calibrated and
  separating sub-actions}, Bull. Braz. Math. Soc. (N.S.), 40 (2009),
  pp.~577--602.

\bibitem{GT}
{\sc E.~Garibaldi and P.~Thieullen}, {\em An ergodic description of ground
  states}, J. Stat. Phys., 158 (2015), pp.~359--371.

\bibitem{Kirk}
{\sc K.~Goebel and W.~A. Kirk}, {\em Topics in metric fixed point theory},
  vol.~28 of Cambridge Studies in Advanced Mathematics, Cambridge University
  Press, Cambridge, 1990.

\bibitem{Gole}
{\sc C.~Gol\'{e}}, {\em Symplectic twist maps}, vol.~18 of Advanced Series in
  Nonlinear Dynamics, World Scientific Publishing Co., Inc., River Edge, NJ,
  2001.
\newblock Global variational techniques.

\bibitem{Gom}
{\sc D.~A. Gomes}, {\em Generalized {M}ather problem and selection principles
  for viscosity solutions and {M}ather measures}, Adv. Calc. Var., 1 (2008),
  pp.~291--307.

\bibitem{Heber}
{\sc J.~Heber}, {\em On the geodesic flow of tori without conjugate points},
  Math. Z., 216 (1994), pp.~209--216.

\bibitem{Her1}
{\sc M.-R. Herman}, {\em Sur les courbes invariantes par les
  diff\'{e}omorphismes de l'anneau. {V}ol. 1}, vol.~103 of Ast\'{e}risque,
  Soci\'{e}t\'{e} Math\'{e}matique de France, Paris, 1983.
\newblock With an appendix by Albert Fathi, With an English summary.

\bibitem{Her2}
\leavevmode\vrule height 2pt depth -1.6pt width 23pt, {\em Sur les courbes
  invariantes par les diff\'{e}omorphismes de l'anneau. {V}ol. 2},
  Ast\'{e}risque,  (1986), p.~248.
\newblock With a correction to: {{\i}t On the curves invariant under
  diffeomorphisms of the annulus, Vol. 1} (French) [Ast\'{e}risque No. 103-104,
  Soc. Math. France, Paris, 1983; MR0728564 (85m:58062)].

\bibitem{ISZ}
{\sc H.~Ibrahim, A.~Siconolfi, and S.~Zabad}, {\em Cycle characterization of
  the {A}ubry set for weakly coupled {H}amilton-{J}acobi systems}, Commun.
  Contemp. Math., 20 (2018), pp.~1750095, 28.

\bibitem{Ilmanen}
{\sc T.~Ilmanen}, {\em The level-set flow on a manifold}, in Differential
  geometry: partial differential equations on manifolds ({L}os {A}ngeles, {CA},
  1990), vol.~54 of Proc. Sympos. Pure Math., Amer. Math. Soc., Providence, RI,
  1993, pp.~193--204.

\bibitem{IsDUKE}
{\sc H.~Ishii}, {\em Perron's method for {H}amilton-{J}acobi equations}, Duke
  Math. J., 55 (1987), pp.~369--384.

\bibitem{IsCPAM}
\leavevmode\vrule height 2pt depth -1.6pt width 23pt, {\em On uniqueness and
  existence of viscosity solutions of fully nonlinear second-order elliptic
  {PDE}s}, Comm. Pure Appl. Math., 42 (1989), pp.~15--45.

\bibitem{Ishii08}
\leavevmode\vrule height 2pt depth -1.6pt width 23pt, {\em Asymptotic solutions
  for large time of {H}amilton-{J}acobi equations in {E}uclidean {$n$} space},
  Ann. Inst. H. Poincar\'e Anal. Non Lin\'eaire, 25 (2008), pp.~231--266.

\bibitem{IsCE}
{\sc H.~Ishii}, {\em An example in the vanishing discount problem for monotone
  systems of {H}amilton-{J}acobi equations}, 2020.

\bibitem{IsSys}
{\sc H.~Ishii}, {\em The vanishing discount problem for monotone systems of
  {H}amilton-{J}acobi equations. {P}art 1: linear coupling}, Math. Eng., 3
  (2021), pp.~Paper No. 032, 21.

\bibitem{IsSys2}
{\sc H.~Ishii and L.~Jin}, {\em The vanishing discount problem for monotone
  systems of {H}amilton-{J}acobi equations: part 2---nonlinear coupling}, Calc.
  Var. Partial Differential Equations, 59 (2020), pp.~Paper No. 140, 28.

\bibitem{IsLi}
{\sc H.~Ishii and P.-L. Lions}, {\em Viscosity solutions of fully nonlinear
  second-order elliptic partial differential equations}, J. Differential
  Equations, 83 (1990), pp.~26--78.

\bibitem{IMT1}
{\sc H.~Ishii, H.~Mitake, and H.~V. Tran}, {\em The vanishing discount problem
  and viscosity {M}ather measures. {P}art 1: {T}he problem on a torus}, J.
  Math. Pures Appl. (9), 108 (2017), pp.~125--149.

\bibitem{IMT2}
\leavevmode\vrule height 2pt depth -1.6pt width 23pt, {\em The vanishing
  discount problem and viscosity {M}ather measures. {P}art 2: {B}oundary value
  problems}, J. Math. Pures Appl. (9), 108 (2017), pp.~261--305.

\bibitem{IsSi}
{\sc H.~Ishii and A.~Siconolfi}, {\em The vanishing discount problem for
  {H}amilton-{J}acobi equations in the {E}uclidean space}, Comm. Partial
  Differential Equations, 45 (2020), pp.~525--560.

\bibitem{IS}
{\sc R.~Iturriaga and H.~S\'{a}nchez-Morgado}, {\em Limit of the infinite
  horizon discounted {H}amilton-{J}acobi equation}, Discrete Contin. Dyn. Syst.
  Ser. B, 15 (2011), pp.~623--635.

\bibitem{Kan}
{\sc L.~Kantorovitch}, {\em On the translocation of masses}, C. R. (Doklady)
  Acad. Sci. URSS (N.S.), 37 (1942), pp.~199--201.

\bibitem{KO}
{\sc Y.~Katznelson and D.~S. Ornstein}, {\em Twist maps and {A}ubry-{M}ather
  sets}, in Lipa's legacy ({N}ew {Y}ork, 1995), vol.~211 of Contemp. Math.,
  Amer. Math. Soc., Providence, RI, 1997, pp.~343--357.

\bibitem{Kirk2}
{\sc W.~A. Kirk and B.~Sims}, eds., {\em Handbook of metric fixed point
  theory}, Kluwer Academic Publishers, Dordrecht, 2001.

\bibitem{LL}
{\sc J.-M. Lasry and P.-L. Lions}, {\em A remark on regularization in {H}ilbert
  spaces}, Israel J. Math., 55 (1986), pp.~257--266.

\bibitem{LeC}
{\sc P.~Le~Calvez}, {\em Propri\'{e}t\'{e}s des attracteurs de {B}irkhoff},
  Ergodic Theory Dynam. Systems, 8 (1988), pp.~241--310.

\bibitem{Li}
{\sc P.-L. Lions}, {\em Generalized solutions of {H}amilton-{J}acobi
  equations}, vol.~69 of Research Notes in Mathematics, Pitman (Advanced
  Publishing Program), Boston, Mass.-London, 1982.

\bibitem{LPV}
{\sc P.-L. Lions, G.~Papanicolaou, and S.~Varadhan}, {\em Homogenization of
  {H}amilton-{J}acobi equation}.
\newblock unpublished preprint, 1987.

\bibitem{MClosed}
{\sc R.~Ma\~{n}\'{e}}, {\em On the minimizing measures of {L}agrangian
  dynamical systems}, Nonlinearity, 5 (1992), pp.~623--638.

\bibitem{Mane}
\leavevmode\vrule height 2pt depth -1.6pt width 23pt, {\em Generic properties
  and problems of minimizing measures of {L}agrangian systems}, Nonlinearity, 9
  (1996), pp.~273--310.

\bibitem{Massart}
{\sc D.~Massart}, {\em On {A}ubry sets and {M}ather's action functional},
  Israel J. Math., 134 (2003), pp.~157--171.

\bibitem{MS}
{\sc D.~Massart and A.~Sorrentino}, {\em Differentiability of {M}ather's
  average action and integrability on closed surfaces}, Nonlinearity, 24
  (2011), pp.~1777--1793.

\bibitem{MatherCrit}
{\sc J.~Mather}, {\em A criterion for the nonexistence of invariant circles},
  Inst. Hautes \'{E}tudes Sci. Publ. Math.,  (1986), pp.~153--204.

\bibitem{Mather}
{\sc J.~N. Mather}, {\em Existence of quasiperiodic orbits for twist
  homeomorphisms of the annulus}, Topology, 21 (1982), pp.~457--467.

\bibitem{Mather43}
\leavevmode\vrule height 2pt depth -1.6pt width 23pt, {\em Nonexistence of
  invariant circles}, Ergodic Theory Dynam. Systems, 4 (1984), pp.~301--309.

\bibitem{MatherP}
\leavevmode\vrule height 2pt depth -1.6pt width 23pt, {\em Modulus of
  continuity for {P}eierls's barrier}, in Periodic solutions of {H}amiltonian
  systems and related topics ({I}l {C}iocco, 1986), vol.~209 of NATO Adv. Sci.
  Inst. Ser. C Math. Phys. Sci., Reidel, Dordrecht, 1987, pp.~177--202.

\bibitem{MatherMeasure}
\leavevmode\vrule height 2pt depth -1.6pt width 23pt, {\em Minimal measures},
  Comment. Math. Helv., 64 (1989), pp.~375--394.

\bibitem{MaDiff}
\leavevmode\vrule height 2pt depth -1.6pt width 23pt, {\em Differentiability of
  the minimal average action as a function of the rotation number}, Bol. Soc.
  Brasil. Mat. (N.S.), 21 (1990), pp.~59--70.

\bibitem{Mather1}
\leavevmode\vrule height 2pt depth -1.6pt width 23pt, {\em Action minimizing
  invariant measures for positive definite {L}agrangian systems}, Math. Z., 207
  (1991), pp.~169--207.

\bibitem{MJAMS}
\leavevmode\vrule height 2pt depth -1.6pt width 23pt, {\em Variational
  construction of orbits of twist diffeomorphisms}, J. Amer. Math. Soc., 4
  (1991), pp.~207--263.

\bibitem{MatherF}
\leavevmode\vrule height 2pt depth -1.6pt width 23pt, {\em Variational
  construction of connecting orbits}, Ann. Inst. Fourier (Grenoble), 43 (1993),
  pp.~1349--1386.

\bibitem{MaDif1}
\leavevmode\vrule height 2pt depth -1.6pt width 23pt, {\em Arnol' d diffusion.
  {I}. {A}nnouncement of results}, Sovrem. Mat. Fundam. Napravl., 2 (2003),
  pp.~116--130.

\bibitem{MaDif2}
\leavevmode\vrule height 2pt depth -1.6pt width 23pt, {\em Arnold diffusion by
  variational methods}, in Essays in mathematics and its applications,
  Springer, Heidelberg, 2012, pp.~271--285.

\bibitem{MaFo}
{\sc J.~N. Mather and G.~Forni}, {\em Action minimizing orbits in {H}amiltonian
  systems}, in Transition to chaos in classical and quantum mechanics
  ({M}ontecatini {T}erme, 1991), vol.~1589 of Lecture Notes in Math., Springer,
  Berlin, 1994, pp.~92--186.

\bibitem{MSTY}
{\sc H.~Mitake, A.~Siconolfi, H.~V. Tran, and N.~Yamada}, {\em A {L}agrangian
  approach to weakly coupled {H}amilton-{J}acobi systems}, SIAM J. Math. Anal.,
  48 (2016), pp.~821--846.

\bibitem{MTdisc}
{\sc H.~Mitake and H.~V. Tran}, {\em Selection problems for a discount
  degenerate viscous {H}amilton-{J}acobi equation}, Adv. Math., 306 (2017),
  pp.~684--703.

\bibitem{Monge}
{\sc G.~Monge}, {\em M{\'e}moire sur la th{\'e}orie des d{\'e}blais et des
  remblais}, De l'Imprimerie Royale, 1781.

\bibitem{MVZ}
{\sc A.~Monzner, N.~Vichery, and F.~Zapolsky}, {\em Partial quasimorphisms and
  quasistates on cotangent bundles, and symplectic homogenization}, J. Mod.
  Dyn., 6 (2012), pp.~205--249.

\bibitem{Moser}
{\sc J.~Moser}, {\em Monotone twist mappings and the calculus of variations},
  Ergodic Theory Dynam. Systems, 6 (1986), pp.~401--413.

\bibitem{NR2}
{\sc G.~Namah and J.-M. Roquejoffre}, {\em Comportement asymptotique des
  solutions d'une classe d'\'{e}quations paraboliques et de
  {H}amilton-{J}acobi}, C. R. Acad. Sci. Paris S\'{e}r. I Math., 324 (1997),
  pp.~1367--1370.

\bibitem{Pa}
{\sc P.~Pageault}, {\em Conley barriers and their applications:
  chain-recurrence and {L}yapunov functions}, Topology Appl., 156 (2009),
  pp.~2426--2442.

\bibitem{Ph}
{\sc R.~R. Phelps}, {\em Lectures on {C}hoquet's theorem}, vol.~1757 of Lecture
  Notes in Mathematics, Springer-Verlag, Berlin, second~ed., 2001.

\bibitem{PS}
{\sc M.~Pozza and A.~Siconolfi}, {\em Discounted {H}amilton-{J}acobi equations
  on networks and asymptotic analysis}, 2019.

\bibitem{Reich}
{\sc S.~Reich}, {\em Strong convergence theorems for resolvents of accretive
  operators in {B}anach spaces}, J. Math. Anal. Appl., 75 (1980), pp.~287--292.

\bibitem{Rock}
{\sc R.~T. Rockafellar}, {\em Convex analysis}, Princeton Landmarks in
  Mathematics, Princeton University Press, Princeton, NJ, 1997.
\newblock Reprint of the 1970 original, Princeton Paperbacks.

\bibitem{Roos}
{\sc V.~Roos}, {\em Variational and viscosity operators for the evolutionary
  {H}amilton-{J}acobi equation}, Commun. Contemp. Math., 21 (2019),
  pp.~1850018, 76.

\bibitem{SiZa}
{\sc A.~Siconolfi and S.~Zabad}, {\em Scalar reduction techniques for weakly
  coupled {H}amilton-{J}acobi systems}, NoDEA Nonlinear Differential Equations
  Appl., 25 (2018), pp.~Paper No. 50, 20.

\bibitem{So}
{\sc A.~Sorrentino}, {\em On the total disconnectedness of the quotient {A}ubry
  set}, Ergodic Theory Dynam. Systems, 28 (2008), pp.~267--290.

\bibitem{SHom}
{\sc A.~Sorrentino}, {\em On the homogenization of the hamilton-jacobi
  equation}, 2019.

\bibitem{sou}
{\sc P.~E. Souganidis}, {\em Approximation schemes for viscosity solutions of
  {H}amilton-{J}acobi equations}, J. Differential Equations, 59 (1985),
  pp.~1--43.

\bibitem{Suhr}
{\sc S.~Suhr}, {\em Aubry-{M}ather theory for {L}orentzian manifolds}, J. Fixed
  Point Theory Appl., 21 (2019), pp.~Paper No. 71, 42.

\bibitem{VitHom}
{\sc C.~Viterbo}, {\em Symplectic homogenization}, 2014.

\bibitem{wal}
{\sc J.~A. Walsh}, {\em The dynamics of circle homeomorphisms: A hands-on
  introduction}, Mathematics Magazine, 72 (1999), pp.~3--13.

\bibitem{WWY}
{\sc K.~Wang, L.~Wang, and J.~Yan}, {\em Implicit variational principle for
  contact {H}amiltonian systems}, Nonlinearity, 30 (2017), pp.~492--515.

\bibitem{WWY3}
\leavevmode\vrule height 2pt depth -1.6pt width 23pt, {\em Aubry-{M}ather
  theory for contact {H}amiltonian systems}, Comm. Math. Phys., 366 (2019),
  pp.~981--1023.

\bibitem{WWY2}
\leavevmode\vrule height 2pt depth -1.6pt width 23pt, {\em Variational
  principle for contact {H}amiltonian systems and its applications}, J. Math.
  Pures Appl. (9), 123 (2019), pp.~167--200.

\bibitem{WYZ}
{\sc Y.-N. Wang, J.~Yan, and J.~Zhang}, {\em Convergence of viscosity solutions
  of generalized contact {H}amilton-{J}acobi equations}, Arch. Ration. Mech.
  Anal., 241 (2021), pp.~885--902.

\bibitem{Wei}
{\sc Q.~Wei}, {\em Viscosity solution of the {H}amilton-{J}acobi equation by a
  limiting minimax method}, Nonlinearity, 27 (2014), pp.~17--41.

\bibitem{ZJMD}
{\sc M.~Zavidovique}, {\em Existence of {$C^{1,1}$} critical subsolutions in
  discrete weak {KAM} theory}, J. Mod. Dyn., 4 (2010), pp.~693--714.

\bibitem{Z}
\leavevmode\vrule height 2pt depth -1.6pt width 23pt, {\em Strict sub-solutions
  and {M}a\~{n}\'{e} potential in discrete weak {KAM} theory}, Comment. Math.
  Helv., 87 (2012), pp.~1--39.

\bibitem{Zfixed}
\leavevmode\vrule height 2pt depth -1.6pt width 23pt, {\em Fixed points of
  contractions approximating 1-{L}ipschitz maps}, Grad. J. Math., 4 (2019),
  pp.~56--61.

\bibitem{Ztwisted}
\leavevmode\vrule height 2pt depth -1.6pt width 23pt, {\em Twisted
  {L}ax-{O}leinik formulas and weakly coupled systems of {H}amilton-{J}acobi
  equations}, Ann. Fac. Sci. Toulouse Math. (6), 28 (2019), pp.~209--224.

\bibitem{Zdisc}
\leavevmode\vrule height 2pt depth -1.6pt width 23pt, {\em Convergence of
  solutions for some degenerate discounted {H}amilton-{J}acobi equations},
  Anal. PDE, 15 (2022), pp.~1287--1311.

\bibitem{Zhang}
{\sc J.~Zhang}, {\em Global behaviors of weak {KAM} solutions for exact
  symplectic twist maps}, J. Differential Equations, 269 (2020),
  pp.~5730--5753.

\bibitem{ZC}
{\sc K.~Zhao and W.~Cheng}, {\em On the vanishing contact structure for
  viscosity solutions of contact type {H}amilton-{J}acobi equations {I}:
  {C}auchy problem}, Discrete Contin. Dyn. Syst., 39 (2019), pp.~4345--4358.

\bibitem{Ziliotto}
{\sc B.~Ziliotto}, {\em Convergence of the solutions of the discounted
  {H}amilton-{J}acobi equation: {A} counterexample}, J. Math. Pures Appl. (9),
  128 (2019), pp.~330--338.

\end{thebibliography}
\bibliographystyle{siam}
\addcontentsline{toc}{chapter}{\protect\numberline{}Bibliography}

\end{document}